\documentclass[a4paper,10pt]{extarticle} 
\usepackage[utf8]{inputenc}
\usepackage [english]{babel}
\usepackage{a4wide}
\usepackage{tocloft}
\usepackage{amssymb}
\usepackage{amsmath}
\usepackage{amsthm}
\usepackage{amsbsy}
\usepackage{mathrsfs}
\usepackage{amsfonts}
\usepackage{mathtools}
\usepackage{comment}
\usepackage{tikz}
\usepackage{tikz-cd}
\usetikzlibrary{cd}
\usepackage{amscd}
\usepackage{hyperref}
\usepackage{xcolor}
\hypersetup{
    colorlinks,
    linkcolor={red!50!black},
    citecolor={blue!50!black},
    urlcolor={blue!80!black}
}
\usepackage[normalem]{ulem}
\usepackage{cleveref}   
\usepackage[all]{xy}
\usepackage{xcolor}
\theoremstyle{plain}
\usepackage[margin=1 in, a4paper]{geometry}
\usepackage{relsize}
\usepackage{lipsum}
 \usepackage[autostyle]{csquotes}
\usepackage{relsize}
\usepackage{enumitem}  
\usepackage{titlesec}
\usepackage{setspace}
\allowdisplaybreaks
\titleformat{\section}
{\centering\large\sc}{\thesection.}{1em}{}
\titleformat{\subsection}
{\centering\normalfont\bfseries}{\thesubsection.}{1em}{}
\titleformat{\subsubsection}[runin]
{\normalfont\normalsize\bfseries}{\thesubsubsection.}{1em}{}

\newcommand{\gl}{\operatorname{GL}}

\newcommand{\ind}{\operatorname{ind}}

\newcommand{\N}{\mathbb{N}}
\newcommand{\Z}{\mathbb{Z}}
\newcommand{\Q}{\mathbb{Q}}
\newcommand{\f}{\mathbb{F}_{p}}
\newcommand{\fstar}{\mathbb{F}_{p}^{\ast}}
\newcommand{\Sp}{\Sigma_{p}}

\newtheorem{theorem}{Theorem}[section]
\newtheorem*{theorem*}{Theorem}
\newtheorem{corollary}[theorem]{Corollary}
\newtheorem{lemma}[theorem]{Lemma}
\newtheorem*{lemma*}{Lemma}
\newtheorem{proposition}[theorem]{Proposition}

\theoremstyle{definition}
\newtheorem{remark}[theorem]{Remark}

\numberwithin{equation}{section}		

\newcommand\keywords[1]{%
	\begingroup
	\let\and \\
	\par
	\noindent\textbf{Keywords } #1\par
	\endgroup
}

\newcommand\subjclass[2]{%
	\begingroup
	\let\and \\
	\par
	\noindent\textbf{Mathematics Subject Classification (2010)} #1\par
	\endgroup
}

\begin{document}
	
\title{\textbf{\large{The Monomial Lattice  in  Modular Symmetric Power 
Representations}}}

\author{Eknath Ghate and Ravitheja Vangala}



\date{}
 
 \maketitle    
 
\begin{abstract}
Let $p$ be a prime. We study the structure of and the inclusion relations 
among the terms in the  monomial lattice in the modular symmetric power representations 
of $\mathrm{GL}_2(\mathbb{F}_p)$. We also determine the structure of certain related quotients of the 
symmetric power representations 
which arise when studying the reductions of local Galois representations of slope at most $p$. In particular, 
we show that
these quotients are periodic and depend only on the congruence class modulo $p(p-1)$. Many of our results 
are stated in terms
of the sizes of various sums of digits in base $p$-expansions and in terms of the vanishing or non-vanishing
of certain binomial coefficients modulo $p$.
\end{abstract}

        

\section{Introduction}

Let $V_r$, for $r \geq 0$, be the $r^{\mathrm{th}}$-symmetric power 
representation over the field $k$, of the general linear group 
$\Gamma = \mathrm{GL}_2(k)$. This paper studies the monomial lattice 
in $V_r$ when $k = \mathbb{F}_p$ is the finite field with $p$ elements, 
for $p$ a prime number. It also studies related quotients of $V_r$ which 
arise in number theoretic problems involving Galois representations.

For a general field $k$,  the  symmetric power representations $V_r$ have 
models over $k$ consisting of homogeneous polynomials $F(X,Y)$ in two 
variables of degree $r$ defined over $k$, with action given by:
\begin{align}\label{G action}
	\begin{pmatrix} a & b \\ c & d	\end{pmatrix} \cdot F(X,Y) = 
	F(aX+cY,bX+dY), \quad  \forall \ 
	\begin{pmatrix} a & b \\ c & d \end{pmatrix} \in \Gamma = \gl_{2}(k).
\end{align}
%
When $k = \mathbb{C}$ is the field of complex numbers,
 the representations $V_r$  are  irreducible.
But the $V_r$ are not 
always irreducible if $k$ has characteristic $p$.
In this introduction we consider two special cases. We assume that 
$k  = {\mathbb{F}}_p$, so that $\Gamma =  \mathrm{GL}_2(\mathbb{F}_p)$ is a 
finite group (of Lie type), or  $k = \bar{\mathbb F}_p$,  the algebraic closure 
of $\mathbb{F}_p$, so that $\Gamma =  \mathrm{GL}_2(\bar{\mathbb{F}}_p)$ 
is an algebraic group (more precisely, the $\bar{\mathbb{F}}_p$-valued points 
of the algebraic group $\mathrm{GL}_2$). In both these cases it is well known
that the representations $V_r$ of $\Gamma$ are irreducible if and only if 
$r \leq p-1$. Thus, it is natural to ask what the structure of various 
$\Gamma$-submodules of these representations are.
A particularly important class of $\Gamma$-submodules are
those generated by the monomials $X^{r-i}Y^i \in V_r$, for $0 \leq i \leq r$. 
Set
\[
   X_{r-i} := \langle X^{r-i}Y^i \rangle \subset V_r,
\]
for $0 \leq i \leq r$, which we also denote by $X_{r-i,\,r}$ when we wish to 
specify the ambient space $V_r$. 
We refer to the lattice of  submodules in $V_r$ defined by the
monomial submodules $X_{r-i}$, with partial order determined by
inclusion, as the {\it monomial lattice}.

It is not difficult to see (Lemma~\ref{first row filtration}) 
that the monomial lattice starts off as an
increasing filtration
\begin{eqnarray}
  \label{first row}
  X_r \subset X_{r-1} \subset \cdots \subset X_{r-(p-1)} 
\end{eqnarray}
and so by the Weyl involution 
$w = \left( \begin{smallmatrix} 0 & 1 \\ 1 & 0 \end{smallmatrix} \right)$,
which flips $X$ and $Y$, ends as a decreasing filtration
\begin{eqnarray*}
 X_{p-1} \supset X_{p-2} \supset \cdots \supset X_1 \supset X_0.
\end{eqnarray*}
However, the other inclusion relations between these two extremes 
are not well understood. In fact, the very next monomial submodule 
$X_{r-p}$ behaves erratically with respect to the filtration \eqref{first row}, 
in the sense that there are infinite families of $r$ 
(e.g., $r = p^m + p$, for $m \geq 2$) such that $X_{r-p}$ does not 
contain $X_{r-(p-1)}$. Similarly,  there are other infinite 
families of $r$ (e.g., $r = p^m+(p-1)$, for $m \geq 2$) such that $X_{r-p}$ 
is not contained in $X_{r-(p-1)}$.

When we are in the algebraic group case, so that $\Gamma =
{\mathrm{GL}}_2(k)$ with $k = \bar{\mathbb F}_p$, S. Doty has given 
an elegant description of all inclusion relations among the monomial submodules 
in terms of {\it carry patterns}, even for the general linear groups 
$\mathrm{GL}_n(k)$, for $n \geq 2$, of higher rank \cite{Doty} (see also \cite{DW96}). 
We describe his result in the present setting ($n= 2$). Let $0 \leq i \leq r$. 
Each of the three numbers involved in the identity 
\begin{eqnarray}
  \label{carry}
  r = i + (r-i),
\end{eqnarray} 
has a base $p$-expansion.
However (as every schoolchild will attest to when $p$ is the non-prime
$10$), the sum of the individual $p$-adic digits of $i$ and $r-i$ 
do not necessarily coincide with the corresponding $p$-adic digits of $r$.
The discrepancy is measured by a sequence of $0$s and $1$s, called the
carry pattern (of $i$).
Doty proved that, for $0 \leq i,j \leq r$, there is an inclusion $X_{r-j}
\subset X_{r-i}$ of monomial submodules if and only if the 
carry pattern of $j$ is less than the carry pattern of $i$ with respect 
to the lexicographic ordering on the set of
all sequences of $0$s and $1$s, thereby setting up an isomorphism of posets
between the monomial lattice and the poset of carry patterns with their
lexicographic ordering. Moreover, Doty showed that any $\Gamma$-submodule
of the symmetric power
representation $V_r$ is a sum of monomial submodules essentially
reducing the study of arbitrary submodules of $V_r$ to the monomial
submodules of $V_r$.
A crucial role in his arguments is played by the fact that $k =
\bar{\mathbb F}_p$ is an infinite field.\footnote{In fact, Doty's results
 also hold if $k = {\mathbb F}_q$ is a finite field of cardinality $q > r$. 
However, in the applications we have in mind, one wishes to treat all $r$ 
at the same time for the fixed finite field $k = \mathbb{F}_p$, so this 
more refined version of his result becomes of limited use.}

In this paper, we study the monomial lattice when $k = \mathbb{F}_p$, 
which we assume to be the case from now on. 
We will focus on the first $p$ monomial submodules in the filtration 
\eqref{first row} where there are already many new phenomena to be discovered. 
It seems to be generally acknowledged 
(see, e.g., the book of Humphreys \cite[\S 19.2]{Humph}) 
that the study of modular representations of algebraic groups is in general 
easier than the study of modular representations of finite groups of Lie type. 
This in certainly borne out in the study of the monomial lattice 
as one moves from $k = \bar{\mathbb F}_p$ to $k = {\mathbb F}_p$. 
For instance, while the  fact that the carry pattern of $j$ should be 
less than the carry pattern of $i$ is clearly still a necessary condition 
for the inclusion $X_{r-j } \subset X_{r-i}$, it is no longer a sufficient 
condition, not even for the first two monomial submodules in 
\eqref{first row} above! Indeed, if  the constant term $r_0$ in the base 
$p$-expansion of $r$ is non-zero, then the carry patterns
with respect to the equation \eqref{carry} for $i = 0$ and $i = 1$ 
are the same (a string of $0$s) even though it is easily checked that 
the inclusion $X_r \subsetneq X_{r-1}$ 
is strict, for $r \geq p$ \cite[Lemma 4.1]{BG15}.

For an integer $n \geq 0$, let  $\Sigma_p(n)$ denote the sum of the 
$p$-adic digits in the base $p$-expansion of $n$. The first result 
of this paper that we state here gives necessary and sufficient conditions
for all inclusion relations among the first $p$ monomial submodules in 
the filtration \eqref{first row} in terms of 
additional data involving such sums. 

\begin{theorem}
   Let $p \geq 2$, $0 \leq j < i \leq p-1$, $r \geq p$ and $r_0$ be 
   the constant term in the base $p$-expansion of $r$. 
   Then the monomial submodules $X_{r-j} = X_{r-i}$ are equal if and only if
   $i$ and $j$ have the same carry patterns, $\Sigma_p(r-j) \leq p-1$
   and $\Sigma_p(r-r_0) \leq j$.
\end{theorem}

\noindent 
The theorem is proved in several lemmas leading up to 
Lemma~\ref{final X_r-i = X_r-j}, 
noting that the cases $p = 2$ and $j = 0$ are trivially true. 
It may be viewed as a refinement of Doty's first result recalled
above when one moves from $k = \bar{\mathbb F}_p$ to $k = \mathbb{F}_p$, 
at least for the first $p$ monomial submodules
\eqref{first row} in the monomial lattice.


We now remark that Doty's second result mentioned above also fails when 
one moves from $k = \bar{\mathbb F}_p$ to $k = {\mathbb F}_p$.  
Let $$\theta(X,Y) = X^pY-XY^p$$ be the Dickson polynomial of degree $p+1$, 
and let $V_r^{(1)} \subset V_r$ denote the $\Gamma$-submodule of $V_r$ 
consisting of all polynomials $F(X,Y)$ divisible by $\theta(X,Y)$ studied by
Glover \cite{Glover}, in what was perhaps the first thorough
study of the symmetric power representations $V_r$ when 
$\Gamma = \mathrm{GL}_2({\mathbb F}_p)$.  
Since $\Gamma$ acts on $\theta(X,Y)$ via the determinant character 
(this really uses the fact that one is working over ${\mathbb F}_p$), 
the submodule $V_r^{(1)}$ is indeed stable 
by $\Gamma$. 
But the sum of the coefficients of a polynomial in $V_r^{(1)}$ vanishes, 
hence $V_r^{(1)}$ does not contain any monomials 
and, in particular, $V_r^{(1)}$  cannot be spanned by monomial submodules.

Nonetheless, there are many interesting subquotients of $V_r$ involving
monomial submodules, so it is important to describe the structure of these
submodules when $k = \mathbb{F}_p$. The submodule $X_r$ generated by the 
highest monomial $X^r$ was essentially described by Glover himself \cite{Glover} 
and the structure of the next submodule $X_{r-1}$ generated by the second 
highest monomial $X^{r-1}Y$ was described completely in \cite{BG15}, for $p \geq 3$.
There does not seem to be any general literature beyond these results on
the other submodules in the monomial lattice when 
$\Gamma =\mathrm{GL}_2(\mathbb{F}_p)$.

The first main goal of this paper is to extend the aforementioned results
to give the structure of the first $p$ monomial submodules $X_{r-i}$ in 
\eqref{first row} in the monomial lattice. To this end, we note that there
is a surjection of $\Gamma$-modules
$\phi_i : X_{r-i,\,r-i} \otimes V_i \twoheadrightarrow X_{r-i,\,r}$ given by 
multiplication (Lemma~\ref{surjection1}), where $X_{r-i,\,r-i}$ is the submodule  
generated by the highest monomial $X^{r-i}$ in $V_{r-i}$.
We prove that generically this map is an isomorphism. More precisely, we have

\begin{theorem}
  Let $p \geq 3$, $0 \leq i \leq p-1$, $r \geq (i+1)(p+1)$ and $r_0$ be the 
  constant term in the base $p$-expansion of $r$. If $\Sigma_p(r-r_0) \geq p$,
  then
  \begin{eqnarray*}
    X_{r-i,\,r} \cong X_{r-i,\,r-i} \otimes V_i.
  \end{eqnarray*}
If $\Sigma_p(r - r_0) < p$, the structure of $X_{r-i,\,r}$ can also be described 
in terms of explicit extensions of tensor products of irreducible and principal
series representations.
\end{theorem}

\noindent Here the principal series
representations are by definition $\mathrm{ind}_B^\Gamma (a^{m}d^n)$, 
where $B$ is the Borel subgroup of upper triangular matrices
$\{ \left( \begin{smallmatrix} a & b \\ 0 & d \end{smallmatrix} \right) 
: a,d \in \f^\ast, b \in \f  \} \subset \Gamma$ and $m$, $n$ are integers. 
If $r_0 \geq i$, then $\Sigma_p(r-i) = \Sigma_p(r-r_0) + (r_0 - i)$ and 
the theorem follows from Theorem~\ref{Structure r_0 >i}.
If $r_0 < i$, it follows from Theorem~\ref{Structure r_0<i}. 
We refer the reader to these theorems for more details.
The proofs of these theorems 
reduce the study of the structure of $X_{r-i}$ to that of $X_{r-(i-1)}$, 
by noting that the quotients $X_{r-i}/X_{r-(i-1)}$ are homomorphic images of 
the principal series representation $\mathrm{ind}_B^\Gamma (a^{r-i}d^i)$, 
for $1 \leq i \leq p-1$ (Corollary~\ref{induced and successive}).
In any case, these theorems allow us to state the following explicit
dimension formulas for the first $p$ monomial submodules $X_{r-i}$ in 
\eqref{first row} (see Corollaries~\ref{dimension r_0 > i} and 
\ref{dimension r_0 < i}).

\begin{corollary}
  Let $ p \geq 3$, $0 \leq i \leq p-1$, $r \geq (i+1)(p+1)$ and $r_0$ be 
  the constant term in the base $p$-expansion of $r$. 

  \noindent If $r_{0} \geq i$, then
  \begin{align*}
     \dim X_{r-i} = 
     \begin{cases}
        (r_{0}+1)(\Sp(r-r_{0})+1) , & \mathrm{if} ~ \Sp(r-i) \leq r_{0},  \\
        (i+1)\left( \Sp(r-i) +1 \right), &\mathrm{if} ~ 
        r_{0} \leq \Sp(r-i) \leq p,\\
        (i+1)(p+1), & \mathrm{if} ~ \Sp(r-i) \geq p.
     \end{cases}     
  \end{align*}
   
   \noindent If $r_{0}<i$, then 
   \begin{align*}
     \dim X_{r-i} =
     \begin{cases}
       (p+r_{0}+1) \Sp(r-r_{0}),
        &~\mathrm{if}~ \Sp(r-i) < p, \\
       (i - r_{0})(p+1)+ (\Sp(r-r_{0})+1)(r_{0}+1), &
        ~ \mathrm{if} ~  \Sp(r-i) \geq p, ~ \Sp(r-r_{0}) \leq p, \\
        (i+1)(p+1), & ~ \mathrm{if} ~ \Sp(r-i) \geq p, ~ \Sp(r-r_{0}) \geq p. \\
     \end{cases}
   \end{align*}
\end{corollary}

\noindent We remark that the first dimension formula is continuous at both the 
boundaries, whereas the second formula is not continuous at the first boundary. 


Although this was not part of the initial goal of this paper, we decided to 
also include in the final section, as an example of things to come, the structure 
of the next submodule $X_{r-p}$ in the monomial lattice, generated by $X^{r-p}Y^p$.
The module $X_{r-p}$ behaves more erratically with respect to the filtration 
\eqref{first row} when $k = \mathbb{F}_p$ than when $k = \bar{\mathbb F}_p$, 
since there are infinite families of $r$ (e.g., $r= p^m+p-1$, for $m \geq 2$) 
such that $X_{r-p}$ {\it simultaneously} neither contains nor is contained in 
$X_{r-(p-1)}$. In spite of this, we show that $X_{r-p}$ has a relatively 
simple structure when $k = {\mathbb F}_p$, namely, it is isomorphic to the 
monomial submodule $X_{s-1,\,s} \subset V_s$ generated by the second highest 
monomial, for some $s$, which is often though not always equal to $r$.
More precisely, we establish the following trichotomy in 
Propositions~\ref{p divides r}, \ref{BG Proposition 3.3}, 
\ref{remaining case r = 1 mod p-1}, \ref{BG Lemma 4.5}, 
\ref{4 JH factors} and \ref{BG Lemma 4.8}.
\begin{theorem}
  Let $p \geq 3$ and $r \geq 2p+1$. Then  
  \begin{eqnarray*}
     X_{r-p} \simeq 
     \begin{cases}
        X_{r/p-1, \,r/p},  & \text{if } \Sigma_p(r-p) < \Sigma_p(r-1), \\
        X_{r-1, \,r},      & \text{if } \Sigma_p(r-p) = \Sigma_p(r-1), \\
        X_{rp-1, \,rp},    & \text{if } \Sigma_p(r-p) > \Sigma_p(r-1). \\
      \end{cases}
  \end{eqnarray*}
\end{theorem}
\noindent Since the structure of the submodule $X_{s-1, \,s}$ was determined 
in \cite[\S 2, \S 3]{BG15}, the theorem also gives the structure of the submodule 
$X_{r-p}$, for $p \geq 3$ (see Theorem~\ref{Main theorem part 1 and 2} and 
Theorem~\ref{Main theorem part 4} for the precise structure).\footnote{
Though we caution the reader that there 
are infinite families of $r$ (e.g., $r=p^m+p+1$, for $m \geq 2$) 
for which one is in the middle case of the theorem, so that 
$X_{r-1} \cong X_{r-p}$, 
yet neither $X_{r-p}$ contains nor is contained in $X_{r-1}$.}  
We remark that in the first case we have $p \mid r$, by 
Lemma~\ref{Equality of sum of p-adic digits of (r-1),(r-p)}. As a corollary,
we obtain:

\begin{corollary}
	Let $p\geq 3$ and $r \geq p(2p+1)$.   
    Write $r=p^{n}u$,  
	with $p \nmid u$. Set $\delta =1$ if $n \geq 2$ or 
	$\Sigma_{p}(r-p) > \Sigma_{p}(r-1)$ and $0$ otherwise. 
	Then
	\begin{align*}
       \mathrm{dim} \ X_{r-p} =
        \begin{cases}	                              
    	    2 \Sigma_p(r) + \delta(p+2-\Sigma_p(r)), 
    	    \ & \text{if} \ \Sigma_p(r) \leq p, \\
    	    2p+2,  &   \text{if} \ \Sigma_p(r) > p. 
	    \end{cases}
	\end{align*}
\end{corollary}
%
\noindent	The proof follows from the theorem  and \cite[Corollary 1.6]{BG15} 
(or \Cref{dimension formula for X_{r-1}}), by checking various cases. 

The second main goal of this paper is to give a complete description of
certain quotients of the symmetric power representations $V_r$ involving 
the monomial submodules $X_{r-i}$ in \eqref{first row}. 
Let $V_r^{(m)}$, for $m \geq 0$, be the $\Gamma$-submodule of $V_r$
consisting of all polynomials $F(X,Y)$ divisible by $\theta(X,Y)^m$.  
The submodules $V_r^{(m)}$ cut out an exhaustive decreasing filtration of $V_r$:
\begin{eqnarray}
\label{theta filtration}
  V_r \supset  V_r^{(1)}  \supset V_r^{(2)} \supset
  \cdots \supset V_r^{(m)} \supset \cdots \supset 0.
\end{eqnarray}
This filtration is much better understood than the monomial filtration 
\eqref{first row}. For one, the structure of each of the individual terms in 
this filtration is easy to write down. By what we have said above, we have 
$V_r^{(m)} \cong V_{r-m(p+1)} \otimes D^m$, 
for $D : \Gamma \rightarrow {\mathbb F}_p^\ast$ the determinant character.
Moreover, for all (but possibly the last) non-zero subquotients in the filtration
\eqref{theta filtration}, we have
\begin{eqnarray*}
  V_r^{(m)} / V_r^{(m+1)} \cong \mathrm{ind}_B^\Gamma (a^md^{r-m})
\end{eqnarray*} 
are  principal series representations (Lemma~\ref{induced and star}) depending 
only on the congruence classes of $r$ modulo $(p-1)$ and $m$ modulo $(p-1)$, 
so extensions of length two, which split if and only if $r \equiv 2m \mod (p-1)$ 
(Lemma~\ref{Breuil map}).
%
Now, consider the  quotients $Q(i)$ and $P(i)$ of $V_r$ involving the monomial 
submodules $X_{r-i}$, defined by
\begin{eqnarray*}
  Q(i) & := & \dfrac{V_r}{X_{r-i} + V_r^{(i+1)}},  \\
\end{eqnarray*}
for $i \geq 0$, and 
\begin{eqnarray*}
  P(i) & := & \dfrac{V_r}{X_{r-(i-1)} + V_r^{(i+1)}},  \\
\end{eqnarray*}
for $i > 0$.
%
These quotients are of number theoretic importance. 
They arise when one is trying to compute the reductions of crystalline 
two-dimensional representations of the Galois group of ${\mathbb Q}_p$ of 
Hodge-Tate weights $(0, r+1)$ and positive slope, using the mod $p$ Local 
Langlands Correspondence. The quotient $P(i)$ arises for integral slopes 
$i > 0$ and $Q(i)$ for fractional slopes in the interval $(i, i+1)$ with 
$i \geq 0$, by \cite[Remark 4.4]{BG09}. Thus, it is important to know the 
structure of the quotients $P(i)$ and $Q(i)$ as $\Gamma$-modules. 
In this paper, we will restrict to discussing the structure of these quotients, 
deferring the number theoretic applications to future works, except for an 
example at the end of this introduction (Corollary~\ref{cor reduction}). 
Moreover, since $P(i)$ is closely related to $Q(i-1)$, for $i > 0$, by the 
discussion around the exact sequence \eqref{Q i-1 and P i} below, 
in this paper we will restrict to essentially discussing the structure of the 
quotients $Q(i)$, for $i \geq 0$.

The quotient $Q(0)$ is irreducible \cite{BG09}, whereas the
quotient $Q(1)$ has at most three Jordan-H\"older (JH) factors \cite{BG15}.
In this paper, we prove several results which in principle allow us to deduce 
the JH factors of the quotients $Q(i)$, for all $1 \leq i \leq p-1$. We summarize 
these results now. For an integer $n$, let $[n] \in  \{1, 2, \ldots, p-1\}$ be 
such that $n \equiv [n] \mod (p-1)$.  Let $r \equiv a \mod (p-1)$, for 
$a \in \{1, 2, \ldots, p-1\}$, so that $a = [r]$. 
We will also need to consider certain intervals of the congruence
classes $\{0,1, \ldots, p-1 \}$ mod $p$.

We first treat the case when $i$ is neither $a$ nor $p-1$. 
\begin{theorem}\label{i not a nor p - 1}
Let $p \geq 3$, $1 \leq i \leq p-1$, $r \geq i(p+1)+p$, $r \equiv a \mod (p-1)$
with $a \in \{1, 2, \ldots, p-1\}$ and $r_0$ be the
constant term in the base $p$-expansion of $r$.
If $i \neq a$, $p-1$ and $j = \min \{i, [a-i] \}$, then
there is an exact sequence of $\Gamma$-modules
\begin{eqnarray}
   \label{intro i vs j-1}
   0 \rightarrow W \rightarrow Q(i) \rightarrow Q(j-1) \rightarrow 0,
\end{eqnarray}
where $W$ is an explicit subquotient of $V_r^{(j)}/V_r^{(i+1)}$, completely
determined by the relationship between $r_0$ and certain explicit interval(s) 
of the congruence classes mod $p$ which depend only on $a$ and $i$
(and, in some cases, a comparison between $[a-r_0]$  and $r_0$).
\end{theorem}

The explicit intervals depend on the vanishing (or non-vanishing) of certain binomial coefficients modulo $p$. 
More precisely, if $i$ is neither $a$ nor $p-1$, then we have the following  explicit version of Theorem~\ref{i not a nor p - 1} 
which deal with the cases (i) $i < [a-i]$, (ii) $i = [a-i]$ and (iii) $i > [a-i]$ separately.

\begin{theorem}
        \label{intro i < [a-i]}
	Let  $p \geq 3$, $r \equiv a \mod (p-1)$ with $1 \leq a \leq p-1$,
	$r \equiv r_{0}~\mathrm{mod}~ p$ with $0 \leq r_{0} \leq p-1$ and let 
	$1 \leq i < [a-i] < p-1$. If 
	$r \geq i(p+1)+p $, then we have the following exact sequence
	\begin{align*}
	0 \rightarrow W \rightarrow Q(i) \rightarrow Q(i-1) \rightarrow 0,
	\end{align*}
	where $W= V_{p-1-[a-2i]} \otimes D^{a-i}$ if 
	$\binom{r-[a-i]-1}{i} \not \equiv 0 \mod p$ and 
	$W= V_{r}^{(i)}/V_{r}^{(i+1)}$ otherwise.
\end{theorem}

\begin{theorem}
        \label{intro i = [a-i]}
	Let $p \geq 3$, $r \equiv a ~\mathrm{mod}~(p-1)$ with $1 \leq a  \leq p-1$,
	$r \equiv r_{0} ~\mathrm{mod}~ p$ with $0 \leq r_{0}  \leq p-1$ and 
	let $1 \leq i < p-1$ with $i=[a-i]$. If $ r \geq i(p+1)+p$, then
	\[
	0 \rightarrow W \rightarrow Q(i) \rightarrow Q(i-1) \rightarrow 0,
	\]
	where
	\begin{align*}
	W  \cong
	\begin{cases}
	0, 
	&\mathrm{if}~ \binom{r-i+1}{i+1} \not \equiv 0 \mod p ,  \\
	V_{p-1} \otimes D^{i},  & \mathrm{if}~ r_{0}=i, \\
	V_{0} \otimes D^{i},  & \mathrm{if}~r_{0}=i-1, \\
	V_{r}^{(i)}/V_{r}^{(i+1)},  &\mathrm{otherwise}.
	\end{cases}
	\end{align*}       
\end{theorem}
\begin{theorem}
        \label{intro i > [a-i]}
	Let $p \geq 3$, $r \equiv a ~ \mathrm{mod}~(p-1)$ with $1 \leq a \leq p-1$ and
	let $r \equiv r_{0}~ \mathrm{mod}~p$ with $0 \leq r_{0} \leq p-1$.
	Let  $1 \leq [a-i]<i < p-1$ and $i(p+1)+p \leq r$. 
	Then we have an exact sequence of $\Gamma$-modules
	\begin{align*}
	0 \rightarrow W \rightarrow Q(i) \rightarrow Q([a-i]-1) \rightarrow 0,
	\end{align*}
	where
	\begin{enumerate}[label= \emph{(\roman*)}]
		\item If $ \binom{r-[a-i]+1}{i+1} \not \equiv 0 \mod p$,
		then $W= 0$.
		\item  If $\binom{r-[a-i]+1}{i+1} \equiv 0 \mod p$ and 
		$\binom{r-i-1}{[a-i]} \not \equiv 0 \mod p$, then we have
		\begin{enumerate}            
			\item[\em{(a)}]   If $[a-r_{0}] < r_{0}+1$, then
			$
			0 \rightarrow V_{r}^{([a-r_{0}]+1)}/V_{r}^{(i+1)} \rightarrow W
			\rightarrow V_{[a-2r_{0}]} \otimes D^{r_{0}} \rightarrow 0.
			$         
			\item[\em{(b)}]  If    $[a-r_{0}] = r_{0} +1 $, then
			$W= V_{r}^{([a-r_{0}])}/ V_{r}^{(i+1)}$.
			\item[\em{(c)}] If  $  [a-r_{0}] >  r_{0}+1 $,
			then    
			$
			0 \rightarrow V_{r}^{([a-r_{0}])}/V_{r}^{(i+1)}  \rightarrow W
			\rightarrow V_{p-1-[2r_{0}+2-a]} \otimes D^{r_{0}+1} \rightarrow 0.
			$           
		\end{enumerate}                                                 
		\item  If $\binom{r-i-1}{[a-i]}  \equiv 0 \mod p$ and
		$r \not \equiv [a-i]+i ~\mathrm{mod}~p$,
		then $W= V_{r}^{([a-i])}/V_{r}^{(i+1)}$.
	\end{enumerate}
\end{theorem}

On the other hand, when $i = a$ or $i = p-1$ we have the following theorem.

\begin{theorem}\label{i = a or p - 1}
  Let $p \geq 3$, $r \equiv a \mod (p-1)$ with $a \in \{1, 2, \ldots, p-1\}$. Let 
  $i = a$ or $p-1$. If $r \geq i(p+1)+p$, then there is an exact sequence of 
  $\Gamma$-submodules 
  \begin{eqnarray*}
     0 \rightarrow W' \rightarrow P(i) \rightarrow Q(i) \rightarrow 0,
  \end{eqnarray*} 
where $W'$ is explicitly determined by whether $r_0$ lies in an explicit 
interval of the congruence classes mod $p$ which depends only on $a$.
\end{theorem}

\noindent Explicit versions of this theorem can be found in Theorem~\ref{Structure of Q(i) if i = a} for $i = a$ and
Theorem~\ref{Structure of Q(i) if i = p - 1} for $i = p-1$. 

The proof of Theorem~\ref{i not a nor p - 1} uses several auxiliary results 
which may be of independent interest. 
The key point is to determine $W$ explicitly in the exact sequence \eqref{intro i vs j-1}. To this end, we note that there is
an exact sequence 
(this is the first column in \eqref{commutative diagram} below)
\begin{equation}
\label{W exact sequence intro}
  0  \rightarrow X_{r-i}^{(j)}/X_{r-i}^{(i+1)} \rightarrow 
  V_{r}^{(j)}/V_{r}^{(i+1)} \rightarrow W \rightarrow 0,
 \end{equation}
for $0 \leq j \leq i \leq p-1$, so the proof of Theorem~\ref{i not a nor p - 1}
reduces to determining the quotients 
\begin{equation}\label{successive quotients intro}
 X_{r-i}^{(j)}/X_{r-i}^{(j+1)},
\end{equation}
for all $0 \leq i, j \leq p-1$. A novel argument involving principal series 
representations coming from the $R$-valued points of $\mathrm{GL}_2$, where 
$R$ is the ring of dual numbers $\f[\epsilon]$ with $\epsilon^2 = 0$ and higher 
generalizations of this ring, allows us to show that $Q(i)$ and all the terms in 
\eqref{W exact sequence intro} and \eqref{successive quotients intro} 
are periodic in $r$ modulo $p(p-1)$ (\Cref{arbitrary quotient periodic}). 
Determining the quotients \eqref{successive quotients intro} in the case 
$j = 0$ is easy (\Cref{Structure X(1)}). One may reduce the case of a given 
$j \geq 1$ to three  special values of $i$, namely $i = j$, $[a-j]$ and $0$ 
(\Cref{reduction}). The subcase $i = 0$ is treated in 
Proposition~\ref{singular quotient X_{r}}. The subcases
$j = i$ and $[a-i]$ are treated in Propositions~\ref{singular quotient i < [a-i]},
\ref{singular i= [a-i]} and \ref{singular i>r-i}, by dividing our discussion
into the three cases mentioned above, namely (i) $i < [a-i]$, (ii) $i = [a-i]$ and (iii) $i > [a-i]$. 
The answers are determined by explicit intervals of the congruence classes modulo $p$.
This determines $W$ explicitly in terms of these intervals. 
By the exact sequence \eqref{intro i vs j-1}, we obtain 
the structure of $Q(i)$ in terms of $W$  and $Q(j-1)$ as in Theorems~\ref{Structure of Q(i) if i<[a-i]}, \ref{Structure of Q i=[a-i]}
and \ref{Structure of Q(i) i>[a-i]}. This proves Theorem~\ref{i not a nor p - 1}. The more explicit versions,  
Theorems~\ref{intro i < [a-i]}, \ref{intro i = [a-i]} and \ref{intro i > [a-i]},
follow immediately using a criterion for membership (or lack thereof) in these intervals in terms of the vanishing
(or non-vanishing) of certain binomial coefficients mod $p$ (cf. Lemma~\ref{interval and binomial}). 
The proof of Theorem~\ref{i = a or p - 1} is simpler and uses similar ideas.
The module $W'$ is determined in Propositions~\ref{singular i=a} and 
\ref{singular i=p-1}, proving the theorem (see 
Theorems~\ref{Structure of Q(i) if i = a}, \ref{Structure of Q(i) if i = p - 1}).

We now illustrate how (the explicit versions) of Theorem~\ref{i not a nor p - 1}
and \ref{i = a or p - 1} above can  
in principle be used to recursively 
determine all the JH factors of $Q(i)$, for all $1 \leq i \leq p-1$, reducing the computation to $Q(0)$ or $Q(1)$.
This also allows us to introduce the explicit intervals mentioned above.  
We first treat the case $i \neq a$, $p-1$. 
We divide our discussion according to the three cases (i), (ii), (iii) mentioned above. 

(i) Assume that $a$ is strictly larger than $2i$, that is, $i < a-i = [a-i]$, so $i < a$. 
By Theorem~\ref{i not a nor p - 1}, we see that
$Q(i)$ is determined in terms of $W$ and $Q(i-1)$.
In this case  
the interval of residue classes modulo $p$ mentioned in the statement of Theorem~\ref{i not a nor p - 1} 
is $${\mathcal I}(a,i) = \{a-i+1, a-i+2, \ldots, a-1,a \},$$
and $W$ is all of (respectively, the cosocle of) $V_r^{(i)}/V_r^{(i+1)}$ if $r_0 \in  {\mathcal I}(a,i)$
(respectively, if $r_0 \not\in {\mathcal I}(a,i)$). Indeed, the interval above is precisely 
the residue classes of $r$ modulo $p$ for which the binomial coefficient in Theorem~\ref{intro i < [a-i]} vanishes modulo $p$. 
Applying Theorem~\ref{intro i < [a-i]} recursively, we see that $Q(i)$ has 
all the cosocle JH factors of $V_r^{(j)}/V_r^{(j+1)}$, for $0 \leq j \leq i_0-1$, and all the JH factors
of $V_r^{(i_0)}/V_r^{(i + 1)}$, where $1 \leq i_0 \leq i$ is the smallest integer such that $r_0 \in {\mathcal I}(a,i_0)$ if it exists,
else $i_0 = i+1$.  
%
%
If $i < [a-i]$, but with $i > a$ instead, 
then $i - 1$ still satisfies these inequalities if $i-1 > a$, so we can again recursively apply Theorems~\ref{i not a nor p - 1} and \ref{intro i < [a-i]}, 
this time with the interval
$$\mathcal{I}(a,i) = \{a, a+1, \ldots, [a-i]-1, [a-i] \}^c,$$
where $c$ denotes the complement in the residue classes $\{0,1,\ldots,p-1\}$ mod $p$,
to determine the JH factors of $Q(i)$ in terms of the $W$ and $Q(a)$. We may   
then apply Theorem~\ref{i = a or p - 1} to determine the JH factors of $Q(a)$ and therefore of $Q(i)$.

(ii) When $i \geq [a-i]$ with $i \neq a$, $p-1$, we also need to consider the intervals of residue classes modulo $p$ 
\begin{eqnarray*}
  \mathcal{J}(a,i) =  \begin{cases}
                         \{a-i-1, a-i, \ldots, a-2, a-1\}, & \text{if } i < a, \\ 
                         \{a-1, a, \ldots, [a-i]-3, [a-i]-2\}^c, & \text{if  } i > a.
                      \end{cases}
\end{eqnarray*}
If we are at the boundary of the cases treated in (i), namely $i = [a-i]$, with $i \neq a$, $p-1$, then
it is the interval $\mathcal{J}(a,i)$ that plays a role in determining $W$ in Theorem~\ref{i not a nor p - 1},
since $\mathcal{J}(a,i)$ is precisely the residue classes of $r$ modulo $p$ for which the binomial coefficient in Theorem~\ref{intro i = [a-i]} 
vanishes modulo $p$.
Applying Theorem~\ref{intro i = [a-i]}, 
we are reduced to determining the structure of $Q(i')$ with $i'=i-1$.
If $i'=0$, we are done, else $i'$ satisfies $1 \leq i' < [a-i']$ and we can apply the arguments in (i) to determine $Q(i')$, 
unless $i' = a$, in which case we apply Theorem~\ref{i = a or p - 1} instead.

(iii) Finally, if $[a-i] < i < p-1$ (the hardest case), then 
$j = [a-i]$ in Theorem~\ref{i not a nor p - 1}, and 
both the intervals ${\mathcal I}(a,[a-i])$ 
and $\mathcal{J}(a,i)$ (along with the size of $[a-r_0]$ compared to $r_0$, in some cases) 
play a role in determining $W$, by Theorem~\ref{intro i > [a-i]}. 
So we are reduced to determining $Q(i')$, for $i' = [a-i]-1$. As in case (ii), if $i'=0$ we are done, else $i'$ satisfies 
$1 \leq i' < [a-i']$, and we are reduced to case (i), unless $i' = a$, in which
case we apply Theorem~\ref{i = a or p - 1} instead.

We now make some remarks about determining the JH factors of $Q(i)$ when 
$i = a$ and $p-1$. If $i = a$, then Theorem~\ref{i = a or p - 1} determines 
$W'$, so the JH factors of $Q(a)$ can be determined from those of $P(a)$, 
hence by what we have said above, from those of $Q(a-1)$. If $a = 2$, we are
reduced to $Q(1)$ and are done, else $i' = a-1$ satisfies $[a-i']<i'$, so 
applying Theorem~\ref{i not a nor p - 1} with $j = \min\{i',[a-i']\} = 1$, we 
are reduced to $Q(0)$. Finally, if $i = p-1$ and $i \neq a$, then $Q(p-1)$ 
is determined by $P(p-1)$ and $W'$ by Theorem~\ref{i = a or p - 1}, hence by 
$Q(p-2)$, hence by $Q(a)$ if $a = p-2$, and again by $Q(a)$ if $a \leq p-3$,
applying Theorem~\ref{i not a nor p - 1} with $j = \min\{p-2, [a-(p-2)]\} = a +1$. 
But we have just determined the JH factors of $Q(a)$ in all cases, so we are again 
done.

As an example of the strategy outlined above, we now determine all cases for 
which the quotient $Q(i)$ is irreducible, for $1 \leq i \leq p-1$
(see Theorem~\ref{irreducible Q(i)}).

\begin{theorem}\label{irreducible}
  Let $p \geq 3$, $1 \leq i \leq p-1$, $r \geq i(p+1)+p$, $r \equiv a \mod (p-1)$ 
  with $a \in \{1, 2, \ldots, p-1\}$ and let $r_0$ be the constant term in the 
  base $p$-expansion of $r$. Then the quotient $Q(i)$ of $V_r$ is irreducible 
  if and only if either
  \begin{itemize}
     \item $i = a-1$ or $a$, and $r_0 \in \{a, a+1, \ldots p-1\}$, or, 
     \item $i = p-1$, $a = 1$ and $r_0 = 0$.
  \end{itemize}
\end{theorem}  
This result is of special number theoretic interest since it immediately
solves the reduction problem for the Galois representations mentioned above 
in the cases that $Q(i)$ is irreducible and does not have dimension $p-1$. 
We state the result, assuming some familiarity with the notation.
Let $k \geq 2$ be an integer and $a_p \in \bar{\Q}_p$ have positive $p$-adic 
valuation $v(a_p)$, where $v$ is normalized so that $v(p) = 1$. Let $V_{k,a_p}$ 
be the unique two-dimensional $p$-adic crystalline representation 
defined over $\bar{\Q}_p$ of the Galois group of $\Q_p$ attached to this data, 
having Hodge-Tate weights $(0, r+1)$, for $r =k-2$, and slope $v(a_p)>0$.
\begin{corollary} \label{cor reduction}
  Let $r = k-2 \equiv a \mod (p-1)$, for $a \in \{1, 2, \ldots, p-1\}$, 
  and assume that the constant term $r_0$ lies in the range 
  $\{a, a+1, \ldots, p-1 \}$.
  If the slope $v(a_p)$ is fractional, with either
  \begin{itemize}
    \item[$\bullet$] $v(a_p) \in (a-1, a)$ for $2 \leq a \leq p-1$, or
    \item[$\bullet$] $v(a_p) \in  (a, a+1)$ for $a \neq p-2$,
    \end{itemize}
  then the reduction of the crystalline  representation $\bar{V}_{k,a_p}$ of $V_{k,a_p}$ is {\em irreducible}. 
\end{corollary}
\noindent In fact, one checks that the reduction $\bar{V}_{k,a_p}$ is 
isomorphic to the induced representation $\ind(\omega_2^{a+1})$, where $\omega_2$ 
is the fundamental character of level $2$ of the Galois group of
the quadratic unramified extension of $\Q_p$. When $i = a = p-2$ or $i = p-1$, 
$a = 1$, the quotient $Q(i)$ is irreducible but has dimension $p-1$ and one may 
only conclude that $\bar{V}_{k,a_p} \cong  \ind(\omega_2^{a+1})$ {\it if} it is 
irreducible. 
%
\section{Preliminaries}

The aim of this section is to recall some basic results concerning the 
symmetric power representations $V_{r}$ and the principal series representations
of $\gl_{2}(\f)$, and to prove some explicit results involving binomial 
coefficients in characteristic $p$. 

\textbf{Notation}: We fix a prime number $p$. We write $\Q_p$ (resp. $\Z_p$) 
for the $p$-adic completion of $\Q$ (resp. $\Z$), $\f$ for the field with 
$p$ elements, $\bar{\mathbb{F}}_{p}$ for a fixed algebraic closure of $\f$. 
We let $M := \mathrm{M}_{2}(\f)$, $\Gamma := \gl_{2}(\mathbb{F}_{p})$, 
$B \subset \Gamma$ the subgroup of upper triangular matrices, $U \subset B$ 
the subgroup of unipotent matrices and $H \subset B $ the subgroup of diagonal
matrices. 

For a positive integer $r$, let $\Sigma_{p}(r)$ denote the sum of digits in
the base $p$-expansion of $r$. It is easy to see that 
$\Sp(r) \equiv r \mod (p-1)$, for every $r \in \mathbb{N}$.
Also $\Sigma_{p}(p^{n}r) = \Sigma_{p}(r)$, for all $n$, $r \geq 0$ and 
$\Sigma_{p}(r-1) = \Sigma_{p}(r)-1$ if $p \nmid r$. We will be considering 
the base $p$-expansion of $r$ quite often which we denote by
\begin{align} \label{base p expansion of r}
       r= r_{m}p^{m}+ \cdots +r_{1}p+r_{0},
\end{align}
where $r_{m} \neq 0$ and $0 \leq r_{j} \leq p-1$. The constant term $r_0$
and the linear term $r_1$ will play key roles. 

For $n \in \mathbb{Z}$, define $[n] \in \lbrace 1,\ldots, p-1 \rbrace$ by 
$n \equiv [n] $ mod ($p-1$). Note that $[[m]-[n]] = [m-[n]] = [[m]- n] = [m-n]$,
$\forall$ $m$, $n \in \mathbb{Z}$. We finally recall the Kronecker delta 
function: if $S$ is any set, and $s_{1}$, $s_{2} \in S$, then  we define
\begin{align*}
    \delta_{s_{1},s_{2}} = 
    \begin{cases}
       0,  &\mathrm{if} ~ s_{1} \neq s_{2}, \\
       1,  &\mathrm{if} ~ s_{1}  = s_{2}.
    \end{cases}
\end{align*}
Let $V_{r} $ denote the space of homogeneous polynomials $F(X,Y)$ of degree 
$r$ in two variables $X$, $Y$ with coefficients in $\f$. The semigroup $M$
acts on $V_{r}$ by $\begin{psmallmatrix} a & b \\ c & d\end{psmallmatrix} 
\cdot F(X,Y)= F(aX+cY, bX+dY)$, for  
$\begin{psmallmatrix} a & b \\ c & d\end{psmallmatrix} \in M$. 
An $\f[M]$-module $V$ is called {\it singular} if every singular matrix
$t \in M$ annihilates $V$, i.e., if $t \cdot V =0$, $\forall$ 
$t \in M \smallsetminus \Gamma$. The largest singular 
submodule of an arbitrary $\f[M]$-module $V$ is denoted by $V^{\ast}$. Let 
$D : \Gamma \rightarrow \f^\ast$ denote the determinant character of $\Gamma$.
Recall the Dickson invariant 
\[ 
   \theta := X^{p}Y-XY^{p} = -X \cdot \prod\limits _{\lambda \in \f}  
   (Y- \lambda X) \in V_{p+1} 
\]
on which $\Gamma$ acts by $D$. Also, for each $m \in \mathbb{N}$, define
\[
   V_{r}^{(m)} = \{ F(X,Y)  \in V_{r} : \theta^{m} \text{ divides }  F(X,Y) 
   \text{ in } \mathbb{F}_{p}[X,Y] \},
\]
so that $ V_{r} \supseteq V_{r}^{(1)} \supseteq V_{r}^{(2)} \supseteq \cdots $ 
is a chain of $\Gamma$-modules of length $\lfloor \frac{r}{p+1}\rfloor +1$. By
\cite[(4.1)]{Glover}, we have $V_{r}^{\ast} =V_{r}^{(1)}$ and
$ V_{r}^{(m)} \cong V_{r-m(p+1)} \otimes D^{m}$, for all $m \in \mathbb{N}$.

\subsection{Modular representations of  \texorpdfstring{$M$ and $\Gamma$}{}}

\subsubsection{Results on \texorpdfstring{$V_{r}$}{}.} 
\label{prelim}
Let $X_{r-i,\,r}$ be the $\mathbb{F}_{p}[\Gamma]$-submodule of $V_{r}$ generated 
by the monomial $X^{r-i} Y^{i}$, for $0 \leq i \leq r$. The representations 
$V_r$ were studied by Glover \cite{Glover}. In this subsection, we recall a 
few results from \cite{Glover} and \cite{BG15} about $V_{r}$ and its 
$\Gamma$-submodules $X_{r,\,r}$ and $X_{r-1,\,r}$. One has to be careful with
notation when using the results of \cite{Glover} as Glover indexed the 
symmetric power representations by dimension instead of by the degree of the 
polynomials involved. 

We start with the following well-known result describing the irreducible 
representations of $\Gamma$ (see \cite{BN41}). These representations form the 
Jordan-H\"older (JH) factors of the various representations of $\Gamma$ studied
later. 
\begin{lemma} 
   If $0 \leq r \leq p-1$ and $1  \leq j \leq p-1$, then $V_{r} \otimes D^{j}$ is 
   an irreducible $\Gamma $-module. In fact these $p(p-1)$ modules are the
   set of all irreducible $\Gamma$-modules.
\end{lemma} 
We note the following congruence modulo $p$ which we use often. With the 
convention $0^0=1$ we have for any $i \geq 0$,
\begin{align}\label{sum fp}
       \sum_{\lambda \in \f} \lambda^{i} \equiv 
        \begin{cases}
                 -1,  &\mathrm{if}~i=n(p-1), ~ \mathrm{for ~ some ~} n \geq 1, \\
                  0,  & \mathrm{otherwise}.
        \end{cases}
\end{align} 

Next we show that the $\Gamma$-modules generated by the first $p$ monomials, i.e., 
$X^{r-i}Y^{i}$,  for $0 \leq i \leq p-1$, form an ascending chain of submodules 
of $V_{r}$.
\begin{lemma}\label{first row filtration}
   For $r \geq p$, we have $X_{r,\,r} \subseteq X_{r-1,\,r} \subseteq  \cdots  
   \subseteq X_{r-i,\,r}\subseteq \cdots  \subseteq X_{r-(p-1),\,r}$.
\end{lemma}
\begin{proof}
   Let $1 \leq i \leq p-1$. We have
   \begin{align*}
	  \sum_{a \in \mathbb{F}_{p}^{\ast}} a^{-1} 
	  \begin{pmatrix} 1 & a \\ 0 & 1 \end{pmatrix} \cdot X^{r-i}Y^{i}  
	  & = \sum_{a \in \mathbb{F}_{p}^{\ast}} a^{-1}  X^{r-i} (aX+Y)^{i} \\
	  & = \sum_{a \in \mathbb{F}_{p}^{\ast}} a^{-1}  \sum_{j=0}^ {i} 
	      \binom{i}{j}a^{j} X^{r-i+j} Y^{i-j} \\
	  & = \sum_{j=0}^ {i} \binom{i}{j} X^{r-i+j} Y^{i-j} 
	      \sum_{a \in \mathbb{F}_{p}^{\ast}} a^{j-1} =-i X^{r-(i-1)}Y^{i-1}. 
   \end{align*}
   Since $i \not\equiv 0$ mod $p$, it follows that $X^{r-(i-1)}Y^{i-1} \in X_{r-i,\,r}$,  
	hence  $X_{r-(i-1),\,r} \subseteq X_{r-i,\,r}$. 
\end{proof}
By the lemma, $X_{r-i,\,r}$ is $M$-stable, for  $0 \leq i \leq p-1$, since if 
$t$ is singular, then $t \cdot X^{r-i}Y^i \in X_{r, \, r}$, by \cite[(4.4)]{Glover}.

We next recall a Clebsh-Gordon type result from \cite{Glover} which gives 
the decomposition of the  tensor product of two irreducible representations 
of $\Gamma$. 
\begin{lemma}\label{ClebschGordan}\emph{\cite[(5.5)]{Glover}}
        Let $p \geq 2$ and $0 \leq m \leq n \leq p-1 $.
        \begin{enumerate}[label=\emph{(\roman*)}] 
                \item If $0 \leq m+n \leq p-1$, then 
                       \begin{align*}
                           V_{m} \otimes V_{n} 
                           \cong   V_{m+n} \oplus (V_{m-1} \otimes V_{n-1} \otimes D) 
                           \cong \bigoplus_{l=0}^{m} V_{m+n-2l} \otimes D^{l}.
                        \end{align*}   
                 \item If $p \leq m+n \leq 2p-2$, then
                        \begin{align*}
                            V_{m} \otimes V_{n} & \cong V_{p(m+n+2-p)-1} \oplus 
                            V_{(p-n-2)} \otimes V_{(p-m-2)} \otimes D^{m+n+2-p} \\
                            & \cong V_{p(m+n+2-p)-1} \oplus
                            \bigoplus_{l=0}^{p-n-2} V_{2p-2-m-n-2-2l} \otimes D^{m+n+2-p+l}.    
                         \end{align*}          
        \end{enumerate}
\end{lemma}
The following dimension formula for $X_{r-1,\,r}$ was proved in \cite{BG15}.  
%
\begin{lemma}\emph{\cite[Corollary 1.6]{BG15}}
\label{dimension formula for X_{r-1}}
    Let $p\geq 3$ and $r \geq  2p+1$. Set $\delta =1$ if $p \mid r$ and 
    $\delta=0$ otherwise. Then
    \begin{align*} 
        \mathrm{dim}\ X_{r-1,\,r} = 
	 \begin{cases}
             2\Sigma_{p}(r)+\delta(p+2-\Sigma_{p}(r)), & \mathrm{if} ~ \Sigma_{p}(r) 
             \leq p, \\
              2p+2, &  \mathrm{if} ~ \Sigma_{p}(r) > p.
          \end{cases} 
     \end{align*}
\end{lemma}
\begin{proof}  
	 Write $r=p^{n}u$, with $p\nmid u$. Then $\Sp(u-1) = \Sp(u)-1 = \Sp(p^{n}u)-1 
	 =\Sp(r)-1$. Substituting $\Sp(u-1)=\Sp(r)-1$ in \cite[Corollary 1.6]{BG15} 
	 we obtain the lemma.
\end{proof}
We next recall the structure of $X_{r,\, r}$  
and obtain a dimension formula for it. 
By \cite[(4.5)]{Glover}, for $r \geq p$, $r \equiv a$ mod $(p-1)$ with 
$1 \leq a \leq p-1$, we have an exact sequence of $M$-modules
\begin{align}\label{Glover 4.5}
    0 \rightarrow X_{r,\,r}^{(1)} \rightarrow X_{r,\,r} \rightarrow V_{a} \rightarrow 0,
\end{align} 
where $X_{r,\, r}^{(1)} = X_{r,\,r} \cap V_{r}^{(1)}$.
More precisely, we have
%
\begin{lemma}\label{dimension formula for X_{r}}
    Let $p\geq 3$ and $1 \leq r \equiv a \mod ~(p-1)$ with $1 \leq a \leq p-1$. Then  
    \begin{enumerate}[label=\emph{(\roman*)}]
	\item If $\Sp(r)=a$, then $X_{r,\, r} \cong V_{a}$, as an $M$-module.
	\item If $\Sp(r)>a$, or equivalently $\Sp(r) \geq a+p-1$, then there is a 
              short exact sequence of $M$-modules
	      \begin{align*}
	          0 \rightarrow V_{p-a-1} \otimes D^{a} \rightarrow X_{r,r} 
	          \rightarrow V_{a} \rightarrow 0.
	      \end{align*}
    \end{enumerate} 
    Moreover, $ \mathrm{dim} ~ X_{r,\,r} = p+1$ if and only if
     $X_{r,\,r}^{(1)} \neq 0$, and
    \begin{align*}
	\mathrm{dim} ~ X_{r,\,r} = 
	 \begin{cases}
	     \Sigma_{p}(r)+1, & \mathrm{if} ~~\Sigma_{p}(r) \leq p-1,  \\
	      p+1, & \mathrm{if} ~~ \Sigma_{p}(r) > p-1.
	    \end{cases}
     \end{align*}
\end{lemma}
\begin{proof}
     This is well known, and can be deduced from the results of \cite{Glover}, \cite{BG15}.
     Write $r=p^{n}u$ with $p \nmid u$. Then $\Sigma_{p}(r) = \Sigma_{p}(u) $ and 
     $X_{u,\,u} \cong X_{r,\,r}$ via the map $F \mapsto F^{p^{n}}$. So it is enough
      to prove the lemma with $r$ replaced by $u$. Note that 
      $p \nmid u$ implies that $\Sp(u-1) = \Sp(u)-1$.
     
     For $a=1$, part (i) follows from the fact that $\Sp(u)=1$ is equivalent to
     $u=1$. The cases $\Sp(u)=p$ and $\Sp(u) > p$ of part (ii) follow
     from \cite[Proposition 3.3]{BG15} and \cite[Proposition 3.8]{BG15} respectively.
     
     For $2 \leq a \leq p-1$, part (i) follows from \cite[Lemma 4.5]{BG15}, noting 
     that $X_{u,\, u} \cong V_{a}$ if $u < p$. For part (ii), if $\Sp(u)\geq a+p-1 > p$, 
     then $u \geq2 p+1$ and so by  \Cref{dimension formula for X_{r-1}}, we have 
     dim $X_{u-1,\,u}= 2p+2$. Thus by \cite[Lemma 3.5]{BG15}, we have 
     dim $X_{u,\,u} = p+1$.  Thus by \eqref{Glover 4.5}, we have $X_{u,\,u}^{(1)} \neq 0$.
     Further by \cite[Lemma 4.6]{BG15},  we have  
     $X_{u,\,u}^{(1)} \cong V_{p-a-1} \otimes D^{a}$. This proves part (ii), for 
     $2 \leq a \leq p-1$.  The other assertions are clear from  the exact sequence \eqref{Glover 4.5} 
     and parts (i), (ii).
\end{proof}
\subsubsection{Principal series.}
In this subsection, we recall a few results about
principal series representations. These 
representations play a central role in this article as they are related to  modules 
such as $X_{r-i} /X_{r-(i-1)}$ and $V_{r}^{(i)}/V_{r}^{(i+1)}$, for 
$0 \leq i  \leq p-1$, which we study in later sections.

Let $ (\sigma, V)$ be a representation of $B$. The induced representation 
$\ind_{B}^{\Gamma}(\sigma) $ is defined as the space of functions 
$f: \Gamma \rightarrow V $ satisfying
$ f(b\gamma) = \sigma(b) f(\gamma)$, $\forall$ $b \in B$, $ \gamma \in \Gamma $, 
this space being endowed with a left $\Gamma$-action defined by right
 translation of functions, i.e., 
 $(\gamma\cdot f)(\gamma ') = f(\gamma' \gamma)$,
 $\forall$ $\gamma,\gamma' \in \Gamma$. For any 
$ \gamma \in \Gamma$, $v \in V$, we define the function 
$[\gamma,v] \in \ind_{B}^{\Gamma}(\sigma)$
by
\begin{align*}
    [\gamma, v] (\gamma') =
    \begin{cases}
       \sigma(\gamma' \gamma) v, & \mathrm{if} ~ \gamma' \in B \gamma^{-1},\\
       0, & \mathrm{otherwise}.
    \end{cases}
\end{align*}
Every element of $\ind_{B}^{\Gamma}(\sigma)$ is a linear combination of functions of
the form $[\gamma,v]$, for $\gamma \in \Gamma$ and $v \in V$.  
It can be checked that $\gamma' \cdot [\gamma, v] = [\gamma' \gamma, v]$, 
$\forall$ $\gamma,\gamma' \in \Gamma$. If $\sigma$ is
a 1-dimensional representation, then $\ind_{B}^{\Gamma}(\sigma)$ is called 
a principal series representation. Since $\lvert \Gamma/ B \rvert =p+1$,
the dimension of a principal series representation equals $p+1$.
 
Let $w = \begin{psmallmatrix} 0 & 1 \\ 1 & 0 \end{psmallmatrix}$. 
For a character  $\chi: H \rightarrow \f^{\ast}$, we define 
$\chi^{w}: H \rightarrow \f^{\ast}$ by $\chi^w(h) = \chi(whw)$, for all $h \in H$. 
Let $\chi_{1}$, $\chi_{2} : H \rightarrow \f^{\ast}$ be the characters defined by
\begin{align*}
      \chi_{1}\left(\begin{psmallmatrix} a & 0 \\ 0 & d \end{psmallmatrix}\right) =
      a, \quad \chi_{2}\left(\begin{psmallmatrix} a & 0 \\ 0 & d \end{psmallmatrix}\right) =
      d, ~ \forall ~ \begin{psmallmatrix} a & 0 \\ 0 & d \end{psmallmatrix} \in H.
\end{align*}
These can also be thought of as characters of $B$ via $ B \twoheadrightarrow H$.
Clearly $(\chi_{1}^{i} \chi_{2}^{j})^{w} =\chi_{1}^{j} \chi_{2}^{i}$. 
It is well known that every character   $\chi: B \rightarrow \f^{\ast}$ is of the form
$\chi_{1}^{i} \chi_{2}^{j}$, for $1 \leq i,j \leq p-1$. 
For a character 
$\chi: B \rightarrow \f^{\ast}$, let $e_{\chi}$ denote a (fixed) non-zero element  of the 
1-dimensional representation $(\chi, V_{\chi})$. The following result 
explicitly describes the Jordan-H\"older (JH) factors of 
principal series representations  and the basis elements of the underlying spaces.
%
\begin{lemma}\label{Structure of induced}
     Let $p \geq 2$, $1 \leq i,j \leq p-1$ and $\chi =\chi_{1}^{i} \chi_{2}^{j}$. Then we have
     the following exact sequence of $\Gamma$-modules 
     \begin{align*}
          0 \rightarrow  V_{[ j-i ] } \otimes D^{i} \rightarrow  \ind_{B}^{\Gamma}
          (\chi_{1}^{i}\chi_{2}^{j}) \rightarrow  V_{p-1-[j-i]} \otimes D^{j} 
          \rightarrow 0 .
     \end{align*}
     The sequence splits if and only if $i = j$. Moreover, 
     \begin{enumerate}[label=\emph{(\roman*)}]
            \item  An $\f$-basis of the  image of  $V_{[j-i]}\otimes D^{i} $ 
                   in $\ind_{B}^{\Gamma} (\chi)$  is given by
                    \begin{align*}
                         \sum\limits_{\lambda \in \f} \lambda^{l} 
                         \begin{pmatrix} \lambda & 1 \\ 1 & 0 \end{pmatrix} [1, e_{\chi}], 
                          ~ \mathrm{for} ~ 0 \leq l < [j-i];       ~ ~                              
                          \sum\limits_{\lambda \in \f}  \lambda^{[j-i]} 
                          \begin{pmatrix} \lambda & 1 \\ 1 &0  \end{pmatrix} 
                          [1, e_{\chi}] + (-1)^{j}  [1, e_{\chi}].
                    \end{align*}
              \item The  elements of $\ind_{B}^{\Gamma} (\chi)$ 
                    \begin{align*}
                          \sum\limits_{\lambda \in \f} \lambda^{l} 
                           \begin{pmatrix} \lambda & 1 \\ 1 &0  \end{pmatrix} 
                           [1, e_{\chi}], ~ \mathrm{for} ~ [j-i] \leq l \leq p-1,
                     \end{align*}   
                     map to an $\f$-basis of   $ V_{p-1-[j-i]} \otimes D^{j}$.  
       \end{enumerate}
 \end{lemma}
\begin{proof}
    See \cite[Proposition 2.4]{Morra} and 
     Lemmas 2.3, 2.6, 2.7 and Theorem 2.4 of 
    \cite{BP12}. 
\end{proof}
\begin{corollary}\label{Common JH factor}
      Let $\chi$, $\eta : B \rightarrow \f^{\ast}$ be characters. Then
      \begin{enumerate}[label=\emph{(\roman*)}]
            \item   The Jordan-H\"older factors of $\ind_{B}^{\Gamma}(\chi)$ are distinct.
             \item socle of $\ind_{B}^{\Gamma}(\chi) \cong $  socle of 
             $\ind_{B}^{\Gamma}(\eta)$
                       if and only if $\eta = \chi$.
              \item socle of $\ind_{B}^{\Gamma}(\chi) \cong $ cosocle of 
              $\ind_{B}^{\Gamma}(\eta)$
                       if and only if $\eta = \chi^{w}$.   
       \end{enumerate}    
       Therefore $\ind_{B}^{\Gamma}(\chi)$ and $\ind_{B}^{\Gamma}(\eta)$ have a
        common Jordan-H\"older factor if and only if $\eta = \chi$ or $\eta = \chi^{w}$. 
\end{corollary}
\begin{proof}
      This is an easy consequence of \Cref{Structure of induced}.
\end{proof}
\subsubsection{The filtration \texorpdfstring{$V_{r}^{(m)}$}{}.}    
    We now prove some results concerning the modules $V_{r}^{(m)}$, for $m \in \N$. 
    We begin by giving a criterion that allows one to check when an arbitrary 
    polynomial $F \in V_{r}$ is divisible by $\theta ^{m}$, 
    slightly generalizing  \cite[Lemma 3.1]{SB18}.
\begin{lemma}\label{divisibility1}
     Let $p\geq 2$, $r>p$ and $F(X,Y) = \sum\limits_{j=0}^{r} a_{j} X^{r-j}Y^{j} \in V_{r}$. 
     Then for any $1 \leq m \leq p$, we have $F \in V_{r}^{(m)}$  if and only if
      the  following hold
       \begin{enumerate}[label=\emph{(\roman*)}]
              \item $a_{j} \neq 0  \;  \Longrightarrow  \; m \leq j \leq r-m$,
              \item $\sum\limits_{j \, \equiv \, l \: \mathrm{mod} ~ (p-1)}^{}  
                    \binom{j}{i} a_{j} = 0 $ in $\f$, $\forall$ $0 \leq i \leq m-1 $ and 
                    $1 \leq l \leq p-1 $.
       \end{enumerate}        
\end{lemma} 
\begin{proof}
       We follow the proof of \cite[Lemma 3.1]{SB18}.
       Consider $f(z) = \sum_{j=0}^{r} a_{j} z^{j} \in \f [z]$, so that $F(X,Y)= 
       X^{r}f(Y/X)$.  Note that
       \begin{align*}
           \theta^{m} \mid F(X,Y) 
           & \Longleftrightarrow F(X,Y) =(-XY)^{m} \prod_{\lambda \in \fstar} 
             (Y-\lambda X)^{m} F_{1}(X,Y), ~ \mathrm{for ~ some} ~ F_{1} 
             \in V_{r-mp-m}  \\ 
           & \Longleftrightarrow X^{m}, Y^{m} \mid F(X,Y) ~ 
             \mathrm{and} ~ f(Y/X) = 
             (-1)^{m}\prod_{\lambda \in \fstar} (Y/X - \lambda)^{m} F_{1}(1,Y/X) \\
           & \Longleftrightarrow  X^{m},Y^{m} \mid F(X,Y) ~ \mathrm{and} ~ f(z) 
             = \prod_{\lambda \in \fstar} (z - \lambda)^{m} f_{1}(z), ~ \mathrm{for ~ some} 
             ~ f_{1} \in \f[z]  \\ & \Longleftrightarrow  X^{m}, Y^{m} \mid F(X,Y)~ 
             \mathrm{and} ~ (z - \lambda )^{m} \mid f(z), \  \forall \ \lambda \in \fstar.
       \end{align*}
       The conditions $X^{m},Y^{m} \mid F(X,Y)$ are equivalent to 
       $a_{i} \neq 0 \Longrightarrow
       m \leq i \leq  r-m$, and $(z- \lambda)^{m} \mid f(z)$ if and only if 
       $f(\lambda) = f'(\lambda)= \cdots = f^{(m-1)} (\lambda) = 0$ in $\f$.
       For $i \geq 0$ and  $\lambda \in \f^{\ast}$, 
       we have
        \begin{align*}
              f^{(i)}(\lambda) 
              & = \sum_{j} j (j-1) \cdots (j-i+1)a_{j} \lambda^{j-i} 
                 = \sum_{j}  i ! \binom{j}{i} a_{j} \lambda^{j-i} \\
              &=  \sum_{l=1}^{p-1} \lambda^{l-i} \sum_{j \, \equiv \,  l ~ \mathrm{mod} ~
               (p-1)} i ! \binom{j}{i} a_{j}.
        \end{align*}
        Since $f^{(i)}(\lambda) =0$, $\forall ~ \lambda \in \f^{\ast}$ and  $0 \leq i 
        \leq m-1$,  by the  non-vanishing of  the   Vandermonde  determinant  
        and $p \nmid i!$, we obtain (ii). This completes the proof.
\end{proof}
\begin{corollary}     
       Let $p\geq 2$, $r > p$ and $F(X,Y) = \sum\limits_{j=0}^{r} a_{j} X^{r-j}Y^{j} \in V_{r}$. 
       For $1 \leq l \leq p-1$, define
       $F_{l}(X,Y) = \sum\limits_{j \equiv l~ \mathrm{mod}~(p-1)} 
       a_{j} X^{r-j}Y^{j} \in V_{r}$.   
       Them, for $1 \leq m \leq p$, we have 
       $F(X,Y) \in V_{r}^{(m)}$ if and only if $F_{1}(X,Y), \ldots, F_{p-1}(X,Y)
       \in V_{r}^{(m)}$.
\end{corollary}
For $r \geq p $,   the map  $F  \mapsto 
\left( \gamma \mapsto  F((0,1) \gamma) \right)$ 
defines a $\Gamma$-linear isomorphism from $V_{r}/V_{r}^{(1)}$
to $\ind_{B}^{\Gamma}(\chi_{2}^{r})$,
see, for example,  \cite[Lemma 2.4]{sandra}. 
We generalize this result as follows:
 \begin{lemma}\label{induced and star}
         For $p \geq 2$, $m \geq 0$ and $r \geq  m(p+1)+p$, 
         we have
         $V_{r}^{(m)}/ V_{r}^{(m+1)}  \cong   \ind_{B}^{\Gamma} (\chi_{1}^{m} \chi_{2}^{r-m})$, as 
         $\Gamma$-modules. Furthermore, if $r< m(p+1)+p $, then
         $V_{r}^{(m)}/V_{r}^{(m+1)} \hookrightarrow
         \ind_{B}^{\Gamma}(\chi_{1}^{m} \chi_{2}^{r-m})$.
 \end{lemma}
 \begin{proof}
        Let  $r' = r-m(p+1)$, for $r \geq m(p+1)$. Assume $r \geq m(p+1)+p$. 
        By \cite[(4.1)]{Glover}, we see that
        $V_{r}^{(m)}/ V_{r}^{(m+1)} \cong
        V_{r'}/V_{r'}^{(1)}  \otimes D^{m}$ as $\Gamma$-modules. Since  
        $r' \geq p$,   we have
         $V_{r'}/V_{r'}^{(1)} \cong \ind_{B}^{\Gamma}
         (\chi_{2}^{r'})$. Hence
    \begin{equation} \label{ps local}
         V_{r}^{(m)}/V_{r}^{(m+1)} \cong  \ind_{B}^{\Gamma}
         (\chi_{2}^{r'}) \otimes D^{m}  \cong  \ind_{B}^{\Gamma}
         (\chi_{2}^{r'} \otimes D^{m}) = \ind_{B}^{\Gamma}(\chi_{1}^{m}
         \chi_{2}^{r-m}).  
    \end{equation}
         This proves the first assertion. If $r < m(p+1)$,
then
         $V_{r}^{(m)} =0$, 
         so the
         second assertion is trivial in this case. Assume
         $m(p+1) \leq r < m(p+1)+p $. Then
         $V_{r}^{(m)} \cong V_{r'} \otimes D^{m}$ and
         $V_{r}^{(m+1)} =0$, so 
         $V_{r}^{(m)}/V_{r}^{(m+1)} \cong   V_{r'} \otimes D^{m}$, which is  
         $V_{[r-2m]} \otimes D^{m}$ if $1 \leq r' \leq p-1$,
         since $r' \equiv r-2m \mod (p-1)$, and is $V_0 \otimes D^m$ if $r' = 0$. In either case,
         this submodule is contained in the socle of 
         the principal series representation 
         $\ind_{B}^{\Gamma}(\chi_{1}^{m} \chi_{2}^{r-m})$ in \eqref{ps local},
         by \Cref{Structure of induced} 
         (the latter because we are in the split case of that lemma). 
 \end{proof}
 
It follows from the lemma that
$\dim V_{r}/V_{r}^{(m)} = \sum_{n=0}^{m-1} \dim V_{r}^{(n)}/V_{r}^{(n+1)}
= m(p+1)$, for $m \geq 1$ and $r \geq m(p+1)-1$.  We have:

\begin{lemma} \label{basis}
Let $p \geq 2$, $m \geq 1$ and $r \geq m(p+1)-1$. Then
\[
  \Lambda = \lbrace X^{r-i}Y^{i}+ V_{r}^{(m)} : 0 \leq i \leq m(p+1)-(m+1) \rbrace \cup
                 \lbrace X^{i}Y^{r-i}+ V_{r}^{(m)} : 0 \leq i < m \rbrace
\]
is a basis of $V_{r}/V_{r}^{(m)}$.  
\end{lemma}
\begin{proof}
Since the cardinality of $\Lambda$ equals 
$\dim V_{r}/V_{r}^{(m)}$, it is enough to show that
$\Lambda$ spans $V_{r}/V_{r}^{(m)} $ as an $\f$-vector space. For this we 
 induct on $i$ to show that  $X^{r-i}Y^{i}+ V_{r}^{(m)}$ 
 belongs to the $\f$-span of $\Lambda$, for all $0\leq i \leq r$.
 If $0 \leq i \leq m(p+1)-(m+1)$ or $r-m < i \leq r$, then 
 $X^{r-i}Y^{i}+ V_{r}^{(m)}$ belongs to 
$\Lambda$.  Assume that $mp = m(p+1)-m \leq i \leq r-m$, then
\[
   X^{r-i}Y^{i} - X^{r-i-m} Y^{i-mp}(XY^{p}-X^{p}Y)^{m} =
       - \sum_{j=1}^{m} (-1)^{j} \binom{m}{j} X^{r-i+jp-j} Y^{i-jp+j}.
\] 
Observe that the degree of $Y$ on right hand side above is
strictly less than $i$.  By induction, the right hand side modulo 
$V_{r}^{(m)}$ belongs to the $\f$-span of $\Lambda$, so  $\Lambda$
 is an $\f$-basis of $V_{r}/V_{r}^{(m)}$.
\end{proof}

\begin{lemma}\label{Breuil map}
    Let $p \geq 2$, $m \geq 0$ and $m(p+1)+p \leq r \equiv a ~\mathrm{mod}~(p-1)$
     with $1 \leq a \le p-1$. Then we
     have a short exact sequence of $\Gamma$-modules
    \begin{align}\label{exact sequence Vr}
	         0 \rightarrow  V_{[a-2m]} \otimes D^{m} \rightarrow V_{r}^{(m)}/ V_{r}^{(m+1)}
	          \rightarrow
	         V_{p-1-[a-2m]} \otimes D^{a-m} \rightarrow 0
	  \end{align}
	 and this sequence splits if and only if $a \equiv 2m$ $\mathrm{mod}~(p-1)$.
	If $r' = r-m(p+1)$ and $a' =[a-2m]$, 
    then the rightmost map in \eqref{exact sequence Vr} is given on a basis by
                 $$
                          \theta^{m} X^{r'-i}Y^{i}~\mathrm{mod}~V_{r}^{(m+1)} \longmapsto 
                          \begin{cases}
                              0,                                                 & \text{if} ~ 0 \leq i < a' \text{ or } i = r', \\
                             (-1)^{r'-i} \binom{p-1-a'}{i-a'} X^{p-1-i}Y^{i-a'}, & \text {if} ~ a' \leq i \leq p-1. 
                          \end{cases}   
                   $$       
\end{lemma}
\begin{proof}
      The exact sequence \eqref{exact sequence Vr} follows  from \Cref{induced and star} 
      and \Cref{Structure of induced}. The explicit description
       of the rightmost map in \eqref{exact sequence Vr} follows from results in \cite{Glover} and \cite{Breuil}.
      Indeed, by \cite[(4.1)]{Glover}, we see that $V_{r}^{(m)}/ V_{r}^{(m+1)} \cong 
        V_{r'}/V_{r'}^{(1)}  \otimes D^{m}$ as $\Gamma$-modules. 
       By the proof of  \cite[(4.2)]{Glover}, we see that the map  
       $ \psi: V_{a'+p-1}/V_{a'+p-1}^{(1)} \rightarrow V_{r'}/V_{r'}^{(1)}$ induced by 
       $X^{a'+p-1-i} Y^{i} 
       \mapsto X^{r'-i}Y^{i}$, 
       for $0 \leq i \leq p-1$, 
       and $Y^{a'+p-1} \mapsto Y^{r'}$, is an $M$-linear isomorphism. Thus, the  composition
       $\phi$ $$ V_{r}^{(m)}/ V_{r}^{(m+1)}   \cong  V_{r'}/V_{r'}^{(1)}  \otimes D^{m}
        \stackrel{\psi^{-1}}{\longrightarrow} 
        V_{a'+p-1}/V_{a'+p-1}^{(1)} \otimes D^{m},$$
        given by  
        $\theta^{m} X^{r'-i} Y^{i} \mapsto X^{a'+p-1-i} Y^{i}$, 
        for $0 \leq i \leq p-1$,  and $ \theta^{m} Y^{r'} \mapsto Y^{a'+p-1}$, is a 
        $\Gamma$-linear isomorphism. 
        Applying $\phi$ and using
        \cite[Lemma 5.1.3]{Breuil}  for 
       $V_{a'+p-1}/V_{a'+p-1}^{(1)}$, we obtain the explicit description
       of the rightmost map in \eqref{exact sequence Vr}.
\end{proof}
Let $0 \leq m \leq p-1$ and $F(X,Y) \in V_{r}^{(m)}$ be a polynomial
such that the coefficients of $X^{r-j} Y^{j}$ are non-zero only if 
  $j  \equiv l \mod (p-1)$
for some $l$. The following lemma gives another useful way to compute 
the image of $F(X,Y)$ under the 
rightmost map of the exact sequence \eqref{exact sequence Vr}.
\begin{lemma}\label{breuil map quotient}
	Let $0 \leq m \leq p-1$ and $m(p+1) +p\leq r \equiv a \mod p$ with $1 \leq a 
	\leq p-1$. Let $F(X,Y) = \sum\limits_{j=0}^{r} a_{j} X^{r-j} Y^{j}  \in V_{r}$ with 
	$a_{j} \neq 0$ only if $j \equiv l \mod (p-1)$, for some $l$.  If $F(X,Y) \in V_{r}^{(m)}$,
	 then
	\begin{align*} 
	        F(X,Y) \equiv ~  \theta ^{m} G(X,Y) +  \theta^{m}  \left( a_{m} X^{r-m(p+1)} +
	        (-1)^{m} a_{r-m} Y^{r-m(p+1)} \right) \mod V_{r}^{(m+1)},
	\end{align*}
	where $$G(X,Y)= \left( \sum_{j=0}^{r} a_{j} \binom{j}{m} - a_{m}+ 
	(-1)^{m+1} a_{r-m} \right) X^{r-m(p+1)-[l-m]}Y^{[l-m]}.$$ Further,
	the image of $F(X,Y)$ under the quotient map 
	$V_{r}^{(m)}/ V_{r}^{(m+1)} \rightarrow V_{p-1-[a-2m]} \otimes D^{a-m} $
	 in \eqref{exact sequence Vr}  is the same as the image of $\theta^{m}G(X,Y)$.
	
\end{lemma}
\begin{proof}
	       Let $r'= r-m(p+1)$. Since $\theta^{m} \mid F(X,Y)$, there exist 
	       $H(X,Y)= \sum_{j=0}^{r'} b_{j} X^{r'-j}Y^{j} \in V_{r'}$ such that 
	       $F(X,Y) = \theta^{m} H(X,Y)$.  Differentiating both sides with respect to $Y$ 
	       $m$-times and substituting $X=Y=1$ we get 
	       \begin{align*}
	                 m! \sum_{j=0}^{r} a_{j} \binom{j}{m} = 
	                 \left( \frac{\partial^{m}}{\partial Y^{m}} F(X,Y)  \right) 
	                  \Biggr|_{\substack{X=1\\Y=1}}= m! ~ H(X,Y) \Bigr|_{\substack{X=1\\Y=1}}
	                  =m! \sum_{j=0}^{r'} b_{j}.
	       \end{align*}
	       Comparing the coefficients of $X^{r-m}Y^{m}$, $X^{m}Y^{r-m}$ in $F(X,Y)$  and 
	       $\theta^{m} H(X,Y)$ we get $a_{m}=b_{0}$ and $a_{r-m} = (-1)^{m} b_{r'}$. 
	       Hence 
	       \begin{align}\label{breuil quotient sum relation}
	                \sum\limits_{0<j<r'}b_{j}=  \left( \sum\limits_{j=0}^ {r'} b_{j} \right)
	                 - b_{0}-b_{r'} = \sum\limits_{j=0}^ {r} a_{j} 
	                \binom{j}{m} - a_{m} + (-1)^{m+1} a_{r-m}.
	       \end{align} 
	       Since $a_{j}=0$ if $j \not \equiv l$ mod $(p-1)$, it can be 
	       checked that $b_{j} = 0$ if $j \not \equiv l-m$ mod $(p-1)$.
	       Therefore 
	       \begin{align*}
	                 F(X,Y) & = \theta ^{m} H(X,Y)
	              =  \theta^{m} 
	                  \sum_{\substack{0<j<r' \\ j \equiv l-m ~\mathrm{mod}~(p-1)}} b_{l}
	                 X^{r'-j} Y^{j} +  \theta^{m} (b_{0} X^{r'} +b_{r'}Y^{r'}) \\
	                 & \equiv \theta^{m} \left( 
	                 \sum_{0<j<r'} b_{j} \right) X^{r'-[l-m]} Y^{[l-m]} + 
	                 \theta^{m} (b_{0} X^{r'} +b_{r'}Y^{r'}) 
	                  \mod V_{r}^{(m+1)} \\
	                 &  \stackrel{\eqref{breuil quotient sum relation}}{\equiv} \theta ^{m} G(X,Y) + \theta^{m} \left( a_{m} X^{r-m(p+1)} +
	                  (-1)^{m} a_{r-m} Y^{r-m(p+1)} \right) \mod V_{r}^{(m+1)},
	       \end{align*}
	       as required. 
	       The last assertion is clear from \Cref{Breuil map}, as
	       $\theta^{m}X^{r'}$, $\theta^{m}Y^{r'}$ map to zero under the rightmost map 
	        in \eqref{exact sequence Vr}.
\end{proof}
\subsection{Binomial coefficients mod $p${}.} 

In this subsection, we prove several elementary results involving binomial 
coefficients in characteristic $p$ which will be used in later sections. 
We begin by recalling Lucas' theorem.
\begin{lemma} \label{lucas} \emph{(\textbf{Lucas' theorem})}
      For any prime $p$, let $m$ and  $n$ be two non-negative integers with 
      base  $p$-expansions 
      given by $m= m_{k}p^{k}+m_{k-1}p^{k-1}+ \cdots + m_{0}$ and 
      $n= n_{k}p^{k}+n_{k-1}p^{k-1}+ \cdots + n_{0}$ respectively. Then 
      $\binom{m}{n} \equiv \binom{m_{k}}{n_{k}} \cdot \binom{m_{k-1}}{n_{k-1}} \cdots 
      \binom{m_{0}}{n_{0}} \mod p$, with the convention that $\binom{a}{b}=0$, 
      if $b>a$.
\end{lemma}

We next prove a  lemma concerning sums involving products of binomial 
coefficients. 
\begin{lemma}\label{binomial sum}
    For $m \geq 0$, $1 \leq b \leq p-1$ and $m < r \equiv a \mod (p-1)$ with 
    $1 \leq a \leq p-1$, we have
    $$
       S_{r, b, m}:=\sum\limits_{\substack {0 \leq l \leq r \\ l \equiv b ~ 
       \mathrm{mod}~ (p-1)}} 
       \binom{r}{l} \binom{l}{m} \equiv \binom{r}{m} \binom{[a-m]}{[b-m]}
       + \binom{r}{m} \delta_{p-1,[b-m]} \mod p,    
     $$   
     where $\delta$ is the Kronecker delta function.
\end{lemma}
\begin{proof}
     Observe that 
     $$
       S_{r, b, m}=\sum\limits_{\substack {m \leq l \leq r \\ l \equiv b ~\mathrm{mod}~
        (p-1)}} 
       \binom{r}{l} \binom{l}{m} = 
       \sum\limits_{\substack {m \leq l \leq r \\ l \equiv b ~\mathrm{mod}~ (p-1)}} 
       \binom{r-m}{l-m} \binom{r}{m} = 
       \binom{r}{m} S_{r-m, [b-m], 0}.
     $$
     Put $r' = r-m$ and $a' = [a-m]$ and $b'= [b-m]$.
     To prove the lemma we compute the following sum in two different ways. Let
     \begin{align*}
         T_{r',b'}:= \sum_{\lambda \in \f} \lambda^{p-1-b'} (1+\lambda)^{r'}.
     \end{align*}        
     First we note that 
     \begin{align*}
         T_{r',b'} 
         &= \sum_{\lambda \in \f^\ast} \lambda^{p-1-b'} (1+\lambda)^{r'} + \delta_{p-1,b'}  
         = \sum_{j=0}^{r'} \binom{r'}{j}  \sum_{\lambda \in \f^\ast} \lambda ^{j-b'}                 
           + \delta_{p-1,b'}  \\  
         & \stackrel{\eqref{sum fp}}{\equiv} -S_{r', b', 0}+ \delta_{p-1,b'} \mod p.
     \end{align*}                                         
     On the other hand since $r>m$, we have 
     \begin{align*}
         T_{r',b'} 
         & = \sum_{\lambda \in \f \smallsetminus \lbrace-1 \rbrace}
             \lambda^{p-1-b'} (1+\lambda)^{r'} 
           =  \sum_{\lambda \in \f \smallsetminus \lbrace-1 \rbrace}
              \lambda^{p-1-b'} (1+\lambda)^{a'} \\
         & = \sum_{j=0}^{a'} \binom{a'}{j} \sum_{\lambda \in \f \smallsetminus 
         \lbrace-1 \rbrace}
             \lambda^{p-1+j-b'} \\      
          &= \left(\sum_{j=0}^{a'} \binom{a'}{j}  \sum_{\lambda \in \f } 
             \lambda^{p-1+j-b'} \right) - \left(\sum_{j=0}^{a'} \binom{a'}{j} (-1)^{p-1+j-b'} 
             \right)  
          \equiv - \binom{a'}{b'}  \mod p.
     \end{align*}
      Here the last congruence follows from the observation that the first sum 
      doesn't vanish if and only if   $a'  \geq b' $, in which case it equals $-\binom{a'}{b'}$
      and for the second sum note that $\sum_{j=0}^{a'} \binom{a'}{j} (-1)^j$ 
       $=(1-1)^{a'}=0$. Hence $S_{r', b',0} \equiv  \binom{a'}{b'} + \delta_{p-1,b'} 
       \mod p$, as desired.
\end{proof}
In a few proofs, we need to choose some numbers $s$ satisfying 
appropriate conditions. This choice is made using the following lemma.
\begin{lemma}\label{choice of s}
     Let $p \leq r =r_{m}p^{m}+ \cdots +r_{1}p+r_{0}$ be the base $p$-expansion 
     of $r$. Then, for  every $0 \leq b \leq r_{0} $ and 
     $1 \leq u \leq \Sigma_{p}(r)-r_{0}$, there exists a positive integer $s$ with
     $p \leq s \leq r$, $s \equiv b \mod p$ such that $\Sp(s)=b+u$ and $\binom{r}{s} 
     \not \equiv 0  \mod p$. In addition,  
     if  $u<  \Sp(r)-r_{0}$, 
     then 
     $s \leq r-p$.   
\end{lemma}
\begin{proof}  
     Since  $1 \leq u \leq \Sigma_{p}(r) - r_{0} = \sum_{i=1}^{m}r_{i} $, we can 
     find integers $s_{i}$ for $1 \leq i \leq m$, such that $0 \leq s_{i} \leq r_{i}$ and 
     $\sum_{i=1}^{m}s_{i}=u$. Put $s= s_{m}p^{m}+ \cdots + s_{1}p+b$. 
     Since $s_{i} \leq r_{i}$ and $b \leq r_{0}$ we have $s \leq r$. Also 
     $s \geq \sum_{i=1}^{m} s_{i} p = u p \geq p$. Clearly $s \equiv b \mod p$ 
     and $\Sp(s)=b+u$. By Lucas' theorem and choice of $s_{i}$, we have 
     $\binom{r}{s} \equiv \binom{r_{m}}{s_{m}} \cdots  \binom{r_{1}}{s_{1}} 
      \binom{r_{0}}{b} \not\equiv 0$  mod $p$.
     Further if  $\sum_{i=1}^{m}s_{i} = u < \sum_{i=1}^{m}r_{i}$, then $s_{j}<r_{j} $ 
     for some $j \geq 1$, whence  $r-s \geq (r_{j}-s_{j})p^{j} \geq p$.
\end{proof}
Next we determine when certain matrices built out of binomial 
coefficients are invertible mod $p$. 
These matrices are typical of the ones we encounter later.
 \begin{proposition}\label{matrix det}
       Suppose that $r \geq 2p$ and $1 \leq a \leq p-1$. 
     \begin{enumerate}[label=\emph{(\roman*)}]
	  \item If $0 \leq i \leq j \leq a \leq p-1 $, then the matrix
	        $$ 
	            \left( \binom{a-n}{j-m}\right)_{0 \leq m,n \leq i}  
	        $$
                is invertible modulo $p$.
	  \item If  $0 \leq i \leq j \leq i+j \leq a < p+i$, then
	        $$
	           \det_{0 \leq m,n \leq i } \left( \binom{r-n}{m} \binom{a-m-n}{j-m} \right)
	            = \binom{a-2i}{j-i} 
	            \prod \limits _{l=0}^{i-1} \frac{(a-i-l)!(r-(a-l))^{i-l}}{(j-l)!(a-j-l)!}.
	        $$
	       The corresponding matrix is invertible mod $p$ $\iff r \not \equiv a-i+1$, $a-i+2,\ldots$, 
	       $a-1$, $a \mod p$. 
	  \item If $1 \leq a-i \leq i$ and $r \not \equiv a-i-1 ~\mathrm{mod}~p$ and  
	            $r \not \equiv i+1,\ldots, a-1$, $a ~\mathrm{mod}~ p$, then  the matrix
	        $$ 
	          \left(\begin{array}{c|c}
	              A ' &  \mathbf{v}^{t} \\
	              \hline 
	             \mathbf{w} & 0
	          \end{array}
	          \right) 
                $$ 
                is invertible modulo $p$,
	        where $A'$ is the matrix $$\left( \binom{r-n}{m} 
	        \binom{a-m-n}{i-m} \right)_{0 \leq m,n \leq a-i-1}$$ and 
	        $\mathbf{v}$, $\mathbf{w} $ are the $1 \times a-i$ row vectors 
	        $\left( \binom{i}{0}, \ldots, \binom{i}{a-i-1} \right)$,
	        $  \left( \binom{r}{r-(a-i)}, \ldots, \binom{r-(a-i-1)}{r-(a-i)} \right)$, 
	        respectively.
	\end{enumerate}
\end{proposition} 
\begin{proof}
         We use elementary row and column operations to reduce the given matrices
         to a particular form to which we can apply results from \cite{Viennot}
         and \cite{Kra99}.
        \begin{enumerate}
               \item[(i)] By reversing the rows 
                and columns we have $ \det\limits_{0 \leq m,n \leq i} 
                \left( \binom{a-n}{j-m}\right)
                = \det\limits_{0 \leq m,n \leq i} \left( \binom{a-i+n}{j-i+m}\right)$. 
                Applying the formula on Line -7 of \cite[ p. 308]{Viennot} with 
                $k=(i+1)$, $b= j-i$ and  $a_{1}=a-i$, $a_{2}=a-i+1, \ldots$, 
                $a_{i+1}=a$ we get
                 \begin{align*}
                        \det_{0 \leq m,n \leq i} \left( \binom{a-i+n}{j-i+m}\right)
                         &= ((j-i)!)^{i+1} \frac{\binom{a-i}{j-i} \cdot \binom{a-i+1}{j-i}
                         \cdots \binom{a}{j-i}}{(j-i)! \cdot(j-i+1)!\cdots j!} 1! \cdot 2! \cdots i! \\
                          &=\frac{\binom{a-i}{j-i} \cdots\binom{a}{j-i}}{\binom{j-i}{j-i} \cdots
                         \binom{j}{j-i}}.  
                \end{align*}
               Since, for $0 \leq l \leq i$, we have $0 \leq j-i \leq j-l \leq a-l \leq p-1$, 
               it follows from  Lucas' theorem that the above determinant is non-zero 
               modulo $p$ and hence the  matrix is invertible.
               \item[(ii)]  Pulling out a factor of  $1/(m! (j-m)!)$ and
                $(a-i-n)!/(a-j-n)!$  from the $(m+1)^{th}$-row and the $(n+1)^{th}$-column 
                respectively,  for $0 \leq m,n \leq i$, we get
                \begin{align}\label{Matrix det eqn 1}
                 \det_{0 \leq m,n \leq i } \left( \binom{r-n}{m} \binom{a-m-n}{j-m} \right) 
            &= \prod\limits_{l=0}^{i} \frac{(a-i-l)!}{(a-j-l)!  (j-l)! l !} ~ \times \\
            & ~~~ \det_{0 \leq m,n \leq i } \left( \frac{(r-n)!(a-m-n)!}{(r-m-n)!(a-i-n)!} 
            \right).\nonumber
       \end{align}
        Applying \cite[Lemma 3]{Kra99} with $n=i+1$, and
        $$X_{1} = r, X_{2}=r-1,\ldots,X_{i+1} = r-i, $$
         $$A_{2}= a-r, A_{3} =a-r-1,\ldots, A_{i+1} = a-r-(i-1),$$  
        $$B_{2} =0, B_{3}=-1,\ldots, B_{i+1} = -(i-1),$$ 
        we get  the determinant of the transpose of 
        $\left( \frac{(r-n)!(a-m-n)!}{(r-m-n)! (a-i-n)!}\right)_{0 \leq m,n \leq i } $
        equals
         \begin{eqnarray*}              
               \prod_{ l=1}^{i+1}\prod_{1 \leq l'<l}  (X_{l'} - X_{l})  
              \times  
             \prod_{l=0}^{i-1} \prod_{\substack{ 2 \leq m \leq n \leq i+1 \\ n-m = l} }
                 (B_{m}-A_{n})  
            &=& \prod_{ l=1}^{i+1} (l-1)!  \prod_{l=0}^{i-1} (r-(a-l))^{i-l}.
          \end{eqnarray*}         
          Substituting the above expression in \eqref{Matrix det eqn 1}, we obtain the formula in (ii).
          The statement about the invertibility can be deduced as in
          (i).
          \item[(iii)] 
               If $a-i=1$, then the matrix is equal to
               $\begin{psmallmatrix}  a  & 1 \\  r & 0 \end{psmallmatrix}$,
               which is  invertible in $M_{2}(\f)$ if $p \nmid r$. So assume
               $a-i \geq 2$.
          	 Multiplying the $(n+1)^{th}$-column  by $(r-n+1)/(a-i-n+1)$ and subtracting
	          from the $n^{th}$-column, for $1 \leq n \leq a-i-1 $, we get
	       \begin{align*}
	            \det \left(\begin{array}{c|c}
	              A ' &  \mathbf{v}^{t} \\
	              \hline 
	             \mathbf{w} & 0
	          \end{array} \right) ~ = ~ -
	          \frac{(a-r)^{a-i-1}}{(a-i)!} (r-(a-i-1))
	          \det \left(\begin{array}{c|c}
	              A '' &  \mathbf{v}^{t}
	          \end{array} \right),    
	       \end{align*}
	       where  
	      $A'' = \left( \binom{r-n}{m} \binom{a-1-m-n}{i-m} \right)$ and
	        $0 \leq m \leq a-i-1$,  $0 \leq n \leq a-i-2$.
	       
	       Now multiplying the $(m+1)^{th}$-row by $m/(i-m+1)$  and subtracting
	        from the $m^{th}$-row, for $1 \leq m \leq a-i-1$, we get
	       \begin{align*}
	       \det
	          \left(\begin{array}{c|c}
	              A ' &  \mathbf{v}^{t} \\
	              \hline 
	             \mathbf{w} & 0
	          \end{array}
	          \right) ~  = ~-
	          \frac{(a-r)^{a-i-1}}{(a-i)!} & (r-(a-i-1)) \frac{(a-1-r)^{a-i-1}}{(a-i-1)!}
	          \binom{i}{a-i-1} \\
	          & \times \det\left(\binom{r-n}{m} \binom{a-2-m-n}{i-m} \right)_{0 \leq m,n 
	          \leq a-i-2}.
	      \end{align*}   
	      Now (iii) follows from (ii).  \qedhere
        \end{enumerate}
\end{proof}
Let us set
        \begin{align}\label{A(a,i,j,r) matrix}
               A(a,i,j,r) := \left( \binom{r-n}{m} \binom{[a-m-n]}{j-m} \right)_{0 \leq m, n \leq i},
         \end{align}
for $1 \leq a \leq p-1$ and $0 \leq i \leq j \leq r$. 
We have the  following corollaries.

\begin{corollary}\label{A(a,i,j,r) invertible}   
        Let $2p \leq r \equiv r_{0} ~\mathrm{mod}~p$ with $0 \leq r_{0} \leq p-1$
        and let $1 \leq a \leq p-1 $.
        Suppose that $0 \leq i \leq j < [a-i] < p-1$.
      Then the matrix $A(a,i,j,r)$ 
         is invertible if $r_{0} \not \in \mathcal{I}(a,i)$, where
         \begin{align}\label{interval I for i < a-i}
       \mathcal{I}(a,i) =
       \begin{cases}
              \lbrace a-i+1, a-i+2, \ldots, a-1,a \rbrace, & \mathrm{if}~
              i <[a-i] <a, \\
              \lbrace 0,1, \ldots a-1 \rbrace \cup 
              \lbrace p+a-i, p+a-i+1, \ldots, p-1 \rbrace, & \mathrm{if}~a<i<[a-i].
       \end{cases}
      \end{align}
\end{corollary}
\begin{proof}
        If $i<a$, then $[a-i] =a-i$ and  the condition $i < [a-i] $ 
        implies $ 2i < a$. Thus, for $i<a$ we have
        $[a-m-n] =a-m-n$, for all $0 \leq m,n \leq i$. 
        For $i>a$, we have
        \begin{align}\label{binomial identity i>a}
      \binom{[a-m-n]}{j-m}  & = 
      \begin{cases}
            \binom{p-1+a-m-n}{j-m}, ~ &\mathrm{if} ~ m+n \geq a, 
            \vspace*{1 mm} \\
            \binom{a-m-n}{j-m}, ~  &\mathrm{if} ~ m+n < a,
      \end{cases} \nonumber \\
      & \equiv \binom{p-1+a-m-n}{j-m} ~ \mathrm{mod}~p,
\end{align}
where the first equality follows  from the definition of [ $\cdot$ ] and the 
second  follows
from Lucas' theorem as the assumption $a<i \leq j$ implies 
$\binom{p-1+a-m-n}{j-m}$, $\binom{a-m-n}{j-m} \equiv 0$ mod $p$ 
in the case $m+n <a$.     
Let $a' = [a-i]+i$.  Note that $a'=a$ (resp. $p-1+a$) if 
$i<a$ (resp. $i>a$).    
Using these observations, we see that
        \begin{align}\label{A(a,i,j,r) expression}
                A(a,i,j,r) =
                \left( \binom{r-n}{m} \binom{a'-m-n}{j-m} 
                        \right)_{0\leq m,n \leq i}.      
        \end{align}
        Since $[a-i] \leq p-1$ and $j < [a-i]$, we see that  
        $0 \leq i \leq j \leq i+j \leq a' < p+i$.
        Using \Cref{matrix det} (ii), 
         it follows that $A(a,i,j,r)$ is invertible
          if $r \not \equiv  a'-i+1, a'-i+2, \ldots, a'-1,a' $ mod $p$
          if and only if   $r_{0} \not \in \mathcal{I}(a,i)$.
\end{proof}
\begin{corollary}\label{block matrix invertible}
	    Let 
	    $ 2p \leq r \equiv r_{0} ~\mathrm{mod}~p$
	    with $0 \leq r_{0} \leq p-1$ and $ 1 \leq a \leq p-1$.
	    Suppose $1 \leq [a-i] \leq i < p-1$. Let
	    $A' = A(a,[a-i]-1,i,r)$ 
       and  let
	   $\mathbf{v}$, $\mathbf{w} $ be the $1 \times [a-i]$ row vectors 
	   given by $\left( \binom{i}{0}, \ldots, \binom{i}{[a-i]-1} \right)$,
	   $  \left( \binom{r}{r-[a-i]}, \ldots, \binom{r-([a-i]-1)}{r-[a-i]} \right)$ 
	   respectively.  Then  the matrix
	   $$ 
	   \left(\begin{array}{c|c}
	   A ' &  \mathbf{v}^{t} \\
	   \hline 
	   \mathbf{w} & 0
	   \end{array}
	   \right) 
	   $$    
	   is  invertible  mod $p$  if 
	    $r \not \equiv [a-i]+i ~\mathrm{mod}~p$ and  
	   $r_{0} \not \in \mathcal{J}(a,i) \smallsetminus \lbrace i \rbrace $,
	   where 
	   \begin{align} \label{interval J first}
	           \mathcal{J}(a,i) =
	           \begin{cases}
	                     \lbrace a-i-1, a-i, \ldots, a-1 \rbrace, & \mathrm{if} ~
	                     [a-i] \leq i <a, \\
	                     \lbrace 0,1, \ldots, a-2 \rbrace \cup
	                     \lbrace p-2+a-i, \ldots , p-1 \rbrace, & \mathrm{if} ~
	                    a< [a-i] \leq i .
	           \end{cases}
	   \end{align}
\end{corollary}
\begin{proof}
        Observe that $[a-i]-1 < i < i+1 =[a-([a-i]-1)]$. 
        Let $a'=[a-i]+i$. By a check similar to  
       \eqref{A(a,i,j,r) expression}, we have
      \[
           A' = \left( \binom{r-n}{m} 
           \binom{a'-m-n}{i-m} \right)_{0 \leq m,n \leq [a-i]-1}.
      \]
      Indeed, let  $0 \leq m,n \leq [a-i]-1$.
      If $i < a$, then $a \leq 2i$, so  $a -m - n \geq 2i -a + 2\geq 2$,
      so $[a-m-m] = a-m-n$. If $i > a$, then $p-1 + a \leq 2i$, so if $m+n < a$, then $[a-m-n] = a-m-n$,
      but  $\binom{p-1+a-m-n}{i-m}$, $\binom{a-m-n}{j-m} \equiv 0$ mod $p$ by Lucas' theorem,  
      and if $m+n \geq a$, then $p-1+a-m-n \geq p-1 + a -2([a-i] -1) = 2i - (p-1) -a + 2 \geq 2$, so $[a-m-n] = p-1+a-m-n$.

      Also, by the definition of $[~ \cdot ~]$ (by considering the cases
      $i<a$ and $i>a$ separately), it follows that
      \begin{align*}
              \mathbf{v} &= \left( \binom{i}{0}, \ldots, \binom{i}{a'-i-1} \right) \\
              \mathbf{w} & =
              \left( \binom{r}{r-(a'-i)}, \ldots, \binom{r-(a'-i-1)}{r-(a'-i)} \right).
      \end{align*}
      Applying \Cref{matrix det} (iii) (with $a$ there equal to $a'$),
      we see that $A$ is invertible if $r \not \equiv a'-i-1$ mod $p$ and
      $r \not \equiv i+1, \ldots, a'-1, a'$ mod $p$ which happens if 
      $r \not \equiv a'$ mod $p$ and $r_{0} \not \in \mathcal{J}(a,i)
      \smallsetminus \lbrace i \rbrace$.
  \end{proof}
\section{The first \texorpdfstring{$p$}{} monomial submodules}

In this section, we determine the structure of the monomial submodules 
$X_{r-i,\,r}$, the $\Gamma$-submodule 
of $V_{r}$ generated by $X^{r-i}Y^{i}$, for $0 \leq i \leq p-1$. 
Recall that these modules are $M$-stable. 
 
We begin by describing an $\mathbb{F}_{p}$-generating  
 set for $X_{r-i, \,r}$, for $0 \leq i \leq p-1$. 
\begin{lemma}\label{Basis of X_r-i}
	  If $0 \leq i \leq p-1$, then $\lbrace X^{l}(kX+Y)^{r-l}, X^{r-l}Y^{l} :  0 \leq l \leq i,
	  ~ k \in \f  \rbrace$ generates $X_{r-i,\,r}$ as an $\f$-vector space.
	  Hence $\mathrm{dim}~X_{r-i,\,r} \leq (i+1)(p+1)$.
\end{lemma}
\begin{proof} 
      By Bruhat decomposition, $\Gamma= B \sqcup B w B$, where 
     $w=\begin{psmallmatrix} 0 & 1 \\ 1 & 0 \end{psmallmatrix}$.  We first look at the
      action of the  Borel subgroup $B$ on $X^{r-i} Y^{i}$. Observe that
    \begin{align*}
	   \begin{pmatrix} a & b \\ 0 & d \end{pmatrix} \cdot X^{r-i}Y^{i} = a^{r-i}  
	   X^{r-i}(bX+dY)^{i} =  a^{r-i} \sum\limits_{l=0}^{i} \binom{i}{l} b^{i-l}d^{l}  
	   X^{r-l} Y^{l}.
    \end{align*}	
    Therefore $B \cdot X^{r-i}Y^{i} \subset \mathbb{F}_{p}$-span of  
    $\lbrace X^{r}, X^{r-1}Y, \ldots ,X^{r-i}Y^{i} \rbrace$.  It is clear that any element
     of $Bw$ is of the form $\begin{psmallmatrix}  a & b \\ c & 0  \end{psmallmatrix}$
      with $bc \neq 0$.  For every $0 \leq l \leq i$, we have
	\begin{align*}
	   \begin{pmatrix} a & b \\ c & 0 \end{pmatrix} \cdot X^{r-l}Y^{l} =  (aX+cY)^{r-l} (bX)^{l}= 
	    b^{l} c^{r-l} X^{l}(ac^{-1}X+Y)^{r-l}.
	\end{align*}
	Hence
	\begin{align*}
	  BwB \cdot X^{r-i} Y^{i} & \subset \f \text{-span  of} ~\lbrace B w \cdot  X^{r-l} 
	  Y^{l} : 0 \leq l \leq  i  \rbrace  \\ & \subset \f \text{-span  of} ~ \lbrace X^{l}
	  (kX+Y)^{r-l} : 0 \leq l \leq  i , \ k \in \f \rbrace. 
	\end{align*}
	Combining these observations, 
        we get $\gamma \cdot X^{r-i} Y^{i} 
	\in \f \text{-span  of} ~\lbrace X^{l}(kX+Y)^{r-l}, X^{r-l}Y^{l} :  0 \leq l \leq i,
	  ~ k \in \f  \rbrace$. 
        This
	completes the proof of the lemma as $X^{r-i}Y^{i}$ generates $X_{r-i,\,r}$ 
	as a $\Gamma$-module.
\end{proof} 
 We next define a surjection 
$X_{r-i,\,r-i} 	\otimes V_{i} \twoheadrightarrow X_{r-i,\,r}$,
for $0 \leq i \leq p-1$,
which generalizes \cite[Lemma 3.6]{BG15}.
\begin{lemma}\label{surjection1}
For  $r \geq i$ and $0 \leq i \leq p-1$, there exists an  $M$-linear
 surjection 
	$$\phi_{i}: X_{r-i,\,r-i} 	\otimes V_{i} \twoheadrightarrow X_{r-i,\,r}.$$ 
\end{lemma}
\begin{proof}
 The map $\phi_{r-i, \, i}:V_{r-i} \otimes V_{i} \rightarrow V_{r}$ sending 
 $F \otimes G \mapsto FG$ 
for  $F \in V_{r-i}$ and $G \in V_{i}$, is $M$-linear by \cite[(5.1)]{Glover}. 
 Let $\phi_{i}$ be the restriction of $\phi_{r-i,\, i}$ to the $M$-submodule  
 $ X_{r-i,\,r-i}  \otimes V_{i}  \subseteq  V_{r-i} \otimes V_{i}$.  As an $M$-module
  $X_{r-i, \, r-i} \otimes V_{i}$ is generated by  $X^{r-i} \otimes X^{i}$, $X^{r-i} 
  \otimes X^{i-1}Y, \ldots ,X^{r-i} \otimes Y^{i}$ whose images $X^{r}$, $X^{r-1}Y, 
  \ldots, X^{r-i}Y^{i}$ lie in $X_{r-i, \, r}$, by \Cref{first row filtration}. 
  The surjectivity follows as $\phi_{i}(X^{r-i}
   \otimes Y^{i}) = X^{r-i}Y^{i}$ generates $X_{r-i,\,r}$ as an $M$-module. 
\end{proof}
We next  define a $\Gamma$-linear 
surjection from an induced representation
to the quotient $X_{r-i,\,r}/X_{r-j,\,r}$, for $0 \leq j \leq i \leq p-1$. 
This map will be crucially used in later sections to 
obtain the structure of  $X_{r-i,\,r}$. 

\begin{proposition}\label{Succesive quotient}
Let  $0 \leq j \leq i \leq p-1 < r$. Then there is a 
$\Gamma$-linear surjection  
$$ 
 \ind_{B}^{\Gamma}(V_{i-j-1} \otimes \chi_{1}^{r-i} \chi_{2}^{j+1}) 
 \twoheadrightarrow  X_{r-i, \, r} / X_{r-j,\,r},
$$
where $V_{-1}=0$ by convention.
\end{proposition}
\begin{proof}
       Let $\chi = \chi_{1}^{r-i} \chi_{2}^{j+1}$ and $e_{\chi}$ 
       be a non-zero
       element in  the representation $( \chi, V_{\chi})$. If $i=j$, then 
       $X_{r-i, \, r} / X_{r-j,\,r}=0$ and there is nothing prove. Assume
       $i> j$. Consider $\f$-linear map $\psi : V_{i-j-1} \otimes 
       \chi_{1}^{r-i} \chi_{2}^{j+1} \rightarrow X_{r-i, \, r} / X_{r-j,\,r}$ defined by 
       $$
             X^{i-j-1-l} Y^{l} \otimes e_{\chi} \mapsto \binom{j+l+1}{j+1}^{-1} 
             X^{r-j-l-1} Y^{j+l+1}, 
                   ~ \forall ~ 0 \leq l \leq i-j-1.
       $$
       For $0 \leq l \leq i-j-1$, we have
       $0 \leq j+1 \leq  j+l+1 \leq i \leq p-1$, whence by Lucas' theorem,
       $\binom{j+l+1}{j+1} \not \equiv 0$ mod $p$, so $\psi$ is
       well defined. We claim that $\psi$  is a $B$-linear map. For $0 \leq n \leq i-j-1$ and $\gamma = \begin{psmallmatrix} 
       a & b \\ 0 & d \end{psmallmatrix} \in B$, we have 
       \begin{align*}
              \gamma \cdot (X^{i-j-1-n} Y^{n} \otimes e_{\chi})  & =
              (aX)^{i-j-1-n }(bX+dY)^{n} \otimes a^{r-i}d^{j+1} e_{\chi} \\
              &= \sum_{l=0}^{n} a^{r-j-n-1} b^{n-l} d^{j+l+1} \binom{n}{l}  
                   \left( X^{i-j-1-l } Y^{l} \otimes e_{\chi} \right).
       \end{align*}
       Therefore
      \begin{align*}
              \psi \left( \gamma \cdot (X^{i-j-1-n} Y^{n} \otimes e_{\chi}) \right)
              & =  \sum\limits_{l=0}^{n} a^{r-j-n-1}b^{n-l}d^{j+l+1} \binom{n}{l}
                    \binom{j+l+1}{j+1}^{-1} X^{r-j-l-1} Y^{j+l+1} \\
              & =  \sum\limits_{l=j+1}^{j+n+1} a^{r-j-n-1}b^{j+n+1-l}d^{l} \binom{n}{l-(j+1)}
                    \binom{l}{j+1}^{-1} X^{r-l} Y^{l} \\      
              &=\binom{j+n+1}{j+1}^{-1} \sum\limits_{l=j+1}^{j+n+1} a^{r-j-n-1}b^{j+n+1-l}
              d^{l} \binom{j+n+1}{l}   X^{r-l} Y^{l}  \\
              &=\binom{j+n+1}{j+1}^{-1} (aX)^{r-j-n-1}(bX+dY)^{j+n+1}   \\    
              & \quad -\binom{j+n+1}{j+1}^{-1}  \sum\limits_{l=0}^{j} a^{r-j-n-1}b^{j+n+1-l}
              d^{l} \binom{j+n+1}{l}   X^{r-l} Y^{l}  \\
              & =  \binom{\,j+n+1}{j+1}^{-1} \begin{pmatrix} a & b \\ 0 & d \end{pmatrix}
                    \cdot X^{r-j-n-1}Y^{j+n+1} \quad \mod X_{r-j,\,r} \\
               & =  \gamma \cdot \psi( X^{i-j-1-n} Y^{n} \otimes e_{\chi}).                
       \end{align*}        
       This shows that $\psi$ is $B$-linear. By  Frobenius reciprocity (alternatively  \cite[Lemma 4, $\S 8$]{Alperin}),
       we see that $\psi$ extends to a 
       $\Gamma$-linear map   $\ind_{B}^{\Gamma}(V_{i-j-1} 
       \otimes \chi_{1}^{r-i} \chi_{2}^{j+1}) \rightarrow X_{r-i, \, r} / X_{r-j ,\,r}$.
       As $X^{r-i}Y^{i} = \binom{i}{j+1}\psi(Y^{i-j-1} \otimes e_{\chi})$
        generates $X_{r-i,\,r}$ as  a $\Gamma$-module, the surjectivity of $\psi$ follows.
\end{proof}
In particular, taking $j=i-1$ in the above proposition, we see the successive quotients 
$X_{r-i,\,r}/ X_{r-(i-1),\,r}$ are isomorphic to  quotients of principal series
representations. 
\begin{corollary}\label{induced and successive}
     If $r \geq p$ and $1 \leq i \leq p-1$, then the  map 
     \begin{align*}
            \psi_{i} : \ind_{B}^{\Gamma}(\chi_{1}^{r-i} \chi_{2}^{i}) 
            ~ & \longrightarrow ~ X_{r-i,\,r}/ X_{r-(i-1),\,r} \\
             [\gamma, e_{\chi_{1}^{r-i} \chi_{2}^{i}}] 
            ~ &\longmapsto ~ \gamma \cdot X^{r-i}Y^{i}
     \end{align*}
     is a $\Gamma$-linear surjection.
\end{corollary}

 \subsection{The case  \texorpdfstring{$\boldsymbol{r_{0} \geq i}$}{}}
 \label{section r0>i}
 In this subsection, we determine the structure of $X_{r-i,\,r}$,
 for $0 \leq i \leq p-1$, in the case $r_{0} \geq i$, where $r_{0}$ is
 as in \eqref{base p expansion of r}. 
 
The structure of $X_{r-1,\,r}$ was determined in \cite{BG15} using the surjection 
$\phi_{1} : X_{r-1,\,r-1} \otimes V_{1} \rightarrow X_{r-1,\,r}$. 
It turns out that the map $\phi_{1}$ is an isomorphism if $p \nmid r$, i.e.,
$r_{0} \geq 1$ in \eqref{base p expansion of r}. This fact can be  
verified using \Cref{dimension formula for X_{r}} in conjunction with
\cite[Proposition 3.13]{BG15} and \cite[Proposition 4.9]{BG15} 
(see also \Cref{Structure r_0 >i}, noting that the first case there does not occur when $i = 1$).
So one might  expect that the  surjection $\phi_{i}$ obtained in  
\Cref{surjection1} is an isomorphism in the case $r_{0}\geq i$,
for arbitrary $i$. 
We show that this is indeed true if $r_{0} \geq i$ and $\Sp(r-i) \geq r_{0}$, 
by showing  dim $X_{r-i,\,r}$
equals dim $X_{r-i,\,r-i} \otimes V_{i}$. See 
\Cref{dim large 1} and \Cref{not equal and not full 1} for details. 
Furthermore,  in \Cref{terms equal in filtration r0 >i}, we show that $\Sp(r-i) < r_{0}$ 
if and only if $X_{r-i,\,r} = X_{r-(i-1),\,r}$. Applying this successively allows us 
to reduce to the case just treated by replacing $i$ with $\Sp(r-r_{0})$,
see \Cref{Structure r_0 >i}.
We remark that the case $\Sp(r-i) < r_{0}$ never happens if $i=1$ and $r\geq p$, by \cite[Lemma 4.1]{BG15}.

We first  determine  a necessary 
condition for equality $X_{r-i, \, r} = X_{r-j, \, r}$, for 
$0 \leq j \leq i < p-1$.  For $r_{0} \geq i$, note that $\Sp(r-i) =\Sp(r)-i$.

 
\begin{lemma}\label{X_{r-i}=X_{r-j}}
      Let $p \leq r= r_{m}p^{m}+r_{m-1}p^{m-1}+ \cdots + r_{0}$ be the base 
      $p$-expansion of $r$.  Let $ 1 \leq j < i  \leq p-1$ with $r_{0} \geq i$. If  
      $r_{m}+r_{m-1} \cdots +r_{1} >  j$, then $X_{r-j, \, r} \neq X_{r-i, \, r}$.  
      In particular, if $ \Sigma_{p}(r- r_{0}) = r_{m}+r_{m-1} \cdots +r_{1} \geq i$, then $X_{r-(i-1), \, r} \neq X_{r-i, \, r}$.
\end{lemma}
\begin{proof}
       If $X_{r-j,\,r} = X_{r-(j+1),\,r}$, then 
       by \Cref{Basis of X_r-i},
         there exist $a_{k,l} \in \f$ and $b_{l} \in \f$,  for $k = 0$, $1, \ldots,p-1 $
         and $l = 0$, $1, \ldots, j$, such that
       \begin{equation}\label{expression r_0>i 1}
             X^{r-j-1}Y^{j+1} = \sum_{l=0}^{j} \sum_{k=0}^{p-1} a_{k,l} X^{l}(kX+Y)^{r-l} + 
             \sum_{l=0}^{j}b_{l} X^{r-l}Y^{l}.
       \end{equation}
       For every positive integer $t$, put $A_{t,l} := \sum\limits_{k=1}^{p-1} a_{k,l} k^{r-l-t} $. 
       Comparing the coefficients of $X^{r-t}Y^{t}$ on both sides of  
       \eqref{expression r_0>i 1}, 
       we get 
       \begin{equation}\label{compare coeff r_0 >i 1}
             \sum_{l=0}^{j} \binom{r-l}{t} A_{t,l} = \delta_{j+1,t}, ~ 
             \forall ~ j < t < r-j.
       \end{equation}
       For every $1 \leq s \leq j+1 $, choose $0 \leq t_{n,s} \leq r_{n}$ for $1 \leq n 
       \leq m$ such that $\sum\limits_{n=1}^{m} t_{n,s} =s$. Put $t_{s}= t_{m,s}p^{m}+ 
       \cdots + t_{1,s}p+(j+1-s)$. Clearly  $\Sp(t_{s}) = j+1$ and
       $t_{s} \equiv \Sp(t_{s})\equiv j+1$ mod ($p-1$). Since $ j <  \Sp(t_{s}) \leq t_{s}$ 
        and $r-t_{s} \geq \sum\limits_{n=1}^{m} (r_{n} -t_{n,s} ) p + r_{0} -(j+1-s)
        \geq (j+1-s)(p-1)+r_{0} \geq i > j$, we get $j< t_{s} <r-j$. 
        By \eqref{compare coeff r_0 >i 1},  and 
        $A_{t,l} = A_{t',l}$ if $t \equiv t'$ mod $(p-1)$, we have
       \begin{equation}\label{eq 3.18}
              \sum\limits_{l=0}^{j} \binom{r-l}{t_{s}} A_{j+1,l}
              = \sum\limits_{l=0}^{j} \binom{r-l}{t_{s}} A_{t_{s},l} = 0.
       \end{equation}
       Applying Lucas' theorem we get $\binom{r-l}{t_{s}} \equiv 
       \binom{r_{m}}{t_{m,s}} \cdots \binom{r_{1}}{t_{1,s}} 
       \binom{r_{0}-l}{j+1-s} $ mod $p$, 
       $\forall ~ 0 \leq l \leq j$. Substituting  this in \eqref{eq 3.18} 
       and dividing  both sides by
       $ \binom{r_{m}}{t_{m,s}} \cdots \binom{r_{1}}{t_{1,s}}$,  we obtain
       $$
             \sum_{l=0}^{j} \binom{r_{0}-l}{j+1-s} A_{j+1,l} =0 ,~ \forall ~ 1
              \leq s \leq  j+1.
       $$
       Putting these set of equations in matrix form, we get 
       $$
            \begin{pmatrix} \binom{r_{0}}{j} & \binom{r_{0}-1}{j} & \cdots & 
             \binom{r_{0}-j}{j} \\
            \binom{r_{0}}{j-1} & \binom{r_{0}-1}{j-1} & \cdots &  \binom{r_{0}-j}{j-1} \\
            \vdots & \vdots & \ddots & \vdots \\
            \binom{r_{0}}{0} & \binom{r_{0}-1}{0} & \cdots &  \binom{r_{0}-j}{0}
             \end{pmatrix} 
            \begin{pmatrix} A_{j+1,0} \\ A_{j+1,1} \\ \vdots \\ A_{j+1,j} \end{pmatrix}
            =  \begin{pmatrix} 0 \\ 0 \\ \vdots \\ 0 \end{pmatrix} .
       $$
       Applying  \Cref{matrix det} (i) (with $a=r_{0}$ and $i=j$), 
       we obtain that the above matrix is invertible, whence 
        $A_{j+1,0} = A_{j+1,1} = \cdots = A_{j+1,j}=0$.
        Taking $t=j+1$ in \eqref{compare coeff r_0 >i 1}  we get
       $$
               \sum\limits_{l=0}^{j} \binom{r-l}{j+1} A_{j+1,l} = 1,
       $$
      which leads to a contradiction. 
      Hence $X_{r-j,\,r} \subsetneq X_{r-(j+1),\,r} \subseteq X_{r-i,\,r}$. 
      This  proves the  lemma.
\end{proof}     
 %
 By \Cref{Basis of X_r-i}, we know that dim $X_{r-i,r} \leq (i+1)(p+1)$,
 for all $0 \leq i \leq p-1$. We next show  that this inequality is indeed an
 equality if $\Sigma_{p}(r)$ is large. More precisely:  
\begin{proposition}\label{dim large 1}
	  Let $p \geq 3$, $0 \leq i \leq p-1$ and $p \leq r 
	  \equiv r_{0} \mod p$ with $0 \leq r_{0} \leq p-1$. If $r_{0} \geq i$ and  
	  $\Sigma_{p}(r-i) \geq  p$, then $\mathrm{dim}~X_{r-i,r} = (i+1)(p+1)$.  
	  As a consequence, we have $X_{r-i,\,r} \cong X_{r-i,\,r-i} \otimes V_{i}$ 
	  as $M$-modules.
\end{proposition}
\begin{proof}
      The last assertion is an easy consequence of the first assertion and 
      \Cref{surjection1}, noting that dim($X_{r-i,\,r-i} \otimes V_{i}) 
      \leq (i+1)(p+1)$.  
	  We prove the proposition by induction on $i$. The cases $i=0$, $1$ 
	   follow from Lemmas \ref{dimension formula for X_{r}}, \ref{dimension formula for
	    X_{r-1}} respectively. Assume that the result holds for all $j < i$, for some  
	    $2 \leq i \leq p-1$. We need to prove the proposition is true for $i$. By 
	    \Cref{Basis of X_r-i}, the vectors
	  $\{ X^{l}(kX+Y)^{r-l}, X^{r-l}Y^{l} : k \in \f, 0 \leq l \leq i \} $	span 
	  $X_{r-i, \, r}$ as an $\f$-vector space. We claim that they are $\f$-linearly
	   independent. Suppose there  exist  constants $A_{l}, B_{l}, c_{k,l} \in \f $ 
	   for $l=0$, $1,\ldots, i$ and  $k= 1$, $2, \ldots, p-1 $ such that
	\begin{align}\label{3.4}
	      \sum\limits_{l=0}^{i} A_{l} X^{r-l}Y^{l} +  \sum\limits_{l=0}^{i} B_{l} X^{l} Y^{r-l} + 
	      \sum\limits_{l=0}^{i}  \sum\limits_{k =1}^{p-1} c_{k,l} X^{l}(kX+Y)^{r-l} =0. 
	\end{align}
	We need to  show that $A_{l},B_{l},c_{k,l}=0$ for all $k$, $l$. 
	For $0 \leq l \leq i$  and $t \in \mathbb{Z}$, let 
	$C_{t,l} := \sum\limits_{k=1}^{p-1}  k^{r-l-t}c_{k,l}$.
	Note that $C_{t,l}$ depends only on the congruence class of $t$ mod $(p-1)$. 
	By the non-vanishing of the Vandermonde determinant, we have 
	for every $l$, $C_{t,l} =0$, for all $t$, if and only if 
	$c_{k,l}=0$   for all $k$.  Comparing the coefficients of $X^{r-t}Y^{t}$ on
	both sides of \eqref{3.4}, we get 
	\begin{align}\label{3.5}
	      \sum\limits_{l=0}^{i} \delta_{t,l} A_{l} + \sum\limits_{l=0}^{i} \delta_{r-l,t} B_{l} +
	       \sum\limits_{l=0}^{i} \binom{r-l}{t} C_{t,l}=0.
	\end{align}
	\noindent
	\underline{Claim}: $C_{1,0} = C_{2,0} = \cdots = C_{p-1,0} =0$. \\
	  Assuming the claim, we complete the proof of the proposition. 
	  Taking $t=r$ in \eqref{3.5} and noting that   $C_{1,0} = C_{2,0} 
	  = \cdots = C_{p-1,0} =0$ by the claim, we get $B_{0}=0$. 
	  Also note that by the claim we have 
	  $c_{1,0} = c_{2,0} = \cdots = c_{p-1,0}$. Thus dividing both sides 
	  of  \eqref{3.4} by $X$, we get 
	\begin{align}\label{reducing 3.4}
    	   \sum\limits_{l=0}^{i} A_{l} X^{r-1-l}Y^{l} +  \sum\limits_{l=0}^{i-1} B_{l+1}
    	    X^{l} Y^{r-1-l} +  \sum\limits_{l=0}^{i-1} 
    	    \sum\limits_{k =1}^{p-1} c_{k,l+1} X^{l}(kX+Y)^{r-1-l} = 0. 
	\end{align}
	Let  $r=r_{m}p^{m}+ \cdots + r_{1}p+r_{0}$ be 
	the base $p$-expansion  of  $r$.  Then 
	 $r-1 = r_{m}p^{m}+ \cdots + r_{1}p+(r_{0}-1)$. 
	 Since $r_{m}+ \cdots +r_{1} = \Sp(r)-r_{0} \geq  p+i-r_{0} > i$,
	    we get $r > p$ and $X_{r-1-(i-1), \,r-1} \neq X_{r-1-i, \,r-1} $, 
	    by \Cref{X_{r-i}=X_{r-j}}. This forces $A_{i}=0$, as otherwise 
	    $X_{r-1-(i-1), \,r-1} =X_{r-1-i, \,r-1}$. 
	    By induction for $r-1$, we have dim $X_{r-1-(i-1), \,r-1} = i(p+1)$.
	   Hence  $\lbrace X^{r-1-l}Y^{l}, X^{l}(kX+Y)^{r-1-l} : k \in \f, \ 0 \leq l 
	   \leq i-1  \rbrace$ is an $\f$-basis of $X_{r-1-(i-1), \,r-1} $ by  
	   \Cref{Basis of X_r-i}. Thus  $A_{l}=B_{l}=0$ and $c_{k,l} =0$ for all $k,l$
	   by \eqref{reducing 3.4} as $A_{i} =0$.
	    Therefore dim $X_{r-i,\,r}=(i+1)(p+1)$. 
	    
	\noindent   
	\underline{Proof of the claim} :
	We first show that  $C_{1,0} , \ldots ,C_{i,0}=0$ and 
	$C_{r_{0}+1,0}, \ldots , C_{p-1,0}=0 $.
	Let $s$ be a positive integer congruent to $ r_{0}$ mod $p$. 
	By Lucas' theorem,  for $0 \leq l \leq i \leq r_{0} $, we see that
	 $ \binom{r-l}{s} \not \equiv 0~  \mathrm{mod}  ~p \Rightarrow r-l  \equiv r_{0}$, 
	     $r_{0}+1,\ldots, p-1 ~ \text{mod} ~ p  \Leftrightarrow l =0$.
	 Taking $t=s$ and  applying this in \eqref{3.5},  we get 
	\begin{align}\label{3.11}
     	    \sum\limits_{l=0}^{i} \delta_{l,s} A_{l} +
     	      \sum\limits_{l=0}^{i}\delta_{r-l,s} B_{l} +  \binom{r}{s} C_{s,0}=0.
	\end{align}
	By \Cref{choice of s}, for every $1 \leq u \leq p+i-r_{0}-1 \leq \Sp(r) -r_{0}-1$, 
	there exists $p \leq s_{u} \leq r-p$
	 such that $s_{u} \equiv r_{0} $ mod $p$, $\Sp(s_{u})= r_{0}+u$ and $\binom{r}{s_{u}}
	  \not \equiv 0$ mod
	 $p$.  Therefore  $\delta_{l,s_{u}} = 0 = \delta_{r-l,s_{u}}$, for all $0\leq l \leq i$.
	  So  \eqref{3.11}  
	 implies that $C_{s_{u},0}=0$, for all $1 \leq u \leq p+i-r_{0}-1$. Since 
	 $s_{u} \equiv \Sigma_{p}(s_{u})=u+r_{0}$ mod $(p-1)$, 
	 we get $C_{u+r_{0},0}=0$, for all $1 \leq u \leq p+i-r_{0}-1$.  
	 So $C_{1,0} , \ldots ,C_{i,0}=0$  and $C_{r_{0}+1,0}, \ldots , C_{p-1,0}=0 $.
	 This finishes the proof of the claim, if $r_{0}=i$.
	
	 Else choose $t\geq 1$ such that  $ i+1 \leq t+i \leq  r_{0}$.  Clearly 
	 $i < t+i \leq r_{0} \leq p-1 < \Sigma_p(r-i) \leq r-i$. Since $t+i \leq r_{0} \leq p-1$,  by Lucas' theorem, we
	 have $\binom{r-l}{t+i} \equiv \binom{r_{0}-l}{t+i}$ mod $p$, $\forall$ 
	 $0 \leq l \leq i \leq r_{0}$. 
	  Therefore,  by \eqref{3.5}, we have 
	 %
	%
	\begin{align}\label{3.12}
    	    \sum\limits_{l=0}^{i} \binom{r_{0}-l}{t+i} C_{t+i,l}=0.
	\end{align} 
	 For every $0 \leq w \leq i-1$, note that $0 \leq w+t \leq t+i \leq r_{0}$ 
	 and $1 \leq i-w < i+1 \leq \Sp(r)-r_{0}$. Thus
	  by \Cref{choice of s} (applied with $b= w+t$ and $u=i-w$), 
	  there exists  $ p \leq s_{w} \leq r-p$
	  such that $s_{w} \equiv w+t$ mod $p$, $\Sp(s_{w}) = t+i$ and 
	 $\binom{r}{s_{w}} \not \equiv 0$ mod  $p$, for all $0 \leq w \leq i-1$.  
	 Let $s_{w} = s_{w,m}p^{m}+\cdots+
	  s_{w,1} p+ (w+t)$ be the base $p$-expansion of 
	$s_{w}$.  By Lucas' theorem, we get 
	$\binom{r-l}{s_{w}} \equiv \binom{r_{0}-l}{w+t}  \prod\limits_{j=1}^{m} \binom{r_{j}}
	{s_{w,j}}  \equiv \binom{r_{0}-l}{w+t} \binom{r_{0}}{w+t}^{-1} \binom{r}{s_{w}} $ 
	mod $p$, for $0 \leq l \leq i \leq r_{0}$. Noting  
	$s_{w} \equiv \Sp(s_{w}) \equiv t+i$ mod $(p-1)$ and 
	$\binom{r_{0}}{w+t}$, $\binom{r}{s_{w}} \not \equiv 0$ mod $p$, for
	$0 \leq w \leq i-1 $, it follows from \eqref{3.5} that
	%
	\begin{align*}
	        \sum\limits_{l=0}^{i} \binom{r_{0}-l}{w+t} C_{t+i,l} = 0, \quad \quad 
	        \forall \; \ 0 \leq w \leq i-1.
	\end{align*}  
	Combining the above set of equations with \eqref{3.12}, we get
	\begin{align*}
    	     \begin{pmatrix}
    	           \binom{r_{0}}{t+i} & \binom{r_{0}-1}{t+i}  &\cdots &\binom{r_{0}-i}{t+i} \\
    	           \binom{r_{0}}{t+i-1} & \binom{r_{0}-1}{t+i-1}  &\cdots & \binom{r_{0}-i}{t+i-1} \\
    	           \vdots & \vdots & \ddots & \vdots \\
    	           \binom{r_{0}}{t} & \binom{r_{0}-1}{t} & \cdots &\binom{r_{0}-i}{t} \\
    	     \end{pmatrix}
    	     \begin{pmatrix}
    	            C_{t+i,0} \\ C_{t+i,1} \\ \vdots \\ C_{t+i, i}
    	     \end{pmatrix}
    	     = 
    	     \begin{pmatrix}
    	         0 \\ 0 \\ \vdots \\ 0
    	     \end{pmatrix}.
	\end{align*}
	%
	By \Cref{matrix det} (i) (with $a=r_{0}$ and $j=t+i$), the  above matrix is invertible,
	whence $C_{t+i,0}=0$, where $i+1 \leq t+i \leq r_{0}$. This finishes 
	the proof of the claim, as we have already shown $C_{1,0} , \ldots ,C_{i,0}=0$ and 
	$C_{r_{0}+1,0}, \ldots , C_{p-1,0}=0 $.
\end{proof}	
We next consider the case $r_{0} \leq \Sp(r-i) \leq p$ and still show that
$X_{r-i,\,r} \cong   X_{r-i,\,r-i} \otimes V_{i} $ as $M$-modules. Note that 
the case $\Sp(r-i) = p$ was treated 
in the proposition above. 
\begin{proposition}\label{not equal and not full 1}
       Let $p\geq 3$, $ 0 \leq i \leq p-1$ and $p \leq r 
        \equiv r_{0} \mod p$ with 
       $0 \leq r_{0} \leq p-1$. If $r_{0} \geq i$ and $r_{0} \leq \Sp(r-i) \leq p$, then
       $\dim X_{r-i,\,r} = (i+1)(\Sp(r-i)+1)$. Furthermore, $
       X_{r-i,\,r} \cong   X_{r-i,\,r-i} \otimes V_{i} $ as $M$-modules.
\end{proposition}
\begin{proof}
       We prove the proposition by induction on $i$.  
       The case $i=0$ follows from 
       \Cref{dimension formula for X_{r}}. Assume that the proposition holds for all
       $j < i$ for some $i \geq 1$.  If $\Sp(r-i) =p$, then  as remarked earlier 
       the proposition follows from \Cref{dim large 1}. In the case $r_{0} \leq \Sp(r)-i \leq p-1$, 
       by  \Cref{X_{r-i}=X_{r-j}}  and \Cref{induced and successive}, 
       we have  $X_{r-i, \, r}/X_{r-(i-1),\,r} $ is isomorphic to a non-zero 
       quotient of $\ind_{B}^{\Gamma}(\chi_{1}^{r-i}\chi_{2}^{i})$.
       Since $r \equiv \Sp(r)$ mod $(p-1)$ and  
       $2i \leq r_{0}+i  \leq \Sp(r) < p-1+2i$, we see that 
       $[2i-r] = [2i- \Sp(r)] = p-1+2i-\Sp(r)$. Thus by \Cref{Structure of induced},
       we have any non-zero quotient  of 
       $\ind_{B}^{\Gamma}(\chi_{1}^{r-i}\chi_{2}^{i})$ has dimension
        at least $(\Sp(r)-2 i+1)$ (note that the quantity $(\Sp(r)-2 i+1)$ equals one if 
        $\Sp(r)=2i$). Since $r_{0}+1 \leq \Sp(r-(i-1)) \leq p$, by induction we have
        \begin{align*}
            \dim X_{r-i, \, r} 
            & =  \dim X_{r-(i-1), \, r} + \dim\left(\frac{X_{r-i, \, r}}{X_{r-(i-1),\,r} }\right) \\
            & \geq i(\Sp(r)-i+2)+ (\Sp(r)-2i+1) \\
            & = (i+1)\left( \Sp(r)-i +1 \right).
        \end{align*}             
        Since $\Sp(r-i) \leq p-1$, by \Cref{dimension formula for X_{r}}, 
        we have $\dim X_{r-i,\,r-i} = \Sp(r-i)+1  = 
        \Sp(r)-i+1$. 
          So  $\dim X_{r-i,\,r-i} \otimes V_{i} $ $\leq (i+1)( \Sp(r)-i +1 )
       \leq \dim X_{r-i, \, r} $. Now the proposition follows from \Cref{surjection1}.
\end{proof}
We now prove the converse of \Cref{X_{r-i}=X_{r-j}}.
\begin{lemma}\label{terms equal in filtration r0 >i}
      Let $p \leq r= r_{m}p^{m}+r_{m-1}p^{m-1}+ \cdots + r_{0}$ be the base 
      $p$-expansion of $r$ with $r_{0} \geq i$.  For $ 1 \leq j < i 
      \leq p-1$, we have $X_{r-j, \, r} = X_{r-i, \, r}$  
      if and only  if  $r_{m}+r_{m-1} \cdots +r_{1} \leq j$.
      In particular, $X_{r-(i-1), \, r} = X_{r-i, \, r} \Longleftrightarrow
       r_{m}+r_{m-1} \cdots +r_{1} < i \Longleftrightarrow 
       \Sp(r-i) < r_{0}$.
\end{lemma}
\begin{proof}
        The  \enquote*{only if} part is \Cref{X_{r-i}=X_{r-j}}. 
        For the \enquote*{if} part, assume $r_{m}+r_{m-1} \cdots +r_{1} \leq j$.  
        By   \Cref{dimension formula for X_{r}}, we have 
        $\dim X_{r-r_{0},\,r-r_{0}} = \Sp(r-r_{0})+1$. Thus, 
        by \Cref{surjection1}, we have  $\dim X_{r-r_{0},\,r}  \leq (r_{0}+1)
        (\Sp(r-r_{0})+1)$.  Since $\Sigma_{p}(r-r_{0}) \leq j \leq r_{0}$,  
        we have $\Sigma_{p}(r- \Sigma_{p}(r-r_{0})) 
        = \Sigma_{p}(r)-\Sigma_{p}(r-r_{0})=\Sigma_{p}(r)-\Sigma_{p}(r)+r_{0}
        = r_{0} \leq p-1$. Hence, by \Cref{not equal and not full 1} (with 
        $i$  there equal to $\Sigma_{p}(r-r_{0})$), 
        we have
        $\dim X_{r- \Sp(r-r_{0}), \, r} = (\Sp(r-r_{0})+1)(r_{0}+1)$. As
        $0 \leq \Sigma_{p}(r-r_{0}) \leq j < i \leq r_{0} \leq p-1 $, 
        by \Cref{first row filtration}, we see that
        $X_{r- \Sp(r-r_{0}), \, r} \subseteq X_{r-j,\,r} \subseteq 
        X_{r-i,\,r} \subseteq  X_{r-r_{0},\,r}$.
         As the  dimension of the rightmost term is less than or equal to  the 
         dimension of the leftmost term,
          it follows that   $ X_{r-j,\,r} = X_{r-i,\,r}$.
\end{proof}
%
Putting together all the results obtained so far, we have the following theorem.
\begin{theorem}\label{Structure r_0 >i}
        Let $p \geq 3$, $ 0 \leq i \leq p-1$ and $p    
        \leq r \equiv r_{0} 
        \mod p$ with  $0 \leq r_{0} \leq p-1$. If $r_{0} \geq i$, then
        as $M$-modules, we have
        \begin{align*}
               X_{r-i,\,r} \cong
               \begin{cases}
                        X_{r- \Sp(r-r_{0}),\, r-\Sp(r-r_{0})}  \otimes V_{\Sp(r-r_{0})}, &
                         \mathrm{if} ~ \Sp(r-i) < r_{0}, \\
                        X_{r-i, \, r-i} \otimes V_{i}, & \mathrm{if} ~ \Sp(r-i) 
                        \geq r_{0}.
               \end{cases}
        \end{align*}
\end{theorem}
\begin{proof}
        Note that  $\Sp(r-i) = \Sp(r)-i = \Sp(r-r_{0})+r_{0}-i$.
        The  case  $\Sp(r-i) \geq  r_{0}$ follows immediately from  
        \Cref{dim large 1} and  \Cref{not equal and not full 1}.
        If $\Sp(r-i)< r_{0}$, then by \Cref{terms equal in filtration r0 >i}, we have 
        $X_{r-\Sp(r-r_{0}),\,r} = X_{r-i,\,r} $. Also note that
        $r_{0} \geq i > \Sp(r-r_{0})$ and  $\Sp(r- \Sp(r-r_{0})) =
        \Sp(r) - \Sp(r-r_{0})  = \Sp(r) - \Sp(r)+ r_{0}  = r_{0} $.
        Applying  \Cref{not equal and not full 1} with $i = \Sp(r-r_{0})$,
        we obtain the theorem in the case $\Sp(r-i)<r_{0}$. This completes the proof.
\end{proof}
Using  the above theorem
and \Cref{dimension formula for X_{r}} in conjunction with
\Cref{ClebschGordan}, one can determine the JH factors of 
$X_{r-i,r}$, for all $0\leq i \leq p-1$, in the case $r_{0} \geq i$.
We also have the following dimension formula.
\begin{corollary}\label{dimension r_0 > i}
      Let $p\geq 3$, $ 0 \leq i \leq p-1$ and $p 
      \leq r \equiv r_{0} 
      \mod p$ with $0 \leq r_{0} \leq p-1$. If $r_{0} \geq i$, then
      \begin{align*}
             \dim X_{r-i, \, r} = 
              \begin{cases}
                      (r_{0}+1)(\Sp(r-r_{0})+1) , & \mathrm{if} ~ \Sp(r-i) < r_{0},  \\
                      (i+1)\left( \Sp(r-i) +1 \right), &\mathrm{if} ~ 
                       r_{0} \leq \Sp(r-i) \leq p,\\
                      (i+1)(p+1), & \mathrm{if} ~ \Sp(r-i) \geq p.
                \end{cases}     
      \end{align*}
\end{corollary}
\begin{proof}
         This follows from  \Cref{Structure r_0 >i} and
          \Cref{dimension formula for X_{r}}. Note that the formulas match at the 
          boundary $\Sp(r-i) = p$.
\end{proof}

As a corollary, we obtain the structure of the successive quotients 
$X_{r-i, \, r}/ X_{r-(i-1),\,r}$, for $r_{0} \geq i$.
\begin{corollary}        
Let $p\geq 3$, $1 \leq i \leq p-1$ and $p 
      \leq r \equiv r_{0} 
      \mod p$ with $0 \leq r_{0} \leq p-1$. If $r_{0} \geq i$, then
        \begin{align*}
             \frac{X_{r-i, \, r}}{X_{r-(i-1),\,r}} \cong
              \begin{cases}
                      (0), & \mathrm{if} ~ \Sp(r-i) < r_{0},  \\
                      V_{p-1-[2i-r]} \otimes D^{i}, &\mathrm{if} ~ 
                       r_{0} \leq \Sp(r-i) \leq p-1,\\
                      \ind_{B}^{\Gamma}(\chi_{1}^{r-i}\chi_{2}^{i}),
                       & \mathrm{if} ~ \Sp(r-i) \geq p.
                \end{cases}     
      \end{align*}
\end{corollary}
\begin{proof}
    Note that $\Sp(r-(i-1)) = \Sp(r-i)+1 =\Sp(r)-i+1$, for $r_{0} \geq i$.
    Using \Cref{dimension r_0 > i}, one checks that dim $X_{r-i, \, r}-$ 
    dim $X_{r-(i-1),\,r}$ equals $0$, $\Sp(r)-2i+1$ and $p+1$, in the cases
     described in the corollary. Now the result follows from 
    \Cref{induced and successive} and \Cref{Structure of induced}, noting that
    $p-1-[2i-r] = \Sp(r)-2i$, 
    if $r_{0} \leq \Sp(r-i) \leq p-1$.
\end{proof}

\subsection{The case \texorpdfstring{$\boldsymbol{r_{0} < i}$}{}}
\label{section r0<i}
In this subsection, we determine  the structure of $X_{r-i,\,r}$,
 for $1 \leq i \leq p-1$, in the case $r_{0} < i$, where $r_{0}$ is
 as in \eqref{base p expansion of r}. 
We  first analyze  
the quotients $X_{r-j,\,r}/X_{r-(j-1),\,r}$ for $r_{0} < j \leq i$,
which depend only on 
$\Sp(r-j)$, see \Cref{successive  quotients r_0 < i}.
Building upon this result we determine 
$X_{r-i,\,r} / X_{r-r_{0},\,r}$ in \Cref{final quotient structure}.
This combined with the results 
of the previous section allows us to determine the structure of $X_{r-i,\,r}$,
for $r_{0}<i$, see \Cref{Structure r_0<i}.

We being with the following useful lemma.
\begin{lemma}\label{Sp(r-i), Sp(r-j)}
       Let $1 \leq i \leq p-1$ and $p \leq r = r_{m}p^{m}+ \cdots +r_{1}p+r_{0}$ be the base 
       $p$-expansion of $r$ with $r_{0} < i$. Then
       \begin{enumerate}[label=\emph{(\roman*)}]
              \item  $\Sp(r-j) = \Sp(r-i)+i-j$, for all $r_{0}<j \leq i$.
              \item  If $\Sp(r-i) \leq p-1$, then $r_{1}  \neq 0$, $\Sp(r-i) =  
                     \Sp(r)+p-1-i$ and  $\Sp(r) \geq r _{0}+1$.
       \end{enumerate}
\end{lemma}
\begin{proof}
         We have
        \begin{enumerate}
       \item[(i)] Let $r-i = r_{m}'p^{m}+ \cdots +r_{1}'p+r_{0}'$ be the base $p$-expansion
       of $r-i$. As $p+r_{0}-i \equiv r-i \equiv r_{0}'$ mod $p$ and  
       $0 \leq p+r_{0}-i, r_{0}' \leq p-1$ we
       have $r_{0}'=p+r_{0}-i $. Since $0 \leq r_{0}'+i-j = p+r_{0}-j \leq p-1$, we obtain
       the base $p$-expansion of $r-j$ is given by $r_{m}'p^{m}+ 
       \cdots +r_{1}'p+(r_{0}'+i-j)$.
       Hence $\Sp(r-j) = r_{m}'+ \cdots +r_{0}'+i-j = \Sp(r-i)+i-j$. 
       
       \item[(ii)]  Suppose $r_{1}=0$. As $r \geq p$ 
        there exists $2 \leq l \leq m$ such that $r_l \neq 0$. 
        Let $l$ be minimal. Then
       $r-i = r_{m}p^{m} + \cdots + r_{l+1}p^{l+1}+(r_{l}-1)p^{l}+(p-1)p^{l-1} +
        \cdots +(p-1)p+ (p+r_{0}-i)$ and
       \begin{align*}
                \Sp(r-i) & \geq (p-1)+p+r_{0}-i \quad \mathrm( \because ~ l \geq 2 ) \\
                             & \geq p \quad \mathrm( \because ~  i \leq p-1,~r_{0} \geq 0 ~),
       \end{align*}
         which is a contradiction. Hence $r_{1} \neq 0$ and $r-i = r_{m}p^{m} + \cdots +
         r_{2}p^{2}+(r_{1}-1)p+ (p+r_{0}-i)$. Thus $\Sp(r-i)= r_{m}+ 
         \cdots+r_{1}-1+p+r_{0}-i
         = \Sp(r)+p-1-i$. Since $r_{1} \geq 1$, we have $1+r_{0} \leq \Sp(r)$.
         \qedhere
         \end{enumerate} 
\end{proof}

 By \cite[(5.1)]{Glover}, for every $r \geq p$, we have an exact sequence
 $$
        0 \rightarrow V_{r-2} \otimes V_{0} \otimes D \xrightarrow{\theta_{r-1,1}} 
        V_{r-1} \otimes V_{1}  \xrightarrow{\varphi_{r-1,1}} V_{r} \rightarrow 0.
 $$
As  in \Cref{surjection1}, we have
$\varphi_{r-1,1}(X_{r-i,\,r-1} \otimes V_{1}) = X_{r-i,\,r}$, for all  
$1 \leq i \leq p-1$. Hence, for $2 \leq i \leq p-1$, we have
\begin{align}\label{surjection3}
       \frac{X_{r-i,\,r-1} \otimes V_{1}}
       {X_{r-(i-1),\,r-1} \otimes V_{1}} \twoheadrightarrow \frac{X_{r-i,\,r}}{X_{r-(i-1),r}}.
\end{align}
 
 We now derive a sufficient condition under which $X_{r-(i-1),\,r} = X_{r-i,\,r}$,
 for $r_{0}< i \leq p-1$.
\begin{lemma}\label{r-i small}
       Let $1 \leq i \leq p-1$, $r \geq p$ and $r=r_{m}p^{m}+ \cdots + r_{1}p+ r_{0}$ be
       the base $p$-expansion of $r$. If $r_{0}< i$ and $\Sp(r-i)< p-1$, then 
        $X_{r-(i-1),\,r} = X_{r-i,\,r}$.
\end{lemma}
\begin{proof}
      We prove the lemma by induction on $i$. For $i=1$, the condition 
      $r_{0} < 1$ implies $r_{0}=0$, i.e., $p \mid r$. Thus $\Sp(r-1) \geq p-1$  
       as $r \geq p$. For $r \geq p$, by \cite[Lemma 4.1]{BG15}, we have $X_{r-1,\, r}
       \neq X_{r,\,r}$. So the lemma is vacuously true for $i=1$. So we may assume
       $i \geq 2$.  By \eqref{surjection3}, it is enough to show
       $X_{r-1-(i-2),\,r-1}= X_{r-1-(i-1),\,r-1}$.
       If $r=p$, then $V_{r-1}$ is irreducible, so 
       $X_{r-1-(i-2),\,r-1}= X_{r-1-(i-1),\,r-1}$.
       Assume $r>p$.
        We first consider the case  $r_{0} =0$.
       By  Lemma~\ref{Sp(r-i), Sp(r-j)} (ii), we see  that
      $\Sp(r-i) < p-1$ implies $r_{1} \neq 0$.
      Thus $r-1 = r_{m}p^{m}+\cdots+r_{2}p^{2}+
      (r_{1}-1)p+p-1$. Note that $\Sp(r-1-(i-1))=\Sp(r-i) < p-1$.
      Applying \Cref{terms equal in filtration r0 >i} for $r-1$, we  get
      $X_{r-1-(i-2),\,r-1}= X_{r-1-(i-1),\,r-1}$.
       If $1 \leq r_{0} < i $, then  the base $p$-expansion of $r-1$ 
       is given by $r_{m}p^{m}+\cdots+r_{2}p^{2}+ r_{1}p+(r_{0}-1)$.
       Thus by induction (for $r-1$ and $i-1$), we have
      $X_{r-1-(i-2),\,r-1} = X_{r-1-(i-1),\, r-1}$. 
      This proves the inductive step 
      and the lemma follows.
\end{proof}
We next prove a result analogous to \Cref{terms equal in filtration r0 >i},
in the case $r_{0}<i$.

\begin{lemma}\label{X_r-i = X_ r-j}
       Let $p \leq r=r_{m}p^{m}+ \cdots + r_{1}p+r_{0}$ be the base  $p$-expansion of 
       $r$. For $1 \leq j < i \leq p-1 $ and $r_{0}<j$, we have
       $X_{r-j,\,r} = X_{r-i,\,r}~ \mathrm{ if ~ and ~ only ~ if}~ \Sp(r-j)\leq p-1 $.
\end{lemma}
 \begin{proof}
      Since $r_{0}< j$, by Lemma~\ref{Sp(r-i), Sp(r-j)} (i) 
     (with $i$ there equal to $l$),
     we have  $\Sp(r-l) = \Sp(r-j)-(l-j) < p-1$,  for all $l$ such that
     $j+1 \leq l \leq i$. By \Cref{r-i small}, we have $X_{r-(l-1),\,r} = X_{r-l,\,r}$
     for all $l$ such that
     $j+1 \leq l \leq i$, whence $X_{r-j,\,r} = X_{r-(j+1),\,r} = \cdots = X_{r-(i-1),\,r}
      = X_{r-i,\,r}$. 
     This proves the \enquote*{if} part.
     
      For the converse we claim that if $\Sigma_{p}(r-j) > p-1$, then 
      $X_{r-j,\,r} \subsetneq X_{r-(j+1),\,r}$.
      Suppose not. Then  by \Cref{first row filtration}, we have 
      $X_{r-j,\,r}=X_{r-(j+1),\,r}$.
      Thus by \Cref{Basis of X_r-i}, there exist 
      $a_{k,l} \in \f$ and $b_{l} \in \f$, 
      for $k \in \f$ and $0 \leq l \leq j$, 
      such that
      \begin{align}\label{Relation r-i small}
             X^{r-j-1} Y^{j+1}= \sum_{l=0}^{j} \sum_{k=0}^{p-1} a_{k,l} X^{l} (kX+Y)^{r-l}
             + \sum\limits_{l=0}^{j} b_{l} X^{r-l}Y^{l}.
      \end{align}
       As above, for every  positive integer $t $ and $0 \leq l \leq j$,  define
      $A_{t,l} := \sum\limits_{k=1}^{p-1} a_{k,l}k^{r-l-t}$. For
      every $l$, note that  $A_{t,l} $ depends only on the congruence class of $t$ mod $(p-1)$.
      Comparing the coefficients of $X^{r-t}Y^{t}$  on  both sides of 
      \eqref{Relation r-i small}, we get 
      \begin{align}\label{3.23}
      	       \sum\limits_{l=0}^{j} \binom{r-l}{t} A_{t,l} = \delta_{j+1,t}, ~  \forall  ~ 
      	        j < t < r-j. 
      \end{align}
      Since $r_{0} < j < p-1$, by Lucas' theorem, 
      we see that $ \binom{r-l}{j+1} \equiv 
      \binom{r_{0}-l}{j+1} \equiv 0 \mod p$, for $0 \leq  l  \leq r_{0}$. 
      Thus, taking $t=j+1$ in \eqref{3.23} we get
      \begin{align}\label{Relation r-i small i}
      	      \sum\limits_{l=r_{0}+1}^{j} \binom{r-l}{j+1} A_{j+1,l} =1.
      \end{align} 
      Below we show that 
      $A_{j+1,r_{0}+1},  \ldots , A_{j+1,j} =0$ by solving a system of linear equations.
      This contradicts \eqref{Relation r-i small i}.


     Let  $r_{m}'p^{m}+ \cdots + r_{1}'p+r_{0}'$ be the base 
     $p$-expansion of $r-j$. Clearly $0 \leq r_{0}'$, $p+r_{0}-j \leq p-1$ 
     and $r_{0}' \equiv r-j \equiv p+r_{0}-j$ mod $p$. So $r_{0}' = p+r_{0}-j$.
     Then the assumption $\Sp(r-j) \geq p$ implies 
      $r_{m}'+ \cdots + r_{1}' \geq p-r_{0}' = j-r_{0} \geq 1$. Thus $r= r-j+j
      \geq p +(j+r_{0}) = 2p+r_{0} \geq 2p$. Also note that for 
      all $l$ such that
      $r_{0}+1 \leq l \leq j $, the base $p$-expansion of $r-l$ is given by 
      \[
           r -l = r -j+ (j-l) = r_{m}'p^{m}+ \cdots + r_{1}'p+ (p+r_{0}-l).
       \]
        Thus  
        $\Sp(r - (r_{0}+1)) -(p-1)= r_{m}'+ \cdots + r_{1}' \geq j-r_{0} \geq 1$
        and $r-(r_{0}+1) \geq ( r_{m}'+ \cdots + r_{1}')p \geq p$.
        For every $1 \leq u \leq j-r_{0}$, let $b_{u} = j+1-u $.
        Clearly $1  \leq b_{u} \leq j \leq p-1$.
        Applying  \Cref{choice of s} (for $r- (r_{0}+1)$), 
        for every $1 \leq u \leq j-r_{0}$ 
         there exists 
        $s_{u}$ such that
       $p \leq s_{u} \leq r- r_{0}-1$,  $s_{u} \equiv b_{u} =j+1-u \mod p$, 
       $\Sp(s_{u}) = b_{u}+u = j+1$ and
       $\binom{r-r_{0}-1}{s_{u}} \not \equiv 0 \mod p$. 
       Let $s_{u} = s_{u,m}p^{m}+ \cdots +s_{1,u}p+ s_{u,0}$ be the base 
       $p$-expansion of $s_{u}$.
       By the above conditions for every $1 \leq u \leq j-r_{0}$, 
       we have $s_{u,0} = j+1-u$
       and $s_{u,n} \leq r_{n}'$, by Lucas' theorem, for $1 \leq n \leq m$.       
       If $ 1\leq u <  j- r_{0}\leq \sum_{n=1}^{m}r_{m}' $, then  by \Cref{choice of s} we have
       $s_{u} \leq (r-r_{0}-1)-p \leq r-p< r-j $,  and also if $u=j-r_{0}$, then 
       $r- (r_{0} +1) - s_{u} \geq (p-1)-s_{j-r_{0},0} = p-1 -(r_{0}+1) > j-(r_{0}+1)$. 
       Thus for all $1 \leq u \leq j-r_{0}$,
       we have  $j< p \leq s_{u} < r-j$, so we may apply  \eqref{3.23} to obtain
      \begin{align}\label{Relation r-i small 2}
             \sum\limits_{l=0}^{j} \binom{r-l}{s_{u}} A_{s_{u},l}  =0, ~ 
             \forall~1 \leq u \leq j-r_{0}.
      \end{align}
      Since $ s_{u,0} = j+1-u \geq r_{0}+1 $ by Lucas' theorem, 
      for all $l$ such that $0 \leq l \leq r_{0}$, we have 
    \[
          \binom{r-l}{s_{u}} \equiv \binom{r_{m}}{s_{u,m}} \cdots \binom{r_{1}}{s_{u,1}}
          \binom{r_{0}-l}{s_{u,0}} \equiv 0 \mod p.
     \]
      Again by Lucas' theorem, 
      for all $l$ such that $r_{0}+1 \leq l \leq j$, we have
     \begin{align*}
            \binom{r-l}{s_{u}} & \equiv      
            \binom{r_{m}'}{s_{u,m}} \cdots \binom{r_{1}'}{s_{u,1}} \binom{p+r_{0}-l}{j+1-u}
            \mod p \\        
            &\equiv \binom{r-r_{0}-1}{s_{u}} \binom{p-1}{j+1-u}^{-1}
             \binom{p+r_{0}-l}{j+1-u} \mod p.
     \end{align*}
     Since $s_{u} \equiv \Sp(s_{u}) = j+1 $ mod $(p-1)$ and 
     $\binom{r-r_{0}-1}{s_{u}} \not \equiv 0$ mod $p$,
      it follows from the above  computations and \eqref{Relation r-i small 2} that 
     \begin{align*}
           \sum\limits_{l=r_{0}+1}^{j}  \binom{p+r_{0}-l}{j+1-u} A_{j+1,l} =0,  ~
          \text{for} ~ 1 \leq u \leq j-r_{0}.
     \end{align*}
     Writing the above set of equations in matrix form and 
     applying \Cref{matrix det} (i) (with $a=p-1$ and $i=j-r_{0}-1$), 
     we obtain $A_{j+1,r_{0}+1} = \cdots = A_{j+1,j}=0$.
      Substituting $A_{j+1,r_{0}+1}, \ldots , A_{j+1,j}=0$
       in \eqref{Relation r-i small i} leads to contradiction.
       This shows $X_{r-j,\,r} \subsetneq  X_{r-(j+1),\,r} \subseteq X_{r-i,\,r} $
       and proves the \enquote*{only if} part.
 \end{proof}   
 %
%
Combining the above lemma with \Cref{terms equal in filtration r0 >i}, 
we have the following criterion for when $X_{r-j,\,r} = X_{r-i,\,r}$ are equal, 
for $1 \leq j < i \leq p-1$. 
\begin{lemma}\label{final X_r-i = X_r-j}
          Let $1 \leq j <  i \leq p-1$ and  $p \leq r = r_{m}p^{m}+ \cdots +
           r_{1}p+r_{0}$ be the base $p$-expansion of $r$. Then 
        \[ 
             X_{r-j,\,r} = X_{r-i,\,r} ~ \Longleftrightarrow~
           r_{0} \neq j, j+1, \ldots, i-1 , ~
               \Sp(r-j) \leq p-1 ~ \mathrm{and} ~ \Sp(r-r_{0}) \leq  j.
          \]      
\end{lemma}
\begin{proof}
        Clearly the \enquote*{if} part follows from \Cref{terms equal in filtration r0 >i}
        and \Cref{X_r-i = X_ r-j}  in the case $r_0 \geq i$ and 
         $r_0 < j$ respectively. 
        
         For the converse assume $X_{r-j,\,r} = X_{r-i,\,r}$.
         Then we claim that $ r_{0} \neq j, j+1, \ldots, i-1$.
        Suppose not.
         Then the  coefficient of $X^{r-i}Y^{i}$ in $X^{l}(kX+Y)^{r-l}$ is congruent to $0$ mod $p$, for all
         $k \in \f$ and $0 \leq l \leq j$, since 
         $\binom{r-l}{i}$ vanishes,
         by Lucas' theorem. 
        Hence $X^{r-i} Y^{i} \not \in \f$-span of $\lbrace X^{l}(kX+Y)^{r-l}, X^{r-l}Y^{l} : 
         k \in \f,  0 \leq l \leq j\rbrace = X_{r-j,\,r}$ by \Cref{Basis of X_r-i}, which is a 
         contradiction.
         Thus $r_{0} \neq j$, $ j+1, \ldots, i-1$. If 
        $r_{0} \geq i$, then by 
        \Cref{terms equal in filtration r0 >i}, we have  $\Sp(r-r_{0}) \leq j \leq  p-1$, whence
        $\Sp(r-j) = \Sp(r-r_{0})+r_{0}-j \leq r_{0} \leq p-1$. Similarly if  $r_{0}<j$, 
        then by \Cref{X_r-i = X_ r-j}, we have 
        $\Sp(r-j) \leq p-1$, whence
         $\Sp(r-r_{0}) = \Sp(r)-r_{0}= \Sp(r-j) +j- (p-1)-r_{0}  \leq j $, 
         by Lemma~\ref{Sp(r-i), Sp(r-j)} (ii).
       This finishes the proof.
\end{proof}
\begin{lemma}\label{dim not large}
        Let $1 \leq i \leq p-1$, $p \leq r \equiv r_{0} 
        ~ \mathrm{mod}~p$ with $0 \leq r_{0}<i$. 
        If  $\Sp(r-i)=p-1$, then $X_{r-i,\,r}/ X_{r-(i-1),\,r} \cong V_{p-1-i} 
        \otimes D^{i}$. 
\end{lemma}
\begin{proof}
       For any integer $t \in \lbrace p, \ldots , p+i-2 \rbrace$
       one checks that $\Sp(t-i) = t-i < p-1$. 
       So the condition  $\Sp(r-i) = p-1$ implies that
       $r \geq p+i-1$.
       Let $s= r-i+1$. Then $s \geq p$, $\Sp(s-1)=p-1$
       and $s=(s-1)+1 \equiv \Sp(s-1)+1 \equiv 1 $ mod $(p-1)$. 
       Applying  \cite[Lemma 3.2]{BG15} for $s$, we get
       $$
            \sum_{k=0}^{p-1} X(kX+Y)^{s-1} = -X^{s}.  
       $$
        Multiplying the above equation by $X^{i-1}$ 
        we get
       \begin{align}\label{Sp(r-i) =p-1 relation}
             \sum_{k=0}^{p-1} X^{i}(kX+Y)^{r-i} = -X^{r}.  
      \end{align}
       Since $r-i \equiv \Sp(r-i)= p-1 \equiv 0$ mod $(p-1)$, we have
      $\ind_{B}^{\Gamma}(\chi_{1}^{r-i}\chi_{2}^{i}) = 
       \ind_{B}^{\Gamma}(\chi_{2}^{i})$.  Using 
       \Cref{Structure of induced} and \Cref{induced and successive},
       we get
        \[
             \begin{tikzcd}
                 0 \arrow[r, rightarrow] & V_{i} \otimes  D^{p-1} \arrow[r, rightarrow]
                  & \ind_{B}^{\Gamma} (\chi_{2}^{i}) \arrow[r, rightarrow]
                   \arrow[d,  twoheadrightarrow, "\psi_{i}"] & V_{p-1-i} \otimes D^{i} 
                   \arrow[r, rightarrow] & 0. \\
                  &  & X_{r-i,\,r}/X_{r-(i-1),\,r}  & &
             \end{tikzcd}    
       \]
        By \Cref{Structure of induced} (i) (for $l=0$), we have 
         $\sum\limits_{k \in \f}^{} [ \begin{psmallmatrix} 
        k & 1 \\ 1 & 0 \end{psmallmatrix} , e_{\chi_{2}^{i}}] $ 
        is an element of
         $V_{i} \otimes  D^{p-1}  \hookrightarrow 
        \ind_{B}^{\Gamma} (\chi_{2}^{i})$. 
         Also by \Cref{induced and successive}, \eqref{Sp(r-i) =p-1 relation}
         and  \Cref{first row filtration} we see that
        \[
           \psi_{i} \bigg( \sum_{k \in \f}^{} [\begin{psmallmatrix} 
        k & 1 \\ 1 & 0 \end{psmallmatrix} , e_{\chi_{2}^{i}}]  \bigg)
         = \sum_{k=0}^{p-1} X^{i}(kX+Y)^{r-i} = -X^{r} \in X_{r}
         \subseteq X_{r-(i-1)}.
        \]
        Thus
        we get the composition $V_{i} \otimes  D^{p-1}  \hookrightarrow 
        \ind_{B}^{\Gamma} (\chi_{1}^{r-i}\chi_{2}^{i}) \twoheadrightarrow 
        X_{r-i,\,r}/X_{r-(i-1),\,r}$ is the zero map. Hence
        $V_{p-1-i} \otimes D^{i} \twoheadrightarrow X_{r-i,\,r}/X_{r-(i-1),\,r}$.
        If $r_{0} =i-1$, then $X_{r-i,\,r}/ X_{r-(i-1),\,r} \neq 0$
        by \Cref{final X_r-i = X_r-j}. If $r_{0}< i-1 $, then
        $\Sp(r-(i-1)) = \Sp(r-i)+1 =p$ by Lemma~\ref{Sp(r-i), Sp(r-j)} (i),
        whence $X_{r-i,\,r}/ X_{r-i-1,\,r} \neq 0$,
        by \Cref{final X_r-i = X_r-j}.
        Therefore  $X_{r-i,\,r}/ X_{r-(i-1),\,r} \cong V_{p-1-i} \otimes D^{i}$
        in either case.
\end{proof}
Next we prove that if $\Sp(r-i)> p-1$, then 
$X_{r-i,\,r}/X_{r-(i-1),\,r}$ is isomorphic to the principal series 
representation $\ind_{B}^{\Gamma}(\chi_{1}^{r-i}\chi_{2}^{i})$.
 \begin{lemma}\label{Sp(r-i) large}
       Let $1 \leq i \leq p-1$, $(i+1)(p+1) \leq r   \equiv r_{0} \mod p$ with $0 \leq r_{0} <i$. 
       If $\Sp(r-i)>p-1$, then $X_{r-i,\,r}/X_{r-(i-1),\,r} 
       \cong \ind_{B}^{\Gamma}(\chi_{1}^{r-i}\chi_{2}^{i})$. As a consequence,
          $\dim X_{r-i,\,r}/X_{r-(i-1),\,r} = p+1$.
\end{lemma}
\begin{proof}
        By \Cref{induced and successive}, we have
        $\psi_{i} : \ind_{B}^{\Gamma}(\chi_{1}^{r-i}\chi_{2}^{i}) 
        \twoheadrightarrow X_{r-i,\,r}/X_{r-(i-1),\,r} $. 
        We claim that $\psi_{i}$ is an isomorphism. As $\psi_{i}$ is
        surjective, it is 
        enough to show that  $\psi_{i}$ is injective. By 
        \Cref{Structure of induced} (i) (for $l=0$) and (ii) (for $l=p-1$), 
        we know that
        $\sum_{\lambda \in \f}  \begin{psmallmatrix} 
        \lambda & 1 \\ 1 &0  \end{psmallmatrix} [1, e_{\chi_{1}^{r-i}\chi_{2}^{i}}] $
        and   $ \sum_{\lambda \in \f^{\ast}}  \begin{psmallmatrix} 
        \lambda & 1 \\ 1 &0  \end{psmallmatrix} [1, e_{\chi_{1}^{r-i}\chi_{2}^{i}}] $ are 
        elements  of the two sub-quotients   $V_{[2i-r]} \otimes D^{r-i}$
        and $V_{p-1 - [2i-r]} \otimes D^{i}$ respectively of 
        $\ind_{B}^{\Gamma}(\chi_{1}^{r-i}\chi_{2}^{i})$.   Hence, it is enough
        to prove that these elements have non-zero image under 
        $\psi_{i}$.  Indeed, if
           
       \[ 
           F(X,Y) := \sum_{\lambda \in \f}  \begin{pmatrix} 
            \lambda & 1 \\ 1 &0  \end{pmatrix} \psi_{i} 
            ([1, e_{\chi_{1}^{r-i}\chi_{2}^{i}}])  =
            \sum_{\lambda \in \f} X^{i}( \lambda X+ Y)^{r-i} 
            \in  X_{r-(i-1),\,r},
      \]
     then by \Cref{Basis of X_r-i}, there exist 
     $A_{l}, B_{l}$ and $c_{k,l} \in \f$  for $0 \leq l \leq  i-1$ 
     and $1 \leq k \leq p-1$, such that 
     \begin{align}\label{Relation r-j large}
            F(X,Y)= \sum_{l=0}^{i-1} A_{l} X^{l} Y^{r-l}+ \sum_{l=0}^{i-1} B_{l} X^{r-l}Y^{l}
            - \sum_{l=0}^{i-1} \sum_{k=1}^{p-1} c_{k,l} X^{l} (kX+Y)^{r-l}.
     \end{align}
     Observe that the coefficient of $X^{t}Y^{r-t}$ in $F(X,Y)$ 
     equals 
     $\sum_{\lambda \in \f } \binom{r-i}{r-t}  \lambda^{t-i}$ which is  zero if 
     $t \not \equiv i$ mod $(p-1)$, by \eqref{sum fp}. 
      Therefore \eqref{Relation r-j large} reduces to 
     \begin{align}\label{Relation r-j large 1}
            F(X,Y) =
            \sum_{l=0}^{i-1} B_{l} X^{r-l}Y^{l}
             -\sum_{l=0}^{i-1}  C_{l} \sum_{\substack {l \leq j \leq r \\ 
           j \equiv i~  \text{mod}~ (p-1)}} \binom{r-l}{r-j} X^{j}Y^{r-j},
     \end{align}
     where $C_{l} = \sum\limits_{k=1}^{p-1} k^{i-l} c_{k,l} $.
     Comparing the coefficients of $X^{t}Y^{r-t}$ on both sides of 
     \eqref{Relation r-j large 1}, we get
     \begin{align}\label{compare coeff r-j large}
           \binom{r-i}{t-i} = \sum_{l=0}^{i-1}  C_{l}  \binom{r-l}{t-l},
           ~ \forall~  i < t \leq r-i ~\mathrm{and}~
           t \equiv i ~\mathrm{mod}~(p-1).
     \end{align}
     Let $r = r_{m}p^{m}+ \cdots +r_{1}p+r_{0}$ be the base  $p$-expansion of $r$. 
     Since $r-(r_{0}+1) \equiv p-1$ mod $p$ and $r-(r_{0}+1)  < r$, note that the
     base $p$-expansion of $r-(r_{0}+1)$  is  of the form
     $r_{m}'p^{m}+ \cdots +r_{1}'p+p-1$ for some $r_{m}', \ldots, r_{1}' $. As 
     $r_{0}<i$, we have the base $p$-expansion of $r-i$ is 
     given by $r_{m}'p^{m}+ \cdots + r_{1}'p+(p+r_{0}-i)$. Then the condition 
     $\Sp(r-i) > p-1$ implies $r_{m}'+\cdots+ r_{1}' \geq i-r_{0}$.
     We now show that $C_{r_{0}+1}, \ldots , C_{i-1} =0 $.
     Note that this statement is vacuous if $i=r_{0}+1$.
     So we may assume $i> r_{0}+1$. 
      For every  $1 \leq u < i-r_{0}$, let $b_{u}=i-u$. Clearly 
      $0 \leq r_{0} \leq b_{u} \leq i-1 \leq p-1$.
      By \Cref{choice of s} (applied for $r-(r_{0}+1)$), for every 
      $1 \leq u < i-r_{0}$, we can find an integer $s_{u}$
      such that $p \leq s_{u} \leq r-r_{0}-1-p \leq r-p$,
      $s_{u}  \equiv b_{u}= i-u$ mod $p$, $\Sp(s_{u}) = b_{u}+u=i$ and 
      $\binom{r-(r_{0}+1)}{s_{u}} \not \equiv 0 \mod p$. 
      Noting that $s_{u} \equiv \Sp(s_{u}) =i$ mod $(p-1)$, by 
      \eqref{compare coeff r-j large}, we get
      \[
          \binom{r-i}{s_{u}-i} 
          = \sum_{l=0}^{i-1} C_{l} \binom{r-l}{s_{u}-l}.
      \]
    Since $s_{u} \equiv i-u $ mod $p$ and $i-u > r_{0}$,
    for $1 \leq u < i-r_{0}$, 
     by Lucas' theorem, we have 
     \begin{align*} 
                \binom{r-l}{s_{u}-l} &\equiv  
                \begin{cases}
                      (*) \binom{r_{0}-l}{i-u-l}, & \mathrm{if} ~0 \leq l \leq r_{0},\\
                       (*) \binom{p+r_{0}-l}{p+i-u-l}, & \mathrm{if}~i-u <l \leq i,
                \end{cases} \\
                &  \equiv  0 \mod p,  
       \end{align*}
       where $(*)$ denotes the contribution from the higher order terms in 
       the base $p$-expansion.  Hence, for every $1 \leq u < i-r_{0}$, we have
      \begin{align}\label{Compare r-j large}
        \sum_{l=r_{0}+1}^{i-u} C_{l} \binom{r-l}{s_{u}-l} =0.
      \end{align}
      Since $s_{u} \equiv i-u$ mod $p$ and $s_{u} \leq r$,  
      we get the base $p$-expansion of $s_{u}$ is given by
       $s_{u} =  s_{u,m} p^{m}+ \cdots + s_{u,1}p+ (i-u)$,
      for some $s_{u,m}, \ldots , s_{u,1}$. By Lucas' theorem 
      and  the choice of $s_{u}$, for $1 \leq u <i-r_{0}$,
      we have
      \begin{align*}
            \binom{r-i+u}{s_{u}-i+u}& \equiv \binom{r_{m}'}{s_{u,m}} \cdots 
            \binom{r_{1}'}{s_{u,1}} \binom{p+r_{0}-i+u}{0} \mod p \\
            &\equiv \binom{r-r_{0}-1}{s_{u}} \binom{p-1}{i-u}^{-1} \mod p \\
            &\not \equiv 0 \mod p.
      \end{align*}
      Writing \eqref{Compare r-j large} in matrix form and noting that the 
      anti-diagonal entries $\binom{r-i+u}{s_{u}-i+u} \not \equiv 0$ mod $p$ 
      for $1 \leq u < i-r_{0}$, we see
      that $C_{r_{0}+1}= \cdots = C_{i-1} =0$.  Thus 
      \eqref{compare coeff r-j large} reduces to
      \begin{align}\label{compare coeff r-j large 1}
           \binom{r-i}{r-t} =  \sum_{l=0}^{r_{0}}  C_{l}  \binom{r-l}{r-t}, 
           ~ \forall~  i < t \leq r-i ~\mathrm{and}~
           t \equiv i ~\mathrm{mod}~(p-1).
      \end{align}
      We now show that the above system of equations leads to a contradiction
      by considering various cases.
      
      If $r_{1}=0$, then $m \geq 2$ and $r \geq p^{2}$. Let 
      $j>1$ be the smallest positive integer such that $r_{j} \neq 0$.
      Then 
      \[ r-i= r_{m}p^{m}+ \cdots +r_{j+1}p^{j+1}+  (r_{j}-1)p^{j}+
      (p-1)p^{j-1}+ \cdots+(p-1)p+ (p+r_{0}-i).
      \]
      Let $s=(p-2)p+i+1 \equiv i$ mod $(p-1)$. 
      Since $i \leq p-1$, note that $p \leq s \leq p^{2} -p \leq r-p$.  
      By \eqref{compare coeff r-j large 1},
      we have
      \[
           \binom{r-i}{s-i} =\sum_{l=0}^{r_{0}} C_{l} \binom{r-l}{s-l}.
      \]
      By Lucas' theorem, $\binom{r-i}{s-i} \equiv  \binom{p-1}{p-2} 
      \binom{p+r_{0}-i} {1} \not \equiv 0$ mod $p$, and 
      for $0 \leq l \leq r_{0}$, we have
      \begin{align*}
      \binom{r-l}{s-l} & \equiv 
      \begin{cases} 
       \binom{r_{1}}{p-2}\binom{r_{0}-l}
      {i+1-l}, & \mathrm{if}~ i<p-1, \\
       \binom{r_{1}}{p-1}\binom{r_{0}-l}
      {0} ,& \mathrm{if}~ i = p-1, l=0,  \\
      \binom{r_{1}}{p-2}\binom{r_{0}-l}
      {p-l}, & \mathrm{if}~ i = p-1, l \neq 0,  \\
      \end{cases} \\
      & \equiv 0 ~\mathrm{mod}~ p ~ ( ~\because r_{1}=0, ~ r_{0}<p). 
      \end{align*}
      This leads to a contradiction.
      
      If  $r_{m}+ \cdots + r_{1} \geq i+1$, then by \Cref{choice of s}
      (applied for $r-i$, $u=i$ and $b=p-i-1 \leq p+r_{0}-i$),
      there exists  $p \leq s \leq r-i-p$, $s \equiv p-i-1$ mod $p$,   
      $\Sp(s) = p-1$ and  $\binom{r-i}{s} \not
      \equiv 0$ mod $p$.  Taking $t=s+i$ in
      \eqref{compare coeff r-j large 1}, we get
      $$
              \binom{r-i}{s} =\sum_{l=0}^{r_{0}} C_{l} \binom{r-l}{s+i-l} .
      $$
      As $s+i-l \equiv p-1-l \mod p$, for $0 \leq l \leq r_{0}$, we
      have  $0 \leq r_{0}-l < i-l \leq p-1-l$, whence by Lucas' theorem,
      $\binom{r-l}{s+i-l}  \equiv (*) \binom{r_{0}-l}{p-1-l} \equiv 0$
       mod $p$, where $(*)$ denotes the contribution form higher order terms in the 
       base $p$-expansion. However by the choice of $s$ we have 
      $\binom{r-i}{s} \not \equiv 0$ mod $p$. Again we obtain a contradiction.
      
      Finally, suppose $r_1 \neq 0 $ and $r_{m}+ \cdots + r_{1} \leq i$. 
      So $r-i = r_{m}p^{m}+ \cdots + r_{2}p^{2}+(r_{1}-1)p+p+r_{0}-i$ in this case.
      By the
       hypotheses $\Sp(r-i) > p-1$ and $r_{1} \neq 0 $, it follows $i-r_{0} < r_{m}+
        \cdots + r_{1}$.
      Hence $r_{m}+ \cdots + r_{1} = i+1-r_{0}+w$, for some $0 \leq w \leq r_{0}-1$.
      If $m=1$, then $(i+1)(p+1) \leq r = r_{1}p+r_{0} = (i+1-r_{0}+w)p+r_{0} < ip+i$
      which is not possible. So $m \geq 2$.
      Let $s = r-w-p$ and $s' = r-w-p^{m}$. Since $\Sp(r) \equiv r \mod (p-1)$, we 
      see that $s$, $s' \equiv i \mod (p-1)$. Since $r_{m} , r_{1} \geq 1$ we have 
      $r \geq p^{m}+ p+ r_{0} \geq p^{m}+p+w$, so $p \leq s,s' \leq r-p$.
      By \eqref{compare coeff r-j large 1} with $t = s$, $s'$, we get 
      $$
              \binom{r-i}{s-i} =\sum_{l=0}^{r_{0}} C_{l} \binom{r-l}{s-l} , \quad
              \binom{r-i}{s'-i}= \sum_{l=0}^{r_{0}} C_{l} \binom{r-l}{s'-l} .
      $$
      By Lucas' theorem,  $\binom{r-l}{s-l} \equiv \binom{r_{m}}{r_{m}} \cdots
       \binom{r_{2}}{r_{2}} \binom{r_{1}}{r_{1}-1} \binom{r_{0}-l}{r_{0}-w-l} 
       \equiv r_{1}  \binom{r_{0}-l}{r_{0}-w-l} \mod p$ if $0 \leq l \leq r_{0}-w$. 
       Similarly,  $\binom{r-l}{s'-l} \equiv \binom{r_{m}}{r_{m}-1} \cdots
       \binom{r_{2}}{r_{2}} \binom{r_{1}}{r_{1}} \binom{r_{0}-l}{r_{0}-w-l} 
       \equiv r_{m}  \binom{r_{0}-l}{r_{0}-w-l} \mod p$ if $0 \leq l \leq r_{0}-w$. 
       If
       $r_{0}-w +1 \leq l  \leq r_{0}$, then $s-l$, $s'-l \equiv p+r_{0}-w-l$ 
       mod $p$ and  $p+r_{0}-w-l> r_{0}-l$, so
       it follows from  Lucas' theorem that $\binom{r-l}{s-l}$, $\binom{r-l}{s'-l} \equiv 0$
       mod $p$, for $r_{0}-w+1 \leq l \leq r_{0}$. Therefore
       $$
              \binom{r-i}{s-i} =r_{1}\sum_{l=0}^{r_{0}-w} C_{l} 
              \binom{r_{0}-l}{r_{0}-w-l}, \quad
              \binom{r-i}{s'-i}=r_{m} \sum_{l=0}^{r_{0}-w} C_{l} \binom{r_{0}-l}{r_{0}-w-l}.
      $$
     Thus  $\binom{r-i}{s-i} \binom{r-i}{s'-i}^{-1} \equiv r_{1} r_{m}^{-1} \mod p$. By Lucas' theorem,
       $$
         \binom{r-i}{s-i} \binom{r-i}{s'-i}^{-1} \equiv  r_{m}^{-1} (r_{1}-1) \mod p
       $$  
     This again leads to contradiction. 
     
      Next  consider $G(X,Y) :=  \sum_{\lambda \in \fstar} \begin{psmallmatrix} 
            \lambda & 1 \\ 1 &0  \end{psmallmatrix} \psi_{i} 
            ([1, e_{\chi_{1}^{r-i}\chi_{2}^{i}}]) =    \sum_{\lambda \in \fstar}
     X^{i}( \lambda X+ Y)^{r-i} $ and note that $F(X,Y)-G(X,Y) = X^{i}Y^{r-i}$.
     So the coefficient of $X^{s}Y^{r-s}$ in $F(X,Y)$ and $G(X,Y)$ agree if $s \neq i $.
     In the proof above for $F(X,Y)$,
      we  only compared the coefficients of $X^{s}Y^{r-s}$
     for $s \neq i$. So imitating 
       the proof above for $F(X,Y)$,  we get  $G(X,Y) \not \in X_{r-(i-1),r}$. 
      This finishes the proof of the lemma.          
\end{proof}
Putting together all the results obtained so far, we have the following result.
\begin{corollary}\label{successive  quotients r_0 < i}
        Let  $p \geq 3$, $1 \leq i \leq p-1$, $(i+1)(p+1)<r$ and $r_{0} <i$. Then
        \begin{align*}
               \frac{X_{r-i,\,r}}{ X_{r-(i-1),\,r}} \cong
               \begin{cases}
                       (0), &\mathrm{if} ~ \Sp(r-i)<p-1,  \\
                       V_{p-1-i} \otimes D^{i},  &\mathrm{if} ~ \Sp(r-i)= p-1, \\
                       \ind_{B}^{\Gamma}(\chi_{1}^{r-i} \chi_{2}^{i}),  
                       &\mathrm{if} ~ \Sp(r-i)>p-1.
               \end{cases}
        \end{align*}
\end{corollary}
 \begin{proof}
         The three assertions follow from
         \Cref{r-i small},
         \Cref{dim not large} and \Cref{Sp(r-i) large} respectively.
 \end{proof}
\begin{lemma}\label{final quotient structure}
        Let  $p\geq 3$, $1 \leq i \leq p-1$, 
        $(i+1)(p+1)<r \equiv r_{0} ~ \mathrm{mod}~p$ with
         $0 \leq r_{0}<i$. Then we have
        \begin{enumerate}[label= \emph{(\roman*)}]
               \item If $\Sp(r-i)>p-1$, then 
                         $X_{r-i,\,r}/ X_{r-r_{0},\,r} \cong V_{i-r_{0}-1}
                          \otimes \ind_{B}^{\Gamma}(\chi_{1}^{r-i}\chi_{2}^{r_{0}+1})$.
               \item  If $\Sp(r-i) \leq p-1$, then
                          $$   0 \rightarrow  V_{\Sp(r-r_{0})-2} \otimes \ind_{B}^{\Gamma}
                                 (\chi_{1} \chi_{2}^{r_{0}+1}) \rightarrow   X_{r-i,\,r}/ X_{r-r_{0},\,r} 
                                  \rightarrow  V_{p-1-\Sp(r)} \otimes D^{\Sp(r)} \rightarrow 0.
                           $$
         \end{enumerate}
\end{lemma} 
\begin{proof}
        Note that 
       \begin{enumerate}
              \item[(i)] By Lemma~\ref{Sp(r-i), Sp(r-j)} (i), 
               we have $\Sp(r-j)= \Sp(r-i)+i-j > p-1$, for  all
               $r_{0} < j \leq i$. Hence by \Cref{Sp(r-i) large}, we have 
               dim $X_{r-j,r} / X_{r-(j-1),r} = p+1$, for all  $r_{0} < j \leq i$.  
              Thus $\dim X_{r-i,\,r}/ X_{r-r_{0},\,r} =
              \sum\limits_{j=r_{0}+1}^{i} \dim  X_{r-j,\,r} / X_{r-(j-1),\,r} 
              = (i-r_{0})(p+1)=
              ( \dim V_{i-r_{0}-1}) \times 
             ( \dim\ind_{B}^{\Gamma}(\chi_{1}^{r-i} \chi_{2}^{r_{0}+1}))$.
               Now assertion (i) follows from  \Cref{Succesive quotient}.
               \item[(ii)] By Lemma~\ref{Sp(r-i), Sp(r-j)} (ii), we have
               $\Sp(r) =\Sp(r-i)-(p-1)+ i \leq i$ and $\Sp(r) \geq r_{0}+1$. 
               Thus by  Lemma~\ref{Sp(r-i), Sp(r-j)} (i), 
               we have $\Sp(r-\Sp(r)) = \Sp(r-i)+i-\Sp(r) = p-1$ whence
                $ X_{r-\Sp(r), \, r}  = X_{r-i, \,r}$ and 
               $X_{r-\Sp(r), \, r}/ X_{r-(\Sp(r)-1), \, r} \cong V_{p-1-\Sp(r)}
                \otimes D^{\Sp(r)}$ by \Cref{X_r-i = X_ r-j} and  \Cref{dim not large}
                respectively.
              We now claim that the inequality $\Sp(r) \geq r_{0}+1$ is strict.  
              If 
              $\Sp(r) = r_{0}+1$, then $r=p^{n}+r_{0}$
              for some $n \geq 1$.       
               Since $\Sp(r-i) \leq p-1$ and $r_{0}<i$, we have $r=p+r_{0}$. Thus
               $r= p+r_{0} < 2p  \leq (i+1)(p+1)$ which is a contradiction
               as $i \geq 1$. Hence $\Sp(r) > r_{0}+1$.
                So by Lemma~\ref{Sp(r-i), Sp(r-j)} (i), 
                we have $\Sp(r-(\Sp(r)-1)) = \Sp(r-i)+i-\Sp(r)+1 = p$, whence 
               $X_{r-(\Sp(r)-1), \, r} / X_{r-r_{0},\,r} \cong
                V_{\Sp(r)-r_{0}-2} \otimes \ind_{B}^{\Gamma}(\chi_{1} \chi_{2}^{r_{0}+1})$,
                 by  part (i).
                 Putting all these together in the following 
                  exact sequence
                \[
                    0 \rightarrow X_{r-(\Sp(r)-1), \, r} / X_{r-r_{0},\,r} \rightarrow 
                     X_{r-\Sp(r), \, r} / X_{r-r_{0}, \,r} \rightarrow 
                      X_{r-\Sp(r), \, r}/ X_{r-(\Sp(r)-1),\, r}  \rightarrow 0,
                \]
                  part (ii) follows as 
                 the middle term  equals $X_{r-i, \, r} / X_{r-r_{0}, \,r}$.  \qedhere
       \end{enumerate} 
\end{proof}
Using the above lemma in conjunction with \Cref{Structure r_0 >i}, we obtain the following
result which describes the structure of $X_{r-i,\,r}$ in the case $r_{0}<i$.
\begin{theorem}\label{Structure r_0<i}
         Let  $p\geq 3$, $1 \leq i \leq p-1$ and  $(i+1)(p+1)<r \equiv r_{0} 
         ~ \mathrm{mod}~p$ with $0 \leq r_{0}<i$. Then 
         \begin{enumerate}[label=\emph{(\roman*)}]
               \item If $\Sp(r-r_{0}) \geq p$, then 
               as $M$-modules we have
                         $X_{r-i,\,r} \cong X_{r-i,\,r-i} \otimes V_{i}$.
                \item If  $\Sp(r-r_{0}) \leq p-1$ and $\Sp(r-i) \geq  p $, then 
                          as $\Gamma$-modules    we have 
                        \begin{align*}
                                 0  \rightarrow V_{\Sp(r-r_{0})} \otimes V_{r_{0}} 
                                 \rightarrow X_{r-i,\,r} 
                                 \rightarrow V_{i-r_{0}-1} \otimes \ind_{B}^{\Gamma}
                                 (\chi_{1}^{r-i}\chi_{2}^{r_{0}+1}) \rightarrow 0.
                          \end{align*}
               \item  If $\Sp(r-r_{0}) \leq  p-1 $ and $\Sp(r-i) \leq  p-1$, then  
                          as $\Gamma$-modules we have
                          \begin{align*}
                                   0 \rightarrow V_{\Sp(r-r_{0})} \otimes V_{r_{0}} 
                                   \rightarrow X_{r-i,\,r} 
                                   \rightarrow W \rightarrow 0,
                           \end{align*}                               
                           where $W$ is an extension of $V_{\Sp(r-r_{0})-2} \otimes 
                           \ind_{B}^{\Gamma}
                           (\chi_{1}\chi_{2}^{r_{0}+1})$ by $V_{p-1- \Sp(r)} 
                           \otimes D^{\Sp(r)}$.
         \end{enumerate}
\end{theorem}
\begin{proof}
        By \Cref{first row filtration}, we get the following 
        exact sequence
         \[
              0 \rightarrow X_{r-r_{0},\,r} \rightarrow X_{r-i,\,r} 
              \rightarrow  X_{r-i,\,r}/ X_{r-r_{0},\,r} \rightarrow 0.
         \] 
         If $\Sp(r-i) \leq p-1$, then by Lemma~\ref{Sp(r-i), Sp(r-j)} (ii),
         we have $\Sp(r)-r_{0} = \Sp(r-i) -(p-1)+i -r_{0} \leq i-r_{0} \leq p-1$.
        Thus $\Sp(r)-r_{0} =\Sp(r-r_{0}) \geq p$ implies
        $\Sp(r-i) \geq p$. Hence by  
        \Cref{dim large 1} (for $i =r_{0}$) and 
       \Cref{final quotient structure} (i), we get dim $X_{r-i,\,r} =
       \dim X_{r-r_{0},\,r}  + \dim X_{r-i,\,r}/ X_{r-r_{0},\,r} =
       (r_{0}+1)(p+1)+ (i-r_{0})(p+1)=
      (i+1)(p+1) \geq \dim X_{r-i,\, r-i} \times \dim V_{i} $. 
      Thus by \Cref{surjection1},  we have $X_{r-i,\,r-i} \cong X_{r-i,\,r-i} 
      \otimes V_{i}$. This proves (i). 
      By \Cref{Structure r_0 >i} (applied for  $i=r_{0}$), we have
      \begin{align*}
            X_{r-r_{0},\,r} \cong 
            \begin{cases}
                    X_{r-r_{0},\, r-r_{0}} \otimes V_{r_{0}}, 
                    &~\mathrm{if}~ r_{0} \leq \Sp(r-r_{0}) \leq p-1, \\
                    X_{r-\Sp(r-r_{0}), \,r-\Sp(r-r_{0})}\otimes V_{\Sp(r-r_{0})},
                    & ~\mathrm{if}~  \Sp(r-r_{0}) < r_{0}.
            \end{cases}
      \end{align*}
      Note that  if $r \geq p$, then  $\Sp(r-r_{0}) \geq 1$, whence
      $\Sp(r-r_{0}) \geq 1$ in the setting of lemma.
      If $r_{0} \leq \Sp(r-r_{0}) \leq p-1$, then  since
      $r-r_{0} \equiv \Sp(r-r_{0})$ mod $(p-1)$,
      we have     $X_{r-r_{0},\,r} \cong V_{\Sp(r-r_{0})} \otimes V_{r_{0}}$,
      by \Cref{dimension formula for X_{r}} (i). If 
      $1 \leq  \Sp(r-r_{0}) < r_{0}$, then since
       $r- \Sp(r-r_{0}) \equiv r_{0} $ mod $(p-1)$ with 
       $1 \leq r_{0} \leq p-1$, we have 
      $  X_{r-r_{0},\,r} \cong V_{r_{0}} \otimes V_{\Sp(r-r_{0})}$,
      by \Cref{dimension formula for X_{r}} (i).
       Hence
       $X_{r-r_{0},\,r} \cong V_{\Sp(r-r_{0})} \otimes V_{r_{0}}$ in either 
       case. Now  parts (ii) and (iii)  follow 
       from the short exact sequence above,
      and \Cref{final quotient structure} (i) and (ii) respectively. 
 \end{proof}
 
\begin{remark} \label{r - i vs r - r_0}
       Using Lemma \ref{Sp(r-i), Sp(r-j)} (ii), 
        the condition $\Sp(r-i) \leq p-1$ implies that 
        $\Sp(r-r_{0}) = \Sp(r)-r_{0}= \Sp(r-i)+i-(p-1)-r_{0} \leq i -r_{0} \leq p-1$.
        So the extra assumption $\Sp(r-r_{0}) \leq p-1$ in part (iii) of the above
        theorem
        is redundant. 
\end{remark}
As a corollary, we have the following dimension formula.
\begin{corollary}  \label{dimension r_0 <  i}   
        Let  $p\geq 3$, $1 \leq i \leq p-1$ and  $(i+1)(p+1)<r \equiv r_{0} 
         ~ \mathrm{mod}~p$ with $0 \leq r_{0}<i$. Then 
          \begin{align*}
                \dim X_{r-i,\,r} =
                 \begin{cases}
                        (i+1)(p+1), & ~ \mathrm{if} ~ \Sp(r-r_{0}) \geq p, \\
                       ( i-r_{0})(p+1)+ (\Sp(r-r_{0})+1)(r_{0}+1), &
                        ~ \mathrm{if} ~ \Sp(r-r_{0}) \leq p,~
                                \Sp(r-i) \geq p, \\
                        (p+r_{0}+1) \Sp(r-r_{0}),
                        &~ \mathrm{if} ~ \Sp(r-i) < p.
                  \end{cases}
           \end{align*}
\end{corollary}

\section{Structure of \texorpdfstring{$Q(i)$}{}} 
\label{section Q}

Recall that $\theta = X^{p}Y- XY^{p}$,
$V_{r}^{(m)} = \lbrace F(X,Y) \in \f[X,Y] : \theta^{m} \mid F ~ \text{in} ~ \f[X,Y] 
\rbrace$ and  $X_{r-i,\,r}^{(m)} = V_{r}^{(m)} \cap X_{r-i,\,r}$ for $0 \leq i \leq r$
and $m \in \Z_{\geq 0}$.
In this section, we study the quotient module $Q(i)$ 
of $V_{r}$, for $0 \leq i \leq p-1$,  defined by
\begin{align*}
    Q(i) = \frac{V_{r}}{X_{r-i,\,r}+V_{r}^{(i+1)}}.
\end{align*}
Similarly, for $1 \leq i \leq p-1$, let 
\begin{align*}
       P(i) = \frac{V_{r}}{X_{r-(i-1),\,r}+V_{r}^{(i+1)}}.
\end{align*}
These modules play an important role  in the study of 
the reductions of Galois representations mod $p$
via the mod $p$ local Langlands correspondence. 

Observe that, for $1 \leq i \leq p-1$, we have  the following 
exact sequence 
          \begin{align}\label{Q i-1 and P i}
              0 \rightarrow \frac{X_{r-(i-1),\,r}+V_{r}^{(i)}}{X_{r-(i-1),\,r}+V_{r}^{(i+1)}}
              \rightarrow P(i) \rightarrow Q(i-1) \rightarrow 0,
          \end{align}
          where the first map is the inclusion and the last map is the quotient
          map. By the second isomorphism theorem, we have
          $(X_{r-(i-1),\,r}+V_{r}^{(i)})/(X_{r-(i-1),\,r}+V_{r}^{(i+1)})$
          is isomorphic to $V_{r}^{(i)}/(X_{r-(i-1),\, r}^{(i)}+ V_{r}^{(i+1)})$,
          which is also the cokernel of the map 
          \[ X_{r-(i-1),\,r}^{(i)}/X_{r-(i-1),\,r}^{(i+1)} 
          \hookrightarrow V_{r}^{(i)}/V_{r}^{(i+1)} . \]
           In the process of  determining the $Q(i)$ in 
           this section, we will 
           need to determine  the quotients
          $X_{r-i,\,r}^{(j)}/X_{r-i,\,r}^{(j+1)}$, for $0 \leq i,j \leq p-1$, cf.
          \Cref{Structure X(1)}, \Cref{reduction}, \Cref{singular quotient X_{r}}, 
          \Cref{singular quotient i < [a-i]}, \Cref{singular i= [a-i]},
           \Cref{singular i>r-i}, \Cref{singular i=a} and \Cref{singular i=p-1} below. 
           Since the  structure of
           $V_{r}^{(i)}/V_{r}^{(i+1)}$ is well known, cf. \Cref{Breuil map}, we can 
           deduce the structure of the cokernel of the map above.
           So in principle we may deduce
           the structure of $P(i)$ from $Q(i-1)$, and so for the rest of this paper
           we concentrate on determining just the $Q(i)$.
          
Note that the definition of $Q(i)$ and $P(i)$ involves $r$. 
In this section, to simplify notation we  often denote
 $X_{r-i,\,r}$ by just $X_{r-i}$ etc.  
%
For every $0 \leq i \leq p-1$ and $ m,n \in \mathbb{Z}_{\geq 0}$ with $n \leq m$, we have 
 $$
       X_{r-i}^{(n)}  \hookrightarrow V_{r}^{(n)} \twoheadrightarrow
        V_{r}^{(n)}/V_{r}^{(m)}.                
 $$
 Clearly the kernel of the composition is $X_{r-i}^{(m)} $. 
 For every $ 0 \leq i \leq p-1$, we have an exact sequence
 \begin{align}\label{Q(i) exact sequence}
 0  \rightarrow X_{r-i}/X_{r-i}^{(i+1)} \rightarrow V_{r}/V_{r}^{(i+1)} \rightarrow 
 Q(i) \rightarrow 0,
 \end{align}
 where the leftmost map is induced by inclusion and the rightmost map is the quotient map.
 %

In \cite[(4.2)]{Glover}, Glover showed that 
$V_r/V_r^{(1)}$ is periodic with period $p-1$, for $r \geq p$. We generalize
this result using the ring of dual numbers etc. to show that $V_{r}/V_{r}^{(m)}$ is 
periodic with  period $p(p-1)$, i.e., $V_{r}/V_{r}^{(m)} \cong V_{r+p(p-1)}/V_{r+p(p-1)}^{(m)}$, 
for $1 \leq m \leq p-1$ and $r \geq m(p+1)-1$. 

\begin{lemma}\label{quotient periodic}
       Let $1 \leq m \leq p$,  and let $r \geq s \geq 
       m(p+1)-1 $. If $r \equiv s ~ \mathrm{mod}~p(p-1)$, then
       \begin{align*}
           \frac{V_{r}}{V_{r}^{(m)}}  \cong  \frac{V_{s}}{V_{s}^{(m)}}.
       \end{align*}
\end{lemma}
\begin{proof}
       Let $\epsilon$ be a variable and let $R = \f[\epsilon]/(\epsilon^{m})$.
       Let $G(R) = \mathrm{GL}_{2}(R)$ and $B(R)$ denote the Borel subgroup
       consisting of the upper triangular matrices in $G(R)$. 
       Define a map
       $ \psi_{r}:  V_{r} \rightarrow \ind_{B(R)}^{G(R)}(\chi_2^{r})$, by
       $$ 
               \psi_{r}(P(X,Y))(\gamma)= P((0,1)\gamma) = P(c,d),
        $$ 
        for all  $\gamma = \begin{psmallmatrix}   a & b \\ c & d \end{psmallmatrix}
         \in G(R)$ and $P(X,Y) \in V_{r}$.  
       Observe that, for $\gamma = 
       \begin{psmallmatrix}   a & b \\ c & d \end{psmallmatrix} \in G(R)$,
       $\gamma' = \begin{psmallmatrix}   a' & b '\\ 0 & d' \end{psmallmatrix}
       \in B(R)$ and  $P(X,Y) \in V_{r}$, we have 
       \[
           \psi_{r}(P(X,Y))(\gamma' \gamma) =  P(cd',dd') = d'^{r} P(c,d)
           = \gamma' \cdot  \psi_{r}(P(X,Y))(\gamma),
       \]
       so $\psi_{r}(P(X,Y)) \in \ind_{B(R)}^{G(R)}(\chi_2^{r})$ and 
       $\psi_{r}$ is well-defined.  For   $\gamma_{1} = 
       \begin{psmallmatrix}   a_{1} & b_{1} \\ c_{1} & d_{1} \end{psmallmatrix}
       \in \Gamma$ and $\gamma_{2} = 
       \begin{psmallmatrix}   a_{2} & b_{2} \\ c_{2} & d_{2} \end{psmallmatrix}
       \in G(R)$, we have
       \begin{align*}
              \psi_{r}(\gamma_{1} \cdot P (X,Y)) (\gamma_{2}) & =
              (\gamma_{1} \cdot P)(c_{2}, d_{2}) = 
              P(a_{1}c_{2}+c_{1}d_{2}, b_{1}c_{2}+ d_{1}d_{2})=
              \psi_{r}(P(X,Y)) (\gamma_{2} \gamma_{1}) \\
              & = (\gamma_{1} \cdot    \psi_{r}(P(X,Y)) ) (\gamma_{2}).
       \end{align*}
       Thus $\psi_{r}$ is  $\f[\Gamma]$-linear, as $\psi_{r}$ is  $\f$-linear.
       We next show that ker $\psi_{r} = V_{r}^{(m)}$. For 
       $c = c_{0}+c_{1} \epsilon+ \cdots + c_{m-1} \epsilon^{m-1}$
       and $d =d_{0}+d_{1} \epsilon+ \cdots + d_{m-1} \epsilon^{m-1} \in R$,
       we have 
       \[
             c^{p} d - c d^{p} = c_{0} d - c d_{0}
             \in \epsilon R.
       \]       
       So $(c^{p} d - c d^{p})^{m} =0$ and $V_{r}^{(m)} \subseteq$ ker $\psi_{r}$.
       Conversely let $P(X,Y) = \sum\limits_{i=0}^{r} a_{i} X^{r-i}Y^{i} \in $  ker $\psi_{r}$.
       Then
       \[
           \psi_{r}(P(X,Y)) 
           \begin{psmallmatrix} \epsilon^{m-1}  & -1 \\ 1 & \epsilon  \end{psmallmatrix}
            = P(1,\epsilon) =  \sum\limits_{i=0}^{r} a_{i}  \epsilon^{i} 
            = \sum\limits_{i=0}^{m-1} a_{i}  \epsilon^{i}  \quad 
            ( ~\because \epsilon^{m}=0).
       \]
       So $a_{0}$, $a_{1}, \ldots , a_{m-1} =0$ and $Y^{m} \mid P(X,Y)$.
        Since $\psi_{r}$ is $\Gamma$-linear we see that 
        for all $\gamma \in \Gamma$,
        $\gamma \cdot P(X,Y) \in $ ker $\psi_{r}$, 
        whence  $\gamma^{-1} \cdot Y^{m}   \mid P(X,Y)$.
        Taking $\gamma = 
        \begin{psmallmatrix} -\lambda & -1 \\ 1 & 0 \end{psmallmatrix} $ for 
        $\lambda \in \f$, we see that $(X- \lambda Y)^{m} \mid P(X,Y)$, for 
        $\lambda \in \f$. 
        Hence $P(X,Y) \in V_{r}^{(m)}$  and  ker $\psi_{r} \subseteq V_{r}^{(m)}$.
       This shows that ker $\psi_{r} = V_{r}^{(m)}$.  Thus $\psi_{r}$ induces
       an injective map  $\psi_{r}: V_{r}/V_{r}^{(m)} \rightarrow 
        \ind_{B(R)}^{G(R)}(\chi_2^{r})$. Let $r$, $s$ be as in   hypothesis. Then  
       we have $\ind_{B(R)}^{G(R)}(\chi_2^{s}) = \ind_{B(R)}^{G(R)}(\chi_2^{r})$,
       as $\chi_2^{r-s}$ is the trivial character.
      By Lemma~\ref{basis}, the set $\Lambda$ provides a basis of $V_r/V_{r}^{(m)}$ and similarly of $V_s/V_{s}^{(m)}$.
       For $\gamma =
        \begin{psmallmatrix}  a & b \\ c & d \end{psmallmatrix} \in G(R) $ and
       $0 \leq i \leq s-m$, we have
       \begin{align*}
                \psi_{s}(X^{s-i}Y^{i})(\gamma)
                 & = \begin{cases}
                                0, &\mathrm{if }~ c \in \epsilon R, \\
                                c^{s-i} d^{i}, & \mathrm{if } ~ c \not \in \epsilon R.
                 \end{cases} \\
                 &=  \psi_{r}(X^{r-i}Y^{i})(\gamma), 
       \end{align*}
       since $r \equiv s ~\mathrm{mod}~p(p-1)$.
       Since $ s-m \geq m(p+1)-(m+1)$, we see that $\psi_{s}$ and $\psi_{r}$
       agree on the
       first kind of basis elements in $\Lambda$.
       Similarly for $s-m < i  \leq s$, one checks that     
       $\psi_{s}(X^{s-i}Y^{i})(\gamma) = \psi_{r}(X^{s-i}Y^{r-s+i})(\gamma)$. 
      Since $r-m < r-s+i \leq r$, we see that that $\psi_{s}$ and $\psi_{r}$
      also agree on the
       second kind of basis elements in $\Lambda$.      
       Thus
       $\psi_{s}(V_{s}/V_{s}^{(m)}) = \psi_{r}(V_{r}/V_{r}^{(m)})$ and we
       have a $\Gamma$-linear  isomorphism 
       \begin{equation*}
               \psi: V_{s}/V_{s}^{(m)} \overset{\psi_{s}}{\longrightarrow} 
              \psi_{s}(V_{s}/V_{s}^{(m)}) = \psi_{r}(V_{r}/V_{r}^{(m)})
               \overset{\psi_{r}^{-1}}{\longrightarrow} V_{r}/V_{r}^{(m)}. \qedhere
       \end{equation*}
\end{proof}

For $n \geq 0$ and $0 \leq j  \leq i \leq p-1$, we
  have $X_{r-j}^{(n)} \subseteq X_{r-i}^{(n)}$, by \Cref{first row filtration}. 
  Therefore, for every  $0 \leq j  \leq i \leq p-1$ and $0 \leq n \leq m$, we have
 \[
X_{r-j}^{(n)}/X_{r-j}^{(m)} \subseteq X_{r-i}^{(n)}/X_{r-i}^{(m)}  \subseteq 
 V_{r}^{(n)}/V_{r}^{(m)}.
\]
As a corollary, we obtain

\begin{corollary}\label{arbitrary quotient periodic}  
       Let $0 \leq n <  m \leq p$, $0 \leq i \leq p-1$, $r \geq s \geq 
       m(p+1)-1$. If $r \equiv s ~ \mathrm{mod}~p(p-1)$, then
       \begin{align*}
           \frac{V_{r}^{(n)}}{V_{r}^{(m)}}  \cong  \frac{V_{s}^{(n)}}{V_{s}^{(m)}} \quad \text{and} \quad
       \frac{X_{r-i}^{(n)}}{X_{r-i}^{(m)}}  \cong \frac{X_{s-i}^{(n)}}{X_{s-i}^{(m)}}.
       \end{align*}
       In particular, $Q(i)$ is periodic with period $p(p-1)$, 
       for all $0 \leq i \leq p-1$ and all $r \geq i(p+1) +p$.
\end{corollary}
\begin{proof}
       Let $\psi: V_{s} / V_{s}^{(m)} \rightarrow V_{r}/ V_{r}^{(m)}$ 
       be defined as  in the proof of  \Cref{quotient periodic}. 
       Explicitly we have 
                  \begin{align*}
                         \psi(X^{s-i}Y^{i}+V_{s}^{(m)}) =
                         \begin{cases}
                                X^{r-i}Y^{i}+V_{r}^{(m)}, &~ \mathrm{if}~0\leq i \leq s-m \\                              
                                 X^{s-i}Y^{r-s+i}+V_{r}^{(m)}, &~ \mathrm{if}~s-m < i \leq s. 
                         \end{cases}
                  \end{align*}
       We claim that
       $\psi (V_{s}^{(n)}/V_{s}^{(m)}) = V_{r}^{(n)}/V_{r}^{(m)}$. 
        Let $F(X,Y) =  \sum_{j=0}^{s} a_{j} X^{s-j}Y^{j}  \in V_{s}^{(n)}$. Then 
        \[
           \psi(F(X,Y)+V_{s}^{(m)}) = \sum_{j=0}^{s-m} a_{j} X^{r-j}Y^{j} + 
        \sum_{j=s-(m-1)}^{s} a_{j} X^{s-j}Y^{r-s+j}+V_{r}^{(m)}.
        \]
         Clearly
        the coefficient of $X^{s-i}Y^{i}$ (resp. $X^{i}Y^{s-i}$) and 
        $X^{r-i}Y^{i}$ (resp. $X^{i}Y^{r-i}$) in  $F(X,Y)$ 
        and $\psi(F(X,Y))$
        are equal if $0 \leq i < n$, so both vanish as 
        $F(X,Y) \in V_{s}^{(n)}$.  Also noting that 
        $ r \equiv s$ mod $p$,  for 
        $1 \leq t \leq p-1$ and $0 \leq l < n \leq p-1$, by Lucas' theorem,
          we have 
        \begin{align*}
              \sum_{\substack{0 \leq j \leq s-m \\ j \equiv t~\mathrm{mod}~(p-1)}}^{} 
              a_{j} \binom{j}{l} + 
              \sum_{\substack{s-m<  j \leq s \\ j \equiv t~\mathrm{mod}~(p-1)}}^{} 
              a_{j} \binom{r-s+j}{l} \equiv 
              \sum_{\substack{0 \leq j \leq s \\ j \equiv t~\mathrm{mod}~(p-1)}}^{} 
              a_{j} \binom{j}{l} \mod p.  
        \end{align*}
         Since  $F(X,Y) \in V_{s}^{(n)}$, we see that the right hand side  vanishes, by 
         \Cref{divisibility1}. Hence by  \Cref{divisibility1}, we obtain 
         $\psi(F) \in V_{r}^{(n)}$, so $\psi$ induces an  injective map
         $V_{s}^{(n)} / V_{s}^{(m)} \rightarrow V_{r}^{(n)}/ V_{r}^{(m)}$,
         for all $0 \leq n \leq m$.
         Since  dim $V_{s}^{(n)} / V_{s}^{(m)} = \dim V_{r}^{(n)}/ V_{r}^{(m)}$,
         the above map is an isomorphism.
         
         Since $X^{r-i}Y^{i}$ generates $X_{r-i}$ and 
        $\psi(X^{s-i}Y^{i}+V_{s}^{(m)}) = X^{r-i}Y^{i}+V_{r}^{(m)}$ we see that 
        $\psi(X_{s-i} +V_{s}^{(m)}) = X_{r-i}+V_{r}^{(m)}$. Thus
        $\psi(X_{s-i}^{(n)} +V_{s}^{(m)}) = \psi((X_{s-i} +V_{s}^{(m)}) \cap( V_{s}^{(n)} +V_{s}^{(m)}))
         = \psi(X_{s-i}+V_{s}^{(m)}) \cap \psi(V_{s}^{(n)}+V_{s}^{(m)}) 
         = (X_{r-i} \cap V_{r}^{(n)})+V_{r}^{(m)} = X_{r-i}^{(n)}+V_{r}^{(m)}$.
        By the second isomorphism theorem,  we have  
        \[
          \frac{X_{s-i}^{(n)}}{ X_{s-i}^{(m)} } \cong 
        \frac{X_{s-i}^{(n)} +V_{s}^{(m)}}{V_{s}^{(m)}} \overset{\psi} {\cong}
        \frac{X_{r-i}^{(n)} +V_{r}^{(m)}}{V_{r}^{(m)} }
         \cong \frac{X_{r-i}^{(n)}}{ X_{r-i}^{(m)}}.
         \]
The periodicity of $Q(i)$ now follows immediately from the exact sequence \eqref{Q(i) exact sequence}.
\end{proof}

         For every $0 \leq j \leq  i \leq p-1$,  we have the following
         commutative diagram:
%
\begin{equation}\label{commutative diagram}
  \begin{tikzcd}
      & 0 \arrow{d} & 0 \arrow{d} & 0 \arrow{d} &  \\
     0 \arrow{r} &  \frac{X_{r-i}^{(j)}}{ X_{r-i}^{(i+1)}} \arrow{r} \arrow{d}&
     \frac{X_{r-i}} {X_{r-i}^{(i+1)}} \arrow{r} \arrow{d}&
     \frac{X_{r-i}}{ X_{r-i}^{(j)}} \arrow{r} \arrow{d}  & 0 \\
     0 \arrow{r} &  \frac{V_{r}^{(j)}}{ V_{r}^{(i+1)}} \arrow{r} \arrow{d}&
     \frac{V_{r}} {V_{r}^{(i+1)}} \arrow{r} \arrow{d} &
     \frac{V_{r}}{ V_{r}^{(j)}} \arrow{r} \arrow{d} & 0 \\
     0 \arrow{r} &  \frac{V_{r}^{(j)}}{ X_{r-i}^{(j)}+V_{r}^{(i+1)}} \arrow{r}  \arrow{d}&
     Q(i) \arrow{r}  \arrow{d} &
     \frac{V_{r}}{ X_{r-i}+V_{r}^{(j)}} \arrow{r} \arrow{d} & 0,\\
        & 0 & 0  & 0 & 
   \end{tikzcd}
\end{equation}
where the leftmost bottom entry is isomorphic to 
$(X_{r-i}+V_{r}^{(j)})/ (X_{r-i}+V_{r}^{(i+1)})$, by the second isomorphism theorem.
Recall that $V_{r} \supseteq V_{r}^{(1)} \supseteq \cdots $ is a descending chain of
submodules of $V_{r}$. By \Cref{Breuil map}, we know the structure of
the quotients of successive terms in the chain, hence we know the structure
of arbitrary quotients in the chain. 
Thus by the commutative diagram above, to determine the structure
of $Q(i)$, it is enough to determine the quotients $X_{r-i}^{(n)}/X_{r-i}^{(m)}$,
for all $0 \leq n \leq m$. These quotients in turn are determined by 
$X_{r-i}^{(j)}/X_{r-i}^{(j+1)}$, for $j \geq 0$. 

The first result describes $X_{r-i}/ X_{r-i}^{(1)}$, for $0 \leq i \leq p-1$.
\begin{lemma}\label{Structure X(1)}
          Let $p \geq 2$ and $p \leq r \equiv a \mod(p-1)$ with $1 \leq a \le p-1$. For 
          $ 0 \leq i \leq p-1$ we have
          \begin{align*}
                 \frac{X_{r-i}}{X_{r-i}^{(1)}}  \cong
                 \begin{cases}
                 V_{a}, & \mathrm{if} ~ 0 \leq i <a, \\
                 V_{r}/V_{r}^{(1)}, & \mathrm{if} ~ a \leq i \leq p-1.
                 \end{cases}
          \end{align*}
\end{lemma}
\begin{proof}
      By \cite[(4.5)]{Glover}, we have $X_{r}/X_{r}^{(1)} \cong V_{a}$. 
      For $0 \leq i \leq p-1$, we have $V_{a} \cong X_{r}/X_{r}^{(1)} 
       \subseteq X_{r-i}/X_{r-i}^{(1)}$.
       Since $X^{r-i}Y^{i}$ generates $X_{r-i}$ as a $\Gamma$-module, we see that 
        $X^{r-i}Y^{i}$  generates $X_{r-i}/X_{r-i}^{(1)}$.
      Therefore $X_{r-i}/X_{r-i}^{(1)}= V_{r}/V_{r}^{(1)} $ if and only if the image of   
      $X^{r-i}Y^{i}$  in the quotient $V_{p-1-a} \otimes D^{a}$ of 
      $V_{r}/V_{r}^{(1)}$  in the exact sequence
       \eqref{exact sequence Vr} (for $m=0$) is non-zero 
      if and only if $i \geq a$,  by \Cref{Breuil map} (for $m=0$). This finishes the proof.
\end{proof}       
%

Recall that $[n]$ denotes the congruence class of $n$ modulo $(p-1)$ in 
$\lbrace 1, 2 , \ldots, p-1 \rbrace$, for $n \in \mathbb{Z}$.
We now make an important further reduction. For $1 \leq j \leq p-1$, we show 
that $X_{r-i}^{(j)}/X_{r-i}^{(j+1)}$ equals $X_{r-j'}^{(j)}/X_{r-j'}^{(j+1)}$,
where $j'$ is the largest integer among the integers $0$, $j$ and $[r-j]$ which is less
than or equal to $i$. In order to do this, it is convenient to introduce the following 
notation. For every $0 \leq i , j \leq p-1$, define the following 
subquotient of $V_{r}^{(j)}/V_{r}^{(j+1)}$ by 
\[
    Y_{i,j}  := \frac{X_{r-i}^{(j)}/ X_{r-i}^{(j+1)}}{X_{r-(i-1)}^{(j)}/ X_{r-(i-1)}^{(j+1)}}.
\]
By the second and third  isomorphism theorems respectively, we have
      \begin{align*}
              \frac{X_{r-i}^{(j)}/ X_{r-i}^{(j+1)}}{X_{r-(i-1)}^{(j)}/ X_{r-(i-1)}^{(j+1)}} 
              &  \cong  \frac{X_{r-i}^{(j)}/X_{r-i}^{(j+1)}} {(X_{r-(i-1)}^{(j)}+
               X_{r-i}^{(j+1)})/X_{r-i}^{(j+1)}}  
               \cong  \frac{X_{r-i}^{(j)}}{X_{r-(i-1)}^{(j)}+X_{r-i}^{(j+1)}},
        \end{align*}
      and, by the analog of the Zassenhaus lemma   \cite[Lemma 3.3]{Lang}
      for modules  and  the inclusions
      $X_{r-i}^{(j+1)} \subseteq  X_{r-i}^{(j)} \subseteq V_{r}^{(j)} $, we have
        \begin{align*}
               \frac{X_{r-i}^{(j)}}{X_{r-(i-1)}^{(j)}+X_{r-i}^{(j+1)}}
             = \frac{\left(X_{r-i} \cap V_{r}^{(j)} \right) + X_{r-i}^{(j+1)}}
                   { \left(X_{r-(i-1)} \cap V_{r}^{(j)} \right) + X_{r-i}^{(j+1)}} 
              \cong  
                   \frac{X_{r-i}^{(j)}+X_{r-(i-1)}}{X_{r-i}^{(j+1)}+X_{r-(i-1)}}.
      \end{align*}
      Hence
       \begin{align}\label{Y i,j}
            Y_{i,j} \cong   \frac{X_{r-i}^{(j)}+X_{r-(i-1)}}{X_{r-i}^{(j+1)}+X_{r-(i-1)}}.
       \end{align}

\begin{lemma}\label{reduction}
      Let  $p \geq 2$, $0 \leq i \leq p-1$ and $1 \leq j \leq p-1$. For every 
      $p \leq r \equiv a \mod (p-1)$ with $1 \leq a \leq p-1$, and  we have 
\begin{align*}
       \frac{ X_{r-i}^{(j)}}{X_{r-i}^{(j+1)}}=
       \begin{cases}
             X_{r-j}^{(j)}/ X_{r-j}^{(j+1)}, 
             & \mathrm{if}~ j \leq i < [a-j ]
             ~ \mathrm{or} ~ [a-j] \leq  j   \leq i , \\
             X_{r-[a-j]}^{(j)}/ X_{r-[a-j ]}^{(j+1)}, 
             & \mathrm{if}~ [ a-j ] \leq i < j ~ 
             \mathrm{or}~ j \leq     [a-j ] \leq i , \\
             X_{r}^{(j)}/X_{r}^{(j+1)}, & \mathrm{if}~ i < j \leq [a-j]~
             \mathrm{or} ~ i < [a-j] \leq j.
    \end{cases}
\end{align*}
\end{lemma}
\begin{proof}
      If $i=0$, then there is nothing to prove. So assume $i \geq 1$.
      We  claim that  for $i \geq 1$ and $j \neq i$, $[a-i]$, we have 
      $X_{r-i}^{(j)}/ X_{r-i}^{(j+1)} = X_{r-(i-1)}^{(j)}/ X_{r-(i-1)}^{(j+1)}$.  Suppose not.
      Then $Y_{i,j}$, a non-zero subquotient of $V_{r}^{(j)}/V_{r}^{(j+1)}$, and  
       $(X_{r-i}^{(j)}+ X_{r-(i-1)})/ (X_{r-i}^{(j+1)}+ X_{r-(i-1)})$,
        a subquotient of $X_{r-i}/X_{r-(i-1)}$,  have a common JH factor, by \eqref{Y i,j}.
      Therefore, by  \Cref{induced and star}  and \Cref{induced and successive},  we get
      $\ind_{B}^{\Gamma}(\chi_{1}^{j} \chi_{2}^{r-j})$
      and $\ind_{B}^{\Gamma}(\chi_{1}^{r-i} \chi_{2}^{i})$ have a common  JH factor.
      But this  is not possible by \Cref{Common JH factor}, 
      since $j \neq [a-i] $, $i$.  
      This proves the claim. 
      
      Since $1 \leq i, j \leq p-1$, we have  $j \neq i$, $[a-i]$
      if and only if  $i \neq j$, $[a-j]$.
      If $j \leq i < [a-j ]$ or $ [a-j] \leq j \leq i $, then  applying the 
       claim above repeatedly, we have
      $ X_{r-i}^{(j)}/ X_{r-i}^{(j+1)} = X_{r-(i-1)}^{(j)}/ X_{r-(i-1)}^{(j+1)}  =
      \cdots = X_{r-j}^{(j)}/ X_{r-j}^{(j+1)}$. Similarly if $[a-j] \leq i < j $ or 
      $j \leq [a-j]  \leq i $,  then  by the claim above, we have $ X_{r-i}^{(j)}/ X_{r-i}^{(j+1)} = 
      X_{r-(i-1)}^{(j)}/ X_{r-(i-1)}^{(j+1)}  =\cdots = X_{r-[a-j]}^{(j)}/ X_{r-[a-j]}^{(j+1)}$.
      Finally if $i < j \leq [a-j]$ or $i<[a-j] \leq j$, then by the claim above 
      we have $ X_{r-i}^{(j)}/ X_{r-i}^{(j+1)}
      = X_{r-(i-1)}^{(j)}/ X_{r-(i-1)}^{(j+1)}  = \cdots = X_{r}^{(j)}/ X_{r}^{(j+1)}$.
\end{proof}

\begin{corollary}\label{reduction corollary}     
          Let  $p \leq r \equiv a ~\mathrm{mod}~ (p-1)$ with 
          $1 \leq a \leq p-1$ and $0 \leq i, j \leq p-1$
          with $i \neq a$. Then
          \begin{enumerate}[label=\emph{(\roman*)}]
             \item If $ j \neq i, [a-i]$, then we have 
                      $X_{r-i}^{(j)}/ X_{r-i}^{(j+1)} = X_{r-(i-1)}^{(j)}/ X_{r-(i-1)}^{(j+1)}$.
             \item Let $i'=\min \lbrace i, [a-i] \rbrace$. Then $X_{r-i} = 
                       X_{r-(i-1)} + X_{r-i}^{(i')}$. 
                       As a consequence, $X_{r-i} \subseteq X_{r-(i-1)} +V_{r}^{(i')}$.    
          \end{enumerate}           
\end{corollary}
\begin{proof}
    The first assertion follows from \Cref{Structure X(1)} if $j=0$ and
    the claim proved in \Cref{reduction}
    if $j \geq 1$. 
    	Thus, by the second isomorphism theorem, we have
    	\[
    	\frac{X_{r-i}^{(j)}}{X_{r-i}^{(j+1)}} =  \frac{X_{r-(i-1)}^{(j)}}{X_{r-(i-1)}^{(j+1)}}
    	= \frac{X_{r-(i-1)}^{(j)}+X_{r-i}^{(j+1)}}{X_{r-i}^{(j+1)}}, ~ 
    	\forall \; 0 \leq j < i' .
    	\]
    	So $X_{r-i}^{(j)} = X_{r-(i-1)}^{(j)}+X_{r-i}^{(j+1)} \subseteq 
    	X_{r-(i-1)}+X_{r-i}^{(j+1)}$, for all $0 \leq j < i'$. Thus we have
    	\[
    	X_{r-i} \subseteq  X_{r-(i-1)}+X_{r-i}^{(1)} \subseteq \cdots
        \subseteq X_{r-(i-1)}+X_{r-i}^{(i')}.
    	\]
    	Since both $X_{r-(i-1)}$, $X_{r-i}^{(i')} \subseteq X_{r-i}$, 
    	the second assertion follows from the inclusion above.
    %
    %
\end{proof}
%

By Lemma~\ref{reduction}, to determine $X_{r-i}^{(j)}/ X_{r-i}^{(j+1)}$, for  all
$0 \leq i$, $j \leq p-1$, it is enough to determine $X_{r-j}^{(j)}/ X_{r-j}^{(j+1)}$, 
$X_{r-[a-j]}^{(j)}/ X_{r-[a-j]}^{(j+1)}$
and $X_{r}^{(j)}/ X_{r}^{(j+1)} $, for all $1 \leq j \leq p-1$, as we already know the 
structure of $X_{r-i}/X_{r-i}^{(1)}$, by \Cref{Structure X(1)}.
We begin by determining $X_{r}^{(j)}/X_{r}^{(j+1)}$, for $0 \leq j \leq p-1$.
%
\begin{lemma}\label{star=double star}
      Let  $0 \leq r \equiv a \mod (p-1)$ with $1 \leq a \leq p-1$.
      Then $X_{r}^{(1)} = X_{r}^{(2)} = \cdots = X_{r}^{(a)}$.
\end{lemma} 
 \begin{proof}
        If $a=1$, then there is nothing to prove. Assume $2 \leq a \leq p-1$.
        If $X_{r}^{(1)} =\cdots = X_{r}^{(m)} \neq X_{r}^{(m+1)}$,
        for some $1 \leq m \leq a-1$,  then \cite[Lemma 4.6]{BG15} (which in fact holds for all $r \geq 0$) 
        implies that
        $X_{r}^{(1)} = X_{r}^{(m)} \cong V_{p-a-1} \otimes D^{a}$ and 
        $X_{r}^{(m+1)}= (0)$.
        Therefore  $V_{p-a-1} \otimes D^{a} \cong X_{r}^{(m)}/X_{r}^{(m+1)}
        \hookrightarrow V_{r}^{(m)}/V_{r}^{(m+1)} \hookrightarrow \ind_B^\Gamma(\chi_1^{m} \chi_2^{r-m})$, 
        by Lemma~\ref{induced and star}. But this is not possible by Lemma~\ref{Structure of induced}, as $1 \leq m \leq a-1$. Therefore 
         $X_{r}^{(1)} =  X_{r}^{(a)}$.
\end{proof}
For $r \equiv a$ mod $(p-1)$,
by the above lemma we have $X_{r}^{(j)}/X_{r}^{(j+1)} = (0)$, for all $1 \leq j < a$.
We next derive a necessary and  sufficient condition under which 
$ X_{r}^{(a)}/X_{r}^{(a+1)}$  is non-zero.  We will soon see
that $X_{r}^{(j)}/X_{r}^{(j+1)} = (0)$
 if $1 \leq j \leq p-1$ and $j \neq a$  (cf. \Cref{singular quotient X_{r}}).
\begin{lemma}\label{quotient image}
   Let $ p \leq r \equiv a \mod (p-1)$ with $1 \leq a \leq p-1$. Then
   \begin{align*}
       G_{r}(X,Y) := \sum_{\lambda \in \f} (\lambda X+Y)^{r} +
        \delta_{a,p-1}X^{r}
   \end{align*}
   belongs to  $ X_{r}^{(a)}$.
   Further, $G_{r}(X,Y) \in V_{r}^{(p)} \Leftrightarrow G_{r}(X,Y) \in V_{r}^{(a+1)}
    \Leftrightarrow 
   \binom{r}{a} \equiv 0 \mod p$. Consequently, 
   $X_{r}^{(a)}/
   X_{r}^{(a+1)}$  is non-zero if and only if $\binom{r}{a} \not \equiv 0 \mod p$.  
\end{lemma}
\begin{proof}
   By \Cref{Basis of X_r-i}, we see that $G_{r}(X,Y) \in X_{r}$. Note  that 
   \begin{align}\label{G r expression}
      G_{r}(X,Y)  ~
       & =  ~ \sum_{l=0}^{r} \binom{r}{l} X^{r-l} Y^{l} \sum_{\lambda \in \f} 
         \lambda^{r-l} + \delta_{a,p-1} X^{r}  \nonumber \\
       & \overset{\eqref{sum fp}}{\equiv} ~
      - \sum_{\substack{ 0 < l < r \\ l \equiv a~ \mathrm{mod}~ (p-1)} } 
       \binom{r}{l} X^{r-l} Y^{l} \mod p.
   \end{align}
   %
   Clearly the coefficient of $X^{r}$, $Y^{r}$ in $G_{r}(X,Y)$ are zero. By
   \cite[Lemma 2.5]{BG15}, we have
   \begin{align*}
      \sum_{\substack{ 0 < l < r \\ l \equiv a ~ \mathrm{mod}~ 
       (p-1)} } \binom{r}{l} \equiv 0 \mod p.  
   \end{align*}
   Thus $G_{r}(X,Y) \in X_{r} \cap  V_{r}^{(1)}$ by \Cref{divisibility1},
   whence $G_{r}(X,Y) \in X_{r}^{(a)}$ by \Cref{star=double star}.
   This proves the first part.
   %
   %
   
   If $G_{r}(X,Y) \in V_{r}^{(p)}$, then  clearly $G_{r}(X,Y) \in V_{r}^{(a+1)}$
   as $a \leq p-1$. 
   If $G_{r}(X,Y) \in V_{r}^{(a+1)}$, then the coefficient of $X^{r-a} Y^{a}$ in 
   $G_{r}(X,Y)$ is zero, i.e., $\binom{r}{a} \equiv 0 \mod p$. Next suppose that 
   $\binom{r}{a} \equiv 0 \mod p$. Then by Lucas' theorem, we get $r \equiv 0,1,
    \ldots, a-1 \mod p$. This implies that the coefficients of $X^{r-a}Y^{a}$ and
   $X^{p-1}Y^{r-(p-1)}$ in $G_{r}(X,Y)$  are zero. Since $a$ (resp. $r-(p-1)$) is
   the only  number between $1$ and $p-1$ 
   (resp. $r-1$ and $r-(p-1)$) 
   which  is congruent to $a$ mod $(p-1)$, we get $X^{p}$, $Y^{p} \mid G_{r}(X,Y)$.
   Thus $G_{r}(X,Y)$ satisfies the condition (i) of \Cref{divisibility1} for  $m=p$.
   Moreover,  by \Cref{binomial sum}, for $ 1 \leq a \leq n \leq p-1$, we have
   \begin{align*}
       \sum_{\substack{ 0 < l < r \\ l \equiv a ~ \mathrm{mod}~(p-1) } }  
        \binom{r}{l} \binom{l}{n} &=
      \sum_{\substack{ 0 \leq l \leq r \\ l \equiv a ~ \mathrm{mod}~(p-1)}} 
        \binom{r}{l} \binom{l}{n} - \binom{r}{n}  \\
        & \equiv \binom{r}{n} \binom{[a-n]}{[a-n]} + \delta_{p-1,[a-n]} \binom{r}{n} 
         - \binom{r}{n} \\
        &  = \delta_{n,a} \binom{r}{a} 
        \equiv 0 \mod p,
   \end{align*}
   where in the second last step we used $[a-n]=p-1 \Leftrightarrow a=n$, as 
   $1 \leq a \leq n \leq p-1$. 
    Therefore $G_{r}(X,Y) \in V_{r}^{(p)}$ by \Cref{divisibility1}, as $G_{r}(X,Y)$ 
    already belongs to $V_{r}^{(a)}$.
\end{proof}
%
%


%
The next result describes all the quotients  $X_{r}^{(m)}/X_{r}^{(m+1)}$, for 
$0 \leq m \leq p-1$.
\begin{proposition}\label{singular quotient X_{r}}
	Let $p \geq 3$, $p \leq r\equiv a \mod (p-1)$ with $1 \leq a \leq p-1$. 
	For $0 \leq m \leq p-1$, we have
	\begin{align*}
		    \frac{X_{r}^{(m)}}{X_{r}^{(m+1)}} \cong \begin{cases}
		    V_{a}, & \mathrm{if} ~ m=0,\\
		    V_{p-1-a} \otimes D^{a}, & \mathrm{if}~ m=a~ \mathrm{and}~ 
		    r  \equiv a,a+1, \ldots, p-1 ~\mathrm{mod}~ p,\\
		    0, & \mathrm{if}~  
		    m=a ~ \mathrm{and}~  r  \equiv 0,1, \ldots, a-1 ~\mathrm{mod}~ p, ~ \mathrm{or} ~ m \neq 0,  a.
		    \end{cases}
	\end{align*}
\end{proposition}
\begin{proof}
	If $m = 0$,  we have $X_{r}/X_{r}^{(1)} \cong V_{a}$, by the exact sequence \eqref{Glover 4.5}, so assume $m \geq 1$. 
        Suppose 
	$r \equiv  a, a+1, \ldots, p-1$ mod $p$. Then, by Lucas' theorem, we have 
	$\binom{r}{a} \not \equiv 0$ mod $p$, so by \Cref{quotient image}, 
	we have $0 \neq G_{r}(X,Y) \in X_{r}^{(a)}/X_{r}^{(a+1)}$. 
	By \cite[Lemma 4.6]{BG15}, we have $X_{r}^{(1)} = (0) $ or $V_{p-1-a} \otimes D^{a}$. 
	Since $X_{r}^{(1)} \supseteq X_{r}^{(2)} 
	\supseteq \cdots \supseteq X_{r}^{(a)}  \supsetneq X_{r}^{(a+1)}$ is a descending chain, 
	we see that $X_{r}^{(a)}/X_{r}^{(a+1)} = V_{p-1-a} \otimes D^{a}$ and 
	$X_{r}^{(m)}/X_{r}^{(m+1)} = 0$, for $m \geq 1$ and $m \neq a$. Next suppose 
	that $r \equiv 0,1, \ldots, a-1$ mod $p$. Then, by Lucas' theorem we have 
        $\binom{r}{a} \equiv 0$ mod $p$. If
        $X_{r}^{(1)}=0$,
	then $X_{r}^{(m)}/X_{r}^{(m+1)}=0$, for all $m \geq 1$. If $X_{r}^{(1)} \neq 0$,
	then by \Cref{dimension formula for X_{r}}, we have $\dim X_{r} = p+1$. 
	Thus $G_{r}(X,Y) \neq 0$, as $G_{r}(X,Y)$   is a non-zero linear combination of the 
	basis elements of $X_{r}$ (cf. \Cref{Basis of X_r-i}).
	By  \Cref{quotient image}, we see that  $G_{r}(X,Y) 
	\in X_{r}^{(p)}$. As $X_{r}^{(1)}$
	is irreducible by \cite[Lemma 4.6]{BG15}, we obtain $X_{r}^{(1)} = 
	\cdots = X_{r}^{(p)}$,
	whence $X_{r}^{(m)}/X_{r}^{(m+1)}=0$, for all $1 \leq m \leq p-1$.
\end{proof}
\begin{corollary}\label{reduction corollary 2}
      Let $0 \leq i  \leq p-1$, $1 \leq j  \leq p-1$  and 
      $p \leq r \equiv a ~\mathrm{mod}~
      (p-1)$ with $1 \leq a \leq p-1$. If $j \neq a$ and $i < j \leq [a-j]$
      or $i < [a-j] \leq j$, then    $X_{r-i}^{(j)}/X_{r-i}^{(j+1)} = (0)$. 
\end{corollary}
\begin{proof}
        By the third part of \Cref{reduction}, we have 
        $X_{r-i}^{(j)}/X_{r-i}^{(j+1)} = X_{r}^{(j)}/X_{r}^{(j+1)}$. Since $j \neq 0$,
        $a$, the corollary follows from  \Cref{singular quotient X_{r}}.        
\end{proof}
We now introduce some notation. 
Let  $0 \leq i \leq p-1$ and let $r \geq p$.
Recall that 
by \Cref{induced and successive}, we have 
$[ \begin{psmallmatrix}  \lambda  & 1 \\ 1 & 0
\end{psmallmatrix}, e_{\chi_{1}^{r-i} \chi_{2}^{i}}]$ 
maps to  $X^{i} (\lambda X + Y)^{r-i}$ under
$\psi_{i}: \ind_{B}^{\Gamma}
(\chi_{1}^{r-i}\chi_{2}^{i}) \twoheadrightarrow X_{r-i}/X_{r-(i-1)}$.
 Let
\begin{align}\label{F i,r definition}  
      F_{i,r}(X,Y) ~ &:=~ \sum_{\lambda \in \f} \lambda^{[2i-a]} X^{i} (\lambda X + Y)^{r-i}
      \in X_{r-i}.
\end{align}
By \Cref{Structure of induced} (ii) (with $l = [2i-a]$),
 if $r-i \not \equiv i$ mod $(p-1)$, then  
 $\sum_{\lambda \in \f} \lambda^{[2i-a]} 
[ \begin{psmallmatrix}  \lambda  & 1 \\ 1 & 0
\end{psmallmatrix}, e_{\chi_{1}^{r-i} \chi_{2}^{i}}]$ generates 
$\ind_{B}^{\Gamma}(\chi_{1}^{r-i} \chi_{2}^{i})$ as a 
$\Gamma$-module, so 
its image $F_{i,r}(X,Y)$ mod $X_{r-(i-1)}$  under $\psi_{i}$ 
generates $X_{r-i}/X_{r-(i-1)}$ as a $\Gamma$-module. Let
 \begin{align}\label{G i,r definition}  
      G_{i,r} (X,Y)~ &:= ~ \sum_{\lambda \in \f}  X^{i} (\lambda X + Y)^{r-i}
      \in X_{r-i}.
\end{align}
By \Cref{Structure of induced} (i) (with $l=0$), we have 
$V_{[2i-a]} \otimes D^{a-i} \hookrightarrow  
\ind_{B}^{\Gamma}(\chi_{1}^{r-i}\chi_{2}^{i})$
is generated by $ \sum_{\lambda \in \f} 
[ \begin{psmallmatrix}  \lambda & 1 \\ 1 & 0 
\end{psmallmatrix}, e_{\chi_{1}^{r-i} \chi_{2}^{i}}]$ 
as a $\Gamma$-module, so
\begin{align}\label{W definition}   
      W_{i,r} ~:= ~ \mathrm{Image ~ of ~} V_{[2i-a]} \otimes D^{a-i} ~
      \mathrm{under} ~ \psi_{i}
\end{align}
is generated by $G_{i,r}(X,Y)$ mod $X_{r-(i-1)}$.
Clearly $X^{i}G_{r-i}(X,Y)-G_{i,r}(X,Y) = \delta_{[a-i],p-1} X^{r} 
\in X_{r} \subseteq X_{r-(i-1)}$, by \Cref{first row filtration}.
By \Cref{quotient image} and \Cref{surjection1}, we see that 
\begin{align}\label{G i,r}
      G_{i,r}(X,Y)+ \delta_{[a-i],p-1} X^{r} =
       X^{i}G_{r-i}(X,Y) = \phi_{i}(X^{i} \otimes G_{r-i}(X,Y)) 
       \in X_{r-i}^{[a-i]}. 
\end{align}
Thus
$G_{i,r}(X,Y) \in  X_{r-i}^{([a-i])} + X_{r-(i-1)}$,
whence $ W_{i,r} \subseteq (X_{r-i}^{([a-i])}+X_{r-(i-1)})/ X_{r-(i-1)}$.
\subsection{The case \texorpdfstring{$\boldsymbol{i \neq a, ~p-1}$}{•} }
\label{Section i not a nor p - 1}

In this subsection, we determine the quotients $Q(i)$ when $i \neq a$, $p-1$, 
by determining the structure of $X_{r-i}^{(i)}/X_{r-i}^{(i+1)}$ and $X_{r-i}^{([a-i])}/X_{r-i}^{([a-i]+1)}$.

The structure of $Q(0)$ is well known \cite{Glover}, \cite{BG09}. We have
\begin{lemma}\label{Structure Q(0)}
       Let $p \leq r \equiv a~\mathrm{mod}~p-1$, with $1 \leq a \leq p-1$.
       Then $Q(0) \cong V_{p-1-a} \otimes D^{a} $.
\end{lemma}
\begin{proof}
	By the exact sequence \eqref{exact sequence Vr} (with $m=0$), we have
       \[
           0 \rightarrow V_{a} \rightarrow V_{r}/V_{r}^{(1)} \rightarrow V_{p-1-a} 
           \otimes D^{a} \rightarrow 0.
       \]
       By \Cref{singular quotient X_{r}}, we have $X_{r}/X_{r}^{(1)} \cong V_{a}$.
       Now the lemma follows from the exact sequence \eqref{Q(i) exact sequence}.
\end{proof}
%
%
\begin{lemma}\label{socle term singular}
   Let $r \equiv a \mod (p-1)$ with $1 \leq a \leq p-1$ and $ 1 \leq i < p-1$
   with $i \leq  [a -i]$. Suppose $r \geq [a-i](p+1)+p$. Then
    $$ V_{[2i-a]} \otimes D^{a-i} \hookrightarrow  X_{r-i}^{([a-i])}/
       X_{r-i}^{([a-i]+1)} \Longleftrightarrow
       \binom{r-i}{[a-i]} \not \equiv 0~ \mathrm{mod}~ p.$$ 
\end{lemma}
\begin{proof}
    Let $j=[a-i]$. Then $r-i \equiv j$ mod $(p-1)$ and $1 \leq j \leq p-1$.
    Let  $F(X,Y):= G_{i,r}(X,Y) + \delta_{[a-i],p-1} X^{r}= X^{i} G_{r-i}(X,Y)$.
     By \eqref{G i,r}, we have  $F(X,Y) \in X_{r-i}^{(j)}$.
      Assume $\binom{r-i}{j} \not \equiv 0$ mod $p$. By
      \eqref{G r expression} we have
      \begin{align}\label{F expression in socle term singular}
       F(X,Y) =
      -\sum_{\substack{0 < l < r-i \\ l \equiv j ~ \mathrm{mod}~(p-1)}}^{}
      \binom{r-i}{l} X^{r-l}Y^{l}.
    \end{align}
    The coefficient of $X^{r-j}Y^{j}$ in $F(X,Y)$ equals $-\binom{r-i}{j}$,
    which is non-zero modulo $p$ by assumption. Hence $F(X,Y) 
    \not \in V_{r}^{(j+1)}$, by \Cref{divisibility1} and 
    $0 \neq F(X,Y) \in X_{r-i}^{(j)}/X_{r-i}^{(j+1)}$. We now show that 
    the image of $F(X,Y)$ is zero under the surjection 
    $V_{r}^{(j)}/V_{r}^{(j+1)} \twoheadrightarrow V_{p-1-[a-2j]} \otimes D^{a-j}$
    in \eqref{exact sequence Vr}. By assumption $i \leq j$.
    If $i <j $, then $r-j \equiv i  \not \equiv j$
    mod $p$, so the coefficient of  $X^{j}Y^{r-j}$ in 
    $F(X,Y)$ equals zero,  by \eqref{F expression in socle term singular}.
     If $i=j$, the coefficient of $X^{j}Y^{r-j}$ in  $F(X,Y)$ is still zero,
      by \eqref{F expression in socle term singular}, as $r-j  = r-i $.
      By \Cref{binomial sum}, 
      if $i \leq j$, we have
    \begin{align*}
           \sum_{\substack{0 < l < r-i \\ l \equiv j ~ \mathrm{mod}~(p-1)}}^{}
         - \binom{r-i}{l} \binom{l}{j} &=
          -\sum_{\substack{0 \leq l \leq r-i \\ l \equiv j ~ \mathrm{mod}~(p-1)}}^{}
          \binom{r-i}{l} \binom{l}{j} + \binom{r-i}{j}  \\
           &\equiv -\delta_{p-1,[j-j]}\binom{r-i}{[a-i]} ~\mathrm{mod}~p \\
           & = -\binom{r-i}{j}\\
           & = \mathrm{the ~ coefficient~ of ~} X^{r-j}Y^{j} -
                   (-1)^{j+1}\mathrm{the ~ coefficient~ of ~}~X^{j}Y^{r-j}.
    \end{align*}
    Thus by \Cref{breuil map quotient}, we have the image of $F(X,Y)$ under
    $V_{r}^{(j)}/V_{r}^{(j+1)} \twoheadrightarrow V_{p-1-[a-2j]} \otimes D^{a-j}$
    is zero. Therefore $ 0 \neq F(X,Y) \in V_{[a-2j]} \otimes D^{j} \hookrightarrow
    V_{r}^{(j)}/V_{r}^{(j+1)}$. This proves the \enquote*{only if} part.
    
    For the converse, assume $\binom{r-i}{j} \equiv 0$ mod $p$. 
    Thus by \Cref{quotient image}, we have $F(X,Y) \in X_{r-i}^{(p)}$.
    Since $i-1 <i \leq [a-i] $, 
    by the third part of \Cref{reduction} (with $i$ there equal to $i-1$ and $j=[a-i]$),  
    we have 
    $X_{r-(i-1)}^{(j)} / X_{r-(i-1)}^{(j+1)} = X_{r}^{(j)} / X_{r}^{(j+1)}$.
    By \Cref{singular quotient X_{r}},  we have $ X_{r}^{(j)} / X_{r}^{(j+1)} =(0)$,
    as $j \neq a$ since $i \neq p-1$. 
    Hence by \eqref{Y i,j},  we have
    \[
           \frac{X_{r-i}^{(j)}}{X_{r-i}^{(j+1)}} = Y_{i,j} \cong 
            \frac{ X_{r-i}^{(j)}+ X_{r-(i-1)}} { X_{r-i}^{(j+1)}+X_{r-(i-1)}}.
     \]
     Thus  $ V_{[2i-a]} \otimes D^{a-i} $ is a JH factor of $X_{r-i}^{(j)}/
       X_{r-i}^{(j+1)} $ if and only if 
       $ V_{[2i-a]} \otimes D^{a-i} $ is a JH factor  of
       $(X_{r-i}^{(j)}+ X_{r-(i-1)} )/ (X_{r-i}^{(j+1)}+X_{r-(i-1)})$.
     By \Cref{Structure of induced}  and \Cref{induced and successive}, 
     we have a map
     \[
       V_{[2i-a]} \otimes D^{a-i} \hookrightarrow \ind_{B}^{\Gamma}(\chi_{1}^{r-i}
        \chi_{2}^{i})
       \overset{\psi_{i}}{\twoheadrightarrow } \frac{X_{r-i}}{X_{r-(i-1)}}
        \twoheadrightarrow
       \frac{ X_{r-i}} { X_{r-i}^{(j+1)}+X_{r-(i-1)}} .
    \]
    Let $\psi'$ denote the above composition.  Since $W_{i,r}$
    is the image  $V_{[2i-a]} \otimes D^{a-i}$  under  $\psi_{i}$
    and $G_{i,r}(X,Y)$ generates  $W_{i,r}$,
    we see that  the image of $\psi '$ is  also generated by $G_{i,r}(X,Y)$
    mod  $X_{r-i}^{(j+1)}+X_{r-(i-1)}$. Note that
    \[
          G_{i,r}(X,Y) = X^{i}G_{r-i}(X,Y) - \delta_{j,p-1}X^{r} 
          = F(X,Y) -\delta_{j,p-1}X^{r} \in X_{r-i}^{(p)}+X_{r-(i-1)}.
    \]
    So the map $\psi'$ is zero. By \Cref{Structure of induced},
   the JH factor $V_{[2i-a]} \otimes D^{a-i}$ occurs with multiplicity
   one in $ \ind_{B}^{\Gamma}(\chi_{1}^{r-i} \chi_{2}^{i})$. Thus it follows
   $V_{[2i-a]} \otimes D^{a-i}$  is not a JH factor of 
   $X_{r-i}/ (X_{r-i}^{(j+1)}+X_{r-(i-1)})$, so is not a JH factor of its submodule
   $(X_{r-i}^{(j)}+ X_{r-(i-1)} )/ (X_{r-i}^{(j+1)}+X_{r-(i-1)})$, so of 
   $X_{r-i}^{(j)}/X_{r-i}^{(j+1)}$. This proves
   the \enquote*{if} part.
\end{proof}
%

To state the results in Sections~\ref{section i < a-i}, \ref{section i = a-i}, \ref{section i > a-i}, we 
need to extend the definitions of the sets $\mathcal{I}(a,i)$ and $\mathcal{J}(a,i)$ introduced in \eqref{interval I for i < a-i} 
and \eqref{interval J first} respectively.
For $1 \leq a$, $i \leq p-1$ and $i \neq a,$ $p-1$, we define 
$\mathcal{I}(a,i) \subseteq \lbrace 0,1, \ldots, p-1 \rbrace$,
a subset of the congruence classes  modulo $p$, by
\begin{align}\label{interval I}
   \mathcal{I}(a,i) = 
   \begin{cases}
      \lbrace a-i+1, a-i+2,  \ldots , a-1, a \rbrace, & \mathrm{if}~ i < a-i < a, \\
      \lbrace a-i, a-i+1, \ldots , a-1,  a\rbrace, & \mathrm{if}~ a-i \leq i < a,  \\
      \lbrace a, a+1, \ldots, [a-i]-1, [a-i]\rbrace^{c}, & \mathrm{if}~ a < i < [a-i], \\
      \lbrace a, a+1, \ldots, [a-i]-2, [a-i]-1 \rbrace^{c}, & \mathrm{if}~ a < [a-i] \leq i,
   \end{cases}
\end{align}
where $c$ in the superscript  denotes the complement in  
$\lbrace 0 , 1, \ldots, p-1 \rbrace$. 
Since any $p-1$ consecutive numbers define a congruence classes
modulo $p$, we may view $ \mathcal{I}(a,i) $ as an interval.
We leave it to the reader to check that
the above listed cases are mutually exclusive and cover all possibilities.
Also it can be checked that if
$i\neq a$, $a+1$, then $\mathcal{I}(a,i-1) \subseteq \mathcal{I}(a,i)$, for $i \geq 2$.

Similarly, we define the subset $\mathcal{J}(a,i)  \subseteq \lbrace 0,1, \ldots, p-1 \rbrace$ 
of the congruence classes  modulo $p$
as follows
\begin{align}\label{interval J}
      \mathcal{J}(a,i) =
      \begin{cases}
              \lbrace a-i, a-i+1 , \ldots,  a-2, a-1\rbrace, & \mathrm{if}~
             i < a-i <a ,\\
              \lbrace  a-i-1, a-i, \ldots,  a-2, a-1\rbrace, & \mathrm{if}~
             a-i \leq i <a , \\ 
             \lbrace  a-1,a, \ldots, [a-i]-2, [a-i]-1 \rbrace^{c}, &\mathrm{if}~
             a<i < [a-i], \\
              \lbrace a-1, a, \ldots,  [a-i]-3, [a-i]-2 \rbrace^{c}, &\mathrm{if}~
             a<[a-i] \leq i,
      \end{cases}
\end{align}
where $c$ in the superscript again denotes the complement in 
$\lbrace 0, 1, 2, \ldots, p-1 \rbrace$.  
As in the case of 
$\mathcal{I}(a,i)$, we think of 
$\mathcal{J}(a,i) $ as an interval.
Again we have $\mathcal{J}(a,i-1) \subseteq \mathcal{J}(a,i)$,
for all $2 \leq i \neq a$, $a+1$.  

Note that 
$\mathcal{J}(a,i) $ is essentially the translate of
$\mathcal{I}(a,i)$ to the `left' by 1 in the set of congruence classes
modulo $p$. More precisely, we have

\begin{lemma} \label{I vs J}
Let  $r \equiv a \mod (p-1)$ with $1 \leq a \leq p-1$, $1 \leq i \leq p-1$ with $i \neq a, p-1$, $r \equiv r_0 \mod p$
with $0 \leq r_0 \leq p-1$.  
\begin{enumerate}
      \item[(i)] If $i<[a-i]$, then $r_{0} \in \mathcal{J}(a,i)$ 
        $\Leftrightarrow$
       $r_{0} \in \mathcal{I}(a,i) \cup \lbrace [a-i] \rbrace$  and 
       $r \not \equiv [a-i]+i$ mod $p$.
       \item[(ii)]  If $i  \geq [a-i]$, then $r_{0} \in \mathcal{J}(a,i)$ 
       $\Leftrightarrow$
       $r_{0} \in \mathcal{I}(a,i) \cup \lbrace [a-i]-1 \rbrace$  and 
       $r \not \equiv [a-i]+i$ mod $p$.
\end{enumerate} 
\end{lemma}
\begin{proof}
  This follows from the definitions of the intervals above, noting that 
  $[a-i]+i  \equiv a$ (resp. $a -1$) mod $p$ if $i<a$ (resp. $i>a$).
\end{proof}

Next we give a criterion to check when the congruence class of $r$ modulo
$p$ belongs to the intervals above in terms of binomial coefficients.
\begin{lemma}
\label{interval and binomial}
        Let $r \geq p$ and $r \equiv r_{0} \mod p$ with $0 \leq r_{0} \leq p-1$.
        For $1 \leq a$, $i \leq p-1$ with $i \neq a$, $p-1$, we have
        \begin{enumerate}[label=\emph{(\roman*)}]
                \item If $i<[a-i]$, then $r_{0} \in \mathcal{I}(a,i)$ if and only if
                 $\binom{r-[a-i]-1}{i} \equiv 0 \mod p$.
                 \item If $i \geq [a-i]$, then $r_{0} \in \mathcal{I}(a,i)$ if
and only if
                 $\binom{r-[a-i]}{i+1} \equiv 0 \mod p$.
                 \item  If $i < [a-i]$, then $r_{0} \in \mathcal{J}(a,i)$ if and
only if
                 $\binom{r-[a-i]}{i} \equiv 0 \mod p$.
                 \item  If $i \geq [a-i]$, then $r_{0} \in \mathcal{J}(a,i)$ if
and only if
                 $\binom{r-[a-i]+1}{i+1} \equiv 0 \mod p$.
        \end{enumerate}
\end{lemma}
\begin{proof}
        The lemma follows from Lucas' theorem and the definitions of the intervals
      \eqref{interval I}  and \eqref{interval J}. To illustrate the proof
we prove (i).
      By Lucas' theorem, we have $\binom{r-[a-i]-1}{i} \equiv 0 \mod p$
      if and only if $r-[a-i]-1 \equiv 0, 1, \ldots, i-1 \mod p$
      if and only if $r \equiv [a-i]+1, [a-i]+2, \ldots, [a-i]+i$.
      If $i<a$, then $[a-i] =a-i$ and (i) follows in this case.
      If $i > a$, then $[a-i] =p-1+a-i$ and (i) follows by shifting the
congruence
      classes between $0$ and $p-1$ whenever they are larger than $p$.
      The other assertions are proved in a similar manner.
\end{proof}

We now determine the structure of $Q(i)$, for $i \neq a$, $p-1$, by considering the cases 
$i < [a-i]$, $i =[a-i]$ and $i> [a-i]$.

\subsubsection{The case \texorpdfstring{$\boldsymbol{i < [a-i].}$}{}}
\label{section i < a-i}
In this subsection, we determine the structure of the quotients $Q(i)$, for all $1 \leq i < p-1$,
 such that $i \not \equiv r$ mod $(p-1)$ and
$i < [a-i]$. Taking $j=i$ in diagram \eqref{commutative diagram} and using
\Cref{reduction corollary} (ii), we see that the rightmost bottom entry there equals
$Q(i-1)$. Thus, to determine the structure of $Q(i)$ in terms of $Q(i-1)$,
it is enough to determine the quotient $X_{r-i}^{(i)}/X_{r-i}^{(i+1)}$. 
The structure of $Q(i)$ is described in terms of $Q(i-1)$  in 
\Cref{Structure of Q(i) if i<[a-i]}.  

Before we proceed further, we note that for $1 \leq a,i \leq p-1$,
we have $[a-i] < p-1$ if and only if  $i \neq a$.  Thus,
the conditions $i \neq a$, $p-1$ and $i<[a-i]$ are equivalent to 
$1\leq i <[a-i] < p-1$.  


\begin{lemma}\label{Large Cong class Quotient non zero}
     Let $p \geq 3$, $ r \equiv a ~\mathrm{mod} ~ (p-1)$ with 
     $1 \leq a \leq p-1$, $r \equiv r_{0}~ \mathrm{mod}~p$ with 
     $0 \leq r_{0} \leq p-1$ and suppose $1 \leq i< [a-i] < p-1$. 
     If  $i(p+1)+ p \leq r $ and $r_{0} \not \in \mathcal{I}(a,i)$, then 
     $X_{r-i}^{(i)}/ X_{r-i}^{(i+1)}$ contains $V_{[a-2i]} \otimes D^{i}$ as a 
     $\Gamma$-module.       
\end{lemma} 
\begin{proof}
         Recall that $A(a,i,i,r)= \left( \binom{r-n}{m} \binom{[a-m-n]}{i-m} 
         \right)_{0 \leq m,n \leq i} $ (cf. \eqref{A(a,i,j,r) matrix}). 
         Under the above hypotheses,
         we have  $A(a,i,i,r)$ is invertible, by Corollary~\ref{A(a,i,j,r) invertible}.
         Hence there exist $C_{0}, \ldots , C_{i} \in \f $ such that
         $A(a,i,i,r) (C_{0}, \ldots,C_{i-1}, C_{i})^{t} =(0,\ldots,0,1)^{t} $,
         i.e.,
         \begin{align}\label{choice C_n for i<a-i}
                 \sum_{n=0}^{i} C_{n} \binom{r-n}{m}\binom{[a-m-n]}
                {i-m} \equiv \delta_{i,m} \mod ~ p, ~\forall~  0\leq m \leq i. 
         \end{align}
        Consider the following polynomial
        \begin{align}
                F(X,Y) 
                & :=  \sum _{n=0}^{i} C_{n}  \sum_{k \in \f^{\ast}}
                   k^{i+n-a}  X^{n} (kX+Y)^{r-n}  \nonumber \\
               & \stackrel{\eqref{sum fp}}{=} - \sum_{n=0}^{i} C_{n} 
               \sum_{\substack{0 \leq l \leq r-n \\  l \equiv i \mathrm{~mod}~ (p-1)}}
                         \binom{r-n}{l}X^{r-l} Y^{l} . \label{polynomial F in i < a-i}
        \end{align}
        Note that $F(X,Y) \in X_{r-i}$, by \Cref{Basis of X_r-i}.
        We claim that $0 \neq F(X,Y) \in X_{r-i}^{(i)}/X_{r-i}^{(i+1)}$.
        Since $1 \leq i < p-1$, we have  $0$, $1, \ldots, i-1 \not \equiv i 
        \mathrm{~mod}~(p-1)$, so the 
        coefficients of $X^{r}, \ldots, X^{r-(i-1)}Y^{i-1}$  in $F(X,Y)$ are zero.
        Since $r- [a-i] \equiv i$ mod $(p-1)$ and $r-(r-[a-i]) =[a-i] < p-1$,
        we have $ r-[a-i]+1$, $ \ldots, r-1, r \not \equiv i \mathrm{~mod}~(p-1)$, 
        so the  coefficients 
        of $X^{[a-i]-1}Y^{r-[a-i]+1}, \ldots, XY^{r-1}, Y^{r}$ in $F(X,Y)$ are also zero.
        As $i < [a-i] $, 
        we see that $F(X,Y)$ satisfies  condition (i) of \Cref{divisibility1} for $m=i$. 
        Further, by \Cref{binomial sum}, for $ 0 \leq m \leq i$,  we have
        \begin{align*}
               \sum_{n=0}^{i} C_{n} \sum_{\substack{0 \leq l \leq r-n 
               \\ l \equiv i \mathrm{~mod}~ (p-1)}} \binom{r-n}{l} \binom{l}{m} 
                 &  \equiv   \sum_{n=0}^{i} C_{n} \binom{r-n}{m} \left( \binom{[a-m-n]}
                {[i-m]} + \delta_{[p-1],[i-m]} \right) \mod ~ p \\
                &  \equiv   \sum_{n=0}^{i} C_{n} \binom{r-n}{m}\binom{[a-m-n]}
                {i-m} \mod ~ p \\
                & \stackrel{\eqref{choice C_n for i<a-i}}{\equiv} \delta_{i,m} \mod ~ p,
        \end{align*}     
         where the second last step is obvious if $m<i$, and  if 
         $m =i$, then $\binom{[a-m-n]}{[i-m]} = \binom{[a-i-n]}{p-1}=0
         $ as $1 \leq [a-i-n] = [ [a-i] -n]< p-1$. 
        Thus, by \Cref{divisibility1},  we have $F(X,Y) \in 
        V_{r}^{(i)}$ and $ F(X,Y) \not \in V_{r}^{(i+1)}$, whence
        $0 \neq F(X,Y) \in  X_{r-i}^{(i)}/X_{r-i}^{(i+1)}$, so 
        $X_{r-i}^{(i)}/X_{r-i}^{(i+1)}$ has non-zero  socle.
        Since $i < [a-i]= [r-i]$, the sequence  \eqref{exact sequence Vr} 
         with $m=i$ is non-split.
        Therefore  $V_{[a-2i]} \otimes D^{i} 
        \hookrightarrow X_{r-i}^{(i)}/X_{r-i}^{(i+1)}$.
\end{proof}
We next determine the quotient $X_{r-i}^{([a-i])}/X_{r-i}^{([a-i]+1)}$,
when $r \equiv [r-i]+i$ mod $p$. This result will be used in 
\Cref{exceptional case 1} and \Cref{medium i full}.
\begin{lemma}\label{large a-i}
       Let  $p \geq 3$, $ r \equiv a ~\mathrm{mod}~ (p-1)$ with $1 \leq a \leq p-1$ and 
       $1 \leq i< [a-i] <p-1$. If $2p-4 \leq  r \equiv [a-i]+i~\mathrm{mod}~ p$,
        then 
      \[
            \frac{X_{r-i}^{([a-i])}}{X_{r-i}^{([a-i]+1)}} = \frac{V_{r}^{([a-i])}}{V_{r}^{([a-i]+1)}}.
      \]
\end{lemma}
\begin{proof}
           Note that 
           the smallest positive integer satisfying $s \equiv [a-i]+i$ mod $p$
           and $s \equiv a$ mod $(p-1)$ is $[a-i]+i  \leq 2p-5$. 
           Since  $r \geq 2p-4$, we have  $r \geq (p-1)p+[a-i]+i \geq ([a-i]+1)p+[a-i]
           = [a-i](p+1)+p$. 
            Recall that by \eqref{F i,r definition}, we have 
            \[
               F_{i,r}(X,Y) = \sum_{\lambda \in \f} \lambda^{[2i-a]} X^{i} (\lambda X + Y)^{r-i}
                 \in X_{r-i}.
            \]  
      Consider the following polynomial
        \begin{align*}
               F(X,Y)  & : = (r-i+1) F_{i,r}(X,Y) - (r-2i+1) \sum_{k \in \fstar}
               k^{2i-a-1} X^{i-1}(kX+Y)^{r-(i-1)}  \\
               & \stackrel{\eqref{sum fp}}{=} -  
               \sum_{\substack{0 \leq l \leq r-i \\ l \equiv i ~\mathrm{mod}~(p-1)}} 
                (r-i+1) \binom{r-i}{l} X^{r-l}Y^{l}   \\  
                & \qquad  \qquad \qquad \qquad + 
                \sum_{\substack{0 \leq l \leq r-i+1 \\ l \equiv i ~\mathrm{mod}~(p-1)}} 
                (r-2i+1)\binom{r-i+1}{l}  X^{r-l}Y^{l}.
      \end{align*}
       Observe that $F(X,Y) \in X_{r-i}$, by \Cref{Basis of X_r-i}.
       We claim that $ F(X,Y) \in X_{r-i}^{([a-i])}$
      and generates $V_{r}^{([a-i])}/V_{r}^{([a-i]+1)}$. 
      Clearly the coefficient of $X^{r-i}Y^{i}$ in  $F(X,Y)$ equals 
        \begin{align}\label{coeff X^r-i Y^i}
            (r-2i+1) \binom{r-i+1}{i} - (r-i+1) \binom{r-i}{i} =0.
        \end{align}
       Also the coefficient of  $X^{[a-i]}Y^{r-[a-i]}$ in $F(X,Y)$ equals
       \[
           (r-2i+1) \binom{r-i+1}{r-[a-i]}  - (r-i+1) \binom{r-i}{r-[a-i]} 
           = \binom{r-i+1}{r-[a-i]} (r-2i+1 - [a-i]+i-1) = 0 ,
       \]
       as $r \equiv [a-i]+i $ mod $p$.
     Since  $i$ (resp. $r-[a-i]$) is the only number less than $p$
      (resp. greater than $r-p$)
     congruent to $i$ mod $(p-1)$, we see that
     the coefficients of $X^{r},X^{r-1}Y, \ldots, X^{r-(p-1)}Y^{p-1}$ and 
     $Y^{r}, XY^{r-1}, \ldots , X^{p-1}Y^{r-(p-1)}$ in $F(X,Y)$ are zero. 
     Using \Cref{binomial sum}  and Lucas' theorem, 
     one checks that, for $0 \leq m < [a-i]$,
     we have
      \begin{align*}
            & \sum_{\substack{0 \leq l \leq r-i \\ l \equiv i ~\mathrm{mod}~ (p-1)}}
              (r-i+1) \binom{r-i}{l} \binom{l}{m} -
              \sum_{\substack{0 \leq l \leq r-i+1 \\ l \equiv i ~\mathrm{mod}~ (p-1)}}
               (r-2i+1)\binom{r-i+1}{l}  \binom{l}{m}  \\ 
           & ~~\equiv (r-i+1) \binom{r-i}{m} \binom{[a-i-m]}{[i-m]}
             - (r-2i+1) \binom{r-i+1}{m} \binom{[a-i+1-m]}{[i-m]} \\
              & ~~\qquad \qquad \qquad + \delta_{p-1,[i-m]} \left[ (r-i+1) \binom{r-i}{m}
                 -  (r-2i+1) \binom{r-i+1}{m} \right]~\mathrm{mod}~p \\
             & ~~ =    \binom{r-i+1}{m} \left[ (r-i-m+1) \binom{[a-i]-m}{[i-m]}
                 - (r-2i+1)  \binom{[a-i]-m+1}{[i-m]} \right]    \\
             & ~~ \equiv    \binom{r-i+1}{m} \left[ ([a-i]-m+1) \binom{[a-i]-m}{[i-m]}
                 - ([a-i]-i+1)  \binom{[a-i]-m+1}{[i-m]} \right] \\
                 & \qquad \qquad \qquad \qquad ~\mathrm{mod}~p,          
      \end{align*}
      where in the penultimate step we have used $0 \leq m < [a-i] < p-1$
      and \eqref{coeff X^r-i Y^i} and in the last step we have used
       $r \equiv [a-i]+i$ mod $p$.
      If $0 \leq m < i$, then $[i-m] = i-m$ so the sum vanishes.
      Since $i \neq a$, we see that $1 \leq [a-i] < p-1$. 
      If $i \leq m <[a-i]$, then $[i-m]= p-1+i-m > [a-i]+1 -m$, whence
      by Lucas' theorem the sum vanishes.
       Hence by \Cref{divisibility1}, we get $F(X,Y) \in X_{r-i}^{([a-i])} $. 
       Since the sequence \eqref{exact sequence Vr} for $m=[a-i]$ is non-split,
       to show $F(X,Y)$ generates $V_{r}^{([a-i])}/V_{r}^{([a-i]+1)}$
       it is enough to show the image of $F(X,Y)$ is non-zero
       under the rightmost map of the sequence \eqref{exact sequence Vr},
       for $m=[a-i]$.
       Noting that  $r \equiv [a-i]+i ~\mathrm{mod}~ p$ and $i < [a-i] <  p-1$, 
       by \Cref{binomial sum} and Lucas' theorem, we have
       \begin{align}\label{large a-i breuil map}
            &\sum_{\substack{0 \leq l \leq r-i \\ l \equiv i ~\mathrm{mod}~ (p-1)}}
             (r-i+1) \binom{r-i}{l} \binom{l}{[a-i]}  -
             \sum_{\substack{0 \leq l \leq r-i+1 \\ l \equiv i ~\mathrm{mod}~ (p-1)}}
             (r-2i+1)\binom{r-i+1}{l}  
            \binom{l}{[a-i]} \nonumber \\
            &\equiv (r-i+1) \binom{r-i}{[a-i]} \binom{p-1}{[i-[a-i]]} -
                 (r-2i+1) \binom{r-i+1}{[a-i]} \binom{1}{[i-[a-i]]} \mod p \nonumber \\
            &\equiv  ([a-i]+1) \left[ \binom{p-1}{p-1+i-[a-i]} - 
                ([a-i]-i+1) \binom{1}{p-1+i-[a-i]} \right] \mod p \nonumber  \\
           & = ([a-i]+1) \binom{p-1}{p-1+i-[a-i]},
      \end{align}
      where in the last step we used that $p-1+i-[a-i] > i \geq 1$.
       Since the coefficients of $X^{[a-i]}Y^{r-[a-i]}$ and $X^{r-[a-i]}Y^{[a-i]}$
       in $F(X,Y)$ are zero, by  \Cref{breuil map quotient} and  
       \eqref{large a-i breuil map},  we have
       $$
            F(X,Y) \equiv -([a-i]+1) \binom{p-1}{p-1+i-[a-i]} \theta^{[a-i]}
              X^{r-[a-i](p+1)-[2i-a]} Y^{[2i-a]} \mod V_{r}^{([a-i]+1)}.
       $$
       Thus, by \Cref{Breuil map} and Lucas' theorem, the image 
       of $F(X,Y)$ in the quotient of  
      $V_{r}^{([a-i])}/V_{r}^{([a-i]+1)}$ is non-zero,
       as $1 \leq i <  [a-i] < p-1 $. Hence  $F(X,Y)$ generates 
       $V_{r}^{([a-i])}/V_{r}^{([a-i]+1)}$.
\end{proof}
We next  prove the converse of \Cref{Large Cong class Quotient non zero}.
\begin{lemma}\label{exceptional case 1}
        Let $p\geq 3$, $r \equiv a ~\mathrm{mod} ~ (p-1)$ with $1 \leq a \leq p-1$, 
        $r \equiv r_{0}~\mathrm{mod} ~ p$ with $1 \leq r_{0} \leq p-1$
        and  $1 \leq i  < [a-i] < p-1$.
       If $i(p+1)+p \leq r$ and  $r_{0} \in
        \mathcal{I}(a,i)$, then 
        $X_{r-i}^{(i)}/X_{r-i}^{(i+1)} = (0)$.
\end{lemma}
\begin{proof}      
          Observe that  $ [a-i]+i$ equals $a$ and $p-1+a$ 
          if $ i < a$ and $i>a$ respectively. 
          Note that $a$ (resp. $a-1$) belongs to $\mathcal{I}(a,i)$ in the case
          $i<a$ (resp. $i>a$). 
         We prove the lemma by considering the
          cases $r \equiv [a-i]+i$ mod $p$ and 
          $r \not \equiv [a-i]+i$ mod $p$.
          
           If $r \equiv  [a-i]+i$ mod $p$, then by
          \Cref{large a-i}, we have $X_{r-i}^{([a-i])}/X_{r-i}^{([a-i]+1)} = 
          V_{r}^{([a-i])}/V_{r}^{([a-i]+1)}$.  Since $i-1< i < [a-i]$
          and $[a-i] \neq 0$, $a$, 
          by \Cref{reduction corollary 2},
          we have $X_{r-(i-1)}^{([a-i])}/X_{r-(i-1)}^{([a-i]+1)} = (0)$.
          Thus, by \eqref{Y i,j} (with $j=[a-i]$), we have
          \[
              \frac{X_{r-i}^{([a-i])}+X_{r-(i-1)}}{X_{r-i}^{([a-i]+1)}+X_{r-(i-1)}}
              \cong Y_{i,[a-i]} = \frac{X_{r-i}^{([a-i])}}{X_{r-i}^{([a-i]+1)}} = 
                \frac{V_{r}^{([a-i])}}{V_{r}^{([a-i]+1)}}
          \]
          has dimension $p+1$.
           Since $X_{r-(i-1)} \subseteq  X_{r-i}^{([a-i]+1)}+X_{r-(i-1)} 
           \subseteq X_{r-i}^{([a-i])}+X_{r-(i-1)} \subseteq X_{r-i}$
           and $\dim X_{r-i}/X_{r-(i-1)} \leq p+1$, it follows 
           that  $X_{r-i} =  X_{r-i}^{([a-i])}+X_{r-(i-1)}$. Since $i < [a-i]$,
           we see that $X_{r-i} =  X_{r-i}^{(i+1)}+X_{r-(i-1)}$.
         Since $i \neq 0$, $a$ and $ i-1 < i  < [a-i]$, by 
         \Cref{reduction corollary 2},  we have  
      $X_{r-(i-1)}^{(i)}/X_{r-(i-1)}^{(i+1)}= (0)$, i.e., $X_{r-(i-1)}^{(i)} =
      X_{r-(i-1)}^{(i+1)}$.  Thus
      \[ 
           X_{r-i}^{(i)} = X_{r-i}^{(i)} \cap 
          (X_{r-(i-1)} + X_{r-i}^{(i+1)}) = X_{r-(i-1)}^{(i)} + X_{r-i}^{(i+1)}
           = X_{r-(i-1)}^{(i+1)} + X_{r-i}^{(i+1)} = X_{r-i}^{(i+1)}.
      \]     
           
            Now suppose $r \not \equiv [a-i]+i$ mod $p$. By hypothesis, 
            $r_{0} \in \mathcal{I}(a,i)$. If $i<a$,  then  $a-i+1 \leq 
            r_{0} \leq a-1$.   If $i>a$, then $0 \leq r_{0} \leq a-2$ or 
            $p+a-i \leq r_{0} \leq p-1$.  
             Let
            \begin{align*}
                  s= 
                  \begin{cases}
                         (a-r_{0})p^{3}+r_{0}, &\mathrm{if}~ i<a ~\mathrm{so}~
                          a-i+1 \leq r_{0} \leq a-1, \\
                         (a-r_{0}-1)p^{3}+p+r_{0}, & \mathrm{if} ~ i>a ~\mathrm{and}
                         ~ 0 \leq r_{0} \leq a-2, \\
                         (a-r_{0}+p-1)p^{3}+r_{0},  & \mathrm{if}  ~  i>a ~\mathrm{and}
                         ~p+a-i  \leq  r_{0} \leq p-1.
                \end{cases}
          \end{align*}
          Observe that $s \geq p^{3}$,  $\Sp(s- r_{0}) \leq i-1$ and 
          $\Sp(s-(i-1))=[a-i]+1 \leq p-1$.
          Since $i<[a-i]$, we have   $r_{0} \neq i-1$.
          Then by \Cref{final X_r-i = X_r-j}, we have 
          $X_{s-i} = X_{s-(i-1)}$. Thus $X_{s-i}^{(i)}/X_{s-i}^{(i+1)} 
          = X_{s-(i-1)}^{(i)}/X_{s-(i-1)}^{(i+1)} = (0)$,
          by \Cref{reduction corollary 2}, as 
         $ i-1< i <[a-i]$ and $i \neq 0$, $a$. 
          Since $r \equiv s$ mod $p(p-1)$ by  \Cref{arbitrary quotient periodic},
          we have
          $X_{r-i}^{(i)}/X_{r-i}^{(i+1)}= X_{s-i}^{(i)}/X_{s-i}^{(i+1)} = (0)$.
\end{proof}
\begin{remark}\label{remark i< a-i singular zero}
       The argument in the case $r \not \equiv [a-i]+i$ mod $p$ and 
       $r_{0} \in \mathcal{I}(a,i)$ in the above lemma also works   in the case 
       $i =[a-i]$, as we didn't require $i$ to be  strictly less than $[a-i]$.
\end{remark}
We are now ready to determine the $\Gamma$-modules
$X_{r-i}^{(i)}/X_{r-i}^{(i+1)}$ and $X_{r-i}^{([a-i])}/X_{r-i}^{([a-i]+1)}$ in the 
case $1 \leq i < [a-i] < p-1$ and $i \neq a$, $p-1$. Before we state 
the result observe that, for $r \equiv a $ mod $(p-1)$ with 
$1 \leq a \leq p-1$, $1 \leq i <[a-i] <p-1$ and 
$j \in \lbrace i, [a-i] \rbrace$, we have
\begin{align}\label{dimension singular quotient}
 \dim \left( \frac{X_{r-i}^{(j)}}{X_{r-i}^{(j+1)}} \right) & =
 \dim  X_{r-i}^{(j)}  - \dim X_{r-(i-1)}^{(j)}  + \dim X_{r-(i-1)}^{(j+1)}
 - \dim X_{r-i}^{(j+1)}\nonumber \\
 &= \dim (X_{r-i}^{(j)}+X_{r-(i-1)}) -\dim (X_{r-i}^{(j+1)}+X_{r-(i-1)}),
\end{align}
where the first equality follows from  \Cref{reduction corollary 2}, since $i-1 <i <[a-i]$,
 and the second equality follows from the dimension formula for sum of two vector
subspaces.
  
Recall that $W_{i,r}$ is the image of  $V_{[2i-r]} \otimes D^{a-i} \hookrightarrow
\ind_{B}^{\Gamma}(\chi_{1}^{r-i}\chi_{2}^{i})$ under the $\Gamma$-linear map
$\psi_{i}:\ind_{B}^{\Gamma}(\chi_{1}^{r-i}\chi_{2}^{i}) \rightarrow X_{r-i}/X_{r-(i-1)}$,
as defined in \Cref{surjection1}.   Also $W_{i,r} \subseteq 
(X_{r-i}^{([a-i])}+X_{r-(i-1)})/X_{r-(i-1)}$.
By the second and fourth parts of \eqref{interval I}, for $1 \leq a,i \leq p-1$ with $i< [a-i] < p-1$, we have
\begin{align}\label{interval I for [a-i]}
       \mathcal{I}(a,[a-i]) =
       \begin{cases}
              \lbrace i, i+1, \ldots, a-1, a \rbrace,  & \mathrm{if}~ i<a, \\
              \lbrace 0, 1, \ldots , a-2, a-1 \rbrace \cup
              \lbrace i, i+1, \ldots , p-1 \rbrace, 
              & \mathrm{if}~ i>a.
       \end{cases}
\end{align}
Thus, for $r \equiv r_{0}$ mod $p$ with $0 \leq r_{0} \leq p-1$,
by Lucas' theorem, we have $\binom{r-i}{[a-i]} \equiv 0$ mod $p$ if and only if 
 $r_{0} \in \mathcal{I}(a,[a-i])$ and $r \not \equiv [a-i]+i $ mod $p$.
\begin{proposition}\label{singular quotient i < [a-i]}
        Let $p\geq 3$, $ r \equiv a \mod (p-1)$ with $1 \leq a \leq p-1$, 
        $r \equiv r_{0} ~\mathrm{mod}~p$ with $0 \leq r_{0}\leq p-1$
        and $1 \leq i <[a-i]<p-1$. 
       For $ j \in \lbrace i, [a-i] \rbrace$ and $r \geq j(p+1)+p$, we have
       \begin{align*}
                \frac{X_{r-i}^{(j)}}{X_{r-i}^{(j+1)}} \cong
                \begin{cases}
                 V_{[a-2j]} \otimes D^{j}, & \mathrm{if} ~ r_{0} \not \in \mathcal{I}(a,j), \\[3pt]
                 V_{r}^{(j)}/V_{r}^{(j+1)}, & \mathrm{if}~  j= [a-i] ~
                 \mathrm{and} ~ r \equiv [a-i]+i ~\mathrm{mod}~p, \\[3pt]
                 (0), &\mathrm{otherwise.}
                \end{cases}
       \end{align*}
\end{proposition}
\begin{proof}
     We consider the cases $j=i$ and $j=[a-i]$ separately. 
    
    \textbf{Case}  $\boldsymbol{j=i}$:
If $r_{0} \not \in \mathcal{I}(a,i)$, then by \Cref{Large Cong class Quotient non zero}, 
we have $V_{[a-2i]} \otimes D^{i} \hookrightarrow X_{r-i}^{(i)}/X_{r-i}^{(i+1)}$. 
To prove the lemma in this case,  we need to show that the inclusion is an 
isomorphism. Note that by \eqref{dimension singular quotient} 
and $W_{i,r} \subseteq (X_{r-i}^{([a-i])}+X_{r-(i-1)})/X_{r-(i-1)} 
\subseteq (X_{r-i}^{(i)}+X_{r-(i-1)})/X_{r-(i-1)}$, we have
    \begin{align*}
           \dim \left( \frac{X_{r-i}^{(i)}}{X_{r-i}^{(i+1)}} \right) 
           & \leq \dim X_{r-i} - \dim ( X_{r-i}^{(i+1)}+X_{r-(i-1)} )  \\
           & = \dim \left(   \frac{X_{r-i}}{X_{r-(i-1)}} \right) -  
           \dim  \left(  \frac{X_{r-i}^{(i+1)}+X_{r-(i-1)}}{X_{r-(i-1)}}\right) \\
           & \leq \dim  \left(  \frac{X_{r-i}/X_{r-(i-1)}}{W_{i,r}}\right) 
           \leq p-[a-2i] \leq p.
    \end{align*}
     As the exact sequence \eqref{exact sequence Vr} 
     does not split for $m=i$, we have
      $X_{r-i}^{(i)}/X_{r-i}^{(i+1)} \cong V_{[a-2i]} \otimes D^{i}$.
      If $r_{0} \in \mathcal{I}(a,i)$, then by \Cref{exceptional case 1}, we have 
    $ X_{r-i}^{(i)}/X_{r-i}^{(i+1)} =0$.

     \textbf{Case}  $\boldsymbol{j=[a-i]}$: 
     If $r_{0} \not \in \mathcal{I}(a,j)$, then from above we have 
     $\binom{r-i}{[a-i]} \not \equiv 0$ mod $p$. Thus,
     by \Cref{socle term singular}, we have 
     $V_{[2i-a]} \otimes D^{a-i} 
      \hookrightarrow X_{r-i}^{(j)}/X_{r-i}^{(j+1)}$. As $\mathcal{I}(a,i) 
      \subseteq \mathcal{I}(a, [a-i])=  \mathcal{I}(a,j)$,
      by the case $j=i$, we have $X_{r-i}^{(i)}/X_{r-i}^{(i+1)} \cong V_{[a-2i]} 
      \otimes D^{i}$. Therefore
      \begin{align*}
        p+1  & = [a-2i]+1 + [2i-a]+1 \\
        & \leq \dim \left( \frac{X_{r-i}^{(i)}}{X_{r-i}^{(i+1)}} \right) +
        \dim \left( \frac{X_{r-i}^{(j)}}{X_{r-i}^{(j+1)}} \right)  \\
         & \leq 
         \dim \left( \frac{X_{r-i}^{(i)}+ X_{r-(i-1)}}{X_{r-i}^{(i+1)}+X_{r-(i-1)}} \right) 
         + \dim \left( \frac{X_{r-i}^{(j)}+ X_{r-(i-1)}}{X_{r-i}^{(j+1)}+X_{r-(i-1)}} \right)
        ~~ \mathrm{by}~ \eqref{dimension singular quotient} \\
        & \leq \dim\left( \frac{X_{r-i}^{(i)}+X_{r-(i-1)}}{X_{r-i}^{(j+1)}+X_{r-(i-1)}}\right) 
            \leq \dim\left( \frac{X_{r-i}}{X_{r-(i-1)}}\right) \leq p+1,
      \end{align*}
      where the inequalities on the last line follow from the fact 
      $X_{r-(i-1)} \subseteq 
      X_{r-i}^{(j+1)}+X_{r-(i-1)}  \subseteq 
      X_{r-i}^{(j)}+X_{r-(i-1)} \subseteq 
      X_{r-i}^{(i+1)}+X_{r-(i-1)} \subseteq X_{r-i}^{(i)} + X_{r-(i-1)} \subseteq X_{r-i}$
      and \Cref{induced and successive}. Therefore 
      dim $X_{r-i}^{(j)}/X_{r-i}^{(j+1)} = [2i-a]+1$ and
      $X_{r-i}^{(j)}/X_{r-i}^{(j+1)} \cong V_{[2i-a]} \otimes D^{a-i}
      =V_{[a-2j]} \otimes D^{j}$.
      
     If $r \equiv [a-i]+i$ mod $p$, then by \Cref{large a-i}, we have
     $X_{r-i}^{(j)}/X_{r-i}^{(j+1)} = V_{r}^{(j)}/V_{r}^{(j+1)}$.
      So assume $r_{0} \in \mathcal{I}(a,j)$ and $ r \not \equiv [a-i]+i $
      mod $p$.
      Again from above, we get
        $\binom{r-i}{[a-i]} \equiv 0$ mod $p$. 
       Thus,  by  \Cref{socle term singular}, we have 
       $  V_{[2i-a]} \otimes D^{a-i} 
       \not \hookrightarrow  X_{r-i}^{(j)}/X_{r-i}^{(j+1)}$. 
       Since the exact sequence \eqref{exact sequence Vr} doesn't split for $m=j$
       and $ V_{[2i-a]} \otimes D^{a-i} $
       is the socle of $V_{r}^{(j)}/V_{r}^{(j+1)}$,
       we see that $X_{r-i}^{(j)}/X_{r-i}^{(j+1)}=(0)$. 
\end{proof}
\begin{theorem}\label{Structure of Q(i) if i<[a-i]}
        Let  $p \geq 3$, $r \equiv a \mod (p-1)$ with $1 \leq a \leq p-1$,
        $r \equiv r_{0}~\mathrm{mod}~ p$ with $0 \leq r_{0} \leq p-1$ and let 
        $1 \leq i < [a-i] < p-1$. If 
        $i(p+1)+p \leq r$, then we have the following exact sequence
        \begin{align*}
               0 \rightarrow W \rightarrow Q(i) \rightarrow Q(i-1) \rightarrow 0,
        \end{align*}
        where $W= V_{p-1-[a-2i]} \otimes D^{a-i}$ if $r_{0} \not\in \mathcal{I}(a,i)$ and 
        $W= V_{r}^{(i)}/V_{r}^{(i+1)}$ if  $r_{0} \in \mathcal{I}(a,i)$.
\end{theorem}
\begin{proof}
Since $i<[a-i]$, by \Cref{reduction corollary} (ii), we have
$ X_{r-i}+ V_{r}^{(i)} = X_{r-(i-1)}+V_{r}^{(i)}$.
 Thus $V_{r} / (X_{r-i}+ V_{r}^{(i)})
= V_{r}/(X_{r-(i-1)}+V_{r}^{(i)})=Q(i-1)$. Taking
$j=i$ in \eqref{commutative diagram} and noting that
rightmost bottom entry is $Q(i-1)$ we get
$$
   0 \rightarrow W \rightarrow Q(i) \rightarrow Q(i-1) \rightarrow 0,
$$
where $W$ is the quotient of $V_{r}^{(i)}/V_{r}^{(i+1)}$ by 
$X_{r-i}^{(i)}/X_{r-i}^{(i+1)}$.
Now it follows from \Cref{singular quotient i < [a-i]} and the exact sequence \eqref{exact sequence Vr} that
$W= V_{p-1-[a-2i]} \otimes D^{a-i}$ if $r_{0} \not\in \mathcal{I}(a,i)$ and 
$W= V_{r}^{(i)}/V_{r}^{(i+1)}$ if  $r_{0} \in \mathcal{I}(a,i)$. This finishes 
the proof.
\end{proof}
%
     If $1 \leq i < a $ in the theorem above, then $i-1$ also satisfies the hypotheses if it is positive.
      Thus repeated application of the above theorem gives the structure of $Q(i)$ in terms of $Q(0)$.   
       But if $i>a$, then $i-1$ satisfies the hypotheses provided $i-1 > a$. Thus
       in the case $i>a$,  the theorem above determines the structure of $Q(i)$
       modulo the structure of $Q(a)$. The structure of $Q(a)$ will be determined in 
       $\S \ref{section i = a}$. This will give the structure of $Q(i)$ in all cases when $i < [a-i]$ and $i \neq a$, $p-1$. 
\subsubsection{ The case \texorpdfstring{$ \boldsymbol{i = [a-i].}$}{}}
\label{section i = a-i}
In this subsection, we determine the structure of $Q(i)$ in the case 
$i = [a-i]$ and $i \neq a $, $p-1$.  Assume $r \geq i(p+1)+p$.
Taking $j=i$ in \eqref{commutative diagram} and using
\Cref{reduction corollary} (ii), we see that the rightmost bottom entry equals
$Q(i-1)$. Thus to determine the structure of $Q(i)$ in terms of $Q(i-1)$
it is enough to determine $X_{r-i}^{(i)}/X_{r-i}^{(i+1)}$. 
By the exact sequence \eqref{exact sequence Vr}, we have $V_{r}^{(i)}/V_{r}^{(i+1)} 
\cong V_{0} \otimes D^{i} \oplus V_{p-1} \otimes D^{i}$. 
By \Cref{socle term singular}, we have $V_{p-1} \otimes D^{i}
\hookrightarrow X_{r-i}^{(i)}/X_{r-i}^{(i+1)}$ if and only if 
$\binom{r-i}{i} \not \equiv 0$ mod $p$. So to describe $X_{r-i}^{(i)}/X_{r-i}^{(i+1)}$
completely we are reduced to determining necessary and sufficient conditions
under which $V_{0} \otimes D^{i} \hookrightarrow  X_{r-i}^{(i)}/X_{r-i}^{(i+1)}$. 
The next two lemmas deal with this question.

Recall that, by the second and fourth parts of \eqref{interval J}, for $1 \leq a,i \leq p-1$ 
with $1 \leq [a-i] \leq i < p-1$, we have
\begin{align}\label{interval J for i > a-i}
        \mathcal{J}(a,i) =
        \begin{cases}
               \lbrace  a-i-1, a-i, \ldots, a-1 \rbrace, & \mathrm{if}~i <a, \\
               \lbrace  0,1, \ldots , a-2 \rbrace \cup 
               \lbrace   p-2+a-i, p-1+a-i, \ldots, p-1 \rbrace,
               & \mathrm{if}~i > a.
        \end{cases}
\end{align}
\begin{lemma}\label{a=2i 1-dim is JH factor}
Let $p \geq 3$, $r \equiv a ~\mathrm{mod}~(p-1)$ with 
$1 \leq a \leq p-1$, $r \equiv r_{0}~\mathrm{mod}~p$ with 
$0 \leq r_{0} \leq p-1$ and let $1 \leq i < p-1$ with $i=[a-i]$. 
If $i(p+1)+p \leq r $ and $r_{0} \not \in \mathcal{J}(a,i) 
 \smallsetminus \lbrace i \rbrace$, 
then $X_{r-i}^{(i)}/X_{r-i}^{(i+1)}$ contains $V_{0} \otimes D^{i}$.
\end{lemma}
\begin{proof}
     Since $r \equiv 2i$ mod $(p-1)$ we see that the exact sequence 
     \eqref{exact sequence Vr} is split for $m=i$.
     So to prove the lemma it is enough to exhibit a polynomial 
     $F(X,Y) \in X_{r-i}^{(i)}/X_{r-i}^{(i+1)}$ whose projection under
     $V_{r}^{(i)}/V_{r}^{(i+1)} \twoheadrightarrow V_{0} \otimes D^{i}$
     is non-zero.    
        Let 
	\begin{align*}
	A' & = \left( \binom{r-n}{m} \binom{[a-m-n]}{i-m} 
	\right)_{0 \leq m,n \leq i-1}, \\  
	\mathbf{v} & =\left( \binom{i}{0}, \binom{i}{1}, \ldots , \binom{i}{i-1} \right), \\
	\mathbf{w}&= \left( \binom{r}{r-i}, \binom{r-1}{r-i}, \ldots , \binom{r-(i-1)}
	{r-i} \right).
	\end{align*}
	Note that $[a-i] =i$ and
	 $r_{0} \not \in \mathcal{J}(a,i) \smallsetminus \lbrace i \rbrace$.
	 If $r \not\equiv 2i \mod p$,  by \Cref{block matrix invertible},
        the matrix
	\[
	       A =
	       \left(
	      \begin{array}{c|c}
	            A' & \mathbf{v}^{t} \\
	            \hline
	            \mathbf{w} & 0
	     \end{array}
	     \right) ~~\mathrm{is ~ invertible},
	\]
       so we may choose $C_{0}, \ldots , C_{i} \in \f $ such that $A (C_{0}, \ldots, C_{i})^{t} =
	(0,\ldots, 0, 1)^{t}$, i.e.,
      \begin{align}
          \sum\limits_{n=0}^{i-1} C_{n} \binom{r-n}{m} 
           \binom{[a-m-n]}{i-m} + C_{i} \binom{i}{m} &=
           0, ~ \forall ~ 0 \leq m \leq i-1, \label{1-dim choice C i eq 1}  \\
          \sum\limits_{n=0}^{i-1} C_{n} \binom{r-n}{r-i}  &=1 \label{1-dim choice C i eq 2}.   
      \end{align}	
      If $r \equiv 2i \mod p$, we take $C_0 = C_1 = \cdots = C_{i-2} = 0$, $C_{i-1} = (i+1)^{-1}$ and $C_i = - 1$.
      Then  \eqref{1-dim choice C i eq 1} and \eqref{1-dim choice C i eq 2}  still hold.   
      Indeed, if $0 \leq m \leq i-1$, by Lucas' theorem, we have
     \begin{align*}
              C_{i-1} \binom{r-(i-1)}{m}   \binom{[a-m-(i-1)]}{i-m}
             = (i+1)^{-1} \binom{i+1}{m} \binom{i+1-m}{i-m}
              = \binom{i}{m}
      \end{align*}
     and          
          \[
               \sum\limits_{n=0}^{i-1} C_{n} \binom{r-n}{r-i}
               = C_{i-1} \binom{r-i+1}{r-i} = (i+1)^{-1} (r-i+1)
               \equiv 1 \mod p.
          \]
	 Let
	\begin{align*}
	       F(X,Y) & := C_{i} X^{r-i}Y^{i} -\sum\limits_{n=0}^{i-1} C_{n} 
	       \sum_{k \in \fstar} k^{i+n-a} X^{n} (k X +Y)^{r-n}  \\
	      & \stackrel{\eqref{sum fp}}{=}  C_{i}X^{r-i}Y^{i} +
	       \sum_{n=0}^{i-1} C_{n} \sum_{\substack{0 \leq l \leq r-n \\ l \equiv i 
	       ~ \mathrm{mod}~ (p-1)}} \binom{r-n}{l} X^{r-l}Y^{l}.
	\end{align*}
	Clearly $F(X,Y) \in X_{r-i}$, by \Cref{Basis of X_r-i}. As above, since
	  $i$ (resp. $r-i$) is smallest (resp. largest) number between $0$ and 
	$r$ congruent to $i$ mod $(p-1)$, we see that the coefficients of 
	$X^{r}, \ldots X^{r-(i-1)}Y^{i-1}$ and $X^{i-1}Y^{r-(i-1)}, \ldots, Y^{r}$ 
	in $F(X,Y)$ are zero. Hence $F(X,Y)$ satisfies  condition (i) 
	in \Cref{divisibility1} for $m=i$. 
	For $0 \leq m \leq i-1$,  by  \Cref{binomial sum}, we have
	\begin{align*}
	      C_{i} \binom{i}{m}+\sum\limits_{n=0}^{i-1} C_{n} 
	      \sum_{\substack{0 \leq l \leq r-n \\ l \equiv i ~ \mathrm{mod}~ (p-1)}}
	      \binom{r-n}{l} \binom{l}{m} & \; \equiv  C_{i} \binom{i}{m} + 
	     \sum_{n=0}^{i-1} C_{n} \binom{r-n}{m} \binom{[a-m-n]}{i-m}  \\
	     & \stackrel{\eqref{1-dim choice C i eq 1}}{\equiv} 0 \mod p.
	\end{align*}
	Hence by \Cref{divisibility1}, we have $F(X,Y) \in X_{r-i}^{(i)}$.  
	Clearly the coefficient of 
	$X^{r-i}Y^{i}$ in $F(X,Y)$ is $C_{i}+\sum_{n=0}^{i-1}
	 C_{n} \binom{r-n}{i} $. Also, the coefficient of $X^{i}Y^{r-i}$ in $F(X,Y)$
	is $\sum_{n=0}^{i-1} C_{n} \binom{r-n}{r-i} =1$, by 
	\eqref{1-dim choice C i eq 2}. 
	By  \Cref{binomial sum}, we have
	\[
	    C_{i} \binom{i}{i}+ \sum\limits_{n=0}^{i-1} C_{n} 
	  \sum_{\substack{0 \leq l \leq r-n \\ l \equiv i ~ \mathrm{mod}~ (p-1)}}
	\binom{r-n}{l} \binom{l}{i}  \equiv C_{i} +
	\sum\limits_{n=0}^{i-1} C_{n} \binom{r-n}{i} \mod p, 
	\]
	which also equals the coefficient of $X^{r-i}Y^{i}$ in $F(X,Y)$.
    Thus, by \Cref{breuil map quotient} (with $m=i$), we have 
    \begin{align*}
        F(X,Y) \equiv (-1)^{i+1}  \theta^{i} X^{r-i(p+1)-p-1}Y^{p-1}  
           ~~\mathrm{mod} ~ V_{r}^{(i+1)},
    \end{align*}
    up to  terms of the form $\theta^{i} X^{r-i(p+1)}$, $\theta^{i} Y^{r-i(p+1)}$.
	It follows from \Cref{Breuil map}
	that the image of $F(X,Y)$ in $X_{r-i}^{(i)}/X_{r-i}^{(i+1)} 
	\hookrightarrow V_{r}^{(i)}/V_{r}^{(i+1)} \twoheadrightarrow 
	V_{0} \otimes D^{i}$ equals $(-1)^{i+1} \neq 0$. 
	This completes the proof of the lemma.
\end{proof}
We next prove the converse of the above lemma. Recall, by \eqref{F i,r definition}, that 
for $i=[r-i]$, we have
$F_{i,r}(X,Y) =  \sum_{k \in \f}^{} k^{[2i-r]} X^{i}(kX+Y)^{r-i} =
\sum_{k \in \f}^{} k^{p-1} X^{i}(kX+Y)^{r-i}$.
\begin{lemma}\label{a=2i 1-dim is not JH}
       Let $p\geq3$, $r \equiv a ~\mathrm{mod}~(p-1)$ with $1 \leq a  \leq p-1$,
       $r \equiv r_{0} ~\mathrm{mod}~ p$ with $0 \leq r_{0}  \leq p-1$ and 
      let $1 \leq i < p-1$ with $i=[a-i]$. If $ r \geq i(p+1)+p$ and
      $r_{0} \in \mathcal{J}(a,i)
      \smallsetminus \lbrace i \rbrace$, then 
      $V_{0} \otimes D^{i} \not \hookrightarrow X_{r-i}^{(i)}/X_{r-i}^{(i+1)}$. 
\end{lemma}
\begin{proof}
        Since $[a-i]=i$,  by \eqref{interval J for i > a-i}
        and the definition of   $ \mathcal{I}(a,i)$,  we have
        $ r_{0}  \in \mathcal{J}(a,i) \smallsetminus \lbrace i,i-1 \rbrace$ 
        implies that  $r_{0} \in \mathcal{I}(a,i)$ and $ r \not \equiv  [a-i]+i $ mod $p$.
        So if $r_{0} \in \mathcal{J}(a,i) \smallsetminus \lbrace i,i-1 \rbrace$,
        then by \Cref{remark i< a-i singular zero}, we have
        $X_{r-i}^{(i)}/X_{r-i}^{(i+1)} =(0)$.  So assume $r \equiv i-1$ mod $p$.
         Since $i-1 < i =[a-i] < p-1$, by \Cref{reduction corollary 2}
         (with $i$ there equal to $i-1$ and $j = i$),
         we have $X_{r-(i-1)}^{(i)}/X_{r-(i-1)}^{(i+1)} =0 $. Thus,
         by \eqref{Y i,j},    we have
         \[
             \frac{X_{r-i}^{(i)}+X_{r-(i-1)}}{X_{r-i}^{(i+1)}+X_{r-(i-1)} }
               \cong Y_{i,i} \cong \frac{X_{r-i}^{(i)}}{X_{r-i}^{(i+1)}}.
         \]
         By \Cref{reduction corollary} (ii),  we have
        $   X_{r-i} =X_{r-(i-1)} +X_{r-i}^{(i)} $.
       Let $\chi = \chi_{1}^{r-i} \chi_{2}^{i}$. Thus, by 
       \Cref{induced and successive}, we have
     \begin{align*}
              \ind_{B}^{\Gamma}(\chi) \overset{\psi_{i}}{\twoheadrightarrow}
            \frac{X_{r-i}}{X_{r-(i-1)}} = \frac{X_{r-(i-1)} +X_{r-i}^{(i)}}{X_{r-(i-1)}}
          \twoheadrightarrow  \frac{X_{r-i}^{(i)}+X_{r-(i-1)}}{X_{r-i}^{(i+1)}+X_{r-(i-1)} }
          \cong  \frac{X_{r-i}^{(i)}}{X_{r-i}^{(i+1)}} \hookrightarrow
            \frac{V_{r}^{(i)}}{V_{r}^{(i+1)}}.
       \end{align*}
       By \Cref{Structure of induced}  and \Cref{induced and star}, we have 
       $ V_{0} \otimes D^{i} \oplus V_{p-1} \otimes D^{i}  \cong
       \ind_{B}^{\Gamma}(\chi) \cong  V_{r}^{(i)}/V_{r}^{(i+1)} $. Thus,
       $V_{0} \otimes D^{i}  \hookrightarrow X_{r-i}^{(i)}/X_{r-i}^{(i+1)}$
       if and only if  the map  
       \[ 
              V_{0} \otimes D^{i}  \hookrightarrow  \ind_{B}^{\Gamma}(\chi) 
              \rightarrow  V_{r}^{(i)}/V_{r}^{(i+1)} 
              \twoheadrightarrow V_{0} \otimes D^{i} 
       \]
       induced by the  composition above is non-zero. 
       By \Cref{Structure of induced} (ii) (for $l=p-1$),  
       $\sum_{\lambda\in \f}^{} \lambda^{p-1} \left[ \begin{psmallmatrix}
         \lambda & 1 \\ 1 & 0 \end{psmallmatrix} , e_{\chi} \right]$ is a basis element of
        $V_{0} \otimes D^{i}  \hookrightarrow  \ind_{B}^{\Gamma}(\chi) $. By 
        \Cref{induced and successive}, we see that  
        $\psi_{i}(\sum_{\lambda\in \f}^{} \lambda^{p-1} \left[ \begin{psmallmatrix}
         \lambda & 1 \\ 1 & 0\end{psmallmatrix} , e_{\chi} \right])=F_{i,r}(X,Y)$. 
        So to prove the lemma it is enough to show that the image of $F_{i,r}(X,Y)$ is zero.
          Since $[a-i] =i <p-1$, we have  $[a-(i-1)] = [i+1] =i+1 >i$. Thus,
          by Corollary~\ref{A(a,i,j,r) invertible}, 
          we see that the matrix
          $A(a,i-1,i,r) =
          \left( \binom{r-n}{m} \binom{[a-m-n]}{i-n} \right)_{0 \leq m,n \leq i-1}$
          is invertible if $r \equiv i-1$ mod $p$. So there exist 
          $C_{0}, \ldots$, $C_{i-1} \in \f$ such that
          \begin{align}\label{choice C_i exceptional case a=2i}
                 \sum_{n=0}^{i-1} C_{n} \binom{r-n}{m} \binom{[a-m-n]}{i-n}
                 = \binom{r-i}{m}, ~~ \forall ~ 0 \leq m \leq i-1.
          \end{align}
          Let $C_{i} = -1$.  Consider the following polynomial
          \begin{align*}
                 F(X,Y) &:= F_{i,r}(X,Y) - \sum_{n=0}^{i-1} C_{n} \sum_{k \in \fstar}^{}
                 k^{i+n-a} X^{n} (k X +Y)^{r-n} \\
                & =  - \sum_{n=0}^{i} C_{n} \sum_{k \in \fstar}^{}
                 k^{i+n-a} X^{n} (k X +Y)^{r-n}  \\
                 & \stackrel{\eqref{sum fp}}{\equiv}  \sum_{n=0}^{i} C_{n}  
                   \sum_{\substack{0 \leq l \leq r-n \\ 
                 l \equiv i ~\mathrm{mod}~ (p-1)}}^{}  \binom{r-n}{l} X^{r-l} Y^{l}
                 \mod p.
          \end{align*}
          Since $i$ (resp. $r-i$) is the smallest (resp. largest) between 
          $0$ and  $r$ congruent to $i$ mod $(p-1)$, we see that 
          $X^{i}, Y^{i} \mid F(X,Y)$.
          Further by \Cref{binomial sum}, for $0 \leq m \leq i-1$, we have 
          \begin{align*}
                  \sum_{n=0}^{i} C_{n}  \sum_{\substack{0 \leq l \leq r-n \\ 
                 l \equiv i ~\mathrm{mod}~ (p-1)}}^{}  \binom{r-n}{l} \binom{l}{m}  
                  \equiv
                   \sum_{n=0}^{i-1} C_{n}  \binom{r-n}{m} \binom{[a-m-n]}{i-m}
                   -\binom{r-i}{m} \overset{\eqref{choice C_i exceptional case a=2i}}{\equiv} 0
                   \mod p.
          \end{align*}
          Hence by \Cref{divisibility1}, we have $F(X,Y) \in V_{r}^{(i)}$. 
          Also note that  by \Cref{Basis of X_r-i}, we have
          $F(X,Y)-F_{i,r} (X,Y) \in X_{r-(i-1)}$, whence the images of 
          $F(X,Y)$, $F_{i,r}(X,Y)$ under
          \begin{align*}
            \frac{X_{r-i}}{X_{r-(i-1)}} = \frac{X_{r-(i-1)} +X_{r-i}^{(i)}}{X_{r-(i-1)}}
            \twoheadrightarrow  \frac{X_{r-i}^{(i)}+X_{r-(i-1)}}{X_{r-i}^{(i+1)}+X_{r-(i-1)} }
           \cong  \frac{X_{r-i}^{(i)}}{X_{r-i}^{(i+1)}} \hookrightarrow
            \frac{V_{r}^{(i)}}{V_{r}^{(i+1)}} \twoheadrightarrow V_{0} \otimes D^{i}
          \end{align*}
          are the same. Since $F(X,Y) \in X_{r-i}^{(i)}$, the image of $F(X,Y)$ 
          under the above composition is the same as the image of 
          $F(X,Y)$ under the last surjection 
          which by \Cref{breuil map quotient} equals zero as
          \begin{align*}
                   \sum_{n=0}^{i} & C_{n}   \sum_{\substack{0 \leq l \leq r-n \\ 
                 l \equiv i ~\mathrm{mod}~ (p-1)}}^{}  \binom{r-n}{l} \binom{l}{i} 
                  - \sum_{n=0}^{i} C_{n}  \binom{r-n}{i}  + 
                  (-1)^{i+1} \sum_{n=0}^{i} C_{n}  \binom{r-n}{r-i}         \\
                 &  \equiv C_{i}  \binom{r-i}{i} -(-1)^{i} \sum_{n=0}^{i} C_{n}  \binom{r-n}{r-i} 
                 \mod p~ \mathrm{ (by~ \Cref{binomial sum}) } \\
                &  \equiv C_{i} \binom{p-1}{i} - (-1)^{i} C_{i}\binom{r-i}{r-i}   
                    \equiv 0\mod p ~~   ( \text{by Lucas' theorem  and}
                     ~ r \equiv i-1 ~ \mathrm{mod}~ p).                  
          \end{align*}
          This proves the lemma.
\end{proof}
We are now ready to describe the quotient $X_{r-i}^{(i)}/X_{r-i}^{(i+1)}$ when $i=[r-i]$.
\begin{proposition}\label{singular i= [a-i]}
      Let $p \geq 3$, $r \equiv a ~\mathrm{mod}~(p-1)$ with $1 \leq a  \leq p-1$,
       $r \equiv r_{0} ~\mathrm{mod}~ p$ with $0 \leq r_{0}  \leq p-1$ and 
      let $1 \leq i < p-1$ with $i=[a-i]$. If $ r \geq i(p+1)+p$, then
      \begin{align*}
             \frac{X_{r-i}^{(i)}}{X_{r-i}^{(i+1)}} \cong
            \begin{cases}
                   V_{r}^{(i)}/V_{r}^{(i+1)}, 
                   &\mathrm{if}~ r_{0} \not \in \mathcal{J}(a,i),  \\
                   V_{0} \otimes D^{i},  & \mathrm{if}~ r_{0}=i, \\
                   V_{p-1} \otimes D^{i},  & \mathrm{if}~r_{0}=i-1, \\
                   (0),  &\mathrm{otherwise}.
             \end{cases}
       \end{align*}       
\end{proposition}
\begin{proof}
        Using $[a-i] = i$ and  \eqref{interval J for i > a-i}, one checks that 
        $\binom{r-i}{i} \not \equiv 0$ mod $p$ if and only if 
        $r_{0} \not \in \mathcal{J}(a,i) \smallsetminus \lbrace i-1 \rbrace$.
        Thus, by \Cref{socle term singular}, we have 
        $V_{p-1} \otimes D^{i} \hookrightarrow X_{r-i}^{(i)}/X_{r-i}^{(i+1)} $ 
        if and only if  
        $r_{0} \not \in \mathcal{J}(a,i) \smallsetminus \lbrace i-1 \rbrace$
        if and only if $r_{0} \not \in \mathcal{J}(a,i)$ or $r_{0}= i-1$.
        By Lemmas \ref{a=2i 1-dim is JH factor} and 
        \ref{a=2i 1-dim is not JH}, we see that  
        $V_{0} \otimes D^{i} \hookrightarrow X_{r-i}^{(i)}/X_{r-i}^{(i+1)} $ 
        if and only if  
        $r_{0} \not \in \mathcal{J}(a,i) \smallsetminus \lbrace i \rbrace$
         if and only if $r_{0} \not \in \mathcal{J}(a,i)$ or $r_{0}=i$.
        Since the exact sequence \eqref{exact sequence Vr} splits for $m=i$,
        we have $X_{r-i}^{(i)}/X_{r-i}^{(i+1)} \cong V_{p-1} \otimes D^{i}  \oplus
        V_{0} \otimes D^{i} $. The proposition follows immediately by putting these facts together.
\end{proof}
We now determine the structure of $Q(i)$ in the case $i=[a-i]$.
\begin{theorem}\label{Structure of Q i=[a-i]}
       Let $p \geq 3$, $r \equiv a ~\mathrm{mod}~(p-1)$ with $1 \leq a  \leq p-1$,
       $r \equiv r_{0} ~\mathrm{mod}~ p$ with $0 \leq r_{0}  \leq p-1$ and 
      let $1 \leq i < p-1$ with $i=[a-i]$. If $ r \geq i(p+1)+p$, then
      \[
          0 \rightarrow W \rightarrow Q(i) \rightarrow Q(i-1) \rightarrow 0,
      \]
      where
      \begin{align*}
            W  \cong
            \begin{cases}
                   (0), 
                   &\mathrm{if}~ r_{0} \not \in \mathcal{J}(a,i),  \\
                   V_{p-1} \otimes D^{i},  & \mathrm{if}~ r_{0}=i, \\
                   V_{0} \otimes D^{i},  & \mathrm{if}~r_{0}=i-1, \\
                   V_{r}^{(i)}/V_{r}^{(i+1)},  &\mathrm{otherwise}.
             \end{cases}
       \end{align*}       
\end{theorem}
\begin{proof}
       Taking $j=i$ in the diagram \eqref{commutative diagram} and using
      \Cref{reduction corollary} (ii), we see that the rightmost bottom entry there equals
     $Q(i-1)$. Thus,  we have
     \[
          0 \rightarrow W \rightarrow Q(i) \rightarrow Q(i-1) \rightarrow 0,
      \]
     where $W$ denotes the cokernel  of the map 
     $X_{r-i}^{(i)}/X_{r-i}^{(i+1)} \hookrightarrow V_{r}^{(i)}/V_{r}^{(i+1)}$.
     By the exact sequence \eqref{exact sequence Vr} (with $m=i$), we see that
     $V_{r}^{(i)}/V_{r}^{(i+1)} \cong V_{0} \otimes D^{i} \oplus V_{p-1} \otimes D^{i}$.
     Now the theorem follows from \Cref{singular i= [a-i]}.
\end{proof}

The theorem can be used to give the complete structure of $Q(i)$, when $i = [a-i]$ and 
$i \neq a$, $p-1$, as follows. It already gives the structure of $Q(i)$ in terms of $W$ and $Q(i-1)$.   
If $i = 1$, we are done. Else $1 \leq i-1 < [a-(i-1)] = [a-i] + 1 \leq p-1$. If the last
inequality is strict, we can use
Theorem~\ref{Structure of Q(i) if i<[a-i]} to determine
$Q(i-1)$, as explained at end of the \S \ref{section i < a-i}. If the last inequality is an equality, i.e., 
if $i-1 = a$, then the structure of $Q(a)$ will be determined in 
\S \ref{section i = a}. In either case, we obtain the structure of $Q(i)$.

\subsubsection{The case \texorpdfstring{$ \boldsymbol{i > [a-i].}$}{•}}
\label{section i > a-i}
In this subsection, we  determine the structure of the quotients $Q(i)$ 
in the case $i>[a-i]$ and $i \neq a$, $p-1$. As in Sections~\ref{section i < a-i} and \ref{section i = a-i}, we determine
$Q(i)$ recursively.  By \Cref{reduction corollary} (ii), we see that
$X_{r-i} = X_{r-(i-1)} + X_{r-i}^{([a-i])} \supset
X_{r-(i-1)} + X_{r-i}^{(i)}$. In most  cases it  turns out 
that the last containment is strict. Thus the natural choice for $j=i$ in
 diagram  \eqref{commutative diagram} which was taken 
in \S \ref{section i < a-i} and \S \ref{section i = a-i} doesn't work here. But it turns out that the choice
$j=[a-i]$  works as we will show that  $X_{r-i}+V_{r}^{([a-i])}
= X_{r-([a-i]-1)}+V_{r}^{([a-i])}$, so that one can realize the rightmost
bottom entry of  diagram \eqref{commutative diagram}
(applied with   $j=[a-i]$)
as $Q([a-i]-1)$. This 
reduces the computation of $Q(i)$ to  \S \ref{section i < a-i} once
one knows the 
leftmost bottom module of diagram  \eqref{commutative diagram} 
for $j=[a-i]$. In order to determine  this module we need to determine  
$X_{r-i}^{([a-i])}/X_{r-i}^{(i+1)}$
as we already know the JH factors of $V_{r}^{([a-i])}/V_{r}^{(i+1)}$,
by \Cref{Breuil map}. To do this we need to determine
the successive quotients of the  following ascending chain of
modules
\begin{align}\label{4.1.3 ascending chain}
    X_{r-i}^{(i+1)} \subseteq X_{r-i}^{(i)} \subseteq \cdots \subseteq
    X_{r-i}^{([a-i]+1)} \subseteq X_{r-i}^{([a-i])}.
\end{align}

We start by determining the last quotient $X_{r-i}^{([a-i])}/X_{r-i}^{([a-i]+1)}$
in the chain \eqref{4.1.3 ascending chain} for $[a-i]<i < p-1$.  
For $1 \leq  a,i \leq p-1$, with $1 \leq [a-i] < i < p-1$,
it follows from the first and third parts of the definition of 
 $\mathcal{J}(a,i)$ (cf. \eqref{interval J})  and the fact that
$[a-[a-i]] =i$, that 
\begin{align}\label{interval J for [a-i]}
      \mathcal{J}(a,[a-i]) = 
      \begin{cases}
              \lbrace i, i+1, \ldots, a-2,a-1 \rbrace, & \mathrm{if}~ 
              [a-i]< i<a, \\
              \lbrace 0,1, \ldots, a-2 \rbrace \cup 
              \lbrace i, i+1, \ldots , p-1 \rbrace, & \mathrm{if}~ 
               a<[a-i]<i.
      \end{cases}
\end{align}
\begin{lemma}\label{medium a-i full}
 Let $p \geq 3$, $r \equiv a \mod (p-1) $ with $1 \leq a \leq p-1$, 
 $r \equiv r_{0} ~\mathrm{mod}~p$ with $0 \leq r_{0}\leq p-1$
 and $1 \leq [a-i] < i < p-1$.  
 If $[a-i](p+1)+p \leq r $ and 
 $r_{0} \not \in \mathcal{J}(a,[a-i])$, 
 then $X_{r-i}^{([a-i])}/X_{r-i}^{([a-i]+1)} = 
 V_{r}^{([a-i])}/V_{r}^{([a-i]+1)}$.
\end{lemma}
\begin{proof}
        Put   $j=[a-i]$. Let 
        \begin{align*}
            A &= \left( \binom{r-n}{m} \binom{[a-m-n]}{j-m}\right)_{0 \leq m,n \leq j-1} , \\
           \mathbf{b} &=\left( \binom{r-i}{0}, \binom{r-i}{1}, \ldots, 
           \binom{r-i}{j-1} \right)^{t}.
       \end{align*} 
       Note  that $A= A(a,j-1,j,r)$, cf. \eqref{A(a,i,j,r) matrix}.
       Note that $[a-(j-1)] = [i+1] =i+1 > j >  j-1$. Thus,
       by Corollary~\ref{A(a,i,j,r) invertible}, 
       we see that $A$ is invertible  if 
       $r_{0} \not \in \mathcal{I}(a,[a-i]-1)$
       (if $[a-i]=1$, then $A= (a)$ which is trivially  invertible). 
       From \eqref{interval J for [a-i]}
       we see that  $r_{0} \not \in \mathcal{J}(a,[a-i])$ and 
       $r \not \equiv [a-i]+i$ mod $p$ implies that
        $r_{0} \not \in \mathcal{I}(a,[a-i]-1)$.
       So, $A \mathbf{x} = \mathbf{b}$ has a solution if 
       $r_{0} \not \in \mathcal{J}(a,[a-i])$ and $r \not \equiv [a-i]+i$ mod $p$.
          We claim that  $A\mathbf{x} = \mathbf{b}$ has a
          solution even  if $r \equiv [a-i]+i =j+i \mod p$. 
         Indeed, by Lucas' theorem, we have 
         $\mathbf{b}= \big( \binom{j}{0}, \ldots, \binom{j}{m}, \ldots,
           \binom{j}{j-1} \big)^{t}$ and since $[a-m-(j-1)]= [i+1-m]$,
          we also have  the last column of $A$ is equal to
           $\big( \binom{i+1}{0} \binom{i+1}{j}, \ldots, \binom{i+1}{m} \binom{i+1-m}{j-m},
           \ldots,\binom{i+1}{j-1}\binom{i+1-(j-1)}{j-(j-1)} \big)^{t}$.
           As 
           $\binom{i+1}{m} \binom{i+1-m}{j-m} \binom{i+1}{j}^{-1} = \binom{j}{m}$,
           we see that
           $A(0, \ldots, 0,  \binom{i+1}{j}^{-1}) ^{t}= \mathbf{b}$.
         Therefore the linear system $A\mathbf{x} = \mathbf{b}$ has a solution
         if $r_{0} \not \in \mathcal{J}(a,[a-i])$.
         Let $C_{0}, \ldots, C_{[a-i]-1}$ be a solution to $A\mathbf{x} = \mathbf{b}$,
         i.e., 
         \begin{align}\label{Choice C_i for [a-i]<i}
                 \sum_{n=0}^{[a-i]-1} C_{n} \binom{r-n}{m} \binom{[a-m-n]}{[a-i]-m}
                 = \binom{r-i}{m}, ~ \forall ~ 0 \leq m \leq [a-i]-1.
         \end{align}
          Let
           \begin{align*}
                  F(X,Y) & = 
                 X^{i}Y^{r-i}
                  + \sum_{n=0}^{[a-i]-1} C_{n} \sum_{k \in \fstar} k^{n-i} 
                   X^{n} (kX+  Y)^{r-n}  \\
                  & \stackrel{\eqref{sum fp}}{\equiv}  X^{i}Y^{r-i}
                  - \sum_{n=0}^{[a-i]-1} C_{n} \sum_{\substack{ 0 \leq l 
                  \leq r-n \\ l \equiv [a-i]  \mod (p-1)}} 
                  \binom{r-n}{l}  X^{r-l} Y^{l} \mod p.
           \end{align*}
           We show below that $F(X,Y) \in X_{r-i}^{([a-i])}$ and
           $F(X,Y)$ generates $V_{r}^{([a-i])}/V_{r}^{([a-i]+1)}$.
           Clearly $F(X,Y) \in X_{r-i}$, by \Cref{Basis of X_r-i}.
         Since $[a-i]$  (resp. $r-i$) is the smallest (resp. largest)  
         number between $0$ and $r$ congruent to $a-i$  mod $(p-1)$,  
           we see that  the coefficients of
           $X^{r}, X^{r-1}Y, \ldots$, $X^{r-[a-i]+1}Y^{[a-i]-1}$ (resp.
           $ X^{i-1}Y^{r-i+1}, X^{i-2}Y^{r-i+2}, \ldots,Y^{r}$) in $F(X,Y)$ are zero.
           As $[a-i] <i$, we see that
           $F(X,Y)$ satisfies  condition (i) of \Cref{divisibility1}
           for $m=[a-i]$. For $0 \leq m \leq [a-i]-1$, by \Cref{binomial sum}, 
          we have
           \begin{align*}
                   \sum_{n=0}^{[a-i]-1} C_{n} \sum_{\substack{ 0 \leq l 
                   \leq r-n \\ l \equiv [a-i] ~ \mathrm{mod} ~(p-1)}} 
                   \binom{r-n}{l}  \binom{l}{m} 
                   & \; \; \equiv \sum_{n=0}^{[a-i]-1} C_{n} \binom{r-n}{m}
                         \binom{[a-m-n]}{[a-i]-m} \mod p \\
                   &\overset{\eqref{Choice C_i for [a-i]<i}}{\equiv}  \binom{r-i}{m}  
                   ~\mathrm{mod}~p.
           \end{align*}
           Thus, $F(X,Y) \in V_{r}^{([a-i])}$, by \Cref{divisibility1}.
          Again by \Cref{binomial sum}, for $0\leq n \leq [a-i]-1$, we have 
          \begin{align*}
              \sum_{\substack{ 0 \leq l \leq r-n \\ l \equiv [a-i] ~ \mathrm{mod} ~(p-1)}} 
                   \binom{r-n}{l}  \binom{l}{[a-i]} 
                   & \equiv  \binom{r-n}{[a-i]} 
                   \binom{[a-n-[a-i]]}{p-1} +  \binom{r-i}{[a-i]} ~\mathrm{mod}~p \\
                   &  \equiv   \binom{r-n}{[a-i]}   ~\mathrm{mod}~p,
          \end{align*}         
          where in the last step we used $\binom{[a-n-[a-i]]}{p-1}  =
          \binom{[i-n]}{p-1} = \binom{i-n}{p-1} = 0$,   as $ n< [a-i] < i< p-1$. Hence
          \begin{align*}
              \binom{r-i}{[a-i]}-\sum_{n=0}^{[a-i]-1}  C_{n}
              \sum_{\substack{ 0 \leq l \leq r-n \\ l \equiv [a-i] ~ \mathrm{mod} ~(p-1)}} 
                 \binom{r-n}{l}  \binom{l}{[a-i]}  
               \equiv  \binom{r-i}{[a-i]}-  
               \sum_{n=0}^{[a-i]-1}  C_{n} \binom{r-n}{[a-i]} ~ \mathrm{mod} ~p.
          \end{align*}
          Clearly the coefficient  of $X^{r-[a-i]}Y^{[a-i]}$ in $F(X,Y)$ is 
          $-\sum\limits_{n=0}^{[a-i]-1} C_{n} 
          \binom{r-n}{[a-i]} $. Also, from above we have 
          the coefficient of $X^{[a-i]}Y^{r-[a-i]}$ 
          in $F(X,Y)$ is  zero. 
         Putting all these together, it follows from  \Cref{breuil map quotient}
          with $m=[a-i]$,  that
          \begin{align*}
                 F(X,Y) \equiv & \binom{r-i}{[a-i]} \theta^{[a-i]} X^{r-[a-i](p+1)-(p-1)}Y^{p-1} 
                 \mod V_{r}^{([a-i]+1)},
          \end{align*}
up to terms involving $\theta^{[a-i]}X^{r-[a-i](p+1)}$ and 
$\theta^{[a-i]}Y^{r-[a-i](p+1)}$.
Thus, by \Cref{Breuil map}, we have that the image of 
$F(X,Y)$  under $V_{r}^{([a-i] )}/ V_{r}^{([a-i]+1)}
\twoheadrightarrow V_{p-1-[a-2i]} \otimes D^{i}$ is $(-1)^{r} \binom{r-i}{[a-i]} 
Y^{p-1-[2i-a]}$, which is non-zero by Lemma~\ref{interval and binomial} (iii) with $i$ there equal to $[a-i]$, since  $r_{0} \not \in \mathcal{J}(a,[a-i])$
(this is where we discard $r_{0}=i$, as all the  previous statements are valid 
even if $r_{0}=i$).
As the exact sequence \eqref{exact sequence Vr} doesn't split  for $m=[a-i]$, 
we obtain  
$F(X,Y)$ generates  $V_{r}^{([a-i]) }/ V_{r}^{([a-i]+1)}$. 
\end{proof}
We next describe the first and last  quotients in  the chain
\eqref{4.1.3 ascending chain}
when the hypothesis of the lemma above fails 
to hold.
\begin{lemma}\label{medium a-i not full}
      Let $ p\geq 3$, $p \leq r \equiv a \mod (p-1) $ with $1 \leq a \leq p-1$,
     $r \equiv r_{0} ~\mathrm{mod}~p$ with $0 \leq r_{0}\leq p-1$ 
      and suppose $1 \leq [a-i] < i < p-1$. If 
      $r_{0} \in \mathcal{J}(a,[a-i])$, then $X_{r-i}= X_{r-(i-1)}+X_{r-i}^{(i+1)}$.
      Furthermore, $X_{r-i}^{(i)}/X_{r-i}^{(i+1)} = 
      X_{r-[a-i]}^{(i)}/X_{r-[a-i]}^{(i+1)}$ and $X_{r-i}^{([a-i])}/X_{r-i}^{([a-i]+1)} = 
      X_{r-[a-i]}^{([a-i])}/X_{r-[a-i]}^{([a-i]+1)}$.
\end{lemma}
\begin{proof}
       Since $[a-i] < i$, we have $r - i \not\equiv i \mod p-1$. 
       Recall that the polynomial
       \begin{align*}
             F_{i,r}(X,Y) 
            &\stackrel{\eqref{F i,r definition}}{=}   \sum_{\lambda \in \f} \lambda^{[2i-a]} X^{i}(\lambda X+Y)^{r-i}
             \stackrel{\eqref{sum fp}}{\equiv} 
            - \sum_{\substack{0 \leq l \leq r-i \\ l \equiv i ~\mathrm{mod}~(p-1)}}^{}
             \binom{r-i}{l}X^{r-l}Y^{l} \mod p
        \end{align*}     
       generates $X_{r-i}/X_{r-(i-1)}$.       
       We claim that $F_{i,r}(X,Y) \in V_{r}^{(i+1)}$, 
       which implies the first assertion in the lemma.   
       Clearly the coefficients 
       of $X^{r-i}Y^{i}$ and $X^{[a-i]}Y^{r-[a-i]}$ in $F_{i,r}(X,Y)$ are
       $-\binom{r-i}{i}$ and $-\binom{r-i}{r-[a-i]}$ respectively.
       Since $i> [a-i]$, we have $\binom{r-i}{r-[a-i]}=0$.
       By Lemma~\ref{interval and binomial} (iii) with $i$ there equal to $[a-i]$, and  the hypothesis $i>[a-i]$, we see that 
       $\binom{r-i}{i} \equiv 0$ mod $p$, for $r_{0} \in \mathcal{J}(a,[a-i])$.
       As $i$  
       (resp. $r-[a-i]$)  is the only number between $0$ and $p-1$ 
       (resp. $r$ and $r-(p-1)$) congruent to $i \mod (p-1)$, we see that
       $X^{p},Y^{p} \mid F_{i,r}(X,Y)$. So  $F_{i,r}(X,Y)$ 
       satisfies condition (i) of \Cref{divisibility1} for $m=i+1$. 
       Further, for $0 \leq m \leq i$, by \Cref{binomial sum},  we have
       \begin{align*}
            \sum_{\substack{0 \leq l \leq r-i \\ l \equiv  i ~ \mathrm{mod}~ (p-1) }}
            \binom{r-i}{l} \binom{l}{m} \equiv \binom{r-i}{m} \left( \binom{[a-i-m]}{[i-m]}
            + \delta_{[i-m],p-1} \right) \mod p.
       \end{align*}
       If $0 \leq m < [a-i]$, then $[a-i-m] = [a-i]-m < i-m =[i-m]$ so
        $\binom{[a-i-m]}{[i-m]} =0$. 
       If  $[a-i] \leq m \leq i$  as we just saw
       $\binom{r-i}{m} \equiv 0 ~\mathrm{mod}~ p$ for
       $r_{0} \in  \mathcal{J}(a,[a-i])$. Thus, the sum above vanishes for all 
       $0 \leq m \leq i$. Hence
       by \Cref{divisibility1}, we have $F_{i,r}(X,Y) \in V_{r}^{(i+1)}$.
       This completes the proof of the first statement of the lemma.
       
       Since $[a-i]< i$, we have
       \[
           X_{r-(i-1)}+X_{r-i}^{(i+1)} \subseteq X_{r-(i-1)}+X_{r-i}^{(i)}  
           \subseteq X_{r-(i-1)}+X_{r-i}^{([a-i]+1)} \subseteq 
           X_{r-(i-1)}+X_{r-i}^{([a-i])} \subseteq X_{r-i}.
       \]
       As the extreme terms in the above chain  are equal,
        all the terms are equal.  Thus, for $j \in \lbrace i,[a-i] \rbrace$,
      it follows from \eqref{Y i,j} that
       \[
          (0) = \frac{X_{r-i}^{(j)}+X_{r-(i-1)}}{X_{r-i}^{(j+1)}+X_{r-(i-1)}}
          \cong Y_{i,j} = 
          \frac{X_{r-i}^{(j)}/X_{r-i}^{(j+1)}}{X_{r-(i-1)}^{(j)}/X_{r-(i-1)}^{(j+1)}}.
       \]
       Thus $X_{r-i}^{(j)}/X_{r-i}^{(j+1)} = X_{r-(i-1)}^{(j)}/X_{r-(i-1)}^{(j+1)} $,
       for $j \in \lbrace i,[a-i] \rbrace$. Since $[a-i]\leq  i -1 <i$ and $[a-[a-i]]=i$,
       it follows from the second and first parts of \Cref{reduction} (with $i$ there equal to $i-1$), that 
       $X_{r-(i-1)}^{(j)}/X_{r-(i-1)}^{(j+1)}  = X_{r-[a-i]}^{(j)}/X_{r-[a-i]}^{(j+1)} $,
       for $j \in \lbrace i,[a-i] \rbrace$. This finishes the proof of the lemma.
\end{proof}
Since $[a-i]<i=[a-[a-i]]$, the lemma above  in conjunction 
with \Cref{singular quotient i < [a-i]} (applied with $i$ there equal to $[a-i]$)
can be used to describe the first quotient
$X_{r-i}^{(i)}/X_{r-i}^{(i+1)}$
and the last quotient $X_{r-i}^{([a-i])}/X_{r-i}^{([a-i]+1)}$ in the chain 
\eqref{4.1.3 ascending chain}  if
 $r_{0}  \in \mathcal{J}(a,[a-i])$.

Note that $\mathcal{J}(a,[a-i]) \subset \mathcal{J}(a,i)$
if $[a-i]<i$.
 In the next  lemma
we  determine the first quotient $X_{r-i}^{(i)}/X_{r-i}^{(i+1)}$ in the chain
\eqref{4.1.3 ascending chain} when
$r_{0} \not \in \mathcal{J}(a,i)$.

\begin{lemma}\label{medium i full}
 Let $p \geq 3$, $ r \equiv a~\mathrm{mod}~ (p-1) $ with $1 \leq a \leq p-1$,
 $r \equiv r_{0} ~\mathrm{mod}~p$ with $0 \leq r_{0}\leq p-1$ and 
suppose $1 \leq [a-i]<i < p-1$. 
 If $i(p+1)+p \leq r$ and $r_{0} \not \in \mathcal{J}(a,i)$, 
 then $X_{r-i}^{(i)}/X_{r-i}^{(i+1)} = 
 V_{r}^{(i)}/V_{r}^{(i+1)}$.
\end{lemma}
\begin{proof}
         Note that the congruence class of $[a-i]+i$ mod $p$
         doesn't belong to $\mathcal{J}(a,i)$, by 
         \eqref{interval J for i > a-i}.
   		If $r \equiv [a-i]+i \mod p$, then by \Cref{large a-i}
   		(applied with $i$ there equal to $[a-i]$), we have 
   		$X_{r-[a-i]}^{(i)}/ X_{r-[a-i]}^{(i+1)} = V_{r}^{(i)}/ V_{r}^{(i+1)}$. 
   		Since  $X_{r-[a-i]}^{(i)}/ X_{r-[a-i]}^{(i+1)}  \subseteq 
   		X_{r-i}^{(i)}/ X_{r-i}^{(i+1)} \subseteq V_{r}^{(i)}/ V_{r}^{(i+1)}$, 
   		we get 
   		$X_{r-i}^{(i)}/ X_{r}^{(i+1)} = V_{r}^{(i)}/ V_{r}^{(i+1)}$. 
   		So assume that  $r_{0} \not \in 
   		\mathcal{J}(a,i)$ but $r \not \equiv [a-i]+i $ mod $p$.
   		We claim that there exist constants $B$, $C_{0}, \ldots , C_{i}$ 
   		 such that
   		\begin{align*}
   		      	F(X,Y) & :=B X^{r-i}Y^{i} - \sum_{n=0}^{i} C_{n} 
   		      	\sum_{k \in \fstar} k^{i+n-a} X^{n} (k X+Y)^{r-n}  \\
   		      	& \stackrel{\eqref{sum fp}}{\equiv}
   		      	  B  X^{r-i}Y^{i} + \sum_{n=0}^{i} C_{n} 
   		      	 \sum_{\substack{0 \leq l \leq r-n \\ l \equiv i ~ \mathrm{mod} ~(p-1)}}
   		      	  \binom{r-n}{l} X^{r-l}Y^{l} 
   		      	  ~\mathrm{mod}~p
   		\end{align*}
   		generates $V_{r}^{(i)}/V_{r}^{(i+1)}$.  Clearly $F(X,Y) \in X_{r-i}$,
   		by \Cref{Basis of X_r-i}, for any choice of $B$, $C_{n}$.
   		
   		To show such  a choice exists, let 
   		$A'= \left( \binom{r-n}{m} \binom{[a-m-n]}{i-m} \right)_{0 \leq m,n \leq [a-i]-1}$,
   		$\mathbf{v}= \left( \binom{i}{0}, \ldots , \binom{i}{[a-i]-1} \right)$ and  
   		$\mathbf{w} = \left( \binom{r}{r-[a-i]}, \ldots, \binom{r-([a-i]-1)}{r-[a-i]} \right)$. 
   		By \Cref{block matrix invertible}, we see that the matrix
   		 \[
	         A = \left(    \begin{array}{c|c} A' & \mathbf{v}^{t} \\ \hline \mathbf{w} & 0 
	                            \end{array} \right), 
	     \]
	     is  invertible if 
	     $r_{0} \not \in \mathcal{J}(a,i)$ and $r _{0} \not \equiv [a-i]+i$ mod $p$.
	     Choose constants $B$, $C_{0}, \ldots, C_{[a-i]-1}$ such that 
	     $A (C_{0}, \ldots, C_{[a-i]-1}, B)^{t} = -(\binom{r-[a-i]}{0}, 
	     \ldots,  \binom{r-[a-i]}{[a-i]-1},1)^{t}$, i.e.,
	      \begin{align}
	              \sum_{n=0}^{[a-i]-1}  C_{n} \binom{r-n}{m} \binom{[a-m-n]}{i-m} 
	             + B \binom{i}{m}
	              &= - \binom{r-[a-i]}{m}, ~ \forall~ 0 \leq m \leq [a-i]-1, 
	              \label{i> a-i choice C_i 1} \\
	               \sum_{n=0}^{[a-i]-1}  C_{n} \binom{r-n}{r-[a-i]} 
	               & = -1. \label{i> a-i choice C_i 2}
	      \end{align}
	     Set $C_{[a-i]}=1$. Using \eqref{interval J for i > a-i}, we see that
             if $r_{0} \not \in \mathcal{J}(a,i)$, then
             $r- [a-m] \equiv m$, $m+1, \ldots, p+m-i-2$ mod $p$,  for all
             $[a-i] \leq m \leq i-1$.
             By Lucas' theorem, 
	     we have $\binom{r-[a-m]}{m} \not \equiv 0$ mod $p$, for all
	     $[a-i] \leq m \leq i-1$.  Successively choose 
	     $C_{i}, C_{i-1},  \ldots, C_{[a-i]+1}$, so that
	     for every $[a-i] \leq m < i$ we have
	     \begin{align}\label{choice C_i}
	     		(-1)^{i-m} \binom{r-[a-m]}{m} C_{[a-m]}  \equiv 
	     		 -\sum_{n=0}^{[a-i]} & C_{n} \binom{r-n}{m} 
	     		\binom{[a-m-n]}{i-m} + D \binom{i}{m}  \nonumber \\ 
	     		&-\sum_{n=[a-m]+1}^{i} C_{n} \binom{r-n}{m} \binom{[a-m-n]}{i-m}
	     		 ~\mathrm{mod}~p.	     		
	     \end{align}
	     We claim that for the  above choice of $B$, $C_{0}, \ldots , C_{i}$ 
	     we have $F(X,Y)$ generates $V_{r}^{(i)}/V_{r}^{(i+1)}$. Note that 
	     the coefficient  of $X^{r-l}Y^{l}$ in $F(X,Y)$ is zero  if
	     $ l \not \equiv i $ mod $p$. Since $i$ is the smallest 
	     number between $0$ and $r$ congruent to $i$ mod $(p-1)$, 
	     we see that   $Y^{i} \mid F(X,Y)$.
	     The coefficient of $X^{[a-i]}Y^{r-[a-i]}$ in $F(X,Y)$ equals
	     \begin{align*}
	            - \sum_{n=0}^{i} C_{n} \binom{r-n}{r-[a-i]}  =  
	             -\sum_{n=0}^{[a-i]} C_{n} \binom{r-n}{r-[a-i]} 
	             \overset{\eqref{i> a-i choice C_i 2}}{\equiv} 0  
	             ~\mathrm{mod}~ p ~~ ( ~\because ~ C_{[a-i]} =1).
          \end{align*}	      
          As $r-[a-i]$ is the only number between $r$ and $r-(p-1)$
          congruent to $i$ mod $(p-1)$, we see that $X^{p} \mid F(X,Y)$.
	     So $X^{p},Y^{i} \mid F(X,Y)$, whence $F(X,Y)$ satisfies condition 
	     (i) of \Cref{divisibility1} for $m=i$. For $0 \leq m < [a-i]$, by
	      \Cref{binomial sum}, we	 have
	     \begin{align*}
	     	   \sum_{n=0}^{i} C_{n} \sum_{\substack{0 \leq l \leq r-n \\ 
	     	   l \equiv i ~\mathrm{mod}~(p-1)} }
	     	   \binom{r-n}{l}\binom{l}{m} - D \binom{i}{m} 
	     	   & \; \equiv   \sum_{n=0}^{i} C_{n}  \binom{r-n}{m}
	     	   \binom{[a-m-n]}{i-m}- D \binom{i}{m} ~\mathrm{mod}~ p \\
	     	   & \;=  \sum_{n=0}^{[a-i]} C_{n}  \binom{r-n}{m} 
	     	   \binom{[a-m-n]}{i-m} - D \binom{i}{m} \\
	     	   & \overset{\eqref{i> a-i choice C_i 1}}{\equiv} 
	     	   0 ~\mathrm{mod}~ p,
	     \end{align*}
	     where in the second last step we used that
	     $[a-m-n] = [a-n]-m < i-m$, for all $[a-i] <  n \leq i$,
	     so $\binom{[a-m-n]}{i-m} =0$.
	     For $[a-i] \leq m < i$, by \Cref{binomial sum}, we have
	     \begin{align*}
	           \sum_{n=0}^{i} C_{n} & \sum_{\substack{0\leq l \leq r-n \\
	            l \equiv i ~\mathrm{mod}~(p-1)}}
	           \binom{r-n}{l} \binom{l}{m} - D \binom{i}{m}  \\ 
	           & \equiv \sum_{n=0}^{[a-i]} C_{n}  \binom{r-n}{m} 
	           \binom{[a-m-n]}{i-m} 
	           + \sum_{n=[a-m]}^{i} C_{n} \binom{r-n}{m} \binom{[a-m-n]}{i-m} 
	           -  D \binom{i}{m} ~\mathrm{mod}~ p \\
	           & \overset{\eqref{choice C_i}}{\equiv} 0 ~\mathrm{mod}~ p,          
	     \end{align*}
	     where in the second last step we used that
	     $[a-m-n] = [a-n]-m < i-m$, so   
	     $\binom{[a-m-n]}{i-m} =0$, for all $[a-i] <  n < [a-m]$
	    and 
	     in the last step we used $\binom{p-1}{i-m} \equiv (-1)^{i-m}$ 
	     mod $p$. Hence by \Cref{divisibility1}, we have 
	     $F(X,Y) \in V_{r}^{(i)}/ V_{r}^{(i+1)}$. As the exact sequence
	     \eqref{exact sequence Vr}  doesn't split for $m=i$, to show
	     $F(X,Y)$ generates $V_{r}^{(i)}/V_{r}^{(i+1)}$ it is enough to show 
	     the image of $F(X,Y)$ under the rightmost map of the exact sequence
	     \eqref{exact sequence Vr} is non-zero. 
	     By \Cref{binomial sum},
	      we have
	     \begin{align*}
	            B \binom{i}{i} - \sum_{n=0}^{i} C_{n} \sum_{\substack{0 \leq l \leq r-n\\ l 
	            \equiv i ~\mathrm{mod}~(p-1)}}
	            \binom{r-n}{l} \binom{l}{i} & \equiv 
	            B \binom{i}{i} -
	            \sum_{n=0}^{i} C_{n} \binom{r-n}{i} + \binom{r-[a-i]}{i} 
	            ~~\mathrm{mod}~ p \\
	            &= \mathrm{coefficient ~ of ~} X^{r-i}Y^{i} ~ \mathrm{in}~ F(X,Y) +
	            \binom{r-[a-i]}{i},
	     \end{align*}
	     where in the first congruence we used $[a-n-i] = p-1$ if and only if
	     $n=[a-i]$.
	     Since $X^{p} \mid F(X,Y)$, by \Cref{breuil map quotient}, we have
	     \begin{align*}
	            F(X,Y) \equiv \binom{r-[a-i]}{i} \theta^{i}  & X^{r-i(p+1)-(p-1)}Y^{p-1}
	                   \mod V_{r}^{(i+1)} ,           
	     \end{align*}
	     up to  terms involving $\theta^{i} X^{r-i(p+1)} $ and 
	     $\theta^{i} X^{r-i(p+1)} $.
	     By Lucas' theorem and \eqref{interval J for i > a-i}, we see that 
	     $\binom{r-[a-i]}{i}  \not \equiv 0$ mod $p$ for
	      $r_{0} \not \in \mathcal{J}(a,i)$. Thus, by \Cref{Breuil map},
	     we get the  image of $F(X,Y)$ is non-zero under the rightmost map 
	     of \eqref{exact sequence Vr} as desired.
\end{proof}
The next lemma determines the first quotient in the chain
\eqref{4.1.3 ascending chain} in the remaining cases
 by reducing to the results in 
\S \ref{section i < a-i}, noting that the inequality $i > [a-i]$ implies
$[a-i] < [a-[a-i]]$. 
\begin{lemma}\label{medium i not full}
       Let  $p \geq 3$, $ r \equiv a \mod (p-1) $ with $1 \leq a \leq p-1$,
       $r \equiv r_{0} ~\mathrm{mod}~p$ with $0 \leq r_{0}\leq p-1$
       and $1 \leq [a-i] < i < p-1$.  If $i(p+1)+p \leq r$,
       $ r_{0} \in \mathcal{J}(a,i)$ and 
       $r_{0} \not \in \mathcal{J}(a,[a-i])$, then 
       $X_{r-i}^{(i)}/ X_{r-i}^{(i+1)}=  X_{r-[a-i]}^{(i)}/X_{r-[a-i]}^{(i+1)}$.
\end{lemma}
\begin{proof}
      Since  $[a-i] \leq i-1 < i$, by the second part of \Cref{reduction}
      (with $i$ there equal to $i-1$ and $j=i$), we have
      $X_{r-(i-1)}^{(i)}/ X_{r-(i-1)}^{(i+1)} = X_{r-[a-i]}^{(i)}/X_{r-[a-i]}^{(i+1)}  $.
      So to prove the lemma it is enough to show 
      $X_{r-i}^{(i)}/ X_{r-i}^{(i+1)}=  X_{r-(i-1)}^{(i)}/X_{r-(i-1)}^{(i+1)}$.
        By  \Cref{medium a-i full}, we have 
        $X_{r-i}^{([a-i])}/X_{r-i}^{([a-i]+1)} = V_{r}^{([a-i])}/V_{r}^{([a-i]+1)}$. 
        Since $[a-i] \leq i-1 < i$, by  the first part of
         \Cref{reduction} (with $i$ there equal to $i-1$ and $j=[a-i]$),  we have 
         $X_{r-(i-1)}^{([a-i])}/X_{r-(i-1)}^{([a-i]+1)} = 
        X_{r-[a-i]}^{([a-i])}/X_{r-[a-i]}^{([a-i]+1)} $. By
        \eqref{interval J for i > a-i} and \eqref{interval J for [a-i]},
        we have
         $$r_{0} \in \lbrace [a-i]-1, [a-i], \ldots, i-1 \rbrace.$$
         By the first and third parts of \eqref{interval I} and the fact $[a-i]<[a-[a-i]] =i$,
         we have
         \begin{align}\label{interval [a-i], for i> [a-i]}
            \mathcal{I}(a,[a-i]) =
            \begin{cases}
                   \lbrace i+1,i+2, \ldots, a \rbrace,
                   & \mathrm{if}~ [a-i] < i<a, \\
                   \lbrace 0,1, \ldots, a-1 \rbrace \cup 
                   \lbrace i+1, i+2, \ldots, p-1 \rbrace,
                   & \mathrm{if}~a < [a-i] < i,
            \end{cases}
         \end{align}
        so $r_{0} \not \in \mathcal{I}(a,[a-i])$.
        Thus, by  \Cref{singular quotient i < [a-i]}
        (with $i,j$ there equal to $[a-i]$),  we have 
        $ 
        X_{r-[a-i]}^{([a-i])}/X_{r-[a-i]}^{([a-i]+1)} 
        \cong V_{[2i-a]} \otimes D^{a-i}$, whence
        $X_{r-(i-1)}^{([a-i])}/X_{r-(i-1)}^{([a-i]+1)}  
        \cong V_{[2i-a]} \otimes D^{a-i}$.         
        By the exact sequence
         \eqref{exact sequence Vr}  (with $m =[a-i]$), we have
         $V_{p-1-[2i-a]} \otimes D^{i} \cong  Y_{i,[a-i]}$.
        Thus, by \eqref{Y i,j},  we have
        \begin{align*}
              V_{p-1-[2i-a]} \otimes D^{i} \cong
               \frac{X_{r-i}^{([a-i])} +X_{r-(i-1)}}{X_{r-i}^{([a-i]+1)}
+X_{r-(i-1)}}.
        \end{align*}

        Now assume $r_{0} \neq [a-i]-1$, i.e.,
         $ r_{0} = [a-i],\ldots, i-1$ mod $p$.
%
%
         By Lemma~\ref{I vs J}, 
         we see that $r_0 \in \mathcal{I}(a,i)$. 
         Since $[a-i] \leq i-1 <i$, by  \Cref{reduction},  we have
         $X_{r-(i-1)}^{(i)}/X_{r-(i-1)}^{(i+1)} = 
        X_{r-[a-i]}^{(i)}/X_{r-[a-i]}^{(i+1)} $.
         By \Cref{singular quotient i < [a-i]}, noting $r \not \equiv [a-i]+i \mod p$,  
         we have 
         $X_{r-[a-i]}^{(i)}/X_{r-[a-i]}^{(i+1)}  = (0)$, whence
         $X_{r-(i-1)}^{(i)}/X_{r-(i-1)}^{(i+1)} =(0)$.
          Suppose $X_{r-i}^{(i)}/ X_{r-i}^{(i+1)}
        \neq   X_{r-(i-1)}^{(i)}/X_{r-(i-1)}^{(i+1)}$, to get a contradiction.
        Then  $X_{r-i}^{(i)}/ X_{r-i}^{(i+1)}\neq (0) $ and 
        $V_{[a-2i]} \otimes D^{i} 
        \hookrightarrow X_{r-i}^{(i)}/ X_{r-i}^{(i+1)}$ by the exact sequence \eqref{exact sequence Vr}
        with $m=i$.
        Therefore
        \begin{align*}
              V_{p-1-[2i-a]} \otimes D^{i} =  V_{[a-2i]} \otimes D^{i} \hookrightarrow 
               Y_{i,i} \cong \frac{X_{r-i}^{(i)} +X_{r-(i-1)}}{X_{r-i}^{(i+1)} +X_{r-(i-1)}}.
        \end{align*}
        Since $[a-i]<i$, we have 
         $\frac{X_{r-i}^{([a-i])} +X_{r-(i-1)}}{X_{r-i}^{([a-i]+1)} +X_{r-(i-1)}}$ and
         $\frac{X_{r-i}^{(i)} +X_{r-(i-1)}}{X_{r-i}^{(i+1)} +X_{r-(i-1)}}$ are distinct 
         subquotients of $X_{r-i}/X_{r-(i-1)}$. Thus, we obtain that
           $V_{p-1- [2i-a]} \otimes D^{i}$ 
         occurs twice as a JH factor of  $X_{r-i}/X_{r-(i-1)}$
         and so for $ \ind_{B}^{\Gamma}(\chi_{1}^{r-i} \chi_{2}^{i})$
         by  \Cref{induced and successive}. Clearly this is not possible by 
         \Cref{Common JH factor} (i). This proves the lemma in the case
         $r_{0} \neq  [a-i]-1$.
        
        Next we deal with the case $r \equiv [a-i]-1 \mod p$.  
        By Lucas' theorem, we have 
        $\binom{r-i}{[a-i]} \not \equiv 0$ mod $p$, as $[a-i] < i < p-1$. 
        By \eqref{interval [a-i], for i> [a-i]},
         it follows
         that $[a-i]-1 \not \in \mathcal{I}(a,[a-i])$. Let 
        $F(X,Y)$ be as defined in \eqref{polynomial F in i < a-i}
        with $i$ there equal to $[a-i]$.
        As in the proof of \Cref{Large Cong class Quotient non zero}, we have 
        $F(X,Y) \in X_{r-[a-i]}^{([a-i])}$. Let
         $G_{i,r}(X,Y)$ be as in \eqref{G i,r definition} and let 
        \begin{align*}
              H(X,Y) & = -  G_{i,r}(X,Y) +  \binom{r-i}{[a-i]} F(X,Y)  \\
              & = \sum_{ \substack{0 \leq l < r-i \\ l \equiv [a-i]
              ~\mathrm{mod}~(p-1)}} \binom{r-i}{l} X^{r-l}Y^{l}
                -  \binom{r-i}{[a-i]} \sum_{n=0}^{[a-i]} C_{n} 
                \sum_{\substack{0 \leq l \leq r-n \\  l \equiv [a-i] 
                \mathrm{~mod}~ (p-1)}} \binom{r-n}{l}X^{r-l} Y^{l} ,
        \end{align*}
        where $C_{0}, \ldots, C_{[a-i]}$ satisfy 
        \eqref{choice C_n for i<a-i} with $i$ there equal to $[a-i]$.
        Clearly $H(X,Y) \in X_{r-i}$.
        We claim that  $H(X,Y) \in X_{r-i}^{(i+1)}+ X_{r-(i-1)}$. 
        Assuming the claim we  
        finish the proof of the lemma. 
        Since $[a-i] \leq i-1$ and $F(X,Y) \in X_{r-[a-i]}$, we have 
        $H(X,Y) + G_{i,r}(X,Y) \in X_{r-(i-1)}$, by 
        \Cref{first row filtration}.  Also recall that
         $G_{i,r}(X,Y)$ generates $W_{i,r}$, the image of 
         $V_{[2i-a]} \otimes D^{a-i} \hookrightarrow 
         \ind_{B}^{\Gamma}(\chi_{1}^{r-i} \chi_{2}^{i}) 
         \overset{\psi_{i}}{\twoheadrightarrow} X_{r-i}/X_{r-(i-1)}$,
         as a $\Gamma$-module. Therefore 
          \begin{align*}
                 V_{[2i-a]} \otimes D^{a-i} 
                \overset{\psi_{i}}{\twoheadrightarrow}  W_{i,r} 
                 \hookrightarrow
                 \frac{X_{r-i}^{(i+1)}+X_{r-i}}{X_{r-(i-1)} }.
          \end{align*}
       Using
       \Cref{Structure of induced} in conjunction with
        \Cref{induced and successive}, and noting that
         $p-1-[2i-a]=[a-2i]$,
       we get
        \[
             \begin{tikzcd}
                 0 \arrow[r, rightarrow] & V_{[2i-a]} \otimes  D^{a-i} \arrow[r, rightarrow]
                  & \ind_{B}^{\Gamma} (\chi_{1}^{r-i}\chi_{2}^{i}) \arrow[r, rightarrow]
                   \arrow[d,  twoheadrightarrow, "\psi_{i}"] & V_{[a-2i]} \otimes D^{i} 
                   \arrow[r, rightarrow] & 0. \\
                  &  & X_{r-i,\,r}/X_{r-(i-1),\,r}  & &
             \end{tikzcd}    
       \]
       Together these facts give a surjection $V_{[a-2i]} \otimes D^{i} 
        \twoheadrightarrow
          \frac{X_{r-i}}{X_{r-i}^{(i+1)} +X_{r-(i-1)}}$.
          Since $[a-i]<i$, we have
           \[
           X_{r-(i-1)}+X_{r-i}^{(i+1)} \subseteq X_{r-(i-1)}+X_{r-i}^{(i)}  
           \subseteq X_{r-(i-1)}+X_{r-i}^{([a-i]+1)} \subseteq 
           X_{r-(i-1)}+X_{r-i}^{([a-i])} \subseteq X_{r-i}.
       \]
        But we already know that $V_{[a-2i]} \otimes D^{i} \hookrightarrow 
          \frac{X_{r-i}^{([a-i])} +X_{r-(i-1)}}{X_{r-i}^{([a-i]+1)} +X_{r-(i-1)}}$.
         So 
          $ X_{r-(i-1)} + X_{r-i}^{([a-i]+1)} 
         = X_{r-(i-1)}+ X_{r-i}^{(i+1)} = X_{r-(i-1)}+ X_{r-i}^{(i)}$. Hence,
         by \eqref{Y i,j}, 
         $X_{r-i}^{(i)}/X_{r-i}^{(i+1)} = X_{r-(i-1)}^{(i)}/X_{r-(i-1)}^{(i+1)} $.
         
        We now prove the claim.
        Since $[a-i]<i < p-1$, by \eqref{G i,r}, we have $G_{i,r}(X,Y) = X^{i} G_{r-i}(X,Y)
        \in X_{r-i}^{([a-i])}$. 
        So $H(X,Y) \in X_{r-i}^{([a-i])}$, as
        $F(X,Y) \in X_{r-i}^{([a-i])}$.
        We first show that $H(X,Y) \in V_{r}^{(i)}$.
        The coefficient of $X^{r-[a-i]}Y^{[a-i]}$ in $H(X,Y)$ equals
        \[
            \binom{r-i}{[a-i]} -  \binom{r-i}{[a-i]}
            \sum_{n=0}^{[a-i]} C_{n} \binom{r-n}{[a-i]}
            \stackrel{\eqref{choice C_n for i<a-i}}{=} 
             \binom{r-i}{[a-i]} -  \binom{r-i}{[a-i]} =
            0,
        \]
        where we used \eqref{choice C_n for i<a-i} with
        $i,m$ there equal to $[a-i]$. 
        Since $[a-i]$ is the only number between $0$ and $p-1$
        congruent to $[a-i]$ mod $(p-1)$, we see that $Y^{p} \mid H(X,Y)$.
        Also $r-i$ is the only number between $r$ and $r-(p-1)$ congruent
        to $[a-i]$ mod $(p-1)$, so $X^{i} \mid H(X,Y)$.
         So $H(X,Y)$ satisfies condition (i) of \Cref{divisibility1} for $m=i$.
         It suffices to check condition (ii) of that lemma for $[a-i] \leq m \leq i$,
         since $H(X,Y) \in V_{r}^{([a-i])}$.
        By \Cref{binomial sum}, for $[a-i] \leq m \leq i$, we have
        \begin{align}\label{i> a-i, G i,r binomial sum}    
               \sum_{ \substack{0 \leq l < r-i \\ l \equiv [a-i]
              ~\mathrm{mod}~(p-1)}} \binom{r-i}{l} \binom{l}{m}
             & =  \sum_{ \substack{0 \leq l \leq  r-i \\ l \equiv [a-i]
              ~\mathrm{mod}~(p-1)}} \binom{r-i}{l} \binom{l}{m}
              - \binom{r-i}{m}  \nonumber \\
             &\equiv \binom{r-i}{m} \left[ \binom{[a-i-m]}{[a-i-m]}+
             \delta_{p-1,[a-i-m]} \right]
             - \binom{r-i}{m}
             \mod p \nonumber \\
             &\equiv \binom{r-i}{m} \delta_{[a-i],m} \mod  p. 
        \end{align}
        Again by \Cref{binomial sum}, for $[a-i] \leq m \leq i-1$,   we have
        \begin{align}\label{i> a-i, F binomial sum}
               \sum_{n=0}^{[a-i]} C_{n} 
                \sum_{\substack{0 \leq l \leq r-n \\  l \equiv [a-i] 
                \mathrm{~mod}~ (p-1)}} \binom{r-n}{l}
                \binom{l}{m}  & \equiv
                \sum_{n=0}^{[a-i]} C_{n}  \binom{r-n}{m} 
                \left[\binom{[a-m-n]}{[a-i-m]} +  \delta_{p-1,[a-i-m]} \right]
                 ~\mathrm{mod}~p \nonumber \\
                & \equiv 
                 \sum_{n=0}^{[a-i]} C_{n}  \binom{r-n}{m}  \delta_{[a-i],m}
                   ~\mathrm{mod}~p \nonumber \\ 
                & \stackrel{\eqref{choice C_n for i<a-i}}{\equiv} \delta_{[a-i],m} 
                ~\mathrm{mod}~p,                                 
         \end{align} 
         where in the second last step we used if $0 \leq n \leq [a-i]-1$, then
         $r-n \equiv [a-i]-1-n$ mod $p$ and $0 \leq [a-i]-1-n < [a-i] \leq m$
         so  $\binom{r-n}{m} \equiv 0$ mod $p$, by Lucas' theorem,
         and if $n=[a-i]$, then we have $\binom{[a-m-n]}{[a-i-m]}
         = \binom{i-m}{p-1+[a-i]-m} = 0$ as $i< p-1 \leq p-1+[a-i]$.
         Hence, the difference  of the expressions
          \eqref{i> a-i, G i,r binomial sum} and 
         $\binom{r-i}{[a-i]}$ times  \eqref{i> a-i, F binomial sum}
         is zero.         
         Thus, by  \Cref{divisibility1}, we have $H(X,Y) \in V_{r}^{(i)}$.
         Next we show that
         \begin{align}\label{H in V i+1}
                 H(X,Y) \equiv 
                  - (-1)^{a-i} C_{[a-i]} \binom{r-i}{[a-i]}  
                  \theta ^{i}  Y^{r-i(p+1)}
                  \mod V_{r}^{(i+1)}.
          \end{align}
           Indeed, since
            $r-n \equiv [a-i]-1-n$ mod $p$ and $0 \leq [a-i]-1-n < i -n$,
           for all $0 \leq n \leq [a-i]-1$, so
           $\binom{r-n}{i-n} \equiv 0$ mod $p$, by Lucas' theorem. Thus,
           by Lucas' theorem, 
            the coefficient of $X^{i}Y^{r-i}$ in $H(X,Y)$ is equal to
            \begin{align}\label{i>a, coeff X i in H}
                    - \binom{r-i}{[a-i]} \sum_{n=0}^{[a-i]} C_{n} \binom{r-n}{r-i}
                    &= - \binom{r-i}{[a-i]} \sum_{n=0}^{[a-i]} C_{n} \binom{r-n}{i-n} 
                    \nonumber\\
                     &\equiv - C_{[a-i]} \binom{r-i}{[a-i]} \binom{r-[a-i]}{i-[a-i]}
                     ~\mathrm{mod}~p  \nonumber \\
                      &\equiv - C_{[a-i]} \binom{r-i}{[a-i]}  
                     \binom{p-1}{i-[a-i]} ~\mathrm{mod}~ p
                     \nonumber \\
                     & \equiv -(-1)^{a}C_{[a-i]} \binom{r-i}{[a-i]}  
                     ~\mathrm{mod}~p ,
            \end{align}      
            which is the clearly the coefficient of $X^{i}Y^{r-i}$   
            on the right hand side of \eqref{H in V i+1}.    
            Also, $Y^{p}$ divides the right hand side  of 
             \eqref{H in V i+1} as $r-i(p+1) \geq p$. Thus
             the difference of the two sides satisfies condition
             (i) of \Cref{divisibility1}  with $m=i+1$.
         By \Cref{binomial sum}, we have
         \begin{align}\label{i> a-i, i derivative condition for F}
               \sum_{n=0}^{[a-i]} C_{n} 
                \sum_{\substack{0 \leq l \leq r-n \\  l \equiv [a-i] 
                \mathrm{~mod}~ (p-1)}} \binom{r-n}{l}
                \binom{l}{i}  & \equiv
                \sum_{n=0}^{[a-i]} C_{n}  \binom{r-n}{i} 
                \binom{[a-i-n]}{[a-2i]} 
                 \mod p  \nonumber \\
                  &\equiv  (-1)^ {a-i}C_{[a-i]}   \mod p, 
        \end{align}       
         where in the  last step we have  used that  if $0\leq n \leq [a-i]-1$,
         then  
       $\binom{r-n}{i} \equiv 0$ mod $p$  as above,
         and  if $n=[a-i]$, then 
         $\binom{r-[a-i]}{i} \binom{p-1}{[a-2i]}
         \equiv \binom{p-1}{i} \binom{p-1}{[a-2i]}
         \equiv (-1)^{a-i}$ mod $p$
         since $r \equiv [a-i]-1$ mod $p$ and $\binom{p-1}{j} \equiv (-1)^{j}
         $ mod $p$. Thus, by 
         \eqref{i> a-i, G i,r binomial sum} and
          \eqref{i> a-i, i derivative condition for F},
         we have
         \begin{align}\label{i> a-i, i derivative condition}
                \sum_{ \substack{0 \leq l < r-i \\ l \equiv [a-i]
              ~\mathrm{mod}~(p-1)}} \binom{r-i}{l} \binom{l}{i}       
              -& \binom{r-i}{[a-i]}   \sum_{n=0}^{[a-i]} C_{n} 
                \sum_{\substack{0 \leq l \leq r-n \\  l \equiv [a-i] 
                \mathrm{~mod}~ (p-1)}} \binom{r-n}{l}
                \binom{l}{i} \nonumber \\
                & \equiv  -(-1)^{a-i}\binom{r-i}{[a-i]} C_{[a-i]} 
                \mod p,
         \end{align}
         so that the difference of  both sides of 
         \eqref{H in V i+1} also satisfies condition (ii)
         of \Cref{divisibility1} with $m=i+1$ as desired.
         
          Since $\binom{r-[a-i]}{i} \equiv \binom{p-1}{i} \not\equiv 0 \mod p$, 
          we have $X_{r-[a-i]}^{(i)}/X_{r-[a-i]}^{(i+1)} \neq (0)$,
          by Lemma~\ref{socle term singular} (with $i$ there equal to $[a-i]$).
           Since  $[a-i]< i$, the exact sequence 
           \eqref{exact sequence Vr} doesn't split for $m=[a-i]$
           and  
          $ \theta^{i}Y^{r-i(p+1)} $ 
          belongs to the socle of $V_{r}^{(i)}/V_{r}^{(i+1)}$,
                  by \Cref{Breuil map}.  Thus, by \eqref{H in V i+1}
                  we see  that
                  $H(X,Y) \in  X_{r-[a-i]}^{(i)}/X_{r-[a-i]}^{(i+1)}$.
                  Since $[a-i] \leq i-1  $, we get
          $H(X,Y) \in X_{r-(i-1)} + V_{r}^{(i+1)}$,
          so in $X_{r-(i-1)} + X_{r-i}^{(i+1)}$  because $H(X,Y) \in X_{r-i}$,
          as claimed. 
\end{proof}
We are now ready to determine the quotients $X_{r-i}^{(j)}/X_{r-i}^{(j+1)}$,
for $j \in \lbrace i , [a-i] \rbrace$ and  $i>[a-i]$. 
\begin{proposition}\label{singular i>r-i}
       Let $p \geq 3$, $r \equiv a \mod (p-1)$ with $1 \leq a \leq p-1$,
       $r \equiv r_{0} \mod p$  with $0 \leq r_{0} \leq p-1$
       and suppose $1 \leq [a-i] <  i < p-1$. If $j \in \lbrace i , [a-i] \rbrace$ and $r \geq j(p+1)+p$, then 
       \begin{align*}
             \frac{X_{r-i}^{(j)} }{X_{r-i}^{(j+1)}} =
             \begin{cases}
                     V_{r}^{(j)}/V_{r}^{(j+1)}, & \mathrm{if}~ r_{0} \not \in \mathcal{J}(a,j), \\
                    V_{[a-2i]} \otimes D^{i}, & \mathrm{if}~ r_{0} = [a-j]-1~ 
                    \mathrm{and}~ j \geq [a-j], \\
                    V_{p-1-[a-2i]} \otimes D^{a-i}, & \mathrm{if}~ r_{0} = [a-j]
                     ~ \mathrm{and}~ j \leq [a-j], \\
                    (0), & \mathrm{if}~ r_{0} \in \mathcal{I}(a,j) 
                    ~\mathrm{and}~ r \not \equiv [a-i]+i ~\mathrm{mod}~p.
             \end{cases}
       \end{align*}
\end{proposition}
\begin{proof}
        We prove the  lemma by treating
        the cases $j=i$ and $j=[a-i]$ separately.
        
        \textbf{Case}  $\boldsymbol{j=i}$:
        If $r_{0} \not \in \mathcal{J}(a,i)$,
        then by Lemma~\ref{medium i full} 
        we have $X_{r-i}^{(i)}/X_{r-i}^{(i+1)}=V_{r}^{(i)}/V_{r}^{(i+1)}$.
        So assume $r_{0} \in \mathcal{J}(a,i)$.
        Then, by  Lemmas \ref{medium a-i not full}  and \ref{medium i not full}, 
        we see that 
        $X_{r-i}^{(i)}/X_{r-i}^{(i+1)} = X_{r-[a-i]}^{(i)}/X_{r-[a-i]}^{(i+1)}$.
        As $[a-i]< i $, by Lemma~\ref{I vs J} (ii),
        we see that 
        $ r_{0} \in \mathcal{I}(a, i) \cup  \lbrace  [a-i]-1 \rbrace$ and
         $ r \not \equiv [a-i]+i $ mod $p$. Thus, 
         by \Cref{singular quotient i < [a-i]} (with 
         $i$  there equal to $[a-i]$ and  $j$ equal to $i$),
          we get 
         $X_{r-[a-i]}^{(i)}/X_{r-[a-i]}^{(i+1)} \cong
          V_{[a-2i]} \otimes D^{i}$ if $r \equiv [a-i]-1$ mod $p$
        and zero otherwise.   
        
        \textbf{Case}  $\boldsymbol{j=[a-i]}$:
        If $r_{0} \not \in \mathcal{J}(a,j)$,
        then by Lemma~\ref{medium a-i full} 
        we have $X_{r-i}^{(j)}/X_{r-i}^{(j+1)}=V_{r}^{(j)}/V_{r}^{(j+1)}$.
        So assume $r_{0}  \in \mathcal{J}(a,j)$, then by 
        \Cref{medium a-i not full}, we have 
        $X_{r-i}^{(j)}/X_{r-i}^{(j+1)} = X_{r-[a-i]}^{(j)}/X_{r-[a-i]}^{(j+1)}$.
        Since  $j < [a-j]$, by Lemma~\ref{I vs J} (i), 
        we see that 
         $r_{0} \in \mathcal{I}(a,j) \cup  \lbrace  [a-j] \rbrace$
         and $ r \not \equiv [a-j]+j$ mod $p$. Thus, 
         by \Cref{singular quotient i < [a-i]} (with $i$ (resp. $j$)  there equal to 
         $[a-i]$ (resp. $[a-i]$)),   we see that  
         $X_{r-j}^{(j)}/X_{r-j}^{(j+1)} \cong
         V_{p-1-[a-2i]} \otimes D^{a-i}$ if $r \equiv i \, =[a-j]$ mod $p$
        and is zero otherwise.   
\end{proof}
\begin{remark}
        Observe that  if in the statement of the proposition above,
        we replace the condition \enquote* {$r_{0}
        \in \mathcal{I}(a,j)$ and $r \not \equiv [a-i]+i$ mod $p$}
        by \enquote*{otherwise}, then we may include the case
        $[a-i]=i$ in the statement, by \Cref{singular i= [a-i]}. 
\end{remark}
\begin{corollary}\label{arbitrary singular i>r-i}    
        Let $p \geq 3$, $r \equiv a \mod (p-1)$ with $1 \leq a \leq p-1$, 
        $r \equiv r_{0} \mod p$ with $0 \leq r_{0} \leq p-1$
       and  suppose $1 \leq [a-i] < i < p-1$. Then for 
       $[a-i] \leq l \leq i$ and   $r \geq i(p+1)+p$,
       we have
       \begin{enumerate}[label= \emph{(\roman*)}]
                \item  If $r_{0} \not  \in \mathcal{J}(a,l)$, 
                 then 
                $X_{r-i}^{([a-i])} / X_{r-i}^{(l+1)} = V_{r}^{([a-i])} / V_{r}^{(l+1)}$.
                \item Assume $l \neq [a-l]$. If  $r_{0} \in \mathcal{I}(a,l)$
                and $ r \not \equiv   [a-i]+i ~\mathrm{mod}~p$, then
                $X_{r-i}^{(l)} / X_{r-i}^{(i+1)} = (0)$. 
       \end{enumerate}
\end{corollary}
\begin{proof}
         Let $j' := \max \lbrace j, [a-j] \rbrace$, for all $[a-i] \leq j \leq i$.
         Note that
         $ \lbrace j, [a-j] \rbrace = \lbrace j', [a-j'] \rbrace$.
         One checks that  $[a-i] \leq j, [a-j] \leq i$, for all $[a-i] \leq j \leq i$.
         Thus $[a-i] \leq j' \leq i <p-1$,  
         for all $[a-i] \leq j \leq i$.
        \begin{enumerate}[label= (\roman*)]
                 \item 
               Since $\mathcal{J}(a,[a-i]) \subseteq  \cdots \subseteq 
                  \mathcal{J}(a,l_{1}) \subseteq \cdots \subseteq \mathcal{J}(a,i)$,
                 we see that $r_{0} \not \in \mathcal{J}(a,l)$ implies that
                 $r_{0} \not \in \mathcal{J}(a,j)$, for all $[a-i] \leq j \leq l$.  
                Hence by \Cref{singular i>r-i}, we have
               $X_{r-j'}^{(j)}/X_{r- j'}^{(j+1)} = V_{r}^{(j)}/V_{r}^{(j+1)}$, for every
                $[a-i] \leq j \leq l$. Since $X_{r-j'}^{(j)}/X_{r-j'}^{(j+1)} \subseteq 
                 X_{r-i}^{(j)}/X_{r-i}^{(j+1)} \subseteq V_{r}^{(j)}/V_{r}^{(j+1)}$,
                for all $[a-i] \leq j \leq l$, we see that
              $X_{r-i}^{([a-i])}/X_{r-i}^{(l+1)} = V_{r}^{([a-i])}/X_{r-i}^{(l+1)}$.
              \item 
              Since  $\mathcal{I}(a,[a-i]) \subseteq  \cdots  \subseteq 
              \mathcal{I}(a,l)  \subseteq  \cdots \subseteq \mathcal{I}(a,i)$,
               we see that $r_{0} \in  \mathcal{I}(a,l)$ and  
               $ r \not \equiv  [a-i]+i $ mod $p$ implies
                $r_{0} \in \mathcal{I}(a,j)$ and   
                $ r \not \equiv  [a-i]+i $ mod $p$, for all  $l \leq j \leq i$. 
               One checks that $[a-i]+i = [a-j]+j$, for all $[a-i] \leq j \leq i$.
               So $r_{0} \in \mathcal{I}(a,j)$ and   
                $ r \not \equiv  [a-j]+j  = [a-j']+j'$ mod $p$, for all  $l \leq j \leq i$. 
                We claim that $X_{r-i}^{(j)}/X_{r-i}^{(j+1)}= (0)$, for all
                $l \leq j \leq i$. Clearly (ii) follows from the claim.
                
                 Fix $l \leq j \leq i$. If $[a-j] \leq j \leq i$,  then
                 by the first part of \Cref{reduction},  we have
                $X_{r-i}^{(j)}/X_{r-i}^{(j+1)} = X_{r-j}^{(j)}/X_{r-j}^{(j+1)}$.
                Similarly,  by the second part, if $j \leq [a-j]  \leq i$, then 
                $X_{r-i}^{(j)}/X_{r-i}^{(j+1)} = X_{r-[a-j]}^{(j)}/X_{r-[a-j]}^{(j+1)}$.
                So to prove the claim it is enough to  show that 
                $X_{r-j'}^{(j)}/X_{r-j'}^{(j+1)} =(0)$.
                If   $[a-j'] < j'$, then 
              $X_{r-j'}^{(j)}/X_{r- j'}^{(j+1)} = (0)$, by \Cref{singular i>r-i}.
              If $j'= [a-j']$, then  $j' = j$.
              So $[a-i] \leq l \lneq  j \leq i$ because  $[a-l] \neq l$,  
             whence
               \[  [a-i]  \leq 
               j-1 < j < [a-j]+1 = [a-(j-1)] \leq i+1 \leq p-1.
               \]
              Thus,  by the first and third parts of \eqref{interval I}, we have 
              $j-1, j \not \in \mathcal{I}(a,j-1)$. 
              Since $ \mathcal{I}(a,l) \subseteq \mathcal{I}(a,j-1)
              \subseteq \mathcal{I}(a,j)$,
              we get $j-1, j \not \in \mathcal{I}(a,l)$.
              So $r_{0} \neq j-1$, $j$.
              Thus, $r_{0}  \in \mathcal{I}(a,j) \smallsetminus 
              \lbrace j,j-1 \rbrace $ and 
              $r \not \equiv [a-i]+i$ mod $p$.
              So, by Lemma~\ref{I vs J} (applied with $i$ there equal to $j$), 
              we see that $r_{0} \neq j$, $j-1$ and $r_{0} \in \mathcal{J}(a,j)$.
              Hence, by \Cref{singular i= [a-i]}, 
              we again have $X_{r-j}^{(j)}/X_{r- j}^{(j+1)} = (0)$.
              This proves the claim.
               \qedhere
            \end{enumerate}         
\end{proof}
For $0 \leq  j \leq n \leq m $, we have the following commutative diagram
 \begin{equation}\label{commutative diagram arbitrary}
        \begin{tikzcd}   
            0 \arrow{r} & \frac{X_{r-i}^{(n)}}{X_{r-i}^{(m)}}  \arrow{r}
             \arrow[hookrightarrow]{d}
           & \frac{X_{r-i}^{(j)}}{X_{r-i}^{(m)}} \arrow{r} \arrow[hookrightarrow]{d} 
           & \frac{X_{r-i}^{(j)}}{X_{r-i}^{(n)}} \arrow{r} \arrow[hookrightarrow]{d} & 0 \\    
           0 \arrow{r} & \frac{V_{r}^{(n)}}{V_{r}^{(m)} } \arrow{r}  & \frac{V_{r}^{(j)}}
           {V_{r}^{(m)}} 
            \arrow{r}  & \frac{V_{r}^{(j)}}{V_{r}^{(n)}}  \arrow{r}  & 0. 
        \end{tikzcd} 
   \end{equation} 
     Taking the cokernel  of each the inclusions and applying the snake lemma, we get
     \begin{align}\label{Q(i) exact sequence 2}
        0 \rightarrow \frac{V_{r}^{(n)}}{X_{r-i}^{(n)}+V_{r}^{(m)}} \rightarrow
        \frac{V_{r}^{(j)}}{X_{r-i}^{(j)}+V_{r}^{(m)}} \rightarrow
         \frac{V_{r}^{(j)}}{X_{r-i}^{(j)}+V_{r}^{(n)}}  \rightarrow 0.
     \end{align}

     We now determine the structure of $Q(i)$ in terms  of $Q([a-i]-1)$
     when $1 \leq [a-i]<i < p-1$. 
\begin{theorem}\label{Structure of Q(i) i>[a-i]}
       Let $p \geq 3$, $r \equiv a ~ \mathrm{mod}~(p-1)$ with $1 \leq a \leq p-1$ and
       let $r \equiv r_{0}~ \mathrm{mod}~p$ with $0 \leq r_{0} \leq p-1$.
       Let  $1 \leq [a-i]<i < p-1$ and $i(p+1)+p \leq r$. 
        Then we have an exact sequence of $\Gamma$-modules
       \begin{align*}
              0 \rightarrow W \rightarrow Q(i) \rightarrow Q([a-i]-1) \rightarrow 0,
       \end{align*}
       where
       \begin{enumerate}[label= \emph{(\roman*)}]
              \item If $r_{0} \not \in \mathcal{J}(a,i)$,
                                then $W= (0)$.
              \item  If $r_{0} \in  \mathcal{J}(a,i) 
                          \smallsetminus  \mathcal{I}(a,[a-i])$ and
              \begin{enumerate}            
                      \item   If $[a-r_{0}] < r_{0}+1$, then
                                 $
                                      0 \rightarrow V_{r}^{([a-r_{0}]+1)}/V_{r}^{(i+1)} \rightarrow W
                                      \rightarrow V_{[a-2r_{0}]} \otimes D^{r_{0}} \rightarrow 0.
                                 $         
              \item  If    $[a-r_{0}] = r_{0} +1 $, then
                           $W= V_{r}^{([a-r_{0}])}/ V_{r}^{(i+1)}$.
              \item If  $  [a-r_{0}] >  r_{0}+1 $,
                        then    
                        $
                               0 \rightarrow V_{r}^{([a-r_{0}])}/V_{r}^{(i+1)}  \rightarrow W
                                \rightarrow V_{p-1-[2r_{0}+2-a]} \otimes D^{r_{0}+1} \rightarrow 0.
                        $           
             \end{enumerate}                                                 
              \item  If $r_{0}\in  \mathcal{I}(a,[a-i])$ and
                          $r \not \equiv [a-i]+i ~\mathrm{mod}~p$,
                                then $W= V_{r}^{([a-i])}/V_{r}^{(i+1)}$.
       \end{enumerate}
\end{theorem}
\begin{proof}
          For $[a-i] \leq j \leq i$, we have $[a-i] \leq  [a-j] \leq i$.
          Thus,  $j':= \min \lbrace [a-j], j \rbrace \geq [a-i]$,
          for all $[a-i] \leq j \leq i$.  By \Cref{reduction corollary} (ii),
          we see that $X_{r-j}  \subseteq X_{r-(j-1)}+ V_{r}^{(j')}
          \subseteq X_{r-(j-1)}+ V_{r}^{([a-i])}$, for all
           $[a-i] \leq j \leq i$.
        Hence
          \[
              X_{r-i} \subseteq X_{r-(i-1)}+ V_{r}^{([a-i])}
               \subseteq  \cdots
               \subseteq X_{r-[a-i]}+ V_{r}^{([a-i])}
                \subseteq X_{r-([a-i]-1)}+ V_{r}^{([a-i])}.
          \]
          Therefore, 
               $X_{r-i}+V_{r}^{([a-i])} = X_{r-([a-i]-1)}+V_{r}^{([a-i])}$,
               by \Cref{first row filtration},   and so
               \[ 
                     \frac{V_{r}}{X_{r-i}+V_{r}^{([a-i])}}
                     =  \frac{V_{r}}{X_{r-([a-i]-1)}+V_{r}^{([a-i])}} =Q([a-i]-1).
               \]
               Taking $j= [a-i]$ in diagram \eqref{commutative diagram} we see that
               \[
                   0  \rightarrow W  \rightarrow Q(i) \rightarrow Q([a-i]-1) \rightarrow 0,
               \]
               where $W$ is the quotient of $V_{r}^{([a-i])}/ V_{r-i}^{(i+1)}$
               by $X_{r-i}^{([a-i])}/X_{r-i}^{(i+1)}$. To determine $W$
               explicitly we consider the cases described in the theorem.
               We first prove (i) and (iii).
        \begin{enumerate}
                \item[(i)]  If $r_{0} \not \in \mathcal{J}(a,i)$, then  by 
                 \Cref{arbitrary singular i>r-i} (i), we have
                 $X_{r-i}^{([a-i])}/X_{r-i}^{(i+1)} =  V_{r}^{([a-i])}/V_{r}^{(i+1)}$.  
                 Thus $W =(0)$. 
              \item[(iii)] If $r_{0} \in \mathcal{I}(a,[a-i])$ and 
               $ r \not \equiv [a-i]+i $ mod $p$, then by
                \Cref{arbitrary singular i>r-i} (ii), 
               we have 
              $X_{r-i}^{([a-i])}/X_{r-i}^{(i+1)} =(0)$. Thus 
              $W =V_{r}^{([a-i])}/V_{r}^{(i+1)}$.
        \end{enumerate} 
        We now prove (ii). So  we may  assume $r_{0} \in 
          \mathcal{J}(a,i) \smallsetminus \mathcal{I}(a,[a-i])$.
        By   \eqref{interval J for i > a-i} and
         \eqref{interval [a-i], for i> [a-i]},  we have
              \begin{align*}
                     \mathcal{J}(a,i) \smallsetminus \mathcal{I}(a,[a-i])=
                           \lbrace [a-i]-1, [a-i], \ldots, i-1,i \rbrace. 
              \end{align*}
              So the congruence class
             of $[a-i]+i$ mod $p$  has a representative in $\mathcal{I}(a,[a-i])$ but not  
              in $\mathcal{J}(a,i)$.
              Clearly $[a-i]-1 \leq r_{0}    \leq i$. 
              One checks that $[a-i]-1 \leq [a-r_{0}]-1 \leq i$ as well. 
              We now prove (ii) according
              to how the numbers $r_0$ and $[a-r_0]-1$ compare to each other.
        \begin{enumerate}
              \item[(a)]  
               If  $ [a-r_{0}]-1 < r_{0}$, then 
              $ [a-i]+1 \leq r_{0} \leq i$
              (because $[a-i] < i= [a-[a-i]]$ and $[a-i]-1 < i+1 = [a-([a-i]-1)]$).
              So $[a-i] \leq [a-r_{0}]+1 \leq i$.
              Since  $2 \leq [a-i]+1 \leq r_{0} \leq i $ and 
              $r_{0}-1 \equiv [a-([a-r_{0}]+1)]$ mod $(p-1)$,
              we have $r_{0} = [a-([a-r_{0}]+1)]+1$. Thus,
              by   \eqref{interval I},
              we see that   
             $r_{0} \in \mathcal{I}(a,[a-r_{0}]+1)$.
              Thus by \Cref{arbitrary singular i>r-i} (ii), we have
              $X_{r-i}^{([a-r_{0}]+1)}/X_{r-i}^{(i+1)} =(0)$,
              if $[a-r_{0}]+1 \neq r_{0} -1$.  If $[a-r_{0}]+1 = r_{0}-1$,
              then  $r_{0} =[a-r_{0}]+2 \geq 3$ and 
              $[a-r_{0}]+2 > r_{0}-2 = [a-([a-r_{0}]+2)]$.
              So by \Cref{arbitrary singular i>r-i} (ii),
              we have
              $X_{r-i}^{([a-r_{0}]+2)}/X_{r-i}^{(i+1)} =(0)$.
             Further, by \eqref{interval J for i > a-i} and  \Cref{singular i= [a-i]}, we have
             $X_{r-[a-r_{0}]-1}^{([a-r_{0}]+1)}/X_{r-[a-r_{0}]-1}^{([a-r_{0}]+2)} =(0)$.
             By the second part of \Cref{reduction}, we have
             $X_{r-i}^{([a-r_{0}]+1)}/ X_{r-i}^{([a-r_{0}]+2)} =
              X_{r-[a-r_{0}]-1}^{([a-r_{0}]+1)}/X_{r-[a-r_{0}]-1}^{([a-r_{0}]+2)} =(0)$.
             So in either case, we have  
              $X_{r-i}^{([a-r_{0}]+1)}/X_{r-i}^{(i+1)} =(0)$.
                 Since $[a-r_{0}] \leq r_{0} =[a-[a-r_{0}]]$, by   \Cref{singular i= [a-i]}
                 and   \Cref{singular i>r-i},
                 we have $X_{r-r_{0}}^{([a-r_{0}])}/X_{r-r_{0}}^{([a-r_{0}]+1)}
                 \cong V_{p-1-[a-2r_{0}]} \otimes D^{a-r_{0}}$. 
                 Since $[a-r_{0}] \leq r_{0} \leq i$, by the second part of \Cref{reduction}
                 (with $j$ there equal to $[a-r_{0}]$),  we have
                 $X_{r-i}^{([a-r_{0}])}/X_{r-i}^{([a-r_{0}]+1)} = 
                 X_{r-r_{0}}^{([a-r_{0}])}/X_{r-r_{0}}^{([a-r_{0}]+1)}
                 \cong V_{p-1-[a-2r_{0}]} \otimes D^{a-r_{0}}$. Thus by
                 \Cref{Breuil map} (with $m=[a-r_{0}]$) and the exact sequences
                 \eqref{commutative diagram arbitrary} and 
                 \eqref{Q(i) exact sequence 2} (with $j=[a-r_{0}]$,
                 $n=[a-r_{0}]+1$, $m=i+1$), we see that
                 \[
                    0 \rightarrow V_{r}^{([a-r_{0}]+1)}/V_{r}^{(i+1)} \rightarrow 
                    \frac{V_{r}^{([a-r_{0}])}}{X_{r-i}^{([a-r_{0}])} + V_{r}^{(i+1)}}
                    \rightarrow V_{[a-2r_{0}]} \otimes D^{r_{0}} \rightarrow 0.
                 \]
                 If $r_{0} =i$, then the middle term is $W$ and we are done.
                 If $r_{0}< i$, then  $[a-r_{0}] \geq [a-i]+1$. 
                 Thus $[a-r_{0}]-1 = [a-r_{0}-1]$  so
                 by \eqref{interval J for [a-i]} (with $i$ there equal to $r_{0}+1$), we have 
                 $r_{0} \not \in \mathcal{J}(a,[a-r_{0}]-1)$, whence 
                  by \Cref{arbitrary singular i>r-i} (i) we have
                  $X_{r-i}^{([a-i])}/ X_{r-i}^{([a-r_{0}])} = 
              V_{r}^{([a-i])}/ V_{r}^{([a-r_{0}])} $. Thus, by
                 the exact sequences
                 \eqref{commutative diagram arbitrary}   and
                 \eqref{Q(i) exact sequence 2} (with $j=[a-i]$,
                 $n =[a-r_{0}]$, $m=i+1$),
                 we have 
              $W \cong \frac{V_{r}^{([a-r_{0}])}}{X_{r-i}^{([a-r_{0}])} + V_{r}^{(i+1)}}$,
              whence (a) follows from the above exact sequence.                
                 
               \item[(b)] 
                 Assume $[a-r_{0}]=r_{0}+1$.
               Since $[a-i] < i < i+1$, we have $r_{0} \neq i$, i.e.,
               $[a-i]-1 \leq r_{0}<i$. So $[a-r_{0}] >[a-i]$.
                If $r_{0} =0$, then $a=1$ and $[a-i] \leq 1$. So $[a-i]=[1-i]=1$
                 and $i =0$, $p-1$ which are not possible.
                Therefore, $[a-r_{0}] > r_{0} \geq 1$.
               Note  that
               $r_{0}= [a-r_{0} -1] = [a-r_{0}] -1 < [a-r_{0}] =   r_{0} +1
               =[a-[a-r_{0}-1]] $. 
               Thus, by \eqref{interval J for [a-i]},  we
               see that $r_{0} \not \in \mathcal{J}(a,[a-r_{0}-1])$.
               Also, $ [a-i] \leq [a-r_{0}]-1  = r_{0} < i$
               whence by \Cref{arbitrary singular i>r-i} (i),
               we have  $X_{r-i}^{([a-i])} / X_{r-i}^{([a-r_{0}])} = 
                V_{r}^{([a-i])}/V_{r}^{([a-r_{0}])}$.
               Also, by \eqref{interval I for [a-i]} (with $i$ there equal to 
               $r_{0}$), we see that $r_{0} \in \mathcal{I}(a,[a- r_{0} ]) $
               and
               $ r \not \equiv [a-i]+i$ mod $p$, whence by
                \Cref{arbitrary singular i>r-i} (ii), we have
               $X_{r-i}^{([a-r_{0}])}/X_{r-i}^{(i+1)} = (0)$.
               Thus, by  the exact sequences
                 \eqref{commutative diagram arbitrary}   and
                 \eqref{Q(i) exact sequence 2} 
                 (with $j=[a-i]$, $n=[a-r_{0}]$, $m=i+1$) we see that  
                $W \cong V_{r}^{([a-r_{0}])}/V_{r}^{(i+1)}$. Since 
                $[a-r_{0}] = r_{0}+1$, we are done.
                
                \item[(c)] If $ [a-r_{0}]> r_{0}+1   $, then $r_{0}
                \neq i, i-1$ (because $[a-i] < i+1$ and $[a-(i-1)] = [a-i]+1 
                \leq (i-1)+1 $). So $[a-i]-1 \leq r_{0} < i-1$, hence
                one checks that
                $[a-i] < [a-(r_{0}+1)] \leq i $. Set $l =[a- (r_{0}+1)]$,
                so $ [a-i]<l \leq i$.  Since 
                $[a-l] =r_{0}+1 < i < p-1$, we have $[a-(l-1)] = [a-l]+1$ 
                and so $r_{0} = [a-(l-1)] -2$. By \eqref{interval J}, we see that 
               $r_{0} \not \in \mathcal{J}(a,l-1)$.
               As $[a-i]\leq  l-1 \leq i-1$, it follows from
               \Cref{arbitrary singular i>r-i} (i), that
              $X_{r-i}^{([a-i])} / X_{r-i}^{(l)}  =V_{r}^{([a-i])}/V_{r}^{(l)}$.
              Since $[a-l] =r_{0}+1 \leq l$,  by   \Cref{singular i= [a-i]} 
              and \Cref{singular i>r-i}, we have  $X_{r-l}^{(l)}/X_{r-l}^{(l+1)} = 
               V_{[a-2l]} \otimes D^{l}$, as $r_{0}=[a-l]-1$. 
             Since $[a-l] \leq l \leq i$, by the first part of \Cref{reduction}, we have 
            $X_{r-i}^{(l)}/X_{r-i}^{(l+1)} = X_{r-l}^{(l)}/X_{r-l}^{(l+1)}$. 
            Therefore,   by the exact sequences 
            \eqref{commutative diagram arbitrary}   and
            \eqref{Q(i) exact sequence 2} (with $j=[a-i]$,
            $n= l$ and $m=l+1$) and the exact sequence
            \eqref{exact sequence Vr} (with $m=l$), we see that
                   $$
                          \frac{V_{r}^{([a-i])}}{X_{r-i}^{([a-i])} + V_{r}^{(l+1)}}
                          \cong V_{p-1-[a-2l]} \otimes D^{a-l}.
                  $$   
                  If $l=i$, then we are done. If $l< i$, then
                  $[a-i] \leq l +1 =[a-r_{0}] \leq i$.
                  Noting $r_{0}< [a-r_{0}]$, by \eqref{interval I for [a-i]}, 
                  we see that
                  $r_{0} \in \mathcal{I}(a,[a-r_{0}])$. Thus, by
                 \Cref{arbitrary singular i>r-i} (ii), we have
                   $X_{r-i}^{(l+1)}/ X_{r-i}^{(i+1)} = (0)$.
                   Thus, by  the exact sequences
                 \eqref{commutative diagram arbitrary}   and
                 \eqref{Q(i) exact sequence 2} (with $j=[a-i]$, 
                 $n=l+1$ and $m=i+1$),
                 and above we see that  
                 \[
                      0 \rightarrow V_{r}^{(l+1)}/V_{r}^{(i+1)}  \rightarrow W
                      \rightarrow V_{p-1-[a-2l]} \otimes D^{a-l} \rightarrow 0.
                 \]
                 Since $l =[a-r_{0}] -1$, this proves part (c). \qedhere
        \end{enumerate}       
\end{proof}

\subsection{The case \texorpdfstring{$ \boldsymbol{ i = a, ~p-1}$}{}}
  \label{Section i = a or p - 1}
Let $1 \leq a \leq p-1$ be such that $r \equiv a$ mod $(p-1)$.
In this subsection, we determine the quotients $Q(a)$ and $Q(p-1)$. Recall that,
for $1 \leq i \leq p$, we have defined
\[
    P(i) = \frac{V_{r}}{X_{r-(i-1)}+V_{r}^{(i+1)}}.
\]
          For $1 \leq i \leq p-1$, we have  the following exact sequence 
          \[
              0 \rightarrow \frac{X_{r-i}+V_{r}^{(i+1)}}{X_{r-(i-1)}+V_{r}^{(i+1)}}
              \rightarrow P(i) \rightarrow Q(i) \rightarrow 0,
          \]
          where the first map is the inclusion and the last map is the quotient
          map. Since $X_{r-(i-1)} \subseteq X_{r-i}$, one checks that
          $X_{r-i} \cap (X_{r-(i-1)} + V_{r}^{(i+1)}) =  
          X_{r-(i-1)} + X_{r-i}^{(i+1)}$.  Thus, by the second isomorphism theorem,
           we have an exact sequence
          \begin{align}\label{Q and P exact sequence}
                0 \rightarrow 
                \frac{X_{r-i}}{X_{r-(i-1)} + X_{r-i}^{(i+1)}} 
                 \rightarrow P(i) \rightarrow Q(i) \rightarrow 0.
          \end{align}
Thus to  determine $Q(i)$ in terms of $P(i)$ it is enough to determine
the leftmost module in the above exact sequence. Note that 
 we have an ascending chain 
 \begin{align}\label{ascending chain}
    X_{r-(i-1)} + X_{r-i}^{(i+1)} \subseteq X_{r-(i-1)} + X_{r-i}^{(i)}
    \subseteq \cdots \subseteq X_{r-(i-1)} + X_{r-i}^{(1)}
    \subseteq X_{r-i}.
 \end{align}
 Recall 
 $Y_{i,j} =( X_{r-i}^{(j)}/X_{r-i}^{(j+1)})/( X_{r-(i-1)}^{(j)}/X_{r-(i-1)}^{(j+1)})$.
By \eqref{Y i,j}, to determine the leftmost module
in \eqref{Q and P exact sequence}, it is enough to determine 
the successive quotients $Y_{i,j}$, for all $0 \leq j \leq i$. Note that
by \Cref{Structure X(1)} and  the exact sequence
 \eqref{exact sequence Vr} for $m=0$, 
 we have the successive quotient 
 $  Y_{a,0} \cong V_{p-1-a} \otimes D^{a} $.
\subsubsection{The case \texorpdfstring{ $\boldsymbol{i=a .} $}{}}
\label{section i = a}
We start with the case $i=a$, where $1 \leq a \leq p-1$.
The first result asserts that most $Y_{i,j}$
vanish.
\begin{lemma}\label{i=a smaller quotients}
       Let $p \leq r \equiv a~\mathrm{mod}~(p-1)$ with $1 \leq a \leq p-1$,
       and let
       $r \equiv r_{0} ~\mathrm{mod}~p$ with $0 \leq r_{0} \leq p-1$.
       \begin{enumerate}[label= \emph{(\roman*)}]
               \item   If $a\leq p-1$, then $X_{r-a}^{(j)}/X_{r-a}^{(j+1)} = 
                X_{r-(a-1)}^{(j)}/X_{r-(a-1)}^{(j+1)}$, for all $1 \leq j < p-1$.
                \item If $a=p-1$ and $r_{0} \neq p-2 $,
                then $X_{r-a}^{(j)}/X_{r-a}^{(j+1)} = 
                X_{r-(a-1)}^{(j)}/X_{r-(a-1)}^{(j+1)}$, for all  $1 \leq j \leq p-1$.
       \end{enumerate}
\end{lemma}
\begin{proof}
       Recall  that
        $G_{a,r}(X,Y)$ generates 
       $W_{a,r}$, the image of $V_{a} \hookrightarrow 
       \ind_{B}^{\Gamma}(\chi_{2}^{a}) \overset{\psi_{i}}{\twoheadrightarrow}
       X_{r-a}/X_{r-(a-1)}$ as a $\Gamma$-module. Assume that
       $G_{a,r}(X,Y) \in X_{r-(a-1)}+X_{r-a}^{(n)}$, for some $n \geq 1$
       (to be determined later). Thus,
       $W_{a,r} \subseteq (X_{r-(a-1)}+X_{r-a}^{(n)}) /X_{r-(a-1)}$.
       Note that we have the following diagram
       \[
             \begin{tikzcd}
                 0 \arrow[r, rightarrow] & V_{a}  \arrow[r, rightarrow]
                  & \ind_{B}^{\Gamma} (\chi_{2}^{a}) \arrow[r, rightarrow]
                   \arrow[d,  twoheadrightarrow, "\psi_{i}"] & V_{p-1-a} \otimes D^{a} 
                   \arrow[r, rightarrow] & 0. \\
                  &  & X_{r-a}/X_{r-(a-1)}  & &
             \end{tikzcd}    
       \]
       Therefore,  $ V_{p-1-a} \otimes D^{a} \twoheadrightarrow 
        X_{r-a}/(X_{r-(a-1)} +X_{r-a}^{(n)})$.
       But, by \eqref{Y i,j}, we have
       \[
             V_{p-1-a} \otimes D^{a} \cong Y_{a,0} 
              \cong  \frac{X_{r-a}}{X_{r-(a-1)}+X_{r-a}^{(1)}}.
        \]
        Thus, in the ascending chain of modules
         \[ 
             X_{r-(a-1)}+X_{r-a}^{(n)} \subseteq 
             X_{r-(a-1)}+X_{r-a}^{(n-1)} \subseteq 
             \cdots \subseteq X_{r-(a-1)}+X_{r-a}^{(1)}
             \subseteq X_{r-a},
         \]    
         we have $X_{r-(a-1)}+X_{r-a}^{(j)} 
         =  X_{r-(a-1)}+X_{r-a}^{(j+1)}$, for all $1 \leq  j  < n$.
         Thus, by \eqref{Y i,j}, we see that 
         $X_{r-a}^{(j)}/X_{r-a}^{(j+1)} = 
         X_{r-(a-1)}^{(j)}/X_{r-(a-1)}^{(j+1)}$,
         for all $1 \leq j < n$.
         
         By \eqref{G i,r} we have  $G_{a,r}(X,Y) 
         \in X_{r-(a-1)}+X_{r-a}^{(p-1)} $. So if $a \leq p-1$,
         we may take $n=p-1$ above and we obtain (i). 
         If $a=p-1$ and $r_{0} \neq a-1$,
         then we can do better. Indeed since $\binom{r-a}{p-1} \equiv 0 \mod p$, by \Cref{quotient image}, we have 
         $G_{r-a}(X,Y) \in V_{r-a}^{(p)} \cap X_{r-a,\,r-a}$. 
         So $X^{a} G_{r-a}(X,Y) \in X_{r-a}^{(p)}$, by 
         \Cref{surjection1}, whence $G_{a,r}(X,Y) =
         -X^{r}+ X^{a} G_{r-a}(X,Y) \in X_{r-(a-1)}+X_{r-a}^{(p)}$,
         by \eqref{G i,r}. Taking $n=p$ above we obtain (ii).
\end{proof}

\begin{lemma}\label{i=a, p-1 quotients}
     Let $(p-1)(p+1)+p< r \equiv a ~\mathrm{mod}~ (p-1)$. Then  $V_{a}
      \hookrightarrow X_{r-a}^{(p-1)}/X_{r-a}^{(p)}$ 
     if and only if  $r \equiv a-1 ~\mathrm{mod}~ p$.
\end{lemma}
\begin{proof}
        The case $1 \leq a < p-1$ follows from \Cref{socle term singular}.
        Assume $a=p-1$. 
        By \Cref{Structure X(1)},
        we have $ X_{r-(p-1)} / X_{r-(p-1)}^{(1)} = V_{r}/V_{r}^{(1)}$
        and   $X_{r-(p-2)}/X_{r-(p-2)}^{(1)} \cong V_{p-1}$. Thus,
        by \Cref{Breuil map} and  the fact $X_{r-(p-2)} \subseteq 
        X_{r-(p-1)}$, we have
        \[
        V_{0} \cong 
           \frac{X_{r-(p-1)}/X_{r-(p-1)}^{(1)} }{X_{r-(p-2)}/X_{r-(p-2)}^{(1)}}
           = Y_{p-1,0} \stackrel{\eqref{Y i,j}}{\cong} 
           \frac{X_{r-(p-2)}+X_{r-(p-1)}}{X_{r-(p-2)}+X_{r-(p-1)}^{(1)}}
           \cong
           \frac{X_{r-(p-1)}}{X_{r-(p-2)}+X_{r-(p-1)}^{(1)}}.
        \]
        Also by \Cref{Structure of induced} and \Cref{induced and successive},
        we have 
        $
           V_{0} \oplus V_{p-1} \cong \ind_{B}^{\Gamma}(1)
           \overset {\psi_{p-1}}{\twoheadrightarrow}
           X_{r-(p-1)}/X_{r-(p-2)}.
        $
        So 
        \begin{align}\label{a=p-1, i=p-1}
               V_{p-1} \twoheadrightarrow \frac{X_{r-(p-2)}+X_{r-(p-1)}^{(1)}}{X_{r-(p-2)}}.
        \end{align}
        
       If $r \equiv p-2$ mod $p$, then by \Cref{quotient image}, 
        we have $G_{r-(p-1)}(X,Y) \in  X_{r-(p-1), \, r-(p-1)}^{(p-1)}$. 
        Thus, by   \Cref{surjection1}, we see that
       $X^{p-1}G_{r-(p-1)}(X,Y) \in X_{r-(p-1)}^{(p-1)}$. 
       Using \eqref{G r expression}, one checks that  the coefficient of 
       $X^{r-(p-1)}Y^{p-1}$ in $X^{p-1}G_{r-(p-1)}(X,Y)$ equals
       $-\binom{r-(p-1)}{p-1} \equiv -\binom{p-1}{p-1} \equiv -1  $ mod $p$,
       by Lucas' theorem. Thus,
       $X^{p-1}G_{r-(p-1)}(X,Y)  \not \in  X_{r-(p-1)}^{(p)}$,
         by \Cref{divisibility1}, whence
          $X_{r-(p-1)}^{(p-1)}/X_{r-(p-1)}^{(p)} \neq (0)$. 
         Since $p-2 < p-1 = [p-1-(p-1)]$, by the third part of \Cref{reduction}
         (with $i=p-2$ and $j= p-1$), we have 
         $X_{r-(p-2)}^{(p-1)}/X_{r-(p-2)}^{(p)}  \cong X_{r}^{(p-1)}/X_{r}^{(p)}$.
         Further, by \Cref{singular quotient X_{r}} and $r \equiv p-2$ mod $p$,
          we see that $X_{r}^{(p-1)}/X_{r}^{(p)} =(0)$.
         Thus, by \eqref{Y i,j}, we have
         \[
                \frac{X_{r-(p-1)}^{(p-1)}+X_{r-(p-2)}}{X_{r-(p-1)}^{(p)}+X_{r-(p-2)}} 
                \cong Y_{p-1,p-1} = \frac{X_{r-(p-1)}^{(p-1)}}{X_{r-(p-1)}^{(p)}} 
                \neq (0).
          \]
          Combining this with \eqref{a=p-1, i=p-1}, we get 
          $X_{r-(p-1)}^{(p-1)}/X_{r-(p-1)}^{(p)} \cong V_{p-1}$.
           
           We now prove the converse. Assume  $r \not \equiv p-2$ mod $p$. 
           Then by   \Cref{quotient image},   we have 
           $X^{p-1}G_{r-(p-1)}(X,Y) \in X_{r-(p-1)}^{(p)}$, whence
           $G_{p-1,r-(p-1)}(X,Y)= X^{p-1}G_{r-(p-1)}(X,Y) -
           X^{r} \in  X_{r-(p-1)}^{(p)}+X_{r-(p-2)}$, by \eqref{G i,r}. 
           Recall $G_{p-1,r-(p-1)}(X,Y)$
          generates $W_{p-1,r}$, the image of $V_{p-1} \hookrightarrow
           \ind_{B}^{\Gamma}(1) \overset{\psi_{p-1}}
           {\twoheadrightarrow} X_{r-(p-1)}/X_{r-(p-2)}$.
           Thus, the image of $V_{p-1}$ in \eqref{a=p-1, i=p-1} 
           lies in $\frac{X_{r-(p-1)}^{(p)}+X_{r-(p-2)}}{ X_{r-(p-2)}}$. Thus 
           $X_{r-(p-1)}^{(p)}+X_{r-(p-2)} = X_{r-(p-1)}^{(p-1)}+X_{r-(p-2)}$, whence
           \[
               \frac{X_{r-(p-1)}^{(p-1)}}{X_{r-(p-1)}^{(p)}} 
               \stackrel{\eqref{Y i,j}}{=}
                \frac{X_{r-(p-2)}^{(p-1)}}{X_{r-(p-2)}^{(p)}}  
                = \frac{X_{r}^{(p-1)}}{X_{r}^{(p)}} . 
           \]
           Since $a=p-1$, by \Cref{singular quotient X_{r}}, we see that
           $V_{p-1} \not \hookrightarrow X_{r}^{(p-1)}/X_{r}^{(p)}$. This completes the 
          proof.
\end{proof}
\begin{proposition}\label{singular i=a}
     Let $ r \equiv a ~ \mathrm{mod}~(p-1)$, with $1 \leq a \leq p-1$
     and let  $r \equiv r_{0} ~\mathrm{mod}~p$ with $0 \leq r_{0} \leq p-1$.
     For $j \in \lbrace a, p-1 \rbrace$ and $r \geq j(p+1)+p$, we have
     \begin{align*}
          \frac{X_{r-a}^{(j)}}{X_{r-a}^{(j+1)}} =
          \begin{cases}
               V_{p-1-a} \otimes D^{a}, & \mathrm{if}~j=a 
               ~\mathrm{and}~r_{0}=a,a+1, \ldots, p-1, \\
               V_{a}, & \mathrm{if}~j=p-1 ~\mathrm{and}~r_{0}=a-1, \\
               (0), & \mathrm{otherwise}.
          \end{cases}
     \end{align*}
\end{proposition}
\begin{proof}
    Since $X_{r-(a-1)} \subseteq X_{r-(a-1)}+X_{r-a}^{(p)}
     \subseteq X_{r-(a-1)}+X_{r-a}^{(p-1)} \subseteq 
     \cdots \subseteq X_{r-(a-1)}+X_{r-a}^{(1)} \subseteq  X_{r-a}$,
     we have
     \begin{align*}
         \dim \left( \frac{X_{r-(a-1)}+X_{r-a}^{(p-1)}}{X_{r-(a-1)}+X_{r-a}^{(p)}}
         \right)  +
         \dim \left( \frac{X_{r-a}}{X_{r-(a-1)}+X_{r-a}^{(1)}}
         \right) 
          \leq  \dim \left( \frac{X_{r-a}}{X_{r-(a-1)}} \right)
         \leq p+1.
     \end{align*}
     Since $Y_{a,0} \cong V_{p-1-a} \otimes D^{a}$, we have
     $\dim Y_{a,p-1} \leq a+1$.
     
     $\boldsymbol{\mathrm{Case} ~ a < p-1}$\textbf{:}
     By \Cref{i=a smaller quotients} (i),
     we have $X_{r-a}^{(a)}/X_{r-a}^{(a+1)} = X_{r-(a-1)}^{(a)}/X_{r-(a-1)}^{(a+1)}$.
     Since $a-1 < a < [a-a] = p-1$, by the third part of \Cref{reduction}, we have 
     $X_{r-(a-1)}^{(a)}/X_{r-(a-1)}^{(a+1)} = X_{r}^{(a)}/X_{r}^{(a+1)}$.
     So $X_{r-a}^{(a)}/X_{r-a}^{(a+1)}= X_{r}^{(a)}/X_{r}^{(a+1)}$ and
     the assertion for $j=a$ follows from \Cref{singular quotient X_{r}}.
     By \Cref{reduction corollary 2}, we have $X_{r-(a-1)}^{(p-1)}/X_{r-(a-1)}^{(p)} 
     = (0)$. So    $Y_{a,p-1} = X_{r-a}^{(p-1)}/X_{r-a}^{(p)}$.
     Since the exact sequence \eqref{exact sequence Vr} doesn't split   
     for $m=p-1$, we have $X_{r-a}^{(p-1)}/X_{r-a}^{(p)} \neq (0)$
     if and only if  $V_{a} \hookrightarrow X_{r-a}^{(p-1)}/X_{r-a}^{(p)}$.
     Hence by \Cref{socle term singular}, we  have
     $X_{r-a}^{(p-1)}/X_{r-a}^{(p)} \neq (0)$ if and only if  
     $V_{a} \hookrightarrow X_{r-a}^{(p-1)}/X_{r-a}^{(p)}$ if and only if 
     $r \equiv a-1 $ mod $p$.  So if $r \equiv a-1$ mod $p$,
     then $Y_{a,p-1}  = X_{r-a}^{(p-1)}/X_{r-a}^{(p)} \cong V_{a}$
     as $\dim Y_{a,p-1} \leq a+1$
     and if $ r \not \equiv a-1$ mod $p$, then
     $X_{r-a}^{(p-1)}/X_{r-a}^{(p)} =(0)$. 
     
     $\boldsymbol{\mathrm{Case} ~ a = p-1}$\textbf{:}
     As earlier, one checks that $X_{r-(a-1)}^{(p-1)}/X_{r-(a-1)}^{(p)}
     =X_{r}^{(p-1)}/X_{r}^{(p)}$. 
     If $r \not \equiv p-2$ mod $p$, then 
     $X_{r-a}^{(p-1)}/X_{r-a}^{(p)} = X_{r-(a-1)}^{(p-1)}/X_{r-(a-1)}^{(p)}$,
     by Lemma~\ref{i=a smaller quotients}.
     Thus, by \Cref{singular quotient X_{r}},
      we have $X_{r-a}^{(p-1)}/X_{r-a}^{(p)} =V_{0}$
      if $r \equiv p-1$ mod $p$ and zero if
      $r \equiv 0,1, \ldots, p-3$ mod $p$. 
      Assume $r \equiv p-2$ mod $p$.
     By Lemma~\ref{i=a, p-1 quotients}, we have 
     $V_{p-1} \hookrightarrow X_{r-a}^{(p-1)}/X_{r-a}^{(p)} $.
     Also, by \Cref{singular quotient X_{r}}, we have 
     $X_{r-(a-1)}^{(p-1)}/X_{r-(a-1)}^{(p)}
     =X_{r}^{(p-1)}/X_{r}^{(p)} =(0)$. 
     So $V_{p-1} \hookrightarrow X_{r-a}^{(p-1)}/X_{r-a}^{(p)} 
     = Y_{a,p-1}$. Since 
     $\dim Y_{a,p-1} \leq a+1 = p $, we get 
     $V_{p-1} = X_{r-a}^{(p-1)}/X_{r-a}^{(p)} $.
\end{proof}
\begin{theorem}\label{Structure of Q(i) if i = a}
       Let $a(p+1)+p \leq r \equiv a ~ \mathrm{mod}~(p-1)$ with $1 \leq a \leq p-1$.
       Then
       \begin{enumerate}
             \item[\emph{(i)}] If $a \neq p-1$ or $r \not \equiv p-2 ~ \mathrm{mod}~p$,
                             then
                             \[
                                0 \rightarrow V_{p-1-a} \otimes D^{a} \rightarrow P(a) \rightarrow 
                                Q(a) \rightarrow 0.
                             \]
              \item[\emph{(ii)}] If $a = p-1$ and $r \equiv p-2 ~ \mathrm{mod}~p$,
                              then
                              \[
                                0 \rightarrow V_{0} \oplus V_{p-1} \rightarrow P(a) \rightarrow 
                                Q(a) \rightarrow 0.
                               \]
       \end{enumerate}
\end{theorem}   
\begin{proof}
          By  \eqref{Q and P exact sequence}, we have an exact sequence
        \[
            0 \rightarrow 
            \frac{X_{r-a}}{X_{r-(a-1)} + X_{r-a}^{(a+1)}} 
            \rightarrow P(a) \rightarrow Q(a) \rightarrow 0.
        \]
        Note that we have an ascending chain of modules
        \[
              X_{r-(a-1)} + X_{r-a}^{(a+1)} \subseteq X_{r-(a-1)} + X_{r-a}^{(a)}
              \subseteq  \cdots \subseteq
              X_{r-a}. 
          \]
        By \eqref{Y i,j}, the successive quotients
          are isomorphic to $Y_{a,j}$, for $0 \leq j \leq a$. 
        By Lemma~\ref{i=a smaller quotients}, we have 
        $Y_{a,j} = (0)$, for $1 \leq j <a$. So
        \begin{align}\label{Q a proof exact sequence}
            0 \rightarrow Y_{a,a} \rightarrow 
            \frac{X_{r-a}}{X_{r-(a-1)} + X_{r-a}^{(a+1)}} \rightarrow
            Y_{a,0} \rightarrow 0.
        \end{align}
        Recall  that $Y_{a,0} \cong V_{p-1-a} \otimes  D^{a}$. 
        Thus it remains to determine  $Y_{a,a}$. 
        By Lemma~\ref{i=a smaller quotients},  if $a \neq p-1$ or $r \not \equiv p-2$ mod $p$, then
        $Y_{a,a} = (0)$. 
          If $a=p-1$ and $r \equiv p-2$ mod $p$, then  
          $X_{r-a}^{(p-1)}/X_{r-a}^{(p)} \cong V_{p-1}$, by \Cref{singular i=a}. 
          Also by the third part of \Cref{reduction} and 
          \Cref{singular quotient X_{r}}, we have 
          $X_{r-(a-1)}^{(p-1)}/X_{r-(a-1)}^{(p)} =
          X_{r}^{(p-1)}/X_{r}^{(p)}=(0)$. So $Y_{a,a}
          =Y_{p-1,p-1} \cong V_{p-1}$.
          Since $V_{p-1}$ is projective, the  exact sequence 
          \eqref{Q a proof exact sequence} splits
          and this completes the proof of the lemma.
\end{proof}

The above theorem completely determines the structure of $Q(a)$.
Indeed, by the remarks at the beginning of Section~\ref{section Q}, the structure 
of $P(a)$ is completely determined by $Q(a-1)$ and 
$X_{r-(a-1)}^{(a)}/X_{r-(a-1)}^{(a+1)}$. The latter  module
is equal to $X_{r}^{(a)}/X_{r}^{(a+1)}$, by the third part of
\Cref{reduction}, so is completely determined by 
\Cref{singular quotient X_{r}}. If $a=1$, then 
the  structure of $Q(a-1) = Q(0)$ is known by \Cref{Structure Q(0)}.
If $a \geq 2$, then  $a-1 \geq [a-(a-1)] =1$, so the structure of $Q(a-1)$
can be determined using the results of \S \ref{section i = a-i}, \S \ref{section i > a-i}, again 
in terms of $Q(0)$ and $W$, see
\Cref{Structure of Q i=[a-i]} and \Cref{Structure of Q(i) i>[a-i]}.
\subsubsection{The case \texorpdfstring{$\boldsymbol{i=p-1}$}{}}
\label{section i = p-1}
In this section we determine $Q(i)$ when $i=p-1$. 
Since the case $i=a=p-1$ was already treated in \S \ref{section i = a},
we may assume $a<p-1$.

\begin{lemma}\label{i=p-1 exceptional case}
      Let $p \geq 3$,
      $p \leq r \equiv a ~ \mathrm{mod}~(p-1)$, with $1 \leq a < p-1$.
      If $r \equiv 0, 1, \ldots, a-2$ 
      or 
     $ p-1 ~\mathrm{mod}~p$ , 
      then $X_{r-(p-1)} \subseteq X_{r-(p-2)}+V_{r}^{(p)}$.
      Furthermore, $X_{r-(p-1)}^{(a)}/X_{r-(p-1)}^{(a+1)} 
      = X_{r-a}^{(a)}/X_{r-a}^{(a+1)}$ and 
      $X_{r-(p-1)}^{(p-1)}/X_{r-(p-1)}^{(p)} 
      = X_{r-a}^{(p-1)}/X_{r-a}^{(p)}$. 
\end{lemma}
\begin{proof}
       Recall that 
       \begin{align*}
            F_{p-1,r}(X,Y)  &\stackrel{\eqref{F i,r definition}}{=} \sum_{\lambda \in \f}^{} \lambda^{[2p-2-a]} 
             X^{p-1}( \lambda X+Y)^{r-(p-1)}  \\
             & \stackrel{\eqref{sum fp}}{\equiv} -X^{r} 
             - \sum_{\substack{0 < l \leq r-(p-1) 
             \\ l \equiv 0 ~\mathrm{mod}~(p-1)}}^{} \binom{r-(p-1)}{l}
                   X^{r-l}Y^{l} \mod p
       \end{align*}
       generates the quotient 
       $X_{r-(p-1)}/X_{r-(p-2)}$ as a $\Gamma$-module
       because $a \neq p-1$. So to prove the first statement of the lemma
       it  is enough to show
       $F_{p-1,r}(X,Y) \in X_{r-(p-2)}+V_{r}^{(p)}$. 
       We claim that $F(X,Y):= F_{p-1,r}(X,Y) +X^{r} \in X_{r-(p-2)}+V_{r}^{(p)}$,
       which proves the the first statement of the lemma, 
       as $X^{r} \in X_{r-(p-2)}$, by \Cref{first row filtration}.
       Since $r-(p-1) \equiv a \not \equiv p-1$ mod $(p-1)$,
       we get the coefficient of $X^{p-1}Y^{r-(p-1)}$ in $F(X,Y)$ is zero.
      From the hypothesis we see that  $r-(p-1) \not \equiv p-1$ mod $p$,
      whence   $\binom{r-(p-1)}{p-1} \equiv 0$ mod $p$ by  Lucas' theorem.
      Therefore $X^{p}$, $Y^{p} \mid F(X,Y)$.  
      So $F(X,Y)$ satisfies condition (i) of \Cref{divisibility1}
      with $m=p$.
      For $0 \leq m \leq p-1$, by \Cref{binomial sum}, we have
      \begin{align*}
          \sum_{\substack{0 < l \leq r-(p-1) 
             \\ l \equiv 0 ~\mathrm{mod}~(p-1)}}^{} \binom{r-(p-1)}{l} 
             \binom{l}{m}  &= \sum_{\substack{0 \leq  l \leq r-(p-1) 
             \\ l \equiv 0 ~\mathrm{mod}~(p-1)}}^{} \binom{r-(p-1)}{l} 
             \binom{l}{m}  - \delta_{0,m} \\
             & \equiv \binom{r-(p-1)}{m} \left[ \binom{[a-m]}{[p-1-m]} 
             + \delta_{p-1,[p-1-m]}  \right]  - \delta_{0,m} ~\mathrm{mod}~ p.
      \end{align*}
      If $0 \leq m < a$, then   $\binom{[a-m]}{[p-1-m]}  = \binom{a-m}{p-1-m}
       = 0$ as $1 \leq a<p-1$.  So the above sum vanishes for $0 \leq m < a$.
       By the assumption we have 
       $r-(p-1) \equiv 0$, $1, \ldots, a-1$ mod $p$, whence by Lucas'
       theorem $\binom{r-(p-1)}{m} \equiv 0$ mod $p$, for  $a \leq m \leq p-1$.
       So the above sum vanishes, for all $0 \leq m \leq p-1$.
       Thus by \Cref{divisibility1}, we have $F(X,Y) \in V_{r}^{(p)}$. 
       
       To prove the last assertion,  note that 
       $X_{r-(p-1)} \subseteq X_{r-(p-2)}+V_{r}^{(p)}$ implies 
       $X_{r-(p-1)} = X_{r-(p-2)}+X_{r-(p-1)}^{(p)}$.
        Recall that we have an ascending chain of modules
        \[
            X_{r-(p-2)}+X_{r-(p-1)}^{(p)} \subseteq 
            \cdots \subseteq X_{r-(p-2)}+X_{r-(p-1)}^{(a+1)}
            \subseteq X_{r-(p-2)}+X_{r-(p-1)}^{(a)}
            \subseteq \cdots \subseteq X_{r-(p-1)}.
        \]       
       Since the extreme terms are equal, all the intermediate 
       terms are equal. Hence, by \eqref{Y i,j}, we have
       $Y_{p-1,j} =(0)$ for $j=a$, $p-1$. Thus 
       $X_{r-(p-1)}^{(j)}/X_{r-(p-1)}^{(j)} = X_{r-(p-2)}^{(j)}/X_{r-(p-2)}^{(j)}$
       for $j=a$, $p-1$.
       Since $a \leq p-2 < p-1$, we have 
           $X_{r-(p-2)}^{(j)}/X_{r-(p-2)}^{(j)}= X_{r-a}^{(j)}/X_{r-a}^{(j+1)}$,
       by the first (resp. second) part  of  \Cref{reduction}, for $j=a$
       (resp.  $p-1$).
       Thus $X_{r-(p-1)}^{(j)}/X_{r-(p-1)}^{(j)}=
       X_{r-a}^{(j)}/X_{r-a}^{(j+1)}$.
\end{proof}
\begin{lemma}\label{i=p-1, full}
      Let $a(p+1)+p \leq r \equiv a ~ \mathrm{mod}~(p-1)$, with $1 \leq a < p-1$.
      If $r \equiv a-1$, $a, \ldots, p-3$, $p-2~ \mathrm{mod}~p$, then
      $X_{r-(p-1)}^{(a)}/X_{r-(p-1)}^{(a+1)} = V_{r}^{(a)}/V_{r}^{(a+1)}$.  
\end{lemma}
\begin{proof}
     First we consider the case $r \equiv a-1$ mod $p$. Let 
     $F(X,Y) := X^{p-1}Y^{r-(p-1)}- X^{r-a}Y^{a} \in X_{r-(p-1),\,r}$, by 
     \Cref{first row filtration}. Clearly $F(X,Y) \in V_{r}^{(1)}$, by 
     \Cref{divisibility1}.
      Since $r \equiv a-1$
     mod $p$, we see that $r-(p-1) \equiv a$ mod $p$. Thus,
      by Lucas' theorem, we have
     \[
       \binom{r-(p-1)}{n} \equiv \binom{a}{n}, ~\forall ~0 \leq n \leq a. 
       \]
     Hence, by \Cref{divisibility1}, we have $F(X,Y) \in V_{r}^{(a)}$. Also by 
     \Cref{breuil map quotient}, we have the image of 
     $F(X,Y) \equiv \theta^{a}X^{r-a(p+1)-(p-1)}Y^{p-1} $
     mod $V_{r}^{(a+1)}$ up to terms involving 
     $\theta^{a} X^{r-a(p+1)}$ and $\theta Y^{r-a(p+1)}$. 
     Clearly the image of $\theta^{a}X^{r-a(p+1)-(p-1)}Y^{p-1}$ under the quotient
     map $V_{r}^{(a)}/V_{r}^{(a+1)} \rightarrow V_{p-1-[r-2a]}$ is non-zero by 
     \Cref{Breuil map}, so also the image of $F(X,Y)$.
     Since the sequence \eqref{exact sequence Vr} doesn't split for 
     $m=a$ and $a \neq p-1$,
      we get $X_{r-(p-1)}^{(a)}/X_{r-(p-1)}^{(a+1)} = V_{r}^{(a)}/V_{r}^{(a+1)}$.
     
     So we may assume $r \equiv a, a+1, \ldots , p-2$ mod $p$.  
     If $a=1$, then $F(X,Y) \equiv (r+1)  \theta X^{r-(p+1)-(p-1)}Y^{p-1}$
     mod $V_{r}^{(a+1)}$ up to terms involving 
     $\theta X^{r-(p+1)}$ and $\theta Y^{r-(p+1)}$. 
     As above one checks that $F(X,Y)$ generates
     $V_{r}^{(1)}/V_{r}^{(2)}$. So  we may further assume $2 \leq a < p-1 $.
     Let 
     \begin{align*}
       A &= \left( \binom{r-n}{m} \binom{p-1+a-m-n}{a-m} \right)_{1 \leq m,n \leq a-1} \\
            &=  \left( \binom{r-1-n}{m+1} \binom{p-1+a-2-m-n}{a-1-m} 
             \right)_{0 \leq m,n \leq a-2}\\
            &= \left( \frac{r-1-n}{m+1} \binom{r-2-n}{m} \binom{p-1+a-2-m-n}{a-1-m} 
                   \right)_{0 \leq m,n \leq a-2} \\
             & = D  \left( \binom{r-2-n}{m} \binom{p-1+a-2-m-n}{a-1-m} 
                   \right)_{0 \leq m,n \leq a-2} D',
     \end{align*}
     where  $D=$ 
     diag $(1 , 2^{-1}, \ldots, (a-1)^{-1})$ and  
     $D'=$ diag $(r-1, r-2, \ldots, r-(a-1))$ are  diagonal matrices.
     Applying \Cref{matrix det} (ii) (with $a$ there equal to $p-1+a-2$, 
     $j=a-1$ and $i=a-2$), we see that 
     $A$  is invertible. Choose $C_{1}, \ldots , C_{a-1}$ such that
     \begin{align}\label{choice C , i=p-1}
            \sum_{n=1}^{a-1} C_{n} \binom{r-n}{m} \binom{p-1+a-m-n}{a-m}
             = \binom{r-(p-1)}{m} - \binom{a}{m}, ~ \forall ~ 1 \leq m \leq a-1.
     \end{align}
      Let 
      \begin{align*}
             G(X,Y) & := F(X,Y) + \sum_{n=1}^{a-1} C_{n} \sum_{k \in \fstar}^{}
            k^{n}  X^{n} (k X+ Y)^{r-n}  \\
             & \stackrel{\eqref{sum fp}}{\equiv}   
             X^{p-1}Y^{r-(p-1)}- X^{r-a}Y^{a}  - \sum_{n=1}^{a-1} C_{n} 
             \sum_{\substack{0 \leq l \leq r-n \\ l \equiv a~\mathrm{mod}~(p-1)}}^{}
             \binom{r-n}{l} X^{r-l} Y^{l} \mod p.
      \end{align*}
      By \Cref{Basis of X_r-i}, we see that $G(X,Y) \in X_{r-(p-1)}$.
       We claim that $G(X,Y) \in V_{r}^{(a)}$  and generates
       $V_{r}^{(a)}/V_{r}^{(a+1)}$.
      Clearly $Y^{a} \mid G(X,Y)$ and the coefficient
      of $Y^{r}$ in $G(X,Y)$ is zero. Since the smallest number
      strictly less than $r$ congruent to $a$ mod $(p-1)$ is 
      $r-(p-1)$, we see that   $X^{p-1} \mid G(X,Y)$. 
      So $G(X,Y)$ satisfies condition (i) of  \Cref{divisibility1} for 
      $m=a$.  By \Cref{binomial sum} with $m=0$, we have
      \[
           1-1-\sum_{n=1}^{a-1} C_{n} 
           \sum_{\substack{0 \leq l \leq r-n \\ l \equiv a~\mathrm{mod}~(p-1)}}^{}
           \binom{r-n}{l}  \equiv - \sum_{n=1}^{a-1} C_{n} \binom{a-n}{a} 
           \equiv 0 ~\mathrm{mod}~p.
      \] 
       For $1 \leq m \leq a-1$,
      again by \Cref{binomial sum}, we have
      \[
           \sum_{n=1}^{a-1} C_{n} 
           \sum_{\substack{0 \leq l \leq r-n \\ l \equiv a~\mathrm{mod}~(p-1)}}^{}
           \binom{r-n}{l}  \binom{l}{m}\equiv \sum_{n=1}^{a-1} C_{n} 
           \binom{r-n}{m} \binom{[a-m-n]}{a-m} \mod p.
      \] 
      If   $m+n<a$, then by Lucas' theorem we see that
      $\binom{[a-m-n]}{a-m} \equiv \binom{a-m-n}{a-m} \equiv 0 $ mod $p$
      and $\binom{p-1+a-m-n}{a-m} \equiv \binom{a-m-n-1}{a-m} \equiv 0$
      mod $p$. If $m+n \geq a$, then $[a-m-n] =p-1+a-m-n$. Therefore
      $\binom{[a-m-n]}{a-m} \equiv \binom{p-1+a-m-n}{a-m}$ mod $p$. 
      Thus 
      \begin{align*}
             \sum_{n=1}^{a-1} C_{n}  \binom{r-n}{m} \binom{[a-m-n]}{a-m}
             & \equiv  \sum_{n=1}^{a-1} C_{n}  \binom{r-n}{m}
              \binom{p-1+a-m-n}{a-m}  \mod p\\
            & \stackrel{ \eqref{choice C , i=p-1}}{\equiv}
             \binom{r-(p-1)}{m} - \binom{a}{m} \mod p.
      \end{align*}
      Thus, by \Cref{divisibility1}, we see that $G(X,Y) \in V_{r}^{(a)}$.  Since 
     $a \neq p-1$, the sequence \eqref{exact sequence Vr} doesn't split.
     To show $X_{r-(p-1)}^{(a)}/X_{r-(p-1)}^{(a+1)} = V_{r}^{(a)}/V_{r}^{(a+1)}$,
     it is enough to show  the image of $G(X,Y)$ under the rightmost 
     map in the exact sequence \eqref{exact sequence Vr} is non-zero. 
     By \Cref{binomial sum}, we have 
     \begin{align*}
        -\sum_{n=1}^{a-1} C_{n}  
        \sum_{\substack{0 \leq l \leq r-n \\ l \equiv a~\mathrm{mod}~(p-1)}}^{}
           \binom{r-n}{l} \binom{l}{a}  - \binom{a}{a}
           & \equiv    -\sum_{n=1}^{a-1} C_{n}  \binom{r-n}{a} -1 \mod p \\ 
         & = \mathrm{the ~ coefficient~ of} ~ X^{r-a}Y^{a}~ \mathrm{in}~ G(X,Y).
     \end{align*}
       Since $X^{p-1} \mid G(X,Y)$, by  \Cref{breuil map quotient}, 
       we see that 
      \[
        G(X,Y) \equiv \binom{r-(p-1)}{a} \theta^{a}   X^{r-a(p+1)-(p-1)}Y^{p-1} 
        \mod V_{r}^{(a+1)},
      \]
      up to terms involving $\theta^{a}X^{r-a(p+1)}$ and 
      $\theta^{a} Y^{r-a(p+1)}$.
      Since $r \equiv a, a+1, \ldots, p-2$ mod $p$, by Lucas' theorem,
      we have   $\binom{r-(p-1)}{a} \not \equiv 0$ mod $p$. 
      Thus, by  \Cref{Breuil map}, the image of $G(X,Y)$
      under the rightmost map in the sequence \eqref{exact sequence Vr} is non-zero. 
       This proves that 
      $G(X,Y)$ generates $V_{r}^{(a)}/V_{r}^{(a+1)}$ as a $\Gamma$-module
      and this finishes the proof.
\end{proof}
\begin{proposition}\label{singular i=p-1}
        Let $p \geq 3$, $r \equiv a~\mathrm{mod}~(p-1)$ with $1 \leq a < p-1$
        and $r \equiv r_{0} ~\mathrm{mod}~p$ with $0 \leq r_{0} \leq p-1$.
        Let $j \in \lbrace a, p-1 \rbrace$. If $r \geq j(p+1)+p$, then
        \begin{align*}
               \frac{X_{r-(p-1)}^{(j)}}{X_{r-(p-1)}^{(j+1)}} =
               \begin{cases}
                       V_{r}^{(a)}/V_{r}^{(a+1)}, & \mathrm{if}~
                       j=a~ \mathrm{and} ~r \equiv a-1, \ldots,p-2 
                       ~\mathrm{mod}~p \\
                       X_{r-a}^{(j)}/X_{r-a}^{(j+1)}, & \mathrm{otherwise}.
               \end{cases}
        \end{align*}
\end{proposition}
\begin{proof} 
      Since $a  \leq p-2 < p-1$,  by the first and  
      second  parts of \Cref{reduction} with $j=a$ and  $j=p-1$, we have
       $X_{r-(p-2)}^{(a)}/X_{r-(p-2)}^{(a+1)} =
       X_{r-a}^{(a)}/X_{r-a}^{(a+1)}$ and
       $X_{r-(p-2)}^{(p-1)}/X_{r-(p-2)}^{(p)}=
       X_{r-a}^{(p-1)}/X_{r-a}^{(p)}$ 
       respectively.  We consider the cases $j=a$ and $j=p-1$
       separately.
       
       $\boldsymbol{\mathrm{Case} ~ j = a}$\textbf{:}
       If  $ r \equiv a-1, \ldots,p-2 $ mod $p$, then 
       $X_{r-(p-1)}^{(a)}/X_{r-(p-1)}^{(a+1)} = V_{r}^{(a)}/V_{r}^{(a+1)}$,
       by Lemma~\ref{i=p-1, full}.  If
       $ r  \not \equiv a-1, \ldots,p-2 $ mod $p$, then 
       $X_{r-(p-1)}^{(a)}/X_{r-(p-1)}^{(a+1)} =
       X_{r-a}^{(a)}/X_{r-a}^{(a+1)}$, by \Cref{i=p-1 exceptional case}.
       
       $\boldsymbol{\mathrm{Case} ~ j = p-1}$\textbf{:}
       So $j \neq a$.
       If $r \not \equiv a-1, \ldots, p-2$ mod $p$, then
       again by \Cref{i=p-1 exceptional case}, we have
       $X_{r-(p-1)}^{(p-1)}/X_{r-(p-1)}^{(p)}
       = X_{r-a}^{(p-1)}/X_{r-a}^{(p)}$ and we are done. So assume 
       $r  \equiv a-1, \ldots, p-2$ mod $p$. Then by the
       $j=a$ case, we have $X_{r-(p-1)}^{(a)}/X_{r-(p-1)}^{(a+1)}
       = V_{r}^{(a)}/V_{r}^{(a+1)}$.  Since
       $X_{r-(p-2)}^{(a)}/X_{r-(p-2)}^{(a+1)} = 
       X_{r-a}^{(a)}/X_{r-a}^{(a+1)}$,
      so $Y_{p-1,a} = (V_{r}^{(a)}/V_{r}^{(a+1)})/(X_{r-a}^{(a)}/X_{r-a}^{(a+1)}) $.       
       Thus, by \Cref{singular i=a}
       and the exact sequence \eqref{exact sequence Vr}, we have
       \begin{align}\label{Y p-1 a}
               Y_{p-1,a} =
               \begin{cases}
                       V_{r}^{(a)}/V_{r}^{(a+1)}, & \mathrm{if}~ 
                       r \equiv a-1 ~\mathrm{mod}~p, \\
                       V_{a}, & \mathrm{if}~ r \equiv a, a+1, \ldots,p-2
                       ~\mathrm{mod}~p.
               \end{cases}
        \end{align} 
         By \eqref{Y i,j}
         and the chain \eqref{ascending chain}, we see that
         $\dim Y_{p-1,p-1} + \dim Y_{p-1,a} \leq 
         \dim  X_{r-(p-1)}/X_{r-(p-2)} $. But by 
         \Cref{induced and successive}, we have
         $ \dim X_{r-(p-1)}/X_{r-(p-2)} \leq p+1$. If $r \equiv a-1$ mod $p$,
         then it follows 
         from \eqref{Y p-1 a} that $Y_{p-1,p-1} = (0)$,
         so $X_{r-(p-1)}^{(p-1)}/X_{r-(p-1)}^{(p)} =  
         X_{r-(p-2)}^{(p-1)}/X_{r-(p-2)}^{(p)}  =
         X_{r-a}^{(p-1)}/X_{r-a}^{(p)}$ and again we are done. 
         If $r \equiv a, \ldots, p-2$ mod $p$,
         then by \Cref{singular i=a}, we have 
         $X_{r-(p-2)}^{(p-1)}/X_{r-(p-2)}^{(p)}=X_{r-a}^{(p-1)}/X_{r-a}^{(p)}
         =(0)$, so $Y_{p-1,p-1} =X_{r-(p-1)}^{(p-1)}/X_{r-(p-1)}^{(p)}$. 
         Since the exact sequence \eqref{exact sequence Vr}
         doesn't split for $m=p-1$ and $a< p-1$, we see that
         $Y_{p-1,p-1} \neq (0)$ if and only if  $V_{a}
         \hookrightarrow Y_{p-1,p-1}$.   By  \Cref{induced and successive} and 
         \Cref{Common JH factor} (i),  we see that $X_{r-(p-1)}/X_{r-(p-2)}$
         doesn't have repeated JH factor. This forces $Y_{p-1,p-1} = (0)$,
          as otherwise the distinct subquotients  $Y_{p-1,a}$ and $Y_{p-1,p-1} $  of 
          $X_{r-(p-1)}/X_{r-(p-2)}$ would both contain $V_{a}$, by \eqref{Y p-1 a}.
\end{proof}
The proposition above in conjunction with \Cref{singular i=a}
determines the quotients stated explicitly. We finally determine $Q(p-1)$.
\begin{theorem}\label{Structure of Q(i) if i = p - 1}
         Let $(p-1)(p+1)+p \leq r \equiv a ~ \mathrm{mod}~(p-1)$
         with $1 \leq a < p-1$.  Then
         \begin{enumerate}
               \item[\emph{(i)}]  If $r \equiv 0,1, \ldots, a-2$
               or $ p-1 ~\mathrm{mod}~p$,  then $Q(p-1) \cong P(p-1)$.
               \item[\emph{(ii)}]   If $r \equiv a-1~ \mathrm{mod}~p$,
                                              then   
                         \[                     
                             0 \rightarrow \frac{V_{r}^{(a)}}{V_{r}^{(a+1)}}
                              \rightarrow P(p-1) \rightarrow  Q(p-1) \rightarrow 0. 
                          \]                                           
              \item[\emph{(iii)}]  If $r \equiv a, \ldots,p-2~ \mathrm{mod}~p$,
                                              then   
                         \[                     
                             0 \rightarrow V_{a}  \rightarrow P(p-1) \rightarrow 
                             Q(p-1) \rightarrow 0. 
                          \]                                           
         \end{enumerate}    
\end{theorem}   
\begin{proof}
          First we consider the case $r \equiv 0,1, \ldots, a-2 $ mod $p$
          or $r \equiv p-1$ mod $p$. Then by  the first part of 
          \Cref{i=p-1 exceptional case},
           we have $X_{r-(p-1)}+V_{r}^{(p)}=X_{r-(p-2)}+V_{r}^{(p)}$.
          Thus $Q(p-1) \cong P(p-1)$.
          
          By the exact sequence \eqref{Q and P exact sequence} 
          with $i =p-1$, we have
          \[
              0 \rightarrow    \frac{X_{r-(p-1)}}{X_{r-(p-2)} + X_{r-(p-1)}^{(p)}}                
              \rightarrow P(p-1) \rightarrow Q(p-1) \rightarrow 0.
          \]
          Observe that we have an ascending chain of modules
          \[
              X_{r-(p-2)} + X_{r-(p-1)}^{(p)}  \subseteq 
              X_{r-(p-2)} + X_{r-(p-1)}^{(p-1)}
              \subseteq \cdots \subseteq 
               X_{r-(p-2)} + X_{r-(p-1)}^{(1)}
              \subseteq X_{r-(p-1)}.
          \]
          Note that by \eqref{Y i,j}, the successive quotients
          are isomorphic to $Y_{p-1,j}$, for $0 \leq j \leq p-1$. 
          If  $j \neq a$, $p-1$, then by \Cref{reduction corollary} (i), 
          we have $Y_{p-1,j} =(0)$. 
          By \Cref{singular i=p-1}, we have 
          $X_{r-(p-2)}^{(p-1)}/X_{r-(p-2)}^{(p)} \subseteq
          X_{r-(p-1)}^{(p-1)}/X_{r-(p-1)}^{(p)} \subseteq
          X_{r-a}^{(p-1)}/X_{r-a}^{(p)} \subseteq 
          X_{r-(p-2)}^{(p-1)}/X_{r-(p-2)}^{(p)}$. So, $Y_{p-1,p-1}=(0)$.
          By \eqref{Y p-1 a},
          $Y_{p-1,a} = V_{r}^{(a)}/V_{r}^{(a+1)}$ if $r \equiv 
          a-1 $ mod $p$ and is equal to $V_{a}$ if $r \equiv a, \ldots, p-2 $
          mod $p$.
\end{proof}   
The  above theorem determines the structure of $Q(p-1)$ in terms of
$P(p-1)$ which in principle is determined by $Q(p-2)$, by the remarks 
made at the beginning of Section~\ref{section Q}. If $a=p-2$, then $Q(p-2)=Q(a)$.
Otherwise $[a-(p-2)] = a+1 \leq p-2 $, so $Q(p-2)$ is in turn determined by
$Q(a)$ by \Cref{Structure of Q i=[a-i]} and \Cref{Structure of Q(i) i>[a-i]}.
Thus, the structure of $Q(p-1)$ can in principle be obtained from $Q(a)$,
which was determined in the previous subsection.
\subsection{Irreducibility of \texorpdfstring{$Q(i)$}{} }
It is possible to completely determine the JH factors of $Q(i)$ in all cases using the results
of Sections~\ref{Section i not a nor p - 1} and \ref{Section i = a or p - 1}. An an example,
in this section we determine when $Q(i)$ is irreducible, 
for $1 \leq i \leq p-1$. 
\begin{lemma}
      Let  $p \geq 3$, $r \equiv a~\mathrm{mod}~(p-1)$ with $1 \leq a \leq p-1$.
      If $0 \leq i < p-1$ and $r \geq i(p+1)+p$, then $Q(i) \neq (0)$. 
\end{lemma}
\begin{proof}
       Suppose $Q(i)=(0)$, towards a contradiction.
       Then, by the exact sequence \eqref{Q(i) exact sequence},
       we have $X_{r-i}/X_{r-i}^{(i+1)} = V_{r}/V_{r}^{(i+1)}$ and so
      $X_{r-i}^{(j)}/X_{r-i}^{(j+1)} = V_{r}^{(j)}/V_{r}^{(j+1)}$, for all 
      $0 \leq j \leq i$.   If $1 \leq i < a$, then by \Cref{Structure X(1)}, we have
      $V_{a} \cong X_{r-i}/X_{r-i}^{(1)} \subsetneq V_{r}/V_{r}^{(1)} $,
      a contradiction.
      If $a \leq i < p-1$, then by the first part of \Cref{reduction}
      (applied with  $i$ there equal to $i$ and $j=a$), we see that 
      $X_{r-i}^{(a)}/X_{r-i}^{(a+1)}
      = X_{r-a}^{(a)}/X_{r-a}^{(a+1)} $.  Hence,
       by \Cref{singular i=a} (applied with $j=a$), we have $X_{r-i}^{(a)}/X_{r-i}^{(a+1)}
       \subsetneq V^{(a)}_{r}/V_{r}^{(a+1)} $, which is again a contradiction.
\end{proof}
\begin{lemma}
         Let  $p \geq 3$, $r \equiv a~\mathrm{mod}~(p-1)$ with $1 \leq a \leq p-1$.
         Let $1 \leq i < p-2$ with $i \neq a-1$, $a$. 
         If $i(p+1)+p \leq r$, then $Q(i)$ is reducible. 
\end{lemma}
\begin{proof}
        If $i< [a-i]$, then by \Cref{Structure of Q(i) if i<[a-i]},
        we see that $Q(i)$ is reducible as both $W$ there and $Q(i-1)$ do not vanish. 
        If $i \geq [a-i]$, then $[a-i]-1 <  i+1 = [a -([a-i]-1)]$.
        Since  $i \neq a-1$, $p-1$, we have $[a-i]-1 \neq 0$. Also
        $[a-i]-1 \neq a$ as $i \neq p-2$. Thus $Q([a-i]-1)$ is reducible by what we
        just proved. Hence, so is $Q(i)$ by Theorems~\ref{Structure of Q i=[a-i]}
        and \ref{Structure of Q(i) i>[a-i]}.
\end{proof}
We next consider the case  $i = a-1$ and $i = a$.
\begin{lemma}\label{irred Q(a), Q(a-1)}
         Let  $p \geq 3$, $r \equiv a~\mathrm{mod}~(p-1)$ with $1 \leq a \leq p-1$.
         Let $r \equiv r_{0} ~\mathrm{mod}~p$ with $0 \leq r_{0} \leq p-1$. 
         Let $\mathcal{J}(a,a-1) = \{ 0,1, \ldots, a-1 \}$.
          \begin{enumerate}
                 \item[\em{(i)}] If $a >1$ and $r \geq (a-1)(p+1)+p$, then   
                 $Q(a-1)$ is irreducible if and only if
                 $r_0 \not \in \mathcal{J}(a,a-1)$.
                 Furthermore, in this case $Q(a-1) \cong 
                 V_{p-1-a} \otimes D^{a}$.
                 \item[\em{(ii)}] If $r \geq a(p+1)+p$, then   $Q(a)$ is 
                 irreducible if and only if
                 $r_0 \not \in \mathcal{J}(a,a-1)$.
                 Furthermore, in this case $Q(a) \cong 
                 V_{a}$.
          \end{enumerate}
\end{lemma}
\begin{proof}
       Note that
       \begin{enumerate}
          \item[(i)]    If $a \geq 2$, then $a-1 \geq 1 =[a-(a-1)]$,
             assertion (i) follows from \Cref{Structure of Q i=[a-i]} 
           (resp.   \Cref{Structure of Q(i) i>[a-i]}) when $a=2$
            (resp. $a \geq 3$) as $W = 0$ if and only if $r_0 \not\in \mathcal{J}(a,a-1)$,
             and $Q(0) \cong V_{p-1-a} \otimes D^{a}$.
            \item[(ii)]  By \Cref{Structure of Q(i) if i = a}, we see that 
            the irreducibility of $Q(a)$ depends on $P(a)$. So we determine $P(a)$.
            By the exact sequence \eqref{Q i-1 and P i}, we  see that
            \[
                0 \rightarrow W'' \rightarrow P(a) \rightarrow Q(a-1) \rightarrow 0,
            \]
            where $W''$ is the cokernel of the map 
            $X_{r-(a-1)}^{(a)}/X_{r-(a-1)}^{(a+1)} \hookrightarrow V_{r}^{(a)}/V_{r}^{(a+1)}$.
            By the third part of  \Cref{reduction} (applied with $i= a-1$ and $j=a$),
            we have   $X_{r-(a-1)}^{(a)}/X_{r-(a-1)}^{(a+1)} = X_{r}^{(a)}/X_{r}^{(a+1)} $.
            Thus, by \Cref{singular quotient X_{r}}, we see that
            $W'' = V_{a}$ if  $r_{0} \not \in \mathcal{J}(a, a-1)$ 
            and $W'' = V_{r}^{(a)}/V_{r}^{(a+1)}$ otherwise. Combining this with
             part (i), we see that $P(a)$ has two (resp. at least four) JH factors
             if $r_{0} \not \in \mathcal{J}(a, a-1)$ (resp. $r_{0} \in \mathcal{J}(a, a-1)$). 
            Now the first assertion of (ii) follows from \Cref{Structure of Q(i) if i = a}. 
            If $r_{0} \not \in \mathcal{J}(a,a-1)$,
            then $W'' = V_{a}$ and $Q(a-1) \cong V_{p-1-a} \otimes D^a$, so 
            the second assertion of (ii) follows from  \Cref{Structure of Q(i) if i = a} (i). \qedhere
           \end{enumerate}
\end{proof}

We next consider the remaining cases of $Q(i)$, for $i = p-2$.
\begin{lemma}
        Let  $p \geq 3$,  $r \geq (p-1)(p+1)+p$, 
        $r \equiv a~\mathrm{mod}~(p-1)$ with $1 \leq a \leq p-1$.
        If $p-2 \neq a$, $a-1$, then $Q(p-2)$ is reducible.
\end{lemma}
\begin{proof}
         By hypothesis we have $[a-(p-2)] = a+1 \leq p-2$. 
         If $r_{0}  \in \mathcal{J}(a,p-2)$, then  by 
         Theorems~\ref{Structure of Q i=[a-i]} and \ref{Structure of Q(i) i>[a-i]},
         we see that $Q(p-2)$ is reducible since
          $W$, $Q(a) \neq 0$.
          If   $r_{0} \not \in \mathcal{J}(a,p-2)$, then again
          by  Theorems~\ref{Structure of Q i=[a-i]} and  
          \ref{Structure of Q(i) i>[a-i]}, 
          we have $Q(p-2) \cong Q(a)$.
          Since $a < [a-(p-2)] = a+1 \leq p-2$, by the fourth part of \eqref{interval J}, we have
         $\mathcal{J}(a,p-2) =\{ a-1\}^{c}$.  So $r_{0} \not \in \mathcal{J}(a,p-2)$
         implies $r_{0} =a-1$, so $r_0 \in \mathcal{J}(a,a-1)$. 
         But by Lemma~\ref{irred Q(a), Q(a-1)} (ii), we have 
         $Q(a)$ is not irreducible and this finishes the proof of the lemma.
\end{proof}
Since $V_{r} \subset V_{r}^{(1)} \subset \cdots \subset V_{r}^{(i+1)} 
\subset \cdots $ is a descending chain of $\Gamma$-modules, 
we have the JH factors of $ V_{r}/V_{r}^{(i+1)} =
\bigcup\limits_{j=0}^{i}$ JH factors of $V_{r}^{(j)}/V_{r}^{(j+1)}$.
Similarly,  the JH factors of $  X_{r-i}/X_{r-i}^{(i+1)}  =
\bigcup\limits_{j=0}^{i}$ JH factors of $ X_{r-i}^{(j)}/X_{r-i}^{(j+1)} $.
Thus, by the exact sequence \eqref{Q(i) exact sequence}, we see that
\begin{align}\label{irred. criteria Q(i)}
    Q(i) ~ \mathrm{is ~ irreducible} \Longleftrightarrow
   \sum_{j=0}^{i}  \lvert \{ \mathrm{JH ~ factors ~ of}~ 
   V_{r}^{(j)}/V_{r}^{(j+1)} \} \rvert -
  \lvert \{  \mathrm{JH ~ factors ~ of}~ 
   X_{r-i}^{(j)}/X_{r-i}^{(j+1)} \} \rvert
   \leq 1. 
\end{align}

Finally, we consider the case $Q(i)$, for $i = p-1$.
\begin{lemma}
        Let $p \geq 3$ and $(p-1)(p+1)+p \leq r \equiv a~\mathrm{mod}~(p-1)$
         with $1 \leq a < p-1$. Then $Q(p-1)$ is irreducible
         if and only if $r \equiv 1 ~\mathrm{mod}~(p-1)$ and $p \mid r$.
         Furthermore, in this case, $Q(p-1) \cong V_{p-2} \otimes D$.
\end{lemma}
\begin{proof}
      We consider the case $r \equiv a-1$ mod $p$ and 
      $r \not \equiv a-1$ mod $p$ separately.
      By \Cref{singular i=p-1}, we have 
      $X_{r-(p-1)}^{(p-1)}/X_{r-(p-1)}^{(p)} = X_{r-a}^{(p-1)}/X_{r-a}^{(p)}$.
      If $r \not \equiv a-1$ mod $p$, then $X_{r-a}^{(p-1)}/X_{r-a}^{(p)}
      =(0)$, by \Cref{singular i=a}, so $Q(p-1)$ is reducible by 
      \eqref{irred. criteria Q(i)}.   
      
      So assume $r \equiv a-1$ mod $p$.  Then 
      $X_{r-a}^{(p-1)}/X_{r-a}^{(p)} = V_{a}$, by \Cref{singular i=a}. 
      By the second part of   \Cref{reduction} (with $j=1$ and $i=p-1$), we see that
        $X_{r-(p-1)}^{(1)}/X_{r-(p-1)}^{(2)} =  X_{r-[a-1]}^{(1)}/X_{r-[a-1]}^{(2)}$.
       If $a \geq 2$, then  $  V_{a-2}  \otimes D  \cong X_{r-(a-1)}^{(1)}/X_{r-(a-1)}^{(2)} 
        \subsetneq V_{r}^{(1)}/V_{r}^{(2)}$,
        by \Cref{singular i= [a-i]} with $i = j = 1$ when $a = 2$  (resp. \Cref{singular i>r-i} 
        with $i = a-1$ and 
        $j = 1$ when $a>2$). So $Q(p-1)$ is reducible again by  \eqref{irred. criteria Q(i)}.
        We finally consider the case $a=1$.
        By \Cref{Structure of Q(i) if i = p - 1} (ii), we have
      \[
           0 \rightarrow V_{r}^{(1)}/V_{r}^{(2)} \rightarrow P(p-1) \rightarrow
           Q(p-1) \rightarrow 0.
      \]
       Since 
       $X_{r-1}^{(p-1)}/X_{r-1}^{(p)} \subseteq 
       X_{r-(p-2)}^{(p-1)}/X_{r-(p-2)}^{(p)} \subseteq 
       X_{r-(p-1)}^{(p-1)}/X_{r-(p-1)}^{(p)}= X_{r-1}^{(p-1)}/X_{r-1}^{(p)} \cong V_{1}$,
       by the exact sequence \eqref{Q i-1 and P i}
       and the exact sequence \eqref{exact sequence Vr} (with $m=p-1$), we have
       \[
           0 \rightarrow V_{p-2} \otimes D \rightarrow P(p-1) \rightarrow Q(p-2)
           \rightarrow 0.
       \]
       If $p=3$, then $Q(p-2)= Q(1)$. If $p >3$, then $2 = a+1 =[a-(p-2)]
       < p-2$, so by \Cref{Structure of Q(i) i>[a-i]}. we see that $Q(p-2)= Q(1)$
       as $W = 0$ since $a-1 =r_{0} \not \in \mathcal{J}(a,p-2) = \{ a-1\}^{c}$. 
       By \cite[Proposition 3.12 (iii)]{BG15}
       and the exact sequence \eqref{exact sequence Vr} (with $m=1$), we see that 
       $Q(1) \cong V_{r}^{(1)}/V_{r}^{(2)}$. Comparing the JH factors in the
       two exact sequences above, we obtain
       $Q(p-1) \cong V_{p-2} \otimes D$.
\end{proof}

Collecting the lemmas above, we obtain the following theorem on the irreducibility
of the $Q(i)$.

\begin{theorem}\label{irreducible Q(i)}
  Let $p \geq 3$, $1 \leq i \leq p-1$, $r \geq i(p+1)+p$, $r \equiv a \mod (p-1)$ with $a \in \{1, 2, \ldots, p-1\}$ and $r_0$ be
  the constant term in the base $p$-expansion of $r$.
  Then $Q(i)$ is irreducible if and only if 
  \begin{itemize}
     \item $i = a-1$ or $a$, and $r_0 \not\in \mathcal{J}(a,a-1) = \{0,1, \ldots, a-2, a-1\}$, or,
     \item $i = p-1$, $a = 1$ and $r_0 = 0$.
  \end{itemize}
\end{theorem}

\section{Structure of \texorpdfstring{$X_{r-p, \,r}$}{}}

In this final section, we determine the structure of $X_{r-p,\,r}$, the monomial submodule generated
by $X^{r-p}Y^p$.  
We exhibit an isomorphism between $X_{r-p,\,r}$ and
 $X_{s-1,\,s}$, for some $s$ depending on $r$, allowing us to use the results of \cite[\S 2, \S 3]{BG15} to  determine the structure of $X_{r-p,\,r}$. 
The proofs in this section use techniques very similar to
those used in \cite[\S 2, \S 3]{BG15}.

 First observe that $X^{r}= \begin{psmallmatrix} 1 & 1 \\ 0 & 1 
 \end{psmallmatrix} X^{r-p} Y^{p} - X^{r-p}Y^{p}$. Hence  $X_{r,\,r} 
 \subseteq X_{r-p,\,r}$ and $X_{r-p,\,r}$ is an $M$-module.
  The next lemma shows that this inclusion is strict if $p\nmid r$.
\begin{lemma}\label{BG Lemma 4.1}
	For any $p\geq 2$, if $p \nmid r$ and $r>p$, then 
	$X_{r,\,r} \subsetneq X_{r-p,\,r}$.
\end{lemma}
\begin{proof}
          Suppose not.
      Then $X_{r-p,\,r} = X_{r,\,r}$. By \Cref{Basis of X_r-i}, 
      we have 
       $\lbrace (kX+Y)^{r}, 
      X^{r} : k \in \mathbb{F}_{p} \rbrace$ 
      is a spanning set for $X_{r,\,r}$ over $\f$. Let 
	  \begin{align*}
	        X^{p}Y^{r-p} = AX^{r}+ \sum\limits_{k=0}^{p-1} c_{k} (kX+Y)^{r}.
	 \end{align*}
	Comparing the coefficients of $XY^{r-1}$ and $X^{p}Y^{r-p}$ on both sides, 
	we get 
	 \begin{align*}
            r \sum\limits_{k =0}^{p-1} c_{k}  k = 0  \quad \mathrm{and} \quad
	         \binom{r}{p} \sum\limits_{k=0}^{p-1} c_{k}  k^{p}=1.
	\end{align*}
	Since $p\nmid r$, we get $\sum\limits_{k=0}^{p-1}c_{k}  k=0$. Hence 
	$ 1= \binom{r}{p} \sum\limits_{k =0}^{p-1} 
	c_{k}  k^{p}  \equiv  \binom{r}{p} \sum\limits_{k=0}^{p-1} c_{k} k \equiv 0 $ 
	mod $p$. This is a contradiction. Therefore $X_{r,\,r} \subsetneq X_{r-p,\,r}$.
\end{proof}
\begin{lemma}\label{surjection}
	There is an $M$-linear surjection $\phi_{p}: X_{r-p,\,r-p} 
	\otimes V_{1} \rightarrow X_{r-p,\,r}$,  given by 
	$\phi_{p}(u \otimes v) = uv^{p}$, for all $u \in X_{r-p,\, r-p}$ and $v \in V_{1}$.
	In particular, $\dim X_{r-p,r} \leq 2p+2$. 
\end{lemma}
\begin{proof}
	  Clearly the map $\psi:V_{1} \rightarrow V_{p}$ defined by $(aX+bY) 
	  \mapsto (aX+bY)^{p}$ is $M$-linear. By \cite[(5.1)]{Glover}, we have an 
	  $M$-linear map $\varphi_{r-p,\, p}: V_{r-p} \otimes V_{p} \rightarrow V_{r}$, 
	  given by $\varphi_{r-p,\,p} (u \otimes v) = uv$, for $u \in V_{r-p}$ and 
	  $v \in V_{p}$. Let $\phi_{p} $ be the restriction of $\varphi_{r-p,\,r} 
	  \circ ( \mathrm{id} \otimes \psi ) $ to  the $M$-submodule $X_{r-p,\,r-p} 
	  \otimes V_{1}$. Since $\begin{psmallmatrix}
	  1 & 1 \\ 0 & 0   \end{psmallmatrix} X^{r-p} \otimes Y = X^{r-p} \otimes X$,
	   we see that $X^{r-p} \otimes Y$   generates 
	 $X_{r-p,\,r-p} \otimes V_{1}$  as an  $M$-module. So $\phi_{p}
	 (X_{r-p,\,r-p} \otimes V_{1})$ is an $M$-module generated by 
	 $ X^{r-p}Y^{p}$. But $X^{r-p} Y^{p}$ is a generator of $X_{r-p,\,r}$ as 
	 an $M$-module. Hence $\phi_{p}(X_{r-p,\, r-p} \otimes V_{1}) = X_{r-p,\,r} $.
\end{proof}

 In the course of determining  the structure of $X_{r-p,\,r}$ we often  define
  an $\f$-linear map $\eta: X_{s-1,\,s} \rightarrow X_{r-p,\,r}$, for some $s$.  The 
   next result gives a criteria  under which such a map $\eta$ is $M$-linear.  
 
\begin{lemma}
\label{M-linearity of map}
	  Let $s,s' \geq p$ be integers such that $s\equiv s' ~\mathrm{mod}~ (p-1)$.  Let  
	   $\eta: X_{s-1,\,s} \rightarrow X_{s'-p,\,s'}$  be an $\f$-linear map satisfying
	   \begin{enumerate}[label=\emph{(\roman*)}]
		     \item $\eta(X^{s})=X^{s'}$ and $\eta(X(kX+Y)^{s-1}) = X^{p}(kX+Y)^{s'-p}$,
		      $\forall \ k \in \f$, \item $\eta(Y^{s})=Y^{s'}$ and $\eta((X+kY)^{s-1}Y) = 
		      (X+kY)^{s'-p}Y^{p}$, $\forall \ k \in \f$.
	\end{enumerate}
	Then $\eta$ is an $M$-linear surjection.
\end{lemma}
\begin{proof}
	  We first claim that $\eta(\gamma \cdot X^{s-1}Y)= \gamma\cdot 
	  \eta(X^{s-1}Y), ~ \forall ~\gamma \in M$. For  $\gamma= \begin{psmallmatrix} a 
	  & b  \\ c & d \end{psmallmatrix} \in M$, we have
	   \begin{align}\label{4.1}
    	          \eta(\gamma \cdot X^{s-1}Y) 
    	          &= \eta((aX+cY)^{s-1}(bX+dY)) \nonumber \\
         	 & = b \eta(X(aX+cY)^{s-1})+ d\eta((aX+cY)^{s-1}Y)
	  \end{align}
	  Since $s\equiv s'$ mod $(p-1)$ it follows that
	  \begin{align*}
	        X(aX+cY)^{s-1} =  \begin{cases} a^{s'-p}X^{s},  \ &\mathrm{if} \ c =0, \\
	             c^{s'-p}X(ac^{-1}X+Y)^{s-1}, \  &\mathrm{if} \ c \neq 0.\end{cases} 
	   \end{align*}
	  This implies that 
	   \begin{align*}
	         \eta(X(aX+cY)^{s-1} )=  
	         \begin{cases} a^{s'-p}X^{s'},  \ &\mathrm{if} \ c =0, \\
	                                 c^{s'-p}X^{p}(ac^{-1}X+Y)^{s'-p}, \  &\mathrm{if} \ c \neq 0.
	         \end{cases} 
	   \end{align*}
        Hence $\eta(X(aX+cY)^{s-1} )= X^{p}(aX+cY)^{s'-p}$. A similar argument as
         above shows that  $\eta((aX+cY)^{s-1} Y)= (aX+cY)^{s'-p}Y^{p}$. Therefore
	   \begin{align*}
           	\eta(\gamma \cdot X^{s-1}Y)
           	~& \stackrel{\mathclap{\eqref{4.1}}}{=} ~b \eta(X(aX+cY)^{s-1})+
           	 d\eta((aX+cY)^{s-1}Y)\\
         	~ & = ~ bX^{p}(aX+cY)^{s'-p}+d(aX+cY)^{s'-p}Y^{p} \\
	        ~ &= ~(bX+dY)^{p}(aX+cY)^{s'-p}\\
	         ~&= ~ \gamma\cdot \eta(X^{s-1}Y).
	   \end{align*} 
	  Thus, for all $\gamma_{1},\gamma_{2} \in M$, we have 
	   \begin{align}\label{eq4.2}
        	       \eta((\gamma_{1} \gamma_{2}) \cdot X^{s-1}Y) 
        	       = (\gamma_{1} \gamma_{2}) \cdot \eta(X^{s-1}Y)
        	       = \gamma_{1}\cdot( \gamma_{2} \cdot \eta(X^{s-1}Y))
        	       = \gamma_{1}\cdot\eta( \gamma_{2} \cdot X^{s-1}Y).
	  \end{align}
	  Let $F(X,Y) \in X_{s-1,\,s}$. Since $X^{s-1}Y$ generates $X_{s-1,\,s}$ as 
	  an $\f[M]$-module, we can write $F(X,Y)$ as $ \sum_{i=1}^{n} a_{i}\gamma 
	  \cdot X^{s-1}Y$, for some $a_{i}\in \f$ and $\gamma_{i} \in M$.  For every
	$\gamma \in M$, we have
	\begin{align*}
       	\eta(\gamma \cdot F(X,Y)) ~ &= ~ ~\sum\limits_{i=1}^{n} a_{i} 
       	\eta(\gamma\gamma_{i}\cdot X^{s-1}Y) \quad
       	( \because \eta  ~ \mathrm{is} ~ \f \text{-} \mathrm{linear}  ) \\
        ~ &\stackrel{\mathclap{\eqref{eq4.2}}}{=} ~~\gamma \cdot \sum
        \limits_{i=1}^{n} a_{i}  \eta(\gamma_{i}\cdot X^{s-1}Y) =  
        \gamma \cdot \eta(F(X,Y)).
	\end{align*}
	This shows that $\eta$ is $M$-linear. Since $\eta(X^{s-1}Y) = X^{s'-p}Y^{p}$ is a
	 generator of $X_{s'-p,\,s'}$ as $M$-module, we get $\eta$ is onto.
\end{proof}
As a consequence we have the following result.
\begin{corollary}\label{M linear isomorhism}
        Let $s,s' \geq p$ be integers such that 
        $s,s' \equiv a ~\mathrm{mod}~ (p-1)$.
        If $\dim X_{s-1,\,s} = \dim X_{s'-p,\,s'} =2p+2 $, then
        $X_{s-1,\,s} \cong X_{s'-p,\,s'}$ as $M$-modules.
\end{corollary}
\begin{proof}
        By \Cref{Basis of X_r-i}, we have $\lbrace (kX+Y)^{s}, 
        X(lX+Y)^{s-1}, X^{s}, X^{s-1}Y : l,k \in \fstar \rbrace $ forms a basis of 
        $X_{s-1,\,s}$. Define an $\f$-linear map $\eta:
        X_{s-1,\,s} \rightarrow X_{s'-p,\,s'}$ by 
        $\eta( (kX+Y)^{s}) =  (kX+Y)^{s'}$,
        $\eta( X(l X+Y)^{s-1}) =  X^{p}(lX+Y)^{s'-p}$,
        $\eta(X^{s}) = X^{s'}$ and $\eta(X^{s-1}Y) = X^{s'-p}Y^{p}$.
        Observe that for $k \in \fstar$, we have
        \begin{align*}
               \eta((X+kY)^{s-1} Y) &= k^{s-1} \eta ( (k^{-1}X+Y)^{s}) - 
                  k^{s-2}\eta(X(k^{-1}X+Y)^{s-1}) \\
                  & = k^{-1}(X+kY)^{s'} - k^{-1}X^{p}(X+kY)^{s'-p}
                  = (X+kY)^{s'-p} Y^{p}.
        \end{align*}
        Thus, by \Cref{M-linearity of map}, we have $\eta$ is an $M$-linear 
        surjection. Since $\dim X_{s-1,s} = \dim X_{s'-p,s'}$, $\eta$ is
        an isomorphism.
\end{proof}
We now give a criterion which allows us to compare  
  the sum of $p$-adic digits of
 $r-1$ and $r-p$  in terms of the constant and linear terms in the
   base $p$-expansion of $r$. 
\begin{lemma}\label{Equality of sum of p-adic digits of (r-1),(r-p)}
     Let $p \leq r= r_{m}p^{m}+\cdots+r_{1}p+r_{0}$ be the base $p$-expansion of $r$. 
     If $r>p$, then
     \begin{enumerate}[label=\emph{(\roman*)}]
	         \item $\Sigma_{p}(r-1) = \Sigma_{p}(r-p)$ if and only if $r_{1} , r_{0} \neq 0$.
	         \item $\Sigma_{p}(r-1) < \Sigma_{p}(r-p)$ if and only if $r_{1}=0$ and 
	         $r_{0} \neq 0$. 
	         \item $\Sigma_{p}(r-1) > \Sigma_{p}(r-p)$ if and only if $r_{0} = 0$.
      \end{enumerate}
\end{lemma}	
\begin{proof}
     It is enough to  only prove the  \enquote*{if} part of the above assertions.  
     
     Case (i): If $r_{1}, r_{0}  \neq 0$, then $\Sigma_{p}(r-1)= (\sum_{i}  r_{i})-1 =
      \Sigma_{p}(r-p)$.   
      
     Case (ii): Assume  $r_{1}=0$ and $r_{0} \neq 0$.
     Let $i>1$ be smallest integer such that $r_{i} \neq 0$. Then 
     $r-1 = r_{m}p^{m}+\cdots+r_{i}p^{i}+ (r_{0}-1)$ and 
     $r-p = r_{m}p^{m}+ \cdots + r_{i+1}p^{i+1}+(r_{i}-1)p^{i}+(p-1)p^{i-1}+ \cdots +
     (p-1)p+r_{0}$. Hence  $\Sigma_{p}(r-1) =(\sum_{i} r_{i})-1 <
     \sum_{i} r_{i}+(i-1)(p-1)-1= \Sigma_{p}(r-p)$.   
     
     Case (iii): Assume $r_{0}=0$. 
     Let $i \geq 1$ be smallest positive integer such that $r_{i} \neq 0$. Then 
     $r-1 = r_{m}p^{m}+\cdots+r_{i+1}p^{i+1}+(r_{i}-1)p^{i}+(p-1)p^{i-1}+ \cdots + 
     (p-1)$ and  $r-p=r_{m}p^{m}+\cdots+r_{i+1}p^{i+1}+ (r_{i}-1)p^{i} + (p-1)p^{i-1}+
     \cdots +(p-1)p$. Hence $\Sigma_{p}(r-1)= \sum_{i}^{}r_{i}+i(p-1)-1 > 
     \sum_{i}^{}r_{i}+(i-1)(p-1)-1 =
     \Sigma_{p}(r-p)$. Therefore $\Sigma_{p}(r-1) > \Sigma_{p}(r-p)$.     
\end{proof}
The following proposition shows that  $X_{rp-p,\,rp} \cong X_{r-1,r}$ as $M$-modules.
\begin{proposition}
\label{p divides r}
       If $r \geq 2p$ and $p \mid r$, then the map sending $X^{r/p -1}Y$ to $X^{r-p}
       Y^{p}$ defines an $M$-linear  isomorphism between $X_{r/p-1,r/p}$ and 
       $X_{r-p,\,r}$.
\end{proposition}
\begin{proof}
	  The map $V_{r/p} \rightarrow V_{r}$ defined by  $F(X,Y) \mapsto F(X,Y)^{p}
	  =F(X^{p},Y^{p}) $ is
	   an injective  $M$-linear homomorphism. Restricting  this  map to
	    $X_{r/p-1,r/p}$ completes the proof.
\end{proof}
This result, combined with the results of \cite[$\S$2, $\S$3]{BG15}
determining the structure of $X_{s-1,\,s}$,
determines the
 structure of $X_{r-p,\,r}$ in the case  $p \mid r$.
%
  Hereafter, we will assume that 
 $p\nmid r$ which, by the above lemma, is equivalent to 
 $\Sp(r-p) \geq \Sp(r-1)$.  If $p \leq r < 2p$, then $0 \leq r-p \leq p-1$, so
the structure of $X_{r-p,\,r}$ can be treated by the methods of $\S3$
and for $r=2p$, we have   $X_{r-p,\,r} \cong V_{2}$.
So from now on we will also assume $r>2p$.
%

\subsection{The case \texorpdfstring{$r \equiv 1 ~\mathrm{mod} ~(p-1)$} {}}
 
In this section, we determine the structure of $X_{r-p,\,r}$ if 
$p \nmid r$ and $r \equiv 1$ mod $(p-1)$. 
Since $\Sp(r-p) \equiv r-p \equiv 0$ mod $(p-1)$, we 
have $\Sp(r-p)$ is a non-zero multiple of $p-1$. We first  consider the
case $\Sp(r-p) = p-1$.

\begin{lemma}\label{BG Lemma 3.2} 
        If $p\geq 2$, $2p < r\equiv 1 ~\mathrm{mod}~ (p-1)$ 
         and $\Sigma_{p}(r-p)=p-1$, then 
	    \begin{align*}
	        \sum\limits_{k=0}^{p-1}X^{p}(kX+Y)^{r-p} \equiv -X^{r} \quad \text{and} \quad  
	         \sum\limits_{k=0}^{p-1}(X+kY)^{r-p}Y^{p} \equiv -Y^{r} ~~\mathrm{mod } ~  p.
	 \end{align*} As a consequence, $\dim X_{r-p,\,r} \leq 2p$.
\end{lemma}
\begin{proof}
	Let $s=r-p+1$. Clearly $s \equiv 1$ mod $(p-1)$ and $\Sigma_{p}(s-1) =
	 \Sigma_{p}(r-p) =p-1$. Further, $s-1 = r-p \geq p$. Therefore,
	  by \cite[Lemma 3.2]{BG15}, we have 
	\begin{align*}
	   \sum\limits_{k=0}^{p-1}X (kX+Y)^{s-1} \equiv -X^{s} \ \text{ and  }\  \sum
	   \limits_{k=0}^{p-1}(X+kY)^{s-1}Y
	   \equiv -Y^{s} \text{ mod }  p. 
	\end{align*} 
	Multiplying the first and second equation above by $X^{p-1}$ and $Y^{p-1}$
	 respectively we obtain the lemma.
\end{proof}
\begin{proposition}\label{BG Proposition 3.3}
For $p\geq 2$, if $p \nmid r$, $2p<r\equiv 1 ~\mathrm{mod}~ (p-1)$ and 
$\Sigma_{p}(r-p)=p-1 $, then	 $X_{r-p,\,r} \cong X_{r-1,r} 
\cong V_{2p-1}$, as $M$-modules.
\end{proposition}
\begin{proof}
	First we  claim that $\Sigma_{p}(r-1)=p-1$. By hypothesis  we have $\Sigma_{p}
	(r-1) \equiv r-1 \equiv  0$ mod $(p-1)$. Also,
	$\Sigma_{p}(r-1) \leq \Sigma_{p}(r-p) =p-1$.  Therefore, 
	$\Sigma_{p}(r-1)=0$ or $p-1$. Since $r>1$, we have
	$\Sigma_{p}(r-1) \neq 0$, whence $\Sigma_{p}(r-1)=p-1$. Hence, 
	by \cite[Proposition 3.3]{BG15}, we have $X_{r-1,r} \cong V_{2p-1}$ and 
	$\lbrace X(kX+Y)^{r-1}, (X+lY)^{r-1}Y : k,l \in \f \rbrace$ is a basis of 
	$X_{r-1,r}$ over $\f$. Define an $\f$-linear map $\eta: X_{r-1,r} 
	\rightarrow X_{r-p,\,r}$, by $\eta(X(kX+Y)^{r-1})=X^{p}(kX+Y)^{r-p}$ and 
	$\eta((X+lY)^{r-1}Y)= (X+lY)^{r-p}Y^{p}$, for  $k,l \in \f$. Then
	\begin{alignat*}{3}
	       \eta(X^{r}) &= -\eta \Big ( \sum\limits_{k=0}^{p-1} X(kX+Y)^{r-1} \Big ) 
	       && \quad \quad \text{( by \cite[Lemma~3.2]{BG15}) }\\
	       &=  - \sum\limits_{k=0}^{p-1} X^{p}(kX+Y)^{r-p} \\
	       &= X^{r} && \quad \quad \text{(by \Cref{BG Lemma 3.2})}.
	\end{alignat*}
     Similarly  $\eta(Y^{r})=Y^{r}$. Therefore, $\eta$ satisfies hypotheses 
     of \Cref{M-linearity of map} with $s=s'=r$. So $\eta$ is $M$-linear and onto. 
      Further, by $M$-linearity  we have 
     $\eta(\sum_{i} a_{i} \gamma_{i}X^{r}) = \sum_{i} a_{i} \gamma_{i} \eta(X^{r}) =
      \sum_{i} a_{i} \gamma_{i}X^{r}$. Therefore, the restriction of  $\eta$  to $X_{r,\,r}$ is
       the identity map. By \cite[Proposition 3.3]{BG15}, and the fact that
       soc$(V_{2p-1})= V_{2p-1}^{(1)}$,
        we have soc$(X_{r-1,r}) = X_{r-1,r}^{(1)} = X_{r,r}^{(1)}$. So ker$(\eta) \cap
         \mathrm{soc}(X_{r-1,r}) = $ ker$(\eta) \cap X_{r,r}^{(1)} = (0)$,
         since $\eta$ is injective on $X_{r,r}$. Hence
          $\eta: X_{r-1,r} \rightarrow  X_{r-p,\,r}$ is an isomorphism. 
\end{proof}
We next consider the remaining case, i.e., $\Sp(r-p)>p-1$.
\begin{proposition}\label{remaining case r = 1 mod p-1}
	Let  $p \geq 3$, $p\nmid r$, $2p<r\equiv1 ~\mathrm{mod}~ (p-1)$ 
	and $\Sigma_{p}(r-p)> p-1$. Then
	 $X_{r-p,\,r} \cong X_{rp-1,rp}$ as 
	$M$-modules, and  we have a short exact sequence of $M$-modules
	$$
	       0 \rightarrow V_{1} \otimes D^{p-1} \rightarrow X_{r-p,\,r} \rightarrow 
	       V_{2p-1} \rightarrow 0 .
	$$
	Moreover, if $\Sp(r-p) = \Sp(r-1) > p-1$, then $X_{r-p,\,r} \cong X_{r-1,\,r}$.
\end{proposition}
\begin{proof}
     We claim  that $\dim X_{r-p,\,r} =2p+2$.
      We prove the  proposition assuming the claim.
     Note that $\Sigma_{p}(rp-1) = \Sigma_{p}((r-1)p+p-1) 
     = \Sigma_{p}(r-1)+p-1 >	 p-1$. 
	 Thus, by \cite[Proposition 3.13 (ii)]{BG15}, we have dim $X_{rp-1,\,rp}=2p+2$.
     Now the first two assertions  of the  proposition follow from \Cref{M linear isomorhism}	 
	  and \cite[Proposition 3.13 (ii)]{BG15}.   
	  For the last assertion, by \cite[Proposition 3.8]{BG15}, we have
     $\dim X_{r-1,\,r} =2p+2$ so $X_{r-1,\,r} \cong X_{r-p,\,r}$,
     again by \Cref{M linear isomorhism}. 
	  
	  We now prove the claim.  Note that
     \begin{align}
            \dim X_{r-p,\,r} & = \dim \left( \frac{X_{r-p,\,r}}{X_{r-p,\,r}^{(1)}} \right)
            +  \dim \left( \frac{X_{r-p,\,r}^{(1)}}{X_{r-p,\,r}^{(2)}} \right)
            + \dim X_{r-p,\,r}^{(2)} \nonumber \\
           & \geq  \dim \left( \frac{X_{r-p,\,r}}{X_{r-p,\,r}^{(1)}} \right) + 
            \dim \left( \frac{X_{r,\,r}^{(1)}}{X_{r,\,r}^{(2)}} \right)
            + \dim X_{r-p,\,r}^{(2)} .
     \end{align}
     We now compute each of the terms on the right hand side 
     of the inequality.
     Note that $X^{r-p}Y^{p} = (X^{r-p}Y^{p} - X^{r-1}Y) + X^{r-1}Y
     \in X_{r-1,\,r} +V_{r}^{(1)}$. Thus,  $X_{r-p,\,r}+ V_{r}^{(1)}
     = X_{r-1,\,r} +V_{r}^{(1)}$. By the second isomorphism theorem, we have
     \[
         \frac{X_{r-1,\,r}}{X_{r-1,\,r}^{(1)}} \cong
         \frac{X_{r-1,\,r} + V_{r}^{(1)}}{V_{r}^{(1)}} =
         \frac{X_{r-p,\,r} + V_{r}^{(1)}}{V_{r}^{(1)}} \cong
          \frac{X_{r-p,\,r}}{X_{r-p,\,r}^{(1)}}.
     \] 
      Thus,  
      $\dim X_{r-p,\,r}/X_{r-p,\,r}^{(1)}= \dim X_{r-1,\,r}/X_{r-1,\,r}^{(1)}
      =p+1$, by \Cref{Structure X(1)}.
      By \cite[Lemma 3.1 (i)]{BG15}, we have 
      $\dim X_{r,\,r}^{(1)}/X_{r,\,r}^{(2)}= p-1$.
      By \Cref{dimension formula for X_{r}}, we have 
      $\dim X_{r-p,\,r-p} = p+1$ and $X_{r-p,\,r-p}^{(1)} \neq (0)$. 
      Since $r-p \equiv p-1$ mod $(p-1)$, by  Lemma~\ref{star=double star}, 
      we have $X_{r-p,\,r-p}^{(1)} = X_{r-p,\,r-p}^{(2)}$.
      If $0 \neq G(X,Y) \in X_{r-p,\,r-p}^{(2)}$, then 
      $X^{p}G(X,Y)$ and $Y^{p}G(X,Y)$ are distinct elements
       of $X_{r-p,\,r}^{(2)}$, by \Cref{surjection}. So 
       $\dim X_{r-p,\,r}^{(2)} \geq 2$. 
     Putting all these facts together, the claim follows from \Cref{surjection}.
\end{proof}

	To summarize the structure of $X_{r-p,\,r}$ obtained so far, we record
	 the following theorem.
\begin{theorem}\label{Main theorem part 1 and 2}
	Let $p\geq 3$, $p \nmid r $ and $2p< r \equiv 1 ~\mathrm{mod}~ (p-1)$. 
	\begin{enumerate}[label=\emph{(\roman*)}]
		\item If $\Sigma_{p}(r-p)=p-1$, then $X_{r-p,\,r} \cong V_{2p-1}$ as an 
		$M$-module.
		\item If $\Sigma_{p}(r-p) > p-1$, then
		we have a short exact sequence of $M$-modules
		$$
	        	0 \rightarrow V_{1} \otimes D^{p-1} \rightarrow X_{r-p,\,r}
	        	 \rightarrow V_{2p-1} \rightarrow 0. 
		$$
	\end{enumerate}
\end{theorem}
\subsection{The case \texorpdfstring{$r \not \equiv 1 \mod (p-1)$}{}}

In this section, we determine the structure of $X_{r-p,\,r}$
when $p \nmid r$ and $r \equiv a $ mod $(p-1)$ with  $2 \leq a \leq p-1$.
Thus, $\Sp(r-p) \equiv r-p \equiv a-1$ mod $(p-1)$.
We begin  by considering the case   $\Sigma_{p}(r-p) = a-1$.
For simplicity, below we sometimes denote $r-p$ by $r''$.   
\begin{proposition}\label{BG Lemma 4.5}
	Let $p\geq3, p \nmid r $ and let 
	 $2p < r \equiv a ~\mathrm{mod}~ (p-1)$, with 
	 $2 \leq a \leq  p-1 $. If $\Sigma_{p}(r-p)=\Sigma_{p}(r-1) =a-1$, 
	 then  $X_{r-p,\,r} \cong X_{r-1,\,r} \cong 
	 V_{a-2} \otimes D \oplus V_{a}$ as  $M$-modules. 
\end{proposition}
\begin{proof}
    By  \Cref{dimension formula for X_{r}} (i), we have dim
    $X_{r'',r''}=a$ and $X_{r'',r''} \cong V_{a-1}$.
	By \cite[(5.2)]{Glover},  we have
    $ X_{r'',r''} \otimes V_{1}  \cong  V_{a-2} \otimes D \oplus V_{a}$.	
	Thus, by \Cref{surjection}, we have an $M$-linear surjection 
	\begin{align*}
	V_{a-2} \otimes D \oplus V_{a}  \cong X_{r'',r''} \otimes V_{1}  
	\xrightarrow{\phi_{p}} X_{r-p,\,r}.
	\end{align*} 
	Since  $\Sp(r)=a$, we have $X_{r,r} \cong V_{a}$ and it follows
	from \Cref{BG Lemma 4.1} that $\phi_{p}$ is an isomorphism.
    By \cite[Lemma 4.5]{BG15}, we have $X_{r-1,\,r}
    \cong V_{a-2} \otimes D \oplus V_{a}$ as $M$-modules
    so 	$X_{r-1,\,r} \cong X_{r-p,\,r}$.
\end{proof}
Before we treat the case $\Sp(r-p) > a-1$, or equivalently
$\Sp(r-p) \geq p+a-2  > p-1$, we need a few preparatory results. 
In the next two lemmas we show that  $V_{p-a+1}\otimes D^{a-1}$ 
and   $V_{a-2} \otimes D$ are  JH factors of $X_{r-p,\,r}$
whenever $\Sigma_{p}(r-p)>p-1$. Observe that $\phi_{p}$ is $M$-linear and
 $X_{r'',r''}^{(1)} \otimes V_{1}$ is singular, so we have $\phi_{p}(X_{r'',r''}^{(1)}
  \otimes V_{1}) \subseteq X_{r-p,\,r}^{(1)}$. 
\begin{lemma}\label{JH1}
    	Let $  p \geq 3, p \nmid r $ and 
    	$2p < r \equiv a ~\mathrm{mod}~(p-1)$ with 
    	$2 \leq a \leq p-1$. If  $\Sigma_{p}(r-p) > p-1$, then 
    	$X_{r-p,\,r}^{(1)}$ contains $V_{p-a+1} \otimes D^{a-1}$ 
    	as an $M$-module.
\end{lemma}
\begin{proof}
          By  \Cref{dimension formula for X_{r}}, we have dim $X_{r'',r''}=p+1$
          and $X_{r'',r''}^{(1)} \cong V_{p-a}  \otimes D^{a-1}$. 
          For $F(X,Y) \in V_{m}$, define $\delta_{m}(F)=F_{X}
         \otimes X +F_{Y}\otimes Y \in V_{m-1} \otimes V_{1}$, where $F_{X},F_{Y}$ are
         the  partial derivatives of $F$ w.r.t. $X,Y$, respectively. 
          It is shown on \cite[p. 449]{Glover}, that $\frac{1}{p-a+1}
          \delta_{p-a+1}$ ($\bar{\phi}$ in the notation of \cite{Glover})
          induces an $M$-linear injection $V_{p-a+1} \otimes D^{a-1} \hookrightarrow 
          (V_{p-a}  \otimes  D^{a-1})  \otimes V_{1} $.  Let $F$ be the inverse
           image of $X^{p-a}$ under the isomorphism 
	       $X_{r'',r''}^{(1)} \cong V_{p-a} \otimes D^{a-1}$. Then  
	       the composition of the maps
	  \begin{alignat*}{4}
	      V_{p-a+1} \otimes D^{a-1} & \hookrightarrow (V_{p-a}  \otimes D^{a-1}) 
	      \otimes V_{1}  
	      &&\stackrel{\simeq}{\longrightarrow} X_{r'',r''}^{(1)} \otimes V_{1}  
	      &&&\stackrel{\phi_{p}}{\longrightarrow} X_{r-p,\,r}^{(1)} \\
	     \ \ \ \ \ \ \ \ X^{p-a+1}& \mapsto  \ \ \ \ \ \ \ X^{p-a} \otimes X  \ \ \   
	      &&\longmapsto \ \ \ F\otimes X  \ \ \ \ &&&\longmapsto \ FX^{p}
	      \neq 0
	  \end{alignat*}
	  is a non-zero $M$-linear map. Since $ V_{p-a+1} \otimes D^{a-1}$ is an
	   irreducible $\Gamma$-module, the composition is injective. Hence 
	    $X_{r-p,\,r}^{(1)}$ contains $ V_{p-a+1} \otimes D^{a-1}$ as an $M$-module.
\end{proof}
\begin{lemma}\label{JH2}
	Let $p \geq 3$, $p \nmid r $ and 
    	$2p < r \equiv a ~\mathrm{mod}~(p-1)$ with $2 \leq a \leq p-1$. 
	Then $V_{a-2} \otimes D$ is a JH factor of $X_{r-p,\,r}$.
\end{lemma}
\begin{proof}
      We first treat the case $a=2$.
	  Consider the polynomial $F(X,Y) =  X^{p}Y^{r-p} - X^{r-p}Y^{p} 
	  \in X_{r-p,\,r}$.   By  \Cref{divisibility1}, we have 
	   $F(X,Y) \in V_{r}^{(1)}$. We claim that the following map is non-zero,
	   hence surjective
	     \[
	          \frac{X_{r-p,\,r}^{(1)}}{X_{r-p,\,r}^{(2)}} \hookrightarrow 
	          \frac{V_{r}^{(1)}}{V_{r}^{(2)}} 
	           \twoheadrightarrow V_{0} \otimes D,
	     \]
	  where the  rightmost map is induced by the quotient map of   
	  the exact sequence \eqref{exact sequence Vr}.
	  Since $r > 2p$, by \Cref{breuil map quotient}, we have
	  $F(X,Y) \equiv r \theta X^{r-(p+1)-(p-1)}Y^{p-1}$
	   mod $V_{r}^{(2)}$.
	  By \Cref{Breuil map}, we have the image of $F(X,Y)$ in 
	  $V_{r}^{(1)}/V_{r}^{(2)}
	  \twoheadrightarrow V_{0} \otimes D$ is non-zero, as $p \nmid r$. 
	  This proves the lemma if $a=2$.
	     
    Assume $3 \leq a \leq p-1$. 
    We claim that $ X_{r-p,\,r}^{(1)}/X_{r-p,\,r}^{(2)} \neq 0$.    
    Consider the polynomial 
	 \begin{align*}
	        G(X,Y) & := X^{r-p} Y^{p} + (a-1)^{-1}\sum\limits_{k=1}^{p-1} 
	        k^{2-a}  X^{p}(kX+Y)^{r-p} \\
	        & \stackrel{\eqref{sum fp}}{ \equiv}
	        X^{r-p}  Y^{p} - (a-1)^{-1} \sum\limits_{\substack { 0 \leq j \leq 
	         r-p\\ j \equiv 1 \ \mathrm{mod} \ (p-1)}} \binom{r-p}{j} X^{r-j}Y^{j}
	         \mod p.
	 \end{align*}
	 Clearly $G(X,Y) \in X_{r-p,\,r}$.
	The coefficient of $X^{r-1}Y$  in $G(X,Y)$ is
	 equal to $-(r-p)/(a-1)  \not \equiv 0$ mod $p$, so by \Cref{divisibility1},
	 we have $G(X,Y) \not \in V_{r}^{(2)}$.  Clearly $X,Y \mid G(X,Y)$.
	 Further, by \Cref{binomial sum}, applied with $m=0$, we have
	 \[
	    1 - (a-1)^{-1} \sum\limits_{\substack { 0 \leq j \leq 
	         r-p\\ j \equiv 1 ~ \mathrm{mod} ~ (p-1)}} \binom{r-p}{j}
	     \equiv 1 - (a-1)^{-1} \binom{a-1}{1} \equiv 0 \mod p.    
	 \]
	  Thus, by \Cref{divisibility1}, we have $G(X,Y) \in V_{r}^{(1)}$
	  and $0 \neq G(X,Y) \in X_{r-p,\,r}^{(1)}/X_{r-p,\,r}^{(2)}$.
	  Since $3 \leq a \leq p-1$, the exact sequence \eqref{exact sequence Vr}
	  doesn't split for $m=1$. Hence $V_{a-2} \otimes D \hookrightarrow
	  X_{r-p,\,r}^{(1)}/X_{r-p,\,r}^{(2)}$.
\end{proof}
\begin{proposition}\label{4 JH factors}
  Let $p\geq 3, \ r>2p, \ p\nmid r$ and 
  $r \equiv a ~\mathrm{mod}~ (p-1)$ with $2 \leq a \leq p-1$. 
  If $\Sigma_{p}(r-1)> p-1$, then  $X_{r-p,r} \cong X_{r-1,\,r} \cong
  X_{rp-1,\,rp}$
  as $M$-modules
   and  there is an exact sequence of $M$-modules
   \[
	 0 \rightarrow V_{p-a-1} \otimes D^{a} \oplus V_{p-a+1} \otimes D^{a-1} 
	  \rightarrow X_{r-p,\,r} 
	 \rightarrow V_{a-2} \otimes D \oplus V_{a} \rightarrow 0 .
    \]
\end{proposition}
\begin{proof}
	By \Cref{dimension formula for X_{r}} (ii), we have dim $X_{r,r} =p+1$ and 
	 $V_{a}$, $V_{p-a-1} \otimes D^{a}$ are JH factors of $X_{r,\,r}$, 
	hence  of $X_{r-p,\,r}$. By \Cref{JH2}, we have $V_{a-2} 
	\otimes D$ is a JH factor of $X_{r-p,\,r}$.
        By Lemma~\ref{Equality of sum of p-adic digits of (r-1),(r-p)}, $\Sigma_{p}(r-p) \geq \Sigma_{p}(r-1) \geq p $. Thus, by 
	 Lemma~\ref{JH1},  $V_{p-a+1} \otimes D^{a-1}$  is a JH factor of $X_{r-p,\,r}$.
	 Adding the dimensions of these JH factors we get dim $X_{r-p,\,r} \geq 2p+2$. By
	\Cref{surjection}, we get dim $X_{r-p,\,r} = 2p+2$. 
    Since $p \nmid r $ and $\Sp(r-1)  > p-1$, by 
    \Cref{dimension formula for X_{r-1}}, we have 
    $\dim X_{r-1,\,r} =2p+2$. Now the isomorphism $X_{r-p,r} \cong X_{r-1}$
    follows from \Cref{M linear isomorhism}
    and hence the  exact sequence above
	follows from  \cite[Proposition 4.9 (iii)]{BG15}.	
	Finally, by \cite[Proposition 4.9 (iii)]{BG15}, we also have
	$\dim X_{rp-1,\,rp} =2p+2$,  so $X_{r-p,\,r} \cong  X_{rp-1,\,rp}  $, again
	by \Cref{M linear isomorhism}. 
\end{proof}

We next treat the last  case, i.e.,  $\Sigma_{p}(r-1) \leq  p-1$ and
  $\Sigma_{p}(r-p)>p-1$. 

\begin{proposition} \label{BG Lemma 4.8}
	Let $p\geq 3$, $p \nmid r$ and let 
	$r\equiv a ~\mathrm{mod}~ (p-1)$  with $2 \leq a \leq p-1$. 
	If $\Sp(r-p)>p-1 > \Sigma_{p}(r-1)=a-1$,
	 then $X_{r-p,\,r} \cong X_{rp-1,\,rp}$ as $M$-modules and
	 we have an exact sequence of $M$-modules
	 \[
	     0 \rightarrow V_{p-a+1} \otimes D^{a-1} \rightarrow X_{r-p,\,r} \rightarrow 
	     V_{a-2} \otimes D \oplus V_{a} \rightarrow 0.
	 \]
\end{proposition}
\begin{proof}
         Clearly the map $F(X,Y) \mapsto F(X,Y)^{p}$ induces an 
          $M$-linear isomorphism $\eta': X_{r,\,r} \rightarrow 
         X_{rp,\,rp}$.   Let $\eta: X_{rp,\,rp} \rightarrow X_{r,\,r}$ be the 
         inverse of $\eta'$. 
         We show  that $\eta$ is the restriction of  an $M$-linear
         surjection $X_{rp-1,rp} \rightarrow X_{r-p,\,r}$ which we  denote  
         again by $\eta$.   Let $S= \lbrace
           X^{rp-1}Y, X(kX+Y)^{rp-1}: k \in \f \rbrace \subset X_{rp-1,rp}$ and 
         $W \subset X_{rp-1,rp}$ be the vector space spanned by $S$. Let $W' = 
         X_{rp,\,rp}+W$.  By \Cref{Basis of X_r-i}, we have $W' =X_{rp-1,\,rp}$.
          Note that  $\Sp(rp) = \Sp(r)=a$.           
          By \Cref{dimension formula for X_{r}} and 
          \Cref{dimension formula for X_{r-1}}, we have $\dim X_{rp,\,rp}=a+1$ and dim
           $X_{rp-1,rp}=a+p+2$. So $\dim W \geq \dim W' - \dim X_{rp,rp} =p+1$. 
           Since Card$(S)  \leq p+1$, we
            have dim  $W =p+1$ and 
          $X_{rp-1,rp} = X_{rp,\,rp} \oplus W$. Extend $\eta$ to an $\f$-linear map 
          $\eta: X_{rp-1,rp} \rightarrow X_{r-p,\,r}$ by setting $\eta(X(kX+Y)^{rp-1})=
          X^{p} (kX+Y)^{r-p}$ and $\eta(X^{rp-1}Y)=X^{r-p}Y^{p}$.
         Also,  for $l \in \f^{\ast}$, we have
          \begin{align*}
               \eta( (X+lY)^{rp-1}Y) 
               & =\eta(l^{-1} (X+lY)^{rp} - l^{rp-2} X(l^{-1}X+Y)^{rp-1}) \\
             &= l^{-1} (X+lY)^{r} - l^{rp-2} X^{p} (l^{-1}X+Y)^{r-p}  \\
                                              &=l^{-1} (X+lY)^{r} - l^{-1} X^{p} (X+lY)^{r-p}  \\
                                              &= (X+lY)^{r-p}Y^{p}.
         \end{align*}
         Therefore the extension $\eta$ satisfies the hypotheses 
         of \Cref{M-linearity of map} with
          $s=rp$ and $s'=r$.  Hence $\eta$ is an $M$-linear surjection. 
         
          We  now show that $\eta$ is an isomorphism by showing $\dim X_{rp-1,\,rp}=
         \dim X_{r-p,\,r}$.                     
         Since $\Sigma_{p}(r-p)>p-1$,  by Lemmas 
         \ref{JH1} and \ref{JH2},  we have    $V_{p-a+1} \otimes D^{a-1}$
         and $V_{a-2} \otimes D$  are JH  factors of $X_{r-p,\,r-p}$.  
         Further, as $r \equiv a $ mod $(p-1)$ by   \cite[(4.5)]{Glover}, 
         we have  $V_{a}$ is JH factor for $X_{r,r}$. So $V_{a}$ is  also a 
         JH factor of    $X_{r-p,\,r}$. Adding the dimensions of the above 
         JH factors we get dim $X_{r-p,\,r}  \geq a+p+2 =
         \mathrm{dim} \ X_{rp-1,\,rp}$. So $\eta$ is an isomorphism. 
        Finally, the exact sequence follows from  \cite[Proposition 4.9 (ii)]{BG15}.
 \end{proof}

We collect all the results related to the structure of $X_{r-p,\,r}$ proved in this
 subsection in the following theorem.
\begin{theorem}\label{Main theorem part 4}
	Let $p\geq 3$, $p\nmid r$, 
	$2p < r \equiv a ~\mathrm{mod}~ (p-1)$ with $2 \leq a \leq p-1$.
	 Then 
	\begin{enumerate}[label=\emph{(\roman*)}]	
		\item If $\Sigma_{p}(r-1) =\Sigma_{p}(r-p) = a-1$, then $X_{r-p,\,r} 
		\cong V_{a-2} \otimes D
		      \oplus V_{a}$.
		\item If $\Sigma_{p}(r-1)=a-1$ and $\Sigma_{p}(r-p) >a-1$, then 
		we have the	following exact sequence of $M$-modules
		\begin{align*}
		    0 \rightarrow V_{p-a+1} \otimes D^{a-1} \rightarrow X_{r-p,\,r} \rightarrow
		     V_{a-2} \otimes D
		    \oplus V_{a} \rightarrow 0 .
		\end{align*}
		\item If $\Sigma_{p}(r-1)>a-1$, then  we have the following
		 exact sequence  of 
		$M$-modules
		\begin{align*}
		0 \rightarrow V_{p-a-1} \otimes D^{a} \oplus V_{p-a+1} \otimes D^{a-1} 
		 \rightarrow X_{r-p,\,r} 
		\rightarrow V_{a-2} \otimes D \oplus V_{a} \rightarrow 0 .
		\end{align*}
	\end{enumerate}   
\end{theorem}

\subsubsection*{Acknowledgements}
	We would like to thank Prof. Doty for useful conversations. The second author 
	wishes to thank his advisor (the first author) for his constant encouragement 
	and  the School of Mathematics, TIFR  for its support. We acknowledge support 
	of the Department of Atomic Energy under project number 
	12-R\&D-TFR-5.01-0500.

\quad \\
\noindent {School of Mathematics, Tata Institute of Fundamental Research, Homi Bhabha Road, Mumbai-5, India}

\noindent{\tt email: eghate@math.tifr.res.in, ravithej@math.tifr.res.in}               
\end{document}